\documentclass[10pt]{articleHJ}

\usepackage[centertags]{amsmath}
\usepackage{amsfonts,amssymb,amsthm}
\usepackage{caption,graphicx}
\usepackage[usenames]{xcolor}
\usepackage[text={6.0in,8.6in},centering,letterpaper]{geometry}
\usepackage[numbers,comma,square,sort&compress]{natbib}

\newlength{\defbaselineskip}
\newcommand{\setlinespacing}[1]%
           {\setlength{\baselineskip}{#1 \defbaselineskip}}

\newcommand{\singlespacing}{\setlength{\baselineskip}{\defbaselineskip}}

\def\Re{\mathop\mathrm{Re}\nolimits}			
\def\Im{\mathop\mathrm{Im}\nolimits}			
\newcommand{\rmd}{\mathrm{d}}						
\newcommand{\Real}{\mathbb{R}}							
\newcommand{\Complex}{\mathbb{C}}							
\newcommand{\Wholes}{\mathbb{Z}}							
\newcommand{\abs}[1]{\left\vert#1\right\vert}			
\newcommand{\norm}[1]{\left\Vert#1\right\Vert}		
\newcommand{\sref}[1]{(\ref{#1})}                       

\newcommand{\Z}{\mathbb{Z}}

\newcommand{\be}{\begin{equation}}
\newcommand{\ee}{\end{equation}}
\newcommand{\bea}{\begin{eqnarray}}
\newcommand{\eea}{\end{eqnarray}}
\newcommand{\ba}{\begin{array}}
\newcommand{\ea}{\end{array}}

\newtheorem{thm}{Theorem}[section]
\newtheorem{cor}[thm]{Corollary}
\newtheorem{lem}[thm]{Lemma}
\newtheorem{prop}[thm]{Proposition}

 {\begin{trivlist}\item[]\textbf{Acknowledgments }}{\end{trivlist}}

\begin{document}

\begin{frontmatter}
\title{Entire Solutions \\ for Bistable Lattice Differential Equations \\ with Obstacles}
\journal{...}
\author[OL]{A. Hoffman},
\author[LD]{H. J. Hupkes\corauthref{coraut}},
\corauth[coraut]{Corresponding author. }
\author[KU]{E. S. Van Vleck}
\address[OL]{
  Franklin W. Olin College of Engineering   \\
  1000 Olin Way; Needham, MA 02492; USA \\ Email:  {\normalfont{\texttt{aaron.hoffman@olin.edu}}}
}
\address[LD]{
  Mathematisch Instituut - Universiteit Leiden \\
  P.O. Box 9512; 2300 RA Leiden; The Netherlands \\ Email:  {\normalfont{\texttt{hhupkes@math.leidenuniv.nl}}}
}
\address[KU]{
  Department of Mathematics - University of Kansas \\
  1460 Jayhawk Blvd; Lawrence, KS 66045; USA \\
  Email: {\normalfont{\texttt{erikvv@ku.edu}}}
}
\date{\today}

\begin{abstract}
\singlespacing
We consider scalar lattice differential equations posed on square lattices
in two space dimensions. Under certain natural conditions
we show that wave-like solutions exist when obstacles
(characterized by ``holes'') are present in the
lattice. Our work generalizes to the discrete spatial setting
the results obtained in \cite{BHM}
for the propagation of waves around obstacles
in continuous spatial domains.
The analysis hinges upon the development of sub and super-solutions
for a class of
discrete bistable reaction-diffusion problems
and on a generalization of a classical result due to Aronson and Weinberger
that concerns the spreading of localized disturbances.
\end{abstract}



\end{frontmatter}


\numberwithin{equation}{section}
\renewcommand{\theequation}{\thesection.\arabic{equation}}

\section{Introduction}
\label{sec:int}


Consider a subset $\Lambda \subset \Wholes^2$
that results after removing a finite (possibly zero)
number of points from the standard square grid.
Write $\mathrm{int} ( \Lambda) \subset \Lambda$ for the collection
of grid points for which all four nearest neighbours
are also included in $\Lambda$ and write $\partial \Lambda
= \Lambda \setminus \mathrm{int}(\Lambda)$
for the remaining points, which can be interpreted as the boundary
of $\Lambda$. Fix a detuning parameter
$0 < a < 1$ and consider
the bistable nonlinearity $g(u) = u ( 1 - u) ( u - a)$.

In this paper
we are interested in
the scalar lattice differential equation (LDE)
\begin{equation}
\label{origLDE}
\dot u_{ij}(t) = u_{i+1,j}(t) +u_{i,j+1}(t) +u_{i-1,j}(t)+u_{i,j-1}(t) - 4 u_{ij}(t)
 + g\big(u_{ij}(t) \big),
 \qquad
  (i,j) \in \mathrm{int}(\Lambda),
\end{equation}
augmented by rules on $\partial \Lambda$
that ensure that the homogeneous states $u \equiv 0$, $u \equiv a$
and $u \equiv 1$ are equilibria for the full problem.
We encourage the reader to think of this system
as the discrete analogue of the scalar Nagumo PDE
\begin{equation}
\label{eq:int:pde:obstructed}
\partial_t u(x,y, t) = \partial_{xx} u(x,y, t) + \partial_{yy} u(x,y,t) + g\big(u(x,y,t)\big), \qquad (x,y) \in \Omega,
\end{equation}
posed on an exterior domain $\Omega = \Real^2 \setminus K$ for
some compact (possibly empty) obstacle $K$, with Neumann boundary conditions on $\partial \Omega$.

In the unobstructed case $\Lambda = \Wholes^2$,
it is known that \sref{origLDE} admits
planar travelling wave solutions
\begin{equation}
\label{eq:int:lde:trvWaveAnsatz}
u_{ij}(t) = \Phi(i \sigma_h + j \sigma_v + ct), \qquad \Phi(-\infty) = 0, \qquad \Phi( \infty) = 1,
\end{equation}
which propagate with speed $-c$ in the direction $(\sigma_h, \sigma_v)$
with a fixed monotone wave profile $\Phi$. Our goal here
is to show that these unobstructed waves persist in an appropriate sense,
after removing points from $\Wholes^2$. In particular,
we give conditions under which \sref{origLDE} with $\Lambda \neq \Wholes^2$
admits so-called {\it entire asymptotic plane-wave solutions}.
Such solutions are defined for all $t \in \Real$
and satisfy the temporal limits 
\begin{equation}
\label{eq:int:limAsympWaves}
 \lim_{|t| \to \infty} \sup_{(i,j) \in \Lambda}
   \left| u_{ij}(t) - \Phi(i \sigma_h + j \sigma_v + ct) \right| = 0.
\end{equation}
As such, the present work can be seen as a direct (partial) generalization
of the results concerning the obstructed PDE \sref{eq:int:pde:obstructed}
that were obtained by Berestycki, Hamel and Matano in the landmark paper \cite{BHM}.

Viewed from a dynamical system perspective, the limits \sref{eq:int:limAsympWaves}
suggest that $u(t)$ can be seen as a homoclinic excursion to and from
an unobstructed travelling wave, with large transients when the wave front meets
the obstacle. In particular, the wave-like solutions
constructed in the present paper can be seen in the wider context of
so-called transition fronts, which can roughly be defined as global in time solutions
for which the width of the
interfacial region connecting the limiting values is uniformly bounded in time.
This nomenclature allows classical travelling waves
and more general wave-like solutions to be discussed
within a common framework.
Besides the work \cite{BHM} mentioned above, results concerning the existence of
transition fronts have appeared in a wide range of settings,
including diffusive random media \cite{Shen04},
reaction-diffusion-advection PDEs \cite{BH07}
and time-dependent reaction-diffusion PDEs \cite{BH12}.
As far as we know however, the present paper is the first
that features transition fronts for LDEs in
higher space dimensions.



\subsection*{Reaction-Diffusion Problems}
The discrete and continuous Nagumo systems \sref{origLDE}-\sref{eq:int:pde:obstructed}
can both be seen
as phenomenological models in which two stable equilibria compete for dominance in a
spatial domain. In modelling contexts one often thinks of these equilibria as
representing material phases or
biological species.
The competition is caused by the
opposing dynamical effects of the reaction and diffusion
terms present in \sref{origLDE}-\sref{eq:int:pde:obstructed}.
Indeed, both equations feature a thresholding nonlinearity that promotes high frequencies
together with a diffusion operator that attenuates them.
The main questions center on how the long-term
behavior of these reaction-diffusion systems is impacted by the balance between these dynamical
features.

In the past, the PDE \sref{eq:int:pde:obstructed} 
has served as a prototype system for the understanding
of many basic concepts at the heart of dynamical systems theory, including the
existence and stability of planar travelling waves and the study of obstacles.
Multi-component versions of \sref{eq:int:pde:obstructed} such as the Gray-Scott model \cite{GRAYSCOTT1983}
play an important role in the formation of patterns,
generating spatially periodic structures from equilibria that destabilize through Turing
bifurcations.

More recently, spatially discrete systems such as \sref{origLDE}
have started to attract an increasing amount of attention.
Dramatic increases in computer power
have made these systems considerably more accessible and
have clearly demonstrated
that discrete models can capture dynamical behaviour
that their continuous counterparts can not.
Understanding the causes and consequences
of these differences is a major theme
that continues to drive researchers in this area,
motivated by both mathematical and practical considerations.

On the mathematical side,
the shift from the PDE \sref{eq:int:pde:obstructed} with $\Omega = \Real^2$
to the LDE \sref{origLDE} with $\Lambda = \Wholes^2$
breaks two important symmetries, 
namely the translational invariance and
spatial isotropy of $\Real^2$. In the sequel two
fundamental consequences of these broken symmetries are encountered.
In addition,
the discrete Laplacian in \sref{origLDE} is a bounded operator
while the continuous Laplacian in \sref{eq:int:pde:obstructed}
is unbounded. This requires the use of delicate techniques
when comparing the spectral and dynamical properties of these two operators,
as discussed in \cite{BatesInfRange}.

On the practical side,
many physical and biological systems have a discrete spatial structure.
It is hence important to develop
mathematical modelling tools that
can incorporate such structures effectively.
Indeed, genuinely discrete phenomena
such as phase transitions in
Ising models \cite{BatesDiscConv}, crystal growth in materials \cite{CAHN},
propagation of action potentials in myelinated nerve fibers \cite{Bell1984}
and phase mixing in martensitic structures \cite{VAIN2009}
have all been modelled using equations similar to \sref{origLDE}. We
expect this list to get longer over time as the available mathematical techniques
for discrete systems are improved.

Finally, we remark that
the LDE \sref{origLDE} arises
as the standard finite difference spatial discretization
of the PDE \sref{eq:int:pde:obstructed}.
As such, the study of \sref{origLDE} and its variants
provides information on the
impact of discretization schemes and is therefore of
importance in the field of numerical analysis.
In order to faithfully
replicate the PDE behavior in numerical simulations,
the vested interest here is to actually suppress
as best as possible the novel features appearing in discrete systems.
Such issues are explored in \cite{HJH2011,BAROXPEL2005,DMIKEVYOS2005}.
In this paper we take the neutral perspective that the discrete model
is set and we seek to understand its behavior in its own right.

\subsection*{Existence of Waves}
In the unobstructed PDE \sref{eq:int:pde:obstructed} with $\Omega = \Real^2$,
the balance between diffusion and reaction is
resolved through the formation of planar travelling
waves
\begin{equation}
 u(x,y, t) = \Phi(x \sigma_h + y \sigma_v +ct); \qquad \Phi(-\infty) = 0,
   \qquad \Phi(\infty) = 1,
\end{equation}
which form the skeleton of the global dynamics \cite{AW78}.
These waves can be thought of as a mechanism of transport in which the fitter species
or more energetically
favourable phase invades the spatial domain.
The existence of these waves can be established via phase-plane analysis
\cite{Fife1977},
since the wave profile
necessarily satisfies the planar ODE
\begin{equation}
  \label{eq:int:eq:waveProfile:pde}
 c\Phi' = \Phi'' + g(\Phi).
\end{equation}
Here we have assumed the normalization $\sigma_h^2 + \sigma_v^2 = 1$.
Notice that the underlying spatial dimension
and the direction $(\sigma_h, \sigma_v)$ are not visible in the
travelling wave ODE \sref{eq:int:eq:waveProfile:pde},
which means that existence results in one spatial dimension can easily be
transferred to higher spatial dimensions in a radially symmetric fashion.

By contrast, substitution of the discrete travelling wave Ansatz
\sref{eq:int:lde:trvWaveAnsatz} into the LDE \sref{origLDE}
with $\Lambda = \Wholes^2$ leads
to the mixed type functional differential equation (MFDE)
\begin{equation}
   \label{eq:int:trvWave:mfde}
  \begin{array}{lcl}
  c\Phi'(\xi) &= & \Phi(\xi + \sigma_h) + \Phi(\xi - \sigma_h)
    + \Phi(\xi + \sigma_v) + \Phi(\xi - \sigma_v) - 4\Phi(\xi)
     + g\big(\Phi(\xi)\big).  
  \end{array}
\end{equation}

The broken translational invariance 
is manifested by the
fact that the wavespeed $c$ appears in \sref{eq:int:trvWave:mfde}
in a singular fashion. This leads to the phenomenon of
propagation failure, in which a sufficient energy difference between the two
stable equilibrium states is needed for propagation of waves,
together with the appearance of step-like
wave solutions \cite{Bell1984, VL28, BatesDiscConv, EVV2005AppMath}.

The broken spatial isotropy 
is manifested by the
explicit presence of the propagation direction $(\sigma_h, \sigma_v)$
in \sref{eq:int:trvWave:mfde}.  This leads
to direction-dependent wave speeds, wave forms and even
pinning regions. In particular, waves can fail to propagate
in certain directions that resonate with the lattice
whilst travelling freely in others.
This phenomenon has been
studied when the bistable nonlinearity is piecewise linear via classical analysis
\cite{CMPVV}, as well as for more general nonlinearities via numerics \cite{EVV,HJHVL2005},
homoclinic-unfolding \cite{MPCP} and center manifold analysis \cite{HOFFMPcrys}.

\subsection*{Stability of Waves}
While existence results for
travelling wave solutions to the PDE
\sref{eq:int:pde:obstructed} with $\Omega = \Real^2$ do not depend
on the spatial dimension, stability results do.
In fact, it has long been known \cite{FML}
that travelling fronts for one dimensional
versions of \sref{eq:int:pde:obstructed} are nonlinearly stable
under small perturbations,
provided one allows for a small phase shift in the wave.
Results in four or more dimensional problems
have also been available for some time \cite{XIN1992}.
However, the first stability result in
the critical case of two spatial dimensions
was only obtained relatively recently by Kapitula \cite{KAP1997}.

The key feature that complicates the stability analysis of
plane waves in two dimensions is that
the interface, by which we mean a level set of the function
$u(x,y,t) = \Phi(x+ct)$, is no longer localized in space at a point
as it is in one dimension, but rather spread out over an entire line.
This means that a phase shift is a big perturbation from the perspective of
$L^p$ ($p < \infty$): the difference $\Phi(x+\tau)-\Phi(x)$ does not live in
$L^p$ for any $\tau \ne 0$ (with $p < \infty$).  Thus it is not unreasonable to
expect that a localized perturbation does not lead to a phase shift.
On the other hand, one must now be concerned with long wave
deformations of the interface in the direction transverse
to that of propagation. Such deformations
manifest themselves as curves of essential spectrum
that touch the origin. This typically leads to
slow algebraic decay at the linear level, which
in the two dimensional case is notoriously difficult to control.

There are two main approaches towards
establishing the stability of travelling waves. The first
is based on spectral methods, Green's functions and bootstrapping
or fixed-point arguments.
This approach was developed by Kapitula \cite{KAP1997}
for the planar unobstructed PDE \sref{eq:int:pde:obstructed}.
Very recently \cite{HJHSTB2D},
we were able to extend it to the planar unobstructed LDE \sref{origLDE},
thereby generalizing earlier work by Bates and Chen \cite{BatesChen2002}
featuring a four-dimensional non-local setting.
The advantage of this approach is that only weak
spectral assumptions need to be imposed
on the underlying system. On the other hand,
such methods typically employ rather crude estimates on the nonlinear
terms and hence yield rather weak estimates for the basin of attraction.

The second approach is based on comparison principles
and yields stronger estimates for the basin of attraction,
at the price of requiring more structure on the underlying system.
The inclusion of obstacles in
\sref{origLDE}-\sref{eq:int:pde:obstructed} is a
rather strong perturbation from the unobstructed systems,
so it is not surprising that comparison principles are the method of
choice in the present context.

We remark that comparison principles were used
in an early paper \cite{XIN1992}
to prove the stability of waves in an unobstructed four dimensional version
of the PDE \sref{eq:int:pde:obstructed}. Here the algebraic decay
of the interfacial deformations described above is rather fast,
considerably easing the analysis. However,
besides the partial results in \cite{LEVXIN1992},
the first successful use of the comparison principle in two dimensions was actually
a byproduct of the analysis \cite{BHM} on which this paper is based.
It is therefore not surprising that part of our work here
can be seen as a companion paper to \cite{HJHSTB2D},
in the sense that we give an alternative proof of the nonlinear stability
of travelling waves to the unobstructed planar LDE \sref{origLDE}
with $\Lambda = \Wholes^2$.

\subsection*{The Program}
As mentioned above,
our contribution in this paper is to extend some of the main results in \cite{BHM}
to the spatially discrete setting.
In fact, the main steps that underpin our arguments here
closely mimic the program developed in \cite{BHM}.
In the remainder of this introduction we therefore
describe this program in some detail,
explaining the achievements in
\cite{BHM} and the technical modifications
required to extend them to the discrete setting.

Let us therefore consider either
\sref{origLDE} or \sref{eq:int:pde:obstructed},
with $\Lambda \neq \Wholes^2$ or $\Omega \neq \Real^2$.
One starts by considering the unobstructed wave when it is still far away
from the obstacle, but moving towards it. As time progresses,
the wave scatters from the obstacle and the goal is to prove
that the wave eventually recovers its shape. To this end,
one can distinguish the following three temporal regimes
and study them separately:
\begin{itemize}
\item[(i)] the {\bf pre-interaction regime} in which the obstacle is so
far away from the wave front that the obstacle essentially sees only an
equilibrium solution;
\item[(ii)] the {\bf interaction regime} in which the wave front and the
obstacle interact strongly, producing a large transient; and finally,
\item[(iii)] the {\bf post-interaction regime} in which the wave front is
again far from the obstacle.
\end{itemize}
In order to construct the desired
entire asymptotic plane wave solution,
one roughly needs the following ingredients:
\begin{itemize}
\item[(i)] a stability result which guarantees that the plane wave doesn't change much in the pre-interaction regime;
\item[(ii)] a bound showing that the transient produced during the interaction
regime is spatially localized;
\item[(iii-A)] a result that describes
the final state of compact subsets of the medium,
long after the wave front has passed through the region
and the transients have died down; and finally
\item[(iii-B)] powerful estimates on
the basin of attraction of the travelling planar wave that
can accommodate the large spatially localized transients generated
during the interaction regime.
\end{itemize}


\subsection*{Pre-interaction regime}

\paragraph*{Continuous Setting} The main task in this step is to prove that
the asymptotic requirement
\begin{equation}
 \label{eq:int:asympLimitsMinInftyForU}
 \lim_{t \to - \infty} \sup_{(x,y) \in \Omega}
    \left| u(x,y,t) - \Phi(x + ct) \right| = 0,
\end{equation}
with $\Phi$ as in \sref{eq:int:eq:waveProfile:pde},
fixes a unique solution to the obstructed PDE
\sref{eq:int:pde:obstructed} that is defined for all $t \in \Real$.
The arguments used to achieve this in \cite{BHM} are all
one-dimensional in nature, in the sense that the sub and super-solutions
that are used in the construction of $u$
depend only on the coordinate $x$ parallel to the direction of the wave.

In particular, a well-known
squeezing technique is employed that
goes back to the classic work of Fife and McLeod \cite{Fife1977}.
Roughly speaking, this technique allows additive
initial perturbations to wave-like structures to be traded off for an asymptotic phase offset.
Since this general principle lies at the heart of many other arguments in this paper,
we discuss it here in some detail.

For the unobstructed PDE \sref{eq:int:pde:obstructed}
with $\Omega = \Real^2$, the squeezing mechanism
can be exhibited by constructing a super-solution of the form
\begin{equation}
u(x,y,t) = \Phi\big(x+ct + Z(t)\big) + z(t), \label{1dsup}
\end{equation}
with $z$ decreasing from $z_0$ to $0$ and with $Z$ increasing from $0$ to $Z_\infty$.
To better understand the crucial relationship between the asymptotic phaseshift $Z_\infty$ and
the additive perturbation $z$, we note that the super-solution residual $\mathcal{J} = u_t - u_{xx} - g(u)$
is given by
\begin{equation}
\mathcal{J} = \dot{Z}\Phi'
  + \dot{z} + g(\Phi) - g(\Phi + z).
\end{equation}
Close to the interface, the term $g(\Phi) - g(\Phi + z)  \sim - g'(\Phi)z$ is negative
and must be dominated by the positive term $\dot{Z}\Phi'$.
This requires that $\dot{Z}$ dominate $z$ and $\dot{z}$. On the other hand,
close to the spatial limits $\Phi \to 0$ and $\Phi \to 1$ we have $g'(\Phi) < 0$,
so this regime requires $z$ to dominate $\dot{z}$.
These observations allow us to define $z(t)$ as a slowly decaying exponential,
which gives the relation
%
\begin{equation}
Z_\infty \sim \int_0^\infty z(t)dt \sim z_0  \label{phaseshift}
\end{equation}
between the asymptotic phase shift and the size of the initial perturbation.

In one spatial dimension, this technique (along with others)
has been used to establish the nonlinear stability of travelling waves under asymptotic phase shifts.
In our present context, a slight variant
is used in \cite{BHM} to establish the uniqueness of the entire solution $u$
for \sref{eq:int:pde:obstructed} with \sref{eq:int:asympLimitsMinInftyForU}
that was mentioned above.

On the other hand, the
existence of this solution $u$ is harder to establish.
The main procedure in \cite{BHM} is to split the plane $\Real^2$ into
two half-planes along a vertical line, with the obstacle
contained entirely in one of the half-planes. One can then focus
on time regimes where the bulk of the incoming wave
is contained in the obstacle-free half-plane.  Following
ideas in \cite{HAMNAD1999,GUOMOR2005}, one can then trap
the desired entire solution $u$ between the sum and the difference
of two counter-propagating wave fronts. This allows the use of
a limiting argument to show that the solution $u$ actually exists.

\paragraph*{Discrete Setting}

The main ideas of the PDE analysis described above can be
carried over to the LDE setting without major complications.
Indeed, arguments based on the comparison principle
are well-developed for systems posed on one-dimensional
lattices; see for example \cite{CHENGUOWU2008,HJHNEGDIF}.
Care must however be taken when studying the counter-propagating wave fronts,
as the discrete Laplacian causes cross-talk between the two-half planes
that needs to be carefully controlled.
This analysis is carried out in \S\ref{sec:ent}.

\subsection*{Interaction regime}

\paragraph*{Continuous Setting}
The main issue in this phase is to show that the large perturbations
that arise when the wave front hits the obstacle are spatially localized.
In particular, in \cite{BHM} the authors  show
that whenever the obstacle $\Real^2 \setminus \Omega$ is bounded,
the entire solution $u$
that satisfies the temporal requirement \sref{eq:int:asympLimitsMinInftyForU}
also admits the spatial limits
\begin{equation}
 \lim_{\abs{x} + \abs{y} \to \infty} \left| u(x,y,t) - \Phi(x+ct) \right| = 0,
\end{equation}
locally uniformly in $t$.

For the transverse direction $y \to \infty$,
this result can be established using standard parabolic estimates, since the
obstacle can effectively be pushed to $y = - \infty$ and hence ignored in the limit.
The analysis for $x \to - \infty$ is more subtle and does need to take the obstacle into account.
In order to do this, two branches of modified nonlinearities $g^\pm_{\delta}$
are constructed in \cite{BHM} with the property that $g^-_{\delta} \le g \le g^+_{\delta}$.
These modified nonlinearities have zeroes at $g^-(-\delta) = g^+(\delta) = 0$
and $g^-(1 - \delta) = g^+(1 + \delta) = 0$. This allows
travelling waves for the \textit{unobstructed} problem with nonlinearities $g^\pm$
to be effectively used as sub and super-solutions for the original \textit{obstructed} problem.

\paragraph*{Discrete Setting}
The parabolic estimates and limiting arguments described above can be readily
transferred to the discrete setting; see \S\ref{sec:lims}. On the other hand,
special care needs to be taken in the construction of the modified nonlinearities $g^\pm_\delta$,
which we perform in \S\ref{sec:prlm}.
This is related to the fact that the existence of travelling waves $(c, \Phi)$
for unobstructed lattices hinges on a delicate analysis of the MFDE
\sref{eq:int:trvWave:mfde}; see \cite{MPB}. Together with our generalization
of the classical spreading speed result, this requires smoothness on the part of
the nonlinearities beyond that which was used in \cite{BHM}.

\subsection*{Post-interaction regime - rest state}

\paragraph*{Continuous Setting}
The first key result obtained in \cite{BHM}
for this phase is that one has the temporal convergence
\begin{equation}
\label{eq:int:tempLimitUInfty}
\lim_{t \to \infty} u(x,y,t) \to u_\infty(x,y),
\end{equation}
locally uniformly in $(x,y)$,
where $u_\infty$ is a stationary solution to the obstructed PDE
\sref{eq:int:pde:obstructed}
that admits the spatial limits
\begin{equation}
\lim_{|x| + |y| \to \infty} u_\infty(x,y) = 1,
\end{equation}
together with Neumann boundary conditions on $\partial \Omega$.

Let us briefly comment on the interpretation of this spatial limit.
A classical result due to Aronson and Weinberger \cite[Thm. 5.3]{AW78}
for the unobstructed PDE \sref{eq:int:pde:obstructed} with $\Omega = \Real^2$
states that an initial disturbance from the $u \equiv 0$ rest state
that occupies a large area of the spatial domain and is large in amplitude,
will eventually spread to fill out the entire spatial domain.
The radial speed of this spreading is comparable to that of a travelling planar wave.
In \cite[Lem. 5.2]{BHM} a similar result was established for the obstructed
case $\Omega\neq \Real^2$. In particular,
the authors established that disturbances such as those described above
continue to spread out
in all directions, providing one restricts attention to regions of
space that are sufficiently far away from the obstacle.
This provides a mechanism by which the disturbance
caused by the incoming wave can bend around the obstacle.
In particular, the favoured state $u \equiv 1$ invades the domain
in an asymptotic sense.

It is clear that the homogeneous equilibrium $u_\infty \equiv 1$ is a
candidate for the temporal limit \sref{eq:int:tempLimitUInfty}
discussed above. It is however by no means clear that this is the only candidate.
The second key result in \cite{BHM} gives geometric conditions on
the obstacle $\Real^2 \setminus \Omega$
that are sufficient to guarantee that in fact $u_\infty \equiv 1$. These conditions
are satisfied if the obstacle is star-shaped or directionally convex,
but are far from being sharp. For example,
a recent result due to Bouhours \cite{BOUHU2012} characterizes a class
of admissible perturbations to such obstacles that still
allow one to prove $u_\infty \equiv 1$.

\paragraph*{Discrete Setting}

The main issue in the discrete setting is that an analogue of
the classic Aronson and Weinberger spreading speed result described above
is not readily available in the literature. This situation
is remedied in \S\ref{sec:expblob}, where we construct
expanding sub-solutions based on the gluing together of planar travelling waves.
The main obstruction here is that the speeds and profiles
of these planar waves are direction-dependent, which prevents us from
using a radially symmetric construction and requires a delicate balancing
of angular dependent terms.

Naturally, such a spreading result can only be established if the wave speeds
are strictly positive for every direction. In the discrete case,
an important role is therefore reserved for the
pinning phenomenon discussed above, which can block propagation in certain
directions. In the present paper we avoid such complications
and stay away from the pinning regime. This issue is discussed further
in \S\ref{sec:disc}.

We also do not fully explore the issues concerning
the geometry of the obstacle. We do however show in \S\ref{sec:lims}
that directionally convex obstacles force the identity $u_\infty \equiv 1$,
which allows our results to be applied to an important class of
non-empty bounded obstacles.

\subsection*{Post-interaction regime - convergence to wave}

\paragraph*{Continuous Setting}
In this final step of the program, the large transients generated in the interaction regime
must be controlled in a frame that moves along with the unobstructed wave.
As discussed above, this analysis leads naturally to a large basin nonlinear stability
result for planar travelling wave solutions to the unobstructed
PDE \sref{eq:int:pde:obstructed} with $\Omega = \Real^2$.
For presentation purposes, we will focus our discussion here on this unobstructed
special case. Indeed, the inclusion of the obstacle merely
adds technical complications that do not contribute to the understanding
of the differences between the continuous and discrete frameworks.

The main task is to construct a super-solution
for \sref{eq:int:pde:obstructed} with $\Omega = \Real^2$
of the form
\begin{equation}
  \label{PDEsupansatz}
  u(x,y, t) = \Phi\big(x + ct + \theta(y,t)  + Z(t) \big) + z(t),
\end{equation}
which adds transverse effects to the Ansatz \sref{1dsup}
discussed earlier.
As before, the function $z$ decreases from $z_0$ to $0$
while $Z$ increases from $0$ to $Z_\infty$.
Both $z$ and $Z$ should be thought of as small terms.
By contrast, the new function $\theta$
should be allowed to be arbitrarily large at $t = 0$,
provided that it is localized in the sense  $\theta(\cdot, 0) \in L^2$
and  that it decays to zero as $t \to \infty$ uniformly in $y$.
This function controls deformations of the wave interface in the transverse direction.

We note that any localized initial perturbation
from the wave 
can be dominated by the initial condition in
\sref{PDEsupansatz} by choosing $z_0$ positive and as small as we wish,
at the cost of a larger value for $\|\theta(\cdot,0)\|_{L^2}$. 
Assuming for the moment that $Z_\infty$ scales with $z_0$,
this freedom implies that we can dominate
the transients caused by such a perturbation
by a family of super-solutions that have
arbitrarily small asymptotic phase offsets $Z_\infty$.
A similar argument with sub-solutions then establishes
the convergence 
to the planar wave without any
asymptotic phase shift.

The difference in behaviour between one and two spatial dimensions
is hence caused by the extra transverse direction,
along which perturbations can diffuse in a sense without causing
a phase shift. The assumption that the initial transverse perturbation $\theta(\cdot, 0)$
is localized is crucial here. Indeed,
if $\theta(\cdot, 0)$ is not localized but still very small,
the perturbations can not only cause phase offsets of the underlying wave,
but can also prevent solutions from converging to any translate of the wave at all \cite{MNT}.

We now proceed to discuss the choices made in \cite{BHM}
for the functions $\theta$, $Z$ and $z$, which ensure that \sref{PDEsupansatz}
is indeed a super-solution with $Z_\infty \sim z_0$.
As before, we write $\mathcal{J} = u_t - \Delta u - g(u)$ for the super
solution residual. A short computation shows that we can split $\mathcal{J}$
into the three parts
\begin{equation}
\mathcal{J} = \mathcal{J}_{\mathrm{glb}} + \mathcal{J}_{\mathrm{heat}}
 + \mathcal{J}_{\mathrm{nl}},
\end{equation}
which with the shorthand $\xi = x + ct + \theta(y,t) + Z(t)$ can be written as
\begin{equation}
\label{eq:int:compForResidual}
\begin{array}{lcl}
\mathcal{J}_{\mathrm{glb}}(x,y,t) & = &
  \dot{Z}(t)\Phi'(\xi)
     + \dot{z}(t)  +g\big(\Phi(\xi)\big)   - g\big(\Phi(\xi) + z(t)\big),
\\[0.2cm]
\mathcal{J}_{\mathrm{heat}}(x,y,t) & = &
 \Phi'(\xi)\big(\partial_t \theta(y,t) - \partial_{yy} \theta(y, t) \big),
\\[0.2cm]
\mathcal{J}_{\mathrm{nl}}(x,y,t) & = &
   - \Phi''(\xi) [\partial_y \theta(y , t)]^2.
\\[0.2cm]
\end{array}
\end{equation}
Exactly as discussed earlier, one can ensure that $\mathcal{J}_{\mathrm{glb}}(x,y,t) \ge 0$
by picking $z(t)$ to be a slowly decaying exponential.
This gives the desired proportionality  $Z_\infty \sim z_0$.
Since $\Phi''$ does not admit a sign, the two remaining
terms $\mathcal{J}_{\mathrm{heat}}$ and $\mathcal{J}_{\mathrm{nl}}$
need to be treated together. The structure
of $\mathcal{J}_{\mathrm{heat}}$ strongly suggests a relation with the heat equation
and that is precisely what is exploited in \cite{BHM}.

To set the stage, let us introduce the functions
\begin{equation}
v(y, t) = e^{ - y^2 / 4 t }, \qquad h(y, t) = \frac{1}{\sqrt{4 \pi t}} v(y, t)
\end{equation}
and remark that the one dimensional heat kernel $h(y,t)$ can
be seen in an appropriate sense as the solution of the initial
value problem
\begin{equation}
\label{eq:int:heatEqForH}
\partial_t h = \partial_{yy} h, \qquad h(y, 0) = \delta(y).
\end{equation}
The choice for $\theta$ used in \cite{BHM} is now given by
\begin{equation}
\label{eq:int:cont:defThetaWithGamma}
\theta(y, t) =  \beta  (t + 1)^{-2 \gamma^{-1} } v\big(y, \gamma (t + 1) \big),
\end{equation}
in which $\gamma = \gamma(\beta) \gg \beta$ and $\beta \gg 1$ can be chosen
to be arbitrarily large. In particular,
the function $\theta$ can be seen as a modified
heat-kernel where the diffusion is sped up by a factor $\gamma$
and the decay rate at the center $y = 0 $ is slowed down to
$\beta (t + 1)^{-2 \gamma^{-1}}$.
The partial derivatives $\partial_t \theta$, $\partial_y \theta$
and $\partial_{yy} \theta$ can all be explicitly evaluated
and yield rational functions of $t$ and $y$ multiplied by $v(y,t)$.
Exploiting a uniform bound $\abs{\Phi''(\xi)} \le K \Phi'(\xi)$,
one can decouple $\xi$ from the pair $(y,t)$
and guarantee 
$\mathcal{J}_{\mathrm{heat}} + \mathcal{J}_{\mathrm{nl}} \ge 0$
by picking $\gamma$ appropriately.

\paragraph*{Discrete Setting}
Considerable modifications must be made
to the procedure outlined above before it can be used
in the discrete setting. To appreciate
the difficulties involved, let us consider the
discrete analogue of the super-solution Ansatz \sref{PDEsupansatz}.
Choosing $(\sigma_h, \sigma_v) \in \Wholes^2$,
this can be written as
\begin{equation}
u_{ij}(t) = \Phi\big(i \sigma_h + j\sigma_v + ct + \theta_{i\sigma_v - j\sigma_h}( t) + Z(t) \big)
 + z(t) \label{LDEsupans}.
\end{equation}
Writing $\mathcal{J}$ for the appropriate super-solution residual,
we note that it can again be split up into three
components
$\mathcal{J} = \mathcal{J}_{\mathrm{glb}} + \mathcal{J}_{\mathrm{heat}}
  + \mathcal{J}_{\mathrm{nl}}$
that are the discrete analogues of \sref{eq:int:compForResidual}.
For the presentation here it is only necessary to fully write out one of these terms.
Indeed, 
introducing the shorthands
$\xi = i\sigma_v+j\sigma_h + ct$ and $l = i\sigma_v -j\sigma_h \in \Wholes$,
we restrict our attention to the component
\begin{equation}
\begin{array}{lcl}
  \mathcal{J}_{\mathrm{heat}} & = &
     \Phi'(\xi) \dot{\theta}_l( t)
       - \Phi'(\xi + \sigma_h )\big(\theta_{l + \sigma_v}( t) - \theta_{l}(t) \big)
       - \Phi'(\xi - \sigma_h )\big(\theta_{l - \sigma_v}(t) - \theta_{l}(t) \big)
       \\
       \\
     & & - \Phi'(\xi - \sigma_v)\big(\theta_{l + \sigma_h}(t) - \theta_{l}(t) \big)
       - \Phi'(\xi + \sigma_v)\big(\theta_{l - \sigma_h}( t) - \theta_{l}(t)  \big),
       \label{LDERes}
  \end{array}
\end{equation}
which is a sum of first differences in $\theta$, each multiplied by a
different shifted version of $\Phi'$.
Since $\mathcal{J}_\mathrm{heat}$ no longer factors as a positive function of $\xi$
times a parabolic operator acting on $\theta$, 
it is not at all clear how the approach described above for the continuous setting
can be mimicked.

Our inspiration to deal with this challenge comes from the art of normal
form transformations. In particular, we add four specially chosen auxilliary terms
to the Ansatz \sref{LDEsupans}. This
replaces the shifted coefficients
$\Phi'(\xi \pm \sigma_{h})$  and $\Phi'(\xi \pm \sigma_v)$
that appear in $\mathcal{J}_{\mathrm{heat}}$
by unshifted coefficients $\Phi'(\xi)$,
while generating additional terms that only involve higher order differences in $\theta$
or products of lower order differences.

This construction relies on solving linear
inhomogeneous MFDEs. More specifically, the four new terms
added to \sref{LDEsupans} can all be factorized as
a first difference in $\theta$ multiplied by $\xi$-dependent functions $p^\pm_h$
and $p^\pm_v$
that satisfy
\begin{equation}
[\mathcal{L}_0 p^\pm_{\#}](\xi) =   \Phi'(\xi \pm \sigma_{\#}) -  \alpha^\pm_{\#} \Phi'(\xi), \qquad \# = h,v.
\end{equation}
Here the linear operator $\mathcal{L}_0: W^{1, \infty}(\Real, \Real) \to L^\infty(\Real,\Real)$
is given by
\begin{equation}
\label{eq:int:inhom:sys:for:p}
\begin{array}{lcl}
[\mathcal{L}_0 p](\xi) & = &
   -c p'(\xi)
 +  p(\xi + \sigma_h)
  +p(\xi + \sigma_v)
  + p(\xi - \sigma_h)
  +p(\xi - \sigma_v)
   -4 p(\xi)
\\[0.2cm]
& & \qquad   + g'\big(\Phi(\xi) \big) p(\xi)
\end{array}
\end{equation}
and the constants $\alpha^\pm_{\#}$ are fixed by the requirement that the system
is actually solvable.
The details of this construction rely heavily on the linear
Fredholm theory developed by Mallet-Paret \cite{MPB}.
In any case, let us remark here that $\mathcal{L}_0$
is related to the linearization of the travelling wave MFDE
\sref{eq:int:trvWave:mfde} around the solution $\Phi$.
Indeed, we have $\mathcal{L}_0 \Phi' = 0$,
which means that we expect the asymptotics
$p^\pm_\#(\xi) \sim \xi \Phi'(\xi)$ as $\xi \to \pm \infty$
as the right hand side of \sref{eq:int:inhom:sys:for:p} is resonant.

It is worthwhile to pause here to fully understand
the ramifications of our analysis up to this point.
With the new terms, we obtain
\begin{equation}
\label{eq:int:dis:res:heat:after:mod}
\begin{array}{lcl}
  \mathcal{J}_{\mathrm{heat}} & = &
     \Phi'(\xi) \Big[ \dot{\theta}_{l}(t)
       - \alpha^+_h \big(\theta_{l + \sigma_v}(t) - \theta_{l}(t) \big)
       - \alpha^-_h\big(\theta_{l - \sigma_v}(t) - \theta_{l}(t) \big)
       \\[0.2cm]
     & & \qquad - \alpha^-_v \big(\theta_{l + \sigma_h}(t) - \theta_{l}(t) \big)
       - \alpha^+_v   \big(\theta_{l - \sigma_h}(t) - \theta_{l}(t) \big)
       \Big]  + h.o.t.,
  \end{array}
\end{equation}
in which the higher order terms contain, amongst others, products of second order
differences in $\theta$ and the functions $p^\pm_{\#}(\xi)$.
It is now tempting to try to proceed as in the continuous case,
writing $h_l( t)$ for the solution of the initial value problem
\begin{equation}
\label{eq:int:discrete:heatEq}
\begin{array}{lcl}
\dot{h}_l(t) & = &
  \alpha^+_h [ h_{l + \sigma_v}(t) - h_{l}(t) ]
  +\alpha^-_h [ h_{l - \sigma_v}( t) -h_{l}(t) ]
\\[0.2cm]
& & \qquad  +\alpha^-_v [ h_{l + \sigma_h}( t) - h_{l}(t) ]
  +\alpha^+_v [ h_{l - \sigma_h}( t) - h_{l}(t)],
\\[0.2cm]
h_{l}(0)&  = & \delta_{l, 0}
\end{array}
\end{equation}
and building $\theta$ from $h$ in a fashion similar to
\sref{eq:int:cont:defThetaWithGamma}.
For $\gamma \gg 1$, this choice
would imply that $\sup_{l \in \Wholes} \abs{\theta_l(t)}$ remains roughly constant
over long periods of time, while $k$-th order differences in $\theta$ would
roughly decay at the rate $t^{-k/2}$.

There are however three important
issues that complicate such an attempt.
First of all, the heat-like problem \sref{eq:int:discrete:heatEq}
contains convective terms, in contrast to \sref{eq:int:heatEqForH}.
These terms manifest themselves as $e^{\pm i \omega \sigma_{\#}} - 1 = O(i \omega)$ contributions
when Fourier-expanding the first differences in $\theta$
in the $l$-direction.
When scaling time as in \sref{eq:int:cont:defThetaWithGamma},
one must therefore take care that only the diffusive
effects are sped up, since the convective effects
do not generate terms of a definite sign.

The second problem is that the $\xi$-dependent terms that behave as $\xi \Phi'(\xi)$
can never be dominated by the signed terms proportional to $\Phi'(\xi)$
that can be obtained by carefully exploiting \sref{eq:int:discrete:heatEq}.
It is therefore necessary to use $\mathcal{J}_{\mathrm{glb}}$ for this purpose.
This implies that we would need $\mathcal{J}_{\mathrm{glb}} \sim z(t) \sim t^{-1}$,
since these troublesome terms come with second order differences in $\theta$.
Unfortunately, this gives $Z(t) \sim \ln(t)$, destroying any hope of closing
a stability argument.

Finally, the third problem is caused by the fact that \sref{eq:int:discrete:heatEq}
does not actually capture all the diffusive effects present in the problem.
Indeed, second order differences in $\theta$ also generate
$O( \omega^2)$ contributions in Fourier space.
Naturally, this leads to trouble when trying to
control the second order terms in \sref{eq:int:dis:res:heat:after:mod},
even if they are multiplied by functions proportional to $\Phi'(\xi)$.

In order to solve the second and third problems,
it is necessary to perform an additional step
in the normal form inspired expansion. In particular,
we add a large number of extra terms to the Ansatz \sref{LDEsupans}
in order to explicitly control all second order differences in $\theta$
and all products of first order differences in $\theta$.

Although this operation is a bookkeeping nightmare,
it serves two crucial purposes.
First of all, it ensures that all diffusive effects
are made visible. It is here
that the link to our prior work \cite{HJHSTB2D}
based on spectral techniques becomes apparent,
since we recover a condition on the essential
spectrum that ensures that the diffusion coefficient is non-zero.
Furthermore, the decay rate of all remaining terms
is pushed to at least $t^{-3/2}$, corresponding
to third order differences in $\theta$. Since this latter
function is integrable over $[1, \infty)$,
these terms can be successfully absorbed into $\mathcal{J}_{\mathrm{glb}}$.

In contrast to the continuous setting, we therefore see
that the three separate components $\mathcal{J}_{\mathrm{glb}}$,
$\mathcal{J}_{\mathrm{heat}}$ and $\mathcal{J}_{\mathrm{nl}}$ of the super-solution
residual need to be analysed together. In addition,
one can no longer use decaying exponentials for the global functions
$Z(t)$ and $z(t)$, which requires considerably more care
to be taken at many points in the analysis.
These issues are explored in detail in \S\ref{sec:oblq:subsup},
which we consider to be the heart of this paper.


\subsection*{Organization}
This paper is organized as follows.
In \S\ref{sec:mr} we state our main results along with their underlying
assumptions.  A comparison principle that is used
throughout the paper is established in \S\ref{sec:prlm},
along with asymptotic properties of planar travelling wave solutions
to the unobstructed LDE.
In \S\ref{sec:expblob} we establish a result similar in spirit to \cite[Thm. 5.3]{AW78},
stating that large disturbances from the zero rest state fill the entire
unobstructed lattice $\Z^2$, provided
the support of the initial condition is large enough.
Sub and super-solutions providing a large basin of attraction
for the waves mentioned above are constructed in \S\ref{sec:oblq:subsup}.
In \S\ref{sec:ent} we focus on the behavior
of the obstructed LDE and establish the existence of an entire
solution that resembles an unobstructed wave as $t\to -\infty$.
The behaviour of this entire solution in several
limiting regimes is studied in \S\ref{sec:lims}.
The proof of our main result is presented in \S\ref{sec:pmr}, followed by
a discussion in \S\ref{sec:disc}.

\paragraph{Reader's guide}
In order to prevent disruptions to the flow of ideas,
the proof of many of our results is deferred
to the second part of their respective sections.
Constants that are used across different sections are given names,
constants used only within a section are given numbers,
while constants used only within proofs are given primes.
We also use primes to denote derivatives, but trust that no confusion
shall result.

\paragraph{Acknowledgments}
Hoffman acknowledges support from the NSF (DMS-1108788).
Hupkes acknowledges support from the Netherlands Organization for Scientific Research (NWO).
Van Vleck acknowledges support from the NSF (DMS-1115408).

%
%
%

\section{Main Results}
\label{sec:mr}
In this section we outline our two main results.
The first result concerns a two-dimensional obstacle
free lattice and is essentially a stability
result for travelling planar fronts. The
second result shows that these
travelling planar fronts persist in an appropriate
sense after perturbing the lattice by removing grid points.

\subsection{Homogeneous Lattice}
Let us first consider the spatially homogeneous LDE
\begin{equation}
\label{eq:mr:lde:hom}
\dot{u}_{ij}(t) = [\Delta^+ u(t)]_{ij} + g\big( u_{ij}(t) \big),
\end{equation}
in which the plus-shaped discrete Laplacian acts
on a planar sequence $u \in \ell^\infty(\Wholes^2; \Real)$ as
\begin{equation}
\label{eq:mr:defPlusLaplacian}
 [\Delta^+ u]_{ij} = u_{i+1, j} + u_{i, j+1} + u_{i-1, j} + u_{i, j-1} -4 u_{ij}.
\end{equation}
The conditions on the nonlinearity $g$ that we will need
are summarized in the following assumption.
\begin{itemize}
  \item[(Hg)]{
    The nonlinearity $g: \Real \to \Real$ is $C^2$-smooth and has a bistable structure,
    in the sense that there exists a constant $0 < a < 1$ such that we have
    \begin{equation}
       g(0) = g(a) = g(1) = 0, \qquad g'(0) < 0, \qquad g'(1) < 0,
    \end{equation}
    together with
    \begin{equation}
        g(u)  < 0 \hbox{ for } u \in (0, a) \cup (1, \infty),
       \qquad
       g(u) > 0 \hbox{ for } u \in (-\infty,-1) \cup (a, 1).
    \end{equation}
  }
\end{itemize}
We note that the results in \cite{MPB} guarantee that
\sref{eq:mr:lde:hom} has planar travelling wave solutions
for every possible direction of propagation.
The key requirement in our next assumption is that
all these waves travel at a strictly positive speed.
\begin{itemize}
  \item[(H$\Phi$)]{
    For each angle $\zeta \in [0, 2 \pi]$, there exists a
    wave speed $c_\zeta > 0$ and a profile
    $\Phi_\zeta \in C^1(\Real, \Real)$
    that satisfies the limits
    \begin{equation}
    \label{eq:mr:critHPhi:limitsPhi}
    \lim_{\xi \to - \infty} \Phi_\zeta(\xi) = 0,
    \qquad \lim_{\xi \to + \infty} \Phi_\zeta(\xi) = 1
    \end{equation}
    and yields a solution to the homogeneous LDE \sref{eq:mr:lde:hom}
    upon writing
    \begin{equation}
      u_{ij}(t) = \Phi_\zeta( i \cos \zeta + j \sin \zeta + c_\zeta t ).
    \end{equation}
  }
\end{itemize}

For the remainder of this subsection,
we fix a specific angle $\zeta_*$ for which $\tan \zeta_* \in \mathbb{Q}$.
This allows us to pick a pair $(\sigma_h, \sigma_v) \in \Wholes^2$
with $\gcd(\sigma_h, \sigma_v) = 1$,
for which we have the identity
\begin{equation}
[ \sigma_h^2 + \sigma_v^2]^{1/2}(\cos \zeta_*, \sin \zeta_*) = (\sigma_h, \sigma_v).
\end{equation}
As a consequence of $(H\Phi)$,
there exists a wave speed $c > 0$
and a wave profile $\Phi \in C^1(\Real, \Real)$
that satisfies the limits \sref{eq:mr:critHPhi:limitsPhi},
so that the LDE \sref{eq:mr:lde:hom} admits a solution
\begin{equation}
\label{eq:mr:rescaledWaveSol}
 u_{ij}(t) = \Phi( i \sigma_h + j \sigma_v + ct).
\end{equation}
We emphasize here that the pair $(c, \Phi)$ is a
rescaled version of the pair $(c_{\zeta_*}, \Phi_{\zeta_*})$.
We choose to use this rescaling here for the benefit of the discussion below.

A standard approach towards establishing the stability of the wave
solution \sref{eq:mr:rescaledWaveSol}
under the nonlinear dynamics of the LDE \sref{eq:mr:lde:hom} is to consider the linear
variational problem
\begin{equation}
\dot{v}_{ij}(t) =  [\Delta^+ v(t)]_{ij} + g' \big( \Phi(i \sigma_h + j \sigma_v + ct) \big) v_{ij}(t).
\end{equation}
As can be seen, the linear operator on the right hand side of this system does
not have constant coefficients and hence cannot be diagonalized via Fourier transform.
It can however be partially diagonalized if one takes the Fourier transform
in the direction that is perpendicular to the propagation of the wave, i.e.,
upon taking  $i\sigma_h + j\sigma_v  =  \mathrm{constant}$.
As explained in \cite{HJHSTB2D},
one arrives for each transverse spatial frequency
$\omega \in [-\pi, \pi]$ at an LDE posed on a
one dimensional lattice that is parallel to the direction of propagation. This LDE is given by
\begin{equation}
\label{eq:mr:1dLDEParallel}
\begin{array}{lcl}
\dot{v}_n(t)
& = &  e^{i \sigma_v \omega } v_{n+ \sigma_h}(t)
            +  e^{- i \sigma_h \omega }  v_{n + \sigma_v}(t)
            +  e^{-i \sigma_v \omega} v_{n - \sigma_h}(t)
            + e^{i \sigma_h \omega } v_{n - \sigma_v}(t)
            -4 v_n(t)
\\[0.2cm]
& & \qquad + g'\big(\Phi(n + ct)\big) v_{n}(t),
\end{array}
\end{equation}
in which 
$n =  i \sigma_h +  j \sigma_v \in \Wholes$.
As explained in detail in \cite[{\S}2]{HJHSTBFHN}, there is a close
relationship between the Green's function for the LDE \sref{eq:mr:1dLDEParallel}
and the linear operators
\begin{equation}
\label{eq:mr:defLomega}
\mathcal{L}_{\omega}: W^{1, \infty}(\Real, \Complex) \to L^{\infty}(\Real, \Complex), \qquad \omega \in [-\pi, \pi],
\end{equation}
that act as
\begin{equation}
\begin{array}{lcl}
[\mathcal{L}_{\omega} p](\xi)
& = & -c p'(\xi)
 + e^{i \sigma_v \omega} p(\xi + \sigma_h)
  +e^{- i \sigma_h \omega} p(\xi + \sigma_v)
  +e^{- i \sigma_v \omega} p(\xi - \sigma_h)
  +e^{i \sigma_h \omega} p(\xi - \sigma_v)
\\[0.2cm]
& & \qquad
   -4 p(\xi)
   + g'\big(\Phi(\xi)\big) p(\xi).
\end{array}
\end{equation}
The formal adjoints of these operators are written as
\begin{equation}
\mathcal{L}^*_\omega: W^{1, \infty}(\Real, \Complex) \to L^{\infty}(\Real, \Complex), \qquad \omega \in [-\pi, \pi],
\end{equation}
and act as
\begin{equation}
\begin{array}{lcl}
[\mathcal{L}^*_\omega q](\xi)
& = &  c q'(\xi)
+  e^{-i \sigma_v \omega} q(\xi - \sigma_h)
  +e^{+ i \sigma_h \omega} q(\xi - \sigma_v)
  +e^{+ i \sigma_v \omega} q(\xi + \sigma_h)
  +e^{- \sigma_h \omega} q(\xi + \sigma_v)
\\[0.2cm]
& & \qquad
  -4 q(\xi)
   + g'\big(\Phi(\xi)\big) q(\xi).
\end{array}
\end{equation}
Indeed, the designation of formal adjoint is justified by an easy
computation, which shows that
\begin{equation}
\int_{-\infty}^{\infty} q^*(\xi)  [\mathcal{L}_{\omega} p ](\xi)  \, d \xi
= \int_{-\infty}^{\infty}  [\mathcal{L}^*_\omega q](\xi)^* p(\xi)  \, d \xi
\end{equation}
holds for all pairs $p,q \in W^{1, \infty}(\Real, \Complex)$.

In view of the fact that the shifts  are all
integer valued, an easy computation shows that
\begin{equation}
\label{eq:mr:spCompact}
e_{-2 \pi i  \ell } ( \mathcal{L}_\omega -\lambda) e_{2 \pi i \ell} = \mathcal{L}_\omega - 2 \pi i \ell c - \lambda
\end{equation}
for all $\lambda \in \Complex$ and $\ell \in \Wholes$, in which the exponential shift operator $e_{\nu}$ is defined by
\begin{equation}
[e_{\nu} v](\xi) = e^{\nu \xi} v(\xi).
\end{equation}
In particular, for all $\omega \in [-\pi, \pi]$, the spectrum of $\mathcal{L}_\omega$ is invariant under the operation
$\lambda \mapsto \lambda + 2 \pi i c$.

As discussed in detail in \cite{HJHSTB2D},
the operator $\mathcal{L}_0$ encodes
stability properties of the wave $(c, \Phi)$ under
perturbations that are constant in the direction transverse to propagation.
In any case, it is easy to verify that $\mathcal{L}_0 \Phi' = 0$
holds. The results in \cite{MPB} show that $\mathcal{L}_0$
is a Fredholm operator of index zero
with a one-dimensional kernel.
In addition, these results give
the existence of a strictly positive
bounded function $\Psi \in C^1(\Real, \Real)$
for which
\begin{equation}
\label{eq:mr:defPsi}
\mathcal{L}_0^* \Psi = 0, \qquad \int_{\Real} \Psi(\xi) \Phi'(\xi) \, d \xi = 1.
\end{equation}
In view of the characterization
\begin{equation}
\label{eq:mr:defRangeL0}
\mathrm{Range} ( \mathcal{L}_0 ) = \{ f \in L^\infty(\Real; \Real) : \int_{\Real} \Psi(\xi) f(\xi) \, d \xi = 0 \},
\end{equation}
this directly implies that $\lambda = 0$ is a simple eigenvalue of $\mathcal{L}_0$.

The next result,
taken from \cite{HJHSTB2D}, show how this simple eigenvalue
behaves as $\omega$ is varied.
It basically states that
there is a branch of simple eigenvalues $\lambda_\omega$
and eigenfunctions $\phi_{\omega}$
for the operators  $\mathcal{L}_\omega$ with $\omega \approx 0$.
With the exception of simple eigenvalues at $\lambda_\omega + 2\pi i c \Z$,
the spectrum of $\mathcal{L}_\omega$
lies to the left of the line $\Re \lambda = -\beta$.

\begin{prop}[{\cite[Prop. 2.1]{HJHSTB2D}}]
\label{prp:mr:melnikov}
Consider the unobstructed LDE \sref{eq:mr:lde:hom}
and suppose that (Hg) is satisfied.
Pick a direction $\zeta_* \in \Real$
for which $\tan \zeta_* \in \mathbb{Q}$
and suppose that the requirements stated in (H$\Phi$) all hold,
but only for the angle $\zeta_*$.

Then there exists a constant $ \delta_{\omega} > 0$
together with pairs
\begin{equation}
  (\lambda_{\omega}, \phi_{\omega} ) \in \Complex \times W^{1, \infty}(\Real, \Complex),
\end{equation}
defined for each  $\omega \in (-\delta_{\omega}, \delta_{\omega})$, such that the following holds true.
\begin{itemize}
\item[(i)]{
  There exists $\beta > 0$ such that
  for all $\omega \in (-\delta_{\omega}, \delta_{\omega})$,
  the operator $\mathcal{L}_\omega - \lambda$
  is invertible as a map from $W^{1, \infty}(\Real, \Complex)$ into $L^\infty(\Real, \Complex)$
  for all
  $\lambda \in \Complex$ that have $\Re \lambda \ge -\beta$
  and for which $\lambda - \lambda_{\omega} \notin  2 \pi i c \Wholes$.
}
\item[(ii)]{
  The only nontrivial solutions $p \in W^{1, \infty}(\Real, \Complex)$ of $(\mathcal{L}_\omega - \lambda_{\omega})p = 0$
  are $p = \phi_{\omega}$ and scalar multiples thereof.
}
\item[(iii)]{
  For all $\omega \in (-\delta_{\omega}, \delta_{\omega})$,
  the equation $(\mathcal{L}_\omega - \lambda_{\omega} )v = \phi_{\omega}$
  does not admit a solution $v \in W^{1, \infty}(\Real, \Complex)$.
}
\item[(iv)]{
  The maps $\omega \mapsto \lambda_{\omega}$
  and $\omega \mapsto \phi_{\omega}$ are $C^2$-smooth,
  with $\lambda_{0} = 0$ and $\phi_0 = \Phi'$.
}
\end{itemize}
\end{prop}

In the present paper we focus on developing a comparison-principle based
approach to understand the stability of \sref{eq:mr:rescaledWaveSol},
avoiding the spectral techniques that lie at the heart of \cite{HJHSTB2D}.
Nevertheless, we need to impose
one of the spectral
condition stated in \cite{HJHSTB2D}.
In particular, the next assumption ensures
that the curve $\omega \to \lambda_{\omega}$ touches the origin in a quadratic tangency
that opens up on the left side of the imaginary axis.
\begin{itemize}
\item[$(HS)_{\zeta_*}$]{
Recalling the curves $\omega \mapsto (\lambda_{\omega}, \phi_{\omega})$
defined in Proposition \ref{prp:mr:melnikov}, we have the Melnikov identity
\begin{equation}
\label{eq:mr:melnikov}
\begin{array}{lcl}
[ \frac{d^2}{d \omega^2} \lambda_{\omega}]_{\omega = 0} & < & 0.
\end{array}
\end{equation}
}
\end{itemize}
We emphasize here that $(HS)_{\zeta_*}$
is guaranteed to be satisfied
whenever  $\zeta_* \approx \frac{k\pi}{4}$
and $c_{\frac{k\pi}{4} } > 0$  for some $k \in \Wholes$.
In addition, we have numerical evidence to
suggest that $(HS)_{\zeta_*}$ is satisfied generically;
see 
\cite[{\S}6]{HJHSTB2D} for a detailed discussion.

We are now ready to state our first main result,
which will be proved in \S\ref{sec:oblq:subsup}.
Compared to the small-perturbation stability result in \cite{HJHSTB2D},
we note that here the initial deviation from the wave can be very large
in compact regions. On the other hand, in \cite{HJHSTB2D} the
perturbations are not required to decay in the direction parallel
to the wave propagation as they are here.

\begin{thm}
\label{thm:mr:unobstructed:stb}
Consider the unobstructed LDE \sref{eq:mr:lde:hom}
and suppose that (Hg) is satisfied.
Pick a direction $\zeta_* \in \Real$
for which $\tan \zeta_* \in \mathbb{Q}$
and suppose that the requirements stated in (H$\Phi$) all hold,
but only for the angle $\zeta_*$.
In addition, suppose that $(HS)_{\zeta_*}$ is satisfied.

Consider any $C^1$-smooth function
$U: [0, \infty) \to \ell^{\infty}(\Wholes^2 ; \Real)$
that satisfies the LDE \sref{eq:mr:lde:hom} for all $t \ge 0$.
Suppose furthermore that
we have the spatial limit
\begin{equation}
\abs{U_{ij}(0) - \Phi_{\zeta_*}( i \cos \zeta_* + j \sin \zeta_*  ) } \to 0, \qquad \abs{i} + \abs{j} \to \infty.
\end{equation}
%
Then we have the uniform convergence
\begin{equation}
\sup_{(i,j) \in \Wholes^2} \abs{U_{ij}(t) - \Phi_{\zeta_*}(i \cos \zeta_* + j \sin \zeta_* + c_{\zeta_*} t) } \to 0 , \qquad t \to \infty.
\end{equation}
\end{thm}

\subsection{Obstructed Lattice}
In order to formalize the concept of removing grid points
from a lattice, we start by introducing some notation.
In particular, for any $(i,j) \in \Wholes^2$ we write
\begin{equation}
\mathcal{N}_{\Wholes^2}(i,j) = \{  (i+1, j), (i, j+1), (i - 1, j), (i, j - 1) \} \subset \Wholes^2
\end{equation}
to denote the set of nearest neighbours for the grid point$(i,j)$.
Obviously, we can now restate the definition \sref{eq:mr:defPlusLaplacian} as
\begin{equation}
[\Delta^+ u]_{ij} = \sum_{(i', j') \in \mathcal{N}_{\Wholes^2}(i,j)} [ u_{i', j'} - u_{ij} ].
\end{equation}

Consider now a bounded set $K_{\mathrm{obs}} \subset \Wholes^2$,
which should be interpreted to be missing from the lattice.
We write
\begin{equation}
\Lambda = \Wholes^2 \setminus K_{\mathrm{obs}}
\end{equation}
to refer to the remaining grid points.
In addition, for any point $(i,j) \in \Lambda$,
we write
\begin{equation}
\mathcal{N}_{\Lambda}(i,j) = \mathcal{N}_{\Wholes^2}(i,j) \cap \Lambda,
\end{equation}
which represents the set of traditional nearest neighbours of $(i,j)$ that are not contained in $K_{\mathrm{obs}}$.
We use the suggestive notation
\begin{equation}
\partial \Lambda = \{ (i,j) \in \Lambda : \mathcal{N}_{\Wholes^2}(i,j) \neq \mathcal{N}_{\Lambda}(i,j) \}
\end{equation}
to denote the set of points in the lattice that
are nearest neighbour to a site in the obstacle $K_{\mathrm{obs}}$.

Let us now consider a sequence $v \in \ell^{\infty}( \Lambda , \Real)$
and introduce the punctured discrete Laplacian that acts as
\begin{equation}
[\Delta^+_\Lambda v]_{ij} = \sum_{(i',j') \in \mathcal{N}_\Lambda(i,j) } [ v_{i' j'} - v_{ij} ]
\end{equation}
for $(i,j) \in \Lambda$.
Obviously,
we have
\begin{equation}
[\Delta^+_\Lambda v]_{ij} = [\Delta^+ v]_{ij} , \qquad (i, j) \in \Lambda \setminus \partial \Lambda.
\end{equation}

Our main goal in this paper is to study
the obstructed LDE
\begin{equation}
\label{eq:mr:ldeWithObs}
\dot{u}_{ij}(t) = [\Delta^+_\Lambda u(t)]_{ij} + g\big( u(t) \big), \qquad (i,j) \in \Lambda.
\end{equation}
As a consequence of our choice for $\Delta^+_{\Lambda}$,
one can interpret this LDE as the analogue of a reaction-diffusion PDE
posed on an exterior domain under Neumann boundary conditions.

Besides the conditions imposed above on the homogeneous LDE \sref{eq:mr:lde:hom},
we need to impose the following two conditions on the obstacle $K_{\mathrm{obs}}$.
\begin{itemize}
\item[(HK1)]{
  The obstacle set $K_{\mathrm{obs}}$ is bounded and $\Lambda = \Wholes^2 \setminus K_{\mathrm{obs}}$
   is connected, in the sense that for every $(i,j) \in \Lambda$
   and $(i', j') \in \Lambda$, there exists an integer $N \ge 0$
   and a sequence $\{ (i_k, j_k) \}_{k=0}^N \subset \Lambda$
   with
   \begin{equation}
     (i_0, j_0) = (i', j'), \qquad (i_N, j_N) = (i,j),
   \end{equation}
   so that for every $1 \le k \le N$ we have
   \begin{equation}
    (i_k, j_k) \in \mathcal{N}_{\Lambda}(i_{k-1}, j_{k-1}).
   \end{equation}
}
\item[(HK2)]{
   The obstacle $K_{\mathrm{obs}}$ is directionally convex,
   in the sense that there exists a line $\ell \subset \Real^2$ so that the following holds true.
   For any $(i,j) \in \partial \Lambda$ and $(i', j') \in K_{\mathrm{obs}}$
   that are related via $(i', j') \in \mathcal{N}_{\Wholes^2}(i,j)$,
   we have
   \begin{equation}
     d\big((i', j'), \ell\big) \le d\big((i,j), \ell \big),
   \end{equation}
   in which $d\big( (x,y), \ell\big) \ge 0$ denotes the distance between a point $(x,y) \in \Real^2$
   and the line $\ell \subset \Real^2$.
}
\end{itemize}
These conditions can be seen as the discrete analogues
of the restrictions imposed in \cite{BHM}.
However, we remark here that we exclude the star-shaped obstacles
that are allowed in \cite{BHM}. In any case,
it is easy to see that any obstacle that consists
of a single point automatically satisfies both (HK1) and (HK2).

We are now ready to state our second main result,
which is the discrete analogue of \cite[Thm. 1.3]{BHM}.
It basically states that there is an entire
solution to \sref{eq:mr:ldeWithObs}
that looks like a travelling planar wave
travelling towards the obstacle for $t \ll -1$,
gets scattered by the obstacle at $t = O(1)$
and gradually recovers its shape
as the wavefront moves away from the obstacle for $t \gg +1$.

\begin{thm}
\label{thm:mr:resKConvex}
Consider the obstructed LDE \sref{eq:mr:ldeWithObs}
and pick a direction $\zeta_* \in \Real$
for which $\tan \zeta_* \in \mathbb{Q}$.
Suppose that (Hg), (H$\Phi$), $(HS)_{\zeta_*}$,
(HK1) and (HK2)
are all satisfied.
Then there exists a $C^1$-smooth
function $U: \Real \to \ell^{\infty}(\Lambda, \Real)$
that satisfies the obstructed LDE \sref{eq:mr:ldeWithObs} for all $t \in \Real$, admits the inequalities
\begin{equation}
0 < U_{ij}(t) < 1, \qquad \dot{U}_{ij}(t) > 0
\end{equation}
for all $t \in \Real$ and $(i,j) \in \Lambda$ and enjoys the temporal limits
\begin{equation}
\label{eq:mr:resKConvex:tempLimit}
\sup_{(i,j) \in \Lambda} \abs{ U_{ij}(t) - \Phi_{\zeta_*}( i \cos \zeta_*  + j \sin \zeta_* + c_{\zeta_*}t) } \to 0, \qquad t \to \pm \infty,
\end{equation}
together with the spatial limit
\begin{equation}
\label{eq:mr:resKConvex:spLimit}
\sup_{t \in \Real} \abs{ U_{ij}(t) - \Phi_{\zeta_*}(i \cos \zeta_* + j \sin \zeta_* + c_{\zeta_*}t) } \to 0, \qquad \abs{i} + \abs{j} \to \infty.
\end{equation}
In addition, if $V: \Real \to \ell^\infty(\Lambda; \Real)$ is another
$C^1$-smooth function that satisfies
the obstructed LDE \sref{eq:mr:ldeWithObs} for all $t \in \Real$
with
\begin{equation}
\label{eq:mr:resKConvex:unq}
\sup_{(i,j) \in \Lambda} \abs{ V_{ij}(t) - \Phi_{\zeta_*}( i \cos \zeta_* + j \sin \zeta_* + c_{\zeta_*}t) } \to 0, \qquad t \to - \infty,
\end{equation}
then we have $U = V$.
\end{thm}


%

%

%

%

%

%

\section{Preliminaries}
\label{sec:prlm}
In this section we consider an obstacle $K_{\mathrm{obs}} \subset \Wholes^2$,
which at times will be taken to be empty,
and study the LDE
\begin{equation}
\label{eq:prlm:lde:obs}
\dot{u}_{ij}(t) = [\Delta^+_{\Lambda} u(t) ]_{ij} + g\big( u_{ij}(t) \big),
\qquad (i,j) \in \Lambda,
\end{equation}
in which $\Lambda = \Wholes^2 \setminus K_{\mathrm{obs}}$.
We formulate a comparison principle for \sref{eq:prlm:lde:obs}
and focus on the asymptotic behaviour of travelling wave solutions
to \sref{eq:prlm:lde:obs} with $K_{\mathrm{obs}} = \emptyset$. We are specifically
interested in the dependence of these waves on the direction
of propagation and the specific form of the nonlinearity,
which we will often need to distort.

For the purposes of this section, we need
to relax the smoothness requirements present in (Hg). In particular,
we introduce the following condition on the nonlinearity $g$.
\begin{itemize}
\item[$\textrm{(hg}\textrm{)}_{\textrm{\S\ref{sec:prlm}}}$]{
The nonlinearity $g$ is $C^1$-smooth, while the map $u \mapsto g'(u)$
is locally Lipschitz continuous.
In addition, there exists a constant $0 < a < 1$ such that we have
\begin{equation}
g(0) = g(a) = g(1) = 0, \qquad g'(0) < 0, \qquad g'(1) < 0,
\end{equation}
together with
\begin{equation}
g(u) < 0 \hbox{ for } u \in (0, a) \cup (1, \infty),
\qquad
g(u) > 0 \hbox{ for } u \in (-\infty,-1) \cup (a, 1).
\end{equation}
}
\end{itemize}
We start by formulating a weak
and strong version of a comparison principle
for \sref{eq:prlm:lde:obs}. The weak version does
not need the obstacle to be finite or connected as required by (HK1).
\begin{prop}
\label{prp:prlm:cmpPrinciple}
Pick any subset $K_{\mathrm{obs}} \subset \Wholes^2$,
consider the LDE \sref{eq:prlm:lde:obs}
and suppose that $\textrm{(hg}\textrm{)}_{\textrm{\S\ref{sec:prlm}}}$ is satisfied.
Consider a pair of functions $u,v \in C^1( [0, \infty), \ell^{\infty}(\Lambda, \Real) )$
that satisfy the uniform bounds
\begin{equation}
-1 \le u_{ij}(t) \le 2, \qquad - 1 \le v_{ij}(t) \le 2, \qquad (i,j) \in \Lambda, \qquad t \ge 0,
\end{equation}
together with the initial inequalities
\begin{equation}
u_{ij}(0) \ge v_{ij}(0), \qquad (i,j) \in \Lambda.
\end{equation}
Suppose furthermore that
for any $t \ge 0$ and all $(i,j) \in \Lambda$,
at least one of the following two properties
is satisfied.
\begin{itemize}
\item[(a)]{
We have the differential inequalities
\begin{equation}
\label{eq:lem:cmp:diff:ineq}
\dot{u}_{ij}(t) \ge [\Delta^+_\Lambda u(t)]_{ij} + g\big(u_{ij}(t)\big),
\qquad \dot{v}_{ij}(t) \le [\Delta^+_\Lambda v(t)]_{ij} + g\big(v_{ij}(t)\big).
\end{equation}
}
\item[(b)]{
We have the inequality $u_{ij}(t) \ge v_{ij}(t)$.
}
\end{itemize}
Then we in fact have $u_{ij}(t) \ge v_{ij}(t)$
for all $(i,j) \in \Lambda$ and $t \ge 0$.
\end{prop}

\begin{cor}
\label{cor:prlm:cmpStrong}
Consider the setting of
Proposition \ref{prp:prlm:cmpPrinciple}.
Suppose that the obstacle $K_{\mathrm{obs}}$ satisfies (HK1)
and that there exists $(i_0, j_0) \in \Lambda$
for which $u_{i_0 j_0}(0) > v_{i_0 j_0}(0)$.
Then we have the strict inequality
\begin{equation}
u_{ij}(t) > v_{ij}(t), \qquad (i,j) \in \Lambda, \qquad t > 0.
\end{equation}
\end{cor}
\begin{proof}
Suppose that for some $t_* > 0$ and $(i_*, j_*) \in \Lambda$,
we have $u_{i_* j_*}(t_*) = v_{i_* j_*}(t_*)$.
Since $u_{ij}(t) \ge v_{ij}(t)$ for all $t \ge 0$ and $(i,j) \in \Lambda$,
we must have $\dot{u}_{i_* j_*}(t_*) = \dot{v}_{i_* j_*}(t_*)$.
In particular, this implies
\begin{equation}
[\Delta^+_\Lambda u(t_*)]_{i_* j_*} = [\Delta^+_\Lambda v(t_*)]_{i_* j_*},
\end{equation}
which in turn shows that
we must have $u_{ij}(t_*) = v_{ij}(t_*)$
for all $(i,j) \in \mathcal{N}_{\Lambda}(i_*, j_*)$.
Exploiting the connectedness of $\Lambda$,
this argument can be repeated to show that
$u(t_*) = v(t_*)$, which contradicts
the uniqueness of solutions to \sref{eq:prlm:lde:obs}
in backward time.
\end{proof}

We now turn our attention to travelling
wave solutions of the unobstructed system \sref{eq:prlm:lde:obs}
with $\Lambda = \Wholes^2$.
To this end, we pick an arbitrary pair $(\sigma_h, \sigma_v) \in \Real^2$,
assuming only that $\sigma_h^2 + \sigma_v^2 \neq 0$.
Inserting the travelling wave Ansatz
\begin{equation}
\label{eq:prlm:trvWaveAnsatz}
u_{ij}(t) = \Phi( i \sigma_h + j \sigma_v + ct)
\end{equation}
into the homogeneous LDE \sref{eq:prlm:lde:obs}, we arrive
at the travelling wave MFDE
\begin{equation}
\label{eq:prlm:trvWaveMFDE}
\begin{array}{lcl}
c\Phi'(\xi) & = &
 \Phi(\xi + \sigma_h) + \Phi(\xi + \sigma_v) + \Phi(\xi - \sigma_h) + \Phi(\xi - \sigma_v) - 4 \Phi(\xi)
 + g\big(\Phi(\xi) \big).
\end{array}
\end{equation}
The following assumption
relating to the existence of solutions to \sref{eq:prlm:trvWaveMFDE} with non-zero wave speed
is used frequently throughout this entire paper.
\begin{itemize}
\item[$\textrm{(h}\Phi\textrm{)}_{\textrm{\S\ref{sec:prlm}}}$]{
The travelling wave system \sref{eq:prlm:trvWaveMFDE}
admits a solution $(c, \Phi)$ for some $c \neq 0$
and bounded function $\Phi \in C^2(\Real, \Real)$
that satisfies the limits
\begin{equation}
\lim_{\xi \to -\infty} \Phi(\xi) = 0, \qquad \lim_{\xi \to + \infty} \Phi(\xi) = 1.
\end{equation}
}
\end{itemize}
We note that the $C^1$-smoothness of the nonlinearity $g$
prescribed by $\textrm{(hg}\textrm{)}_{\textrm{\S\ref{sec:prlm}}}$
ensures that the $C^2$-continuity mentioned in
$\textrm{(h}\Phi\textrm{)}_{\textrm{\S\ref{sec:prlm}}}$
is automatic upon assuming that $\Phi$ is merely continuous.

An important role is played by the asymptotic
rates at which the wave profile $\Phi$ approaches its limiting values.
To study these rates, we introduce the
limiting spatial characteristic functions
\begin{equation}
\label{eq:prlm:spCharEqns}
\begin{array}{lcl}
\Delta^-(z) & = &
cz - (2 \cosh(\sigma_h z)  + 2 \cosh(\sigma_v z) - 4) - g'(0),
\\[0.2cm]
\Delta^+(z) & = &
cz - (2 \cosh(\sigma_h z)  + 2 \cosh(\sigma_v z) - 4) - g'(1).
\end{array}
\end{equation}
It is well known that the real roots of the equations $\Delta^\pm(z) = 0$
are directly related to the asymptotic convergence rates discussed above.
\begin{lem}
\label{lem:prlm:defSpatExps}
Consider the characteristic equations
\sref{eq:prlm:spCharEqns}
and suppose that $\textrm{(hg}\textrm{)}_{\textrm{\S\ref{sec:prlm}}}$
and $\textrm{(h}\Phi\textrm{)}_{\textrm{\S\ref{sec:prlm}}}$
both hold.
Then there exist constants $\eta^\pm_\Phi > 0$ that
satisfy the identities
\begin{equation}
\label{eq:lem:prlm:defSpatExps:idForEta}
\begin{array}{lcl}
c \eta^-_\Phi & = & 2 \cosh(\sigma_h \eta^-_\Phi)
   + 2 \cosh(\sigma_v \eta^-_\Phi) - 4 + g'(0),
\\[0.2cm]
-c \eta^+_\Phi & = & 2 \cosh(\sigma_h \eta^+_\Phi) + 2 \cosh(\sigma_v \eta^+_\Phi) - 4 + g'(1),
\\[0.2cm]
\end{array}
\end{equation}
which implies that
\begin{equation}
\begin{array}{lcl}
\Delta^+( - \eta^+_\Phi) = \Delta^-(\eta^-_\Phi) = 0.
\end{array}
\end{equation}
In addition, if either $\Delta^+( - \eta) = 0$
or $\Delta^-(\eta) = 0$ holds for any $\eta \ge 0$, then
we must have $\eta = \eta^+_\Phi$ or $\eta = \eta^-_\Phi$
respectively.

Finally, if the inequalities
\begin{equation}
c > 0, \qquad  \eta^-_\Phi \le \eta^+_\Phi
\end{equation}
are both satisfied,
then the inequality
\begin{equation}
g'(0) > g'(1)
\end{equation}
must also hold.
\end{lem}
\begin{proof}
The statements follow directly from the observation that
$[\Delta^\pm]''(z) < 0$ for all $z \in \Real$,
together with the limits $\lim_{z \to \pm \infty} \Delta^\pm(z) = - \infty$
and the inequalities $\Delta^\pm(0) > 0$.
\end{proof}

\begin{prop}
\label{prp:prlm:asymEsts}
Consider the travelling wave MFDE
\sref{eq:prlm:trvWaveMFDE},
assume that $\textrm{(hg}\textrm{)}_{\textrm{\S\ref{sec:prlm}}}$
and $\textrm{(h}\Phi\textrm{)}_{\textrm{\S\ref{sec:prlm}}}$
both hold and recall the spatial exponents $\eta^\pm_\Phi$
defined in Lemma \ref{lem:prlm:defSpatExps}.
Then there exist constants $K_\Phi > 1$, $\kappa_\Phi > 0$
and $C^\pm_\Phi > 0$
such that for every $\xi \le 0$
we have
\begin{equation}
\label{eq:prlm:asymEsts:estOnMinus}
\begin{array}{lcl}
\abs{ \Phi(\xi) - C^-_\Phi e^{ - \eta^-_\Phi \abs{\xi} }  } & \le & K_\Phi e^{ - (\eta^-_\Phi + \kappa_\Phi) \abs{\xi} },
\\[0.2cm]
\abs{ \Phi'(\xi) - \eta^-_\Phi C^-_\Phi e^{ - \eta^-_\Phi \abs{\xi} }  } & \le & K_\Phi e^{ - (\eta^-_\Phi + \kappa_\Phi) \abs{\xi} },
\\[0.2cm]
\abs{ \Phi''(\xi) - [\eta^-_\Phi]^2 C^-_\Phi e^{ - \eta^-_\Phi \abs{\xi} }  } & \le & K_\Phi e^{ - (\eta^-_\Phi + \kappa_\Phi) \abs{\xi} },
\\[0.2cm]
\end{array}
\end{equation}
while for every $\xi \ge 0$ we have
\begin{equation}
\label{eq:prlm:asymEsts:estOnPlus}
\begin{array}{lcl}
\abs{ (1 -  \Phi(\xi) ) - C^+_\Phi e^{ - \eta^+_\Phi \abs{\xi} }  } & \le & K_\Phi e^{ - (\eta^+_\Phi + \kappa_\Phi) \abs{\xi} },
\\[0.2cm]
\abs{ \Phi'(\xi) - \eta^+_\Phi C^+_\Phi e^{ - \eta^+_\Phi \abs{\xi} }  } & \le & K_\Phi e^{ - (\eta^+_\Phi + \kappa_\Phi) \abs{\xi} },
\\[0.2cm]
\abs{ \Phi''(\xi) + [\eta^+_\Phi]^2 C^+_\Phi e^{ - \eta^+_\Phi \abs{\xi} }  } & \le & K_\Phi e^{ - (\eta^+_\Phi + \kappa_\Phi) \abs{\xi} }.
\\[0.2cm]
\end{array}
\end{equation}
\end{prop}
\begin{proof}
These bounds follow directly from \cite[Thm. 2.2]{MPB}.
\end{proof}

\begin{cor}
\label{cor:prlm:estOnWave}
Consider the setting of Proposition \ref{prp:prlm:asymEsts}.
Then there exist constants
\begin{equation}
0 < \alpha^\pm_{\mathrm{low}} < \alpha^\pm_{\mathrm{up}}, \qquad 0 < \beta^\pm_{\mathrm{low}} < \beta^\pm_{\mathrm{up}}
\end{equation}
such that for every $\xi \le 0$
we have
\begin{equation}
\begin{array}{lclcl}
\beta^-_{\mathrm{low}} e^{ - \eta^-_\Phi \abs{\xi} } & \le & \Phi(\xi)
  & \le & \beta^-_{\mathrm{up}} e^{ - \eta^-_\Phi \abs{\xi} },
\\[0.2cm]
\alpha^-_{\mathrm{low}} e^{ - \eta^-_\Phi \abs{\xi} } & \le & \Phi'(\xi)
 & \le & \alpha^-_{\mathrm{up}} e^{ - \eta^-_\Phi \abs{\xi} },
\\[0.2cm]
\end{array}
\end{equation}
while for every $\xi \ge 0$ we have
\begin{equation}
\begin{array}{lclcl}
\beta^+_{\mathrm{low}} e^{ - \eta^+_\Phi \abs{\xi} } & \le &
  1 - \Phi(\xi) & \le & \beta^+_{\mathrm{up}} e^{ - \eta^+_\Phi \abs{\xi} }
\\[0.2cm]
\alpha^+_{\mathrm{low}}  e^{ - \eta^+_\Phi \abs{\xi} }
  & \le & \Phi'(\xi) &  \le &  \alpha^+_{\mathrm{up}} e^{ - \eta^+_\Phi \abs{\xi} }.
\\[0.2cm]
\end{array}
\end{equation}
\end{cor}
\begin{proof}
These identities follow directly from
Proposition \ref{prp:prlm:asymEsts},
upon exploiting the fact that
the inequalities $0 < \Phi(\xi) < 1$
and $\Phi'(\xi) > 0$ hold for all $\xi \in \Real$.
\end{proof}

\begin{cor}
\label{cor:prlm:shifts}
Consider the travelling wave MFDE
\sref{eq:prlm:trvWaveMFDE}
and assume that $\textrm{(hg}\textrm{)}_{\textrm{\S\ref{sec:prlm}}}$ and
$\textrm{(h}\Phi\textrm{)}_{\textrm{\S\ref{sec:prlm}}}$ both hold.
Then for every $M > 1$, there exists a constant
$K_{\mathrm{shift}} = K_{\mathrm{shift}}(M)$
so that
\begin{equation}
\abs{\frac{ \Phi''(\zeta) }{\Phi'(\xi)} } +
\abs{ \frac{ \Phi'(\zeta)}{ \Phi'(\xi)} } \le K_{\mathrm{shift}}
\end{equation}
holds for every pair $(\zeta, \xi) \in \Real^2$
for which $\abs{\zeta - \xi} \le M$.
\end{cor}
\begin{proof}
This follows from the fact that $\Phi'(\xi) > 0$ for all $\xi \in \Real$
together with the asymptotic bounds stated in Proposition \ref{prp:prlm:asymEsts}.
\end{proof}

Our final main result in this section
roughly states that the properties described above for the wave $(c, \Phi)$
vary continuously upon changing the direction $(\sigma_h, \sigma_v)$
and perturbing the nonlinearity $g$. We note that
the results in \cite[Thm. 2.1]{MPB} cover
neither variations in the direction of propagation
nor smoothness properties of asymptotic expansions, so
we take the opportunity here to discuss
these issues in depth.

In order to state the result, we introduce for any $\eta \in \Real$
and any interval $\mathcal{I} \subset \Real$
the exponentially weighted function spaces
\begin{equation}
\begin{array}{lcl}
BC_{ \eta } ( \mathcal{I}, \Real)  & = & \{ p \in C( \mathcal{I}, \Real) \mid
\sup_{\xi \in \mathcal{I} } e^{ - \eta \abs{\xi } } \abs{p(\xi) } < \infty \},
\\[0.2cm]
BC^1_{ \eta } ( \mathcal{I}, \Real )  & = & \{ p \in C^1(\mathcal{I}, \Real) \mid
\sup_{\xi \in \mathcal{I} } e^{ - \eta \abs{\xi } } [\abs{p(\xi) } + \abs{p'(\xi)} ]
< \infty \}.
\end{array}
\end{equation}

\begin{prop}
\label{prp:prlm:subsupWithDelta}
Consider the travelling wave MFDE
\sref{eq:prlm:trvWaveMFDE}
and suppose that
$\textrm{(hg}\textrm{)}_{\textrm{\S\ref{sec:prlm}}}$
and $\textrm{(h}\Phi\textrm{)}_{\textrm{\S\ref{sec:prlm}}}$ with $c > 0$ are both satisfied.
Fix $\delta_p > 0$ sufficiently small
and consider the set
\begin{equation}
\Omega = \{ ( \delta, \sigma_h', \sigma_v') : 0 \le \delta < \delta_p \hbox{ and } \abs{\sigma_h' - \sigma_h} + \abs{\sigma_v' - \sigma_v} < \delta_p \}.
\end{equation}
Then for any $p = (\delta, \sigma_h', \sigma_v') \in \Omega$,
there exist $C^2$-smooth functions $\Phi^\pm_{p}: \Real \to \Real$
and constants $c^\pm_{p} > 0$
such that the following properties hold.
\begin{itemize}
\item[(i)]{
  For any $p = (\delta, \sigma_h', \sigma_v') \in \Omega$,
  the MFDEs
  \begin{equation}
    \label{eq:prlm:mfdePhiDeltaC}
    \begin{array}{lcl}
    c^\pm_p [\Phi^\pm_{p}]'(\xi) & = &
       \Phi^\pm_{p}(\xi + \sigma_h') + \Phi^\pm_{p}(\xi + \sigma_v')
            + \Phi^\pm_{p}(\xi - \sigma_h') + \Phi^\pm_{p}(\xi - \sigma_v')
            - 4 \Phi^\pm_{p}(\xi)
      \\[0.2cm]
      & & \qquad
        + g^\pm_{\delta}\big( \Phi^\pm_p(\xi) \big)
    \end{array}
  \end{equation}
  are satisfied for all $\xi \in \Real$.
  Here the maps $u \to g^\pm_\delta(u \mp \delta)$
  satisfy $\textrm{(hg}\textrm{)}_{\textrm{\S\ref{sec:prlm}}}$,
  while for any $u \in \Real$ we have $g^\pm_0(u) = g(u)$
  and $g^-_{\delta}(u) \le g(u) \le g^+_{\delta}(u)$.
}
\item[(ii)]{
  For any $p = (\delta, \sigma_h', \sigma_v') \in \Omega$,
  we have $[\Phi^\pm_p]'(\xi) > 0$, together with the four limits
  \begin{equation}
    \begin{array}{lclclcl}
    \lim_{\xi \to - \infty} \Phi^-_{p}(\xi)  &= &  - \delta,
     & & \lim_{\xi \to + \infty} \Phi^-_{p}(\xi)  &= &  1- \delta,
     \\[0.2cm]
     \lim_{\xi \to - \infty} \Phi^+_{p}(\xi)  &= &  + \delta,
     & & \lim_{\xi \to + \infty} \Phi^+_{p}(\xi)  &= &  1 + \delta.
     \end{array}
  \end{equation}
}
\item[(iii)]{
  For any $p = (\delta, \sigma_h', \sigma_v') \in \Omega$,
  the functions $W^\pm_p: \Real \to \ell^\infty(\Wholes^2, \Real)$
  defined by
  \begin{equation}
    [W^\pm_p]_{ij}(t) = \Phi^\pm_{p}(\sigma_h' i + \sigma_v' j + c^\pm_p t )
  \end{equation}
  satisfy the differential inequalities
  \begin{equation}
    \label{eq:prp:prlm:subsupWithDelta:diffIneq}
    \mathcal{J}^-_{ij}(t) \le 0 \le \mathcal{J}^+_{ij}(t),
  \end{equation}
  in which
  \begin{equation}
  \begin{array}{lcl}
    \mathcal{J}^\pm_{ij}(t) & = &
      [\dot{W}^\pm_p]_{ij}(t) - [\Delta^+ W^\pm_p(t)]_{ij} - g\big( [W^\pm_p]_{ij}(t) \big).
  \end{array}
  \end{equation}
  In addition, if $\delta = 0$, then we have $W^-_p = W^+_p$
  and the inequalities in \sref{eq:prp:prlm:subsupWithDelta:diffIneq} are equalities.
}
\item[(iv)]{
 Upon writing $\eta_\Phi = \min\{ \eta^+_{\Phi}, \eta^-_\Phi \}$, the maps
 \begin{equation}
   (\nu, \sigma_h', \sigma_v') \mapsto
     \left\{ \begin{array}{l}
         c^\pm_{\nu^2, \sigma_h', \sigma_v'} \in \Real
         \\[0.2cm]
          \Phi^-_{\nu^2, \sigma_h', \sigma_v'} + \nu^2 - \Phi \in BC_{-\frac{1}{2}\eta_\Phi}(\Real, \Real)
         \\[0.2cm]
           \Phi^+_{\nu^2, \sigma_h', \sigma_v'} - \nu^2 - \Phi \in BC_{-\frac{1}{2}\eta_\Phi}(\Real, \Real)
        \end{array}
     \right.
 \end{equation}
 are $C^1$-smooth,
 with
 \begin{equation}
    c^\pm_{0, \sigma_h, \sigma_v} = c, \qquad  \Phi^\pm_{0, \sigma_h, \sigma_v} = \Phi.
 \end{equation}
}
\item[(v)]{
 For any $p = (\delta, \sigma_h', \sigma_v') \in \Omega$,
 the function $\Phi^-_p + \delta$
 satisfies the asymptotic estimates
 \sref{eq:prlm:asymEsts:estOnMinus}-\sref{eq:prlm:asymEsts:estOnPlus} with
 constants $K_\Phi > 1$ and $\kappa_\Phi > 0$
 that are independent of $p$,
 constants $\eta^\pm_\Phi > 0$
 that depend $C^1$-smoothly on $(\sigma_h', \sigma_v')$
 but are independent of $\delta$,
 together with constants $C_\Phi^\pm$ that depend
 $C^1$-smoothly on $(\sqrt{\delta}, \sigma_h', \sigma_v')$.
 A similar statement holds for the functions $\Phi^+_p - \delta$.
}
\end{itemize}
\end{prop}

In the remainder of this section
we provide the missing proofs for the results stated above.
We start by establishing the weak comparison principle,
closely following the arguments in \cite[Prop. 4.1]{HJHNEGDIF},
which in turn are based on \cite{CHEN1997}.

\begin{proof}[Proof of Proposition \ref{prp:prlm:cmpPrinciple}]
Upon writing $w_{ij}(t) = u_{ij}(t) - v_{ij}(t)$
together with
\begin{equation}
\mathcal{I}_{ij}(t) =  \int_0^1 g'\big( v_{ij}(t) + \vartheta w_{ij}(t) \big) d \vartheta,
\end{equation}
the estimate
\begin{equation}
\label{eq:lem:cmp:int:idForW}
\begin{array}{lcl}
\dot{w}_{ij}(t)  & \ge & [\Delta^+_\Lambda w(t)]_{ij} + g\big(u_{ij}(t)\big) - g\big(v_{ij}(t)\big) \\[0.2cm]
& = & [ \Delta^+_\Lambda w(t)]_{ij} +  \mathcal{I}_{ij}(t) w_{ij}(t)
\end{array}
\end{equation}
holds for all $(i,j) \in \Lambda$ and $t \ge 0$
for which condition (a) is satisfied.

In order to show that $w_{ij}(t) \ge 0$ for all
$t \ge 0$ and $(i,j) \in \Lambda$,
let us assume to the contrary that this is false.
In particular, suppose that there exist
$t_* > 0$, $(i_*, j_*) \in \Lambda$
for which $w_{i_*, j_*}(t_*) = - \vartheta < 0$.
Picking $\epsilon > 0$ and $K' > 0$ in such a way that
$\vartheta = \epsilon e^{ 2 K' t_*}$, we can now define
\begin{equation}
T' : = \sup \{ t \ge 0 \mid w_{ij}( t) >
- \epsilon e^{2 K' t}  \hbox{ for all } (i,j) \in \Lambda \}.
\end{equation}
The $C^1$-smoothness of $w$ guarantees that
$0 < T' \le t_*$.
In addition, we have
\begin{equation}
\inf_{ (i,j) \in \Lambda} w_{ij}(T') = -\epsilon e^{2 K' T'},
\end{equation}
since otherwise the
smoothness of $w$ as map into $\ell^{\infty}(\Lambda, \Real)$
would allow the constant $T'$ to be increased.
Without loss of generality we may therefore assume that
$(0,0) \in \Lambda$ and $w_{0, 0}( T') < - \frac{7}{8} \epsilon e^{2 K' T'}$.

Consider now the function
\begin{equation}
 w^-_{ij}(t; \beta) = - \epsilon \big(\frac{3}{4} + \beta z_{ij} \big) e^{2 K' t},
\end{equation}
in which $\beta > 0$ is a parameter and $z \in \ell^\infty( \Lambda;  \Real)$
has $z_{0,0} = 1$,
$\lim_{ \abs{i} + \abs{j} \to \infty} z_{ij} = 3$,
$1 \le z \le 3$ and $\abs{\Delta^+_\Lambda z} \le 1$.
Write $\beta_* \in (\frac{1}{8}, \frac{1}{4}]$
for the minimal value of $\beta$ for which $w_{ij}(t) \ge w^-_{ij}(t ;  \beta)$ holds
for all $(i,j,t) \in \Lambda \times [0, T']$.
Since
\begin{equation}
\lim_{\abs{i} + \abs{j} \to \infty} w^-_{ij}( t;  \beta_*)
 = -\epsilon [ \frac{3}{4} + 3 \beta_* ] e^{2 K' t}
 < -\frac{9}{8}\epsilon e^{2 K' t} ,
\end{equation}
there exist $(i_0, j_0) \in \Lambda$ and $0 < t_0 \le T'$ such that
$w_{i_0, j_0}(t_0) = w^-_{i_0, j_0}( t_0; \beta_*)$. The definition of $\beta_*$
now implies that
\begin{equation}
\begin{array}{lcl}
\dot{w}_{i_0, j_0}(t_0) & \le & \dot{w}^-_{i_0, j_0}( t_0; \beta_*), \\[0.2cm]
\end{array}
\end{equation}
In addition, by positivity of the off-diagonal coefficients
in $\Delta^+_\Lambda$, we have
\begin{equation}
\begin{array}{lcl}
[\Delta^+_\Lambda w(t_0)]_{i_0, j_0}
& \ge &
  [\Delta^+_\Lambda w^-(t_0; \beta_*)]_{i_0, j_0}. \\[0.2cm]
\end{array}
\end{equation}
Using the fact that $w_{i_0, j_0}(t_0) < 0$,
we see that (a) and hence \sref{eq:lem:cmp:int:idForW}
is satisfied, which
leads to the estimate
\begin{equation}
\begin{array}{lcl}
-\frac{7}{4} \epsilon K' e^{2 K' t_0}
& \ge & \dot{w}^-_{i_0, j_0}(t_0) \ge \dot{w}_{i_0, j_0}(t_0) \\[0.2cm]
& \ge & [\Delta^+_\Lambda w(t_0)]_{i_0, j_0} +  \mathcal{I}_{i_0,j_0}(t_0) w_{i_0, j_0}(t_0) \\[0.2cm]
%
& \ge &
  [\Delta^+_\Lambda w^-(t_0; \beta_*)]_{i_0, j_0}
     +  \mathcal{I}_{i_0,j_0}(t_0) w^-_{i_0, j_0}(t_0 ; \beta_*). \\[0.2cm]
\end{array}
\end{equation}
In particular, we obtain the bound
\begin{equation}
-\frac{7}{4} \epsilon K' e^{2 K' t_0} \ge - 3 \epsilon  \big[ 1 + M' \big]  e^{2 K' t_0},
\end{equation}
in which
\begin{equation}
M' = \sup_{-1 \le u \le 2 } \abs{g'(u) }.
\end{equation}
This leads to a contradiction upon choosing $K' \gg 1$ to be sufficiently large,
showing that indeed $w_{ij}(t) \ge 0$ for all $(i,j) \in \Lambda$ and $ t\ge 0$.
\end{proof}

We now turn our attention to the results stated
in Proposition \ref{prp:prlm:subsupWithDelta}.
Our first concern is to construct the distorted nonlinearities
mentioned in item (i).
\begin{lem}
\label{lem:prlm:exTauFncs}
There exist two $C^1$-smooth functions
\begin{equation}
\tau^\pm: \Real \times (-\infty, \frac{1}{12}] \to \Real
\end{equation}
that satisfy the following properties.
\begin{itemize}
\item[(i)]{
  For any $\nu \in(0, \frac{1}{12})$ and any $u \in [\nu, 1 - \nu]$, we have
  $\tau^\pm(u, \nu) = u$.
}
\item[(ii)]{
  For any $\nu \in (0, \frac{1}{12})$ and any $u \in (-\infty, -\nu^2) \cup (1 - \nu^2 , \infty)$,
  we have
  \begin{equation}
    \tau^+(u, \nu) = u + \nu^2,
  \end{equation}
  while for any $\nu \in (0, \frac{1}{12})$ and any $u \in (-\infty, \nu^2) \cup (1 + \nu^2, \infty)$,
  we have
  \begin{equation}
    \tau^-(u, \nu) = u - \nu^2.
  \end{equation}
}
\item[(iii)]{
  For any $\nu \le 0$, we have $\tau^\pm(s, \nu) = s$ for all $s \in \Real$.
}
\item[(iv)]{
  For any $0 \le \nu \le \frac{1}{12}$ and $u \in \Real$, we have
  the bound
  \begin{equation}
    \label{eq:prlm:defTauPlus:bndDerivative}
   \frac{1}{2} \le 1 - 6 \nu \le \partial_u \tau^\pm(u, \nu) \le 1,
  \end{equation}
  together with
  \begin{equation}
     \tau^-(u, \nu) \le u \le \tau^+(u, \nu).
  \end{equation}
}
\item[(v)]{
  There exists a constant $C_1 > 1$ such that
  \begin{equation}
    [\partial_u \tau^\pm(u, \nu) - \partial_u \tau^\pm(v, \nu) ] \le C_1 \abs{u - v}
  \end{equation}
  holds for all $\nu \le \frac{1}{12}$ and $(u,v) \in \Real^2$.
}
\end{itemize}
\end{lem}
\begin{proof}
We restrict ourselves to defining a function $\tau^+: \Real^2 \to \Real$
that satisfies the stated properties whenever $u \le \frac{1}{2}$ and $\nu \le \frac{1}{12}$.
To this end, we define the three open sets
\begin{equation}
\begin{array}{lcl}
\mathcal{V}_1 & = & \{\nu < 0 \} \cup \{ u >  \nu \}  \subset \Real^2,
\\[0.2cm]
\mathcal{V}_2 & = & \{\nu > 0 \hbox{ and } u < - \nu^2 \} \subset \Real^2,
\\[0.2cm]
\mathcal{V}_3 & = & \{\nu > 0 \hbox{ and }  - \nu^2 < u < \nu \} \subset \Real^2,
\\[0.2cm]
\end{array}
\end{equation}
together with the three smooth functions
\begin{equation}
\begin{array}{lcl}
\tau_1 : \Real^2 \to \Real &  & (u, \nu) \mapsto  u,
\\[0.2cm]
\tau_2 : \Real^2 \to \Real &  & (u, \nu) \mapsto u + \nu^2,
\\[0.2cm]
\tau_3 : \Real \times \{ \nu > 0 \} \to \Real & & (u, \nu) \mapsto
\int_{-\nu^2}^u [1 + 6 \nu^{-1} (1 + \nu)^{-3}(s + \nu^2)(s - \nu) ] \, d s.
\\[0.2cm]
\end{array}
\end{equation}
Upon writing $\tau^+(u, \nu) = \tau_i (u, \nu)$ whenever $(u, \nu) \in \mathcal{V}_i$,
we have constructed a $C^1$-smooth function $\tau^+: \mathcal{V} \to \Real$,
where $\mathcal{V} = \mathcal{V}_1 \cup \mathcal{V}_2 \cup \mathcal{V}_3$.

We now set out to show that $\tau^+$ can be extended to the
boundary $\partial \mathcal{V}$ in a smooth fashion.
First of all, notice that
\begin{equation}
\begin{array}{lcl}
\partial_u \tau_3(u, \nu) & = &
1 + 6 \nu^{-1}(1 +\nu)^{-3}(u + \nu^2)(u - \nu)
\\[0.2cm]
& \ge &
1 - 6 \nu^{-1}(1 + \nu)^{-3}\frac{1}{4} (\nu + \nu^2)^2
\\[0.2cm]
& \ge &
1 - 6 \nu \frac{1}{4}(1 + \nu)^2
\\[0.2cm]
& \ge &
1 -6 \nu
\\[0.2cm]
& \ge & \frac{1}{2}.
\end{array}
\end{equation}
In addition,  for any $\nu > 0$ we have
\begin{equation}
\partial_u \tau_3(-\nu^2 ,\nu) =
\partial_u \tau_3(\nu ,\nu) = 1
\end{equation}
by construction,
while a short computation shows that
\begin{equation}
\tau_3(-\nu^2, \nu) = 0, \qquad \tau_3( \nu , \nu) = \nu.
\end{equation}
Differentiating these last two identities with respect to $\nu$,
we obtain
\begin{equation}
-2 \nu \partial_{u} \tau_3( - \nu^2, \nu)
+ \partial_{\nu} \tau_3(-\nu^2, \nu) = 0,
\qquad
\partial_u \tau_3(\nu, \nu) + \partial_{\nu}\tau_3(\nu, \nu) = 1,
\end{equation}
which shows that for all $\nu > 0$ we have
\begin{equation}
\partial_\nu \tau_3(-\nu^2, \nu) = 2 \nu,
\qquad
\partial_{\nu}\tau_3(\nu, \nu) = 0.
\end{equation}
This suffices to show that $\tau^+$ can be extended to a $C^1$-smooth
function on $\Real^2 \setminus \{0, 0\}$.

In order to establish $C^1$-smoothness at $(0,0)$, we compute
\begin{equation}
\tau_3(u, \nu) =
u + \nu^2 + 2 \nu^{-1} (1 + \nu)^{-3} u^3
+ (1 + \nu)^{-3}\big[ 2 \nu^5 + 3 (\nu - 1)( u^2 -  \nu^4) -6 \nu^2 (u + \nu^2) \big].
\end{equation}
We define $\tau^+(0, 0) = 0$, which together with
\begin{equation}
\lim_{\nu \downarrow 0} \frac{1}{\nu} \tau^+(0, \nu)
= \lim_{\nu \downarrow 0} \frac{1}{\nu} \tau_3(0, \nu) = 0
\end{equation}
allows us to conclude
\begin{equation}
\partial_u \tau^+(0, 0) =1, \qquad \partial_\nu \tau^+(0, 0) = 0.
\end{equation}
It hence remains to show that for any sequence $\{(u_k, \nu_k)\} \subset \mathcal{V}_3$
with $(u_k, \nu_k) \to (0, 0)$ we have $\partial_u \tau_3(u_k, \nu_k) \to 1$
and $\partial_\nu \tau_3(u_k, \nu_k) \to 0$.
We compute
\begin{equation}
\label{eq:prlm:defTauPM:defTau3}
\begin{array}{lcl}
\partial_u \tau_3(u, \nu) & = &
1  + 6 \nu^{-1} (1 + \nu)^{-3} u^2
+ (1 + \nu)^{-3}\big[  6 (\nu - 1) u  -6 \nu^2  \big],
\\[0.2cm]
\partial_\nu \tau_3(u, \nu) & = &
2\nu - 2 \nu^{-2} (1 + \nu)^{-3} u^3
- 6 \nu^{-1} (1 + \nu)^{-4} u^3
\\[0.2cm]
& & \qquad -3 (1 + \nu)^{-4}\big[ 2 \nu^5 + 3 (\nu - 1)( u^2 -  \nu^4) -6 \nu^2 (u + \nu^2) \big]
\\[0.2cm]
& & \qquad + (1 + \nu)^{-3}\big[ 10 \nu^4 + 3 ( u^2 -  \nu^4) - 12(\nu -1)\nu^3 -12 \nu (u + \nu^2) - 12 \nu^3 \big].
\\[0.2cm]
\end{array}
\end{equation}
In view of the fact that $\abs{ u_k } \le \nu_k \le 1$ and $\nu_k \downarrow 0$,
the desired limits can now be read off.

Finally, the Lipschitz property (v) can be easily verified using
\sref{eq:prlm:defTauPM:defTau3}.
\end{proof}

\begin{lem}
\label{lem:prlm:nlGDeltaPM}
Suppose that $\textrm{(hg}\textrm{)}_{\textrm{\S\ref{sec:prlm}}}$ is satisfied.
For any sufficiently small $\delta > 0$, there exist
nonlinearities $g^\pm_\delta: \Real \to \Real$
such that the following properties are satisfied.
\begin{itemize}
\item[(i)]{
For all $u \in \Real$ we have the inequalities
\begin{equation}
\label{eq:lem:prlm:nlGDeltaPM:ineqOnGPM}
g^-_\delta(u) \le g(u) \le g^+_\delta(u).
\end{equation}
}
\item[(ii)]{
  Recalling the constant $a$ appearing in $\textrm{(hg}\textrm{)}_{\textrm{\S\ref{sec:prlm}}}$, we have
  the identities
  \begin{equation}
    g^-_\delta( - \delta)  = g^-_\delta(a) = g^-_{\delta}( 1 - \delta) = 0,
    \qquad g^+_\delta( \delta)  = g^+_\delta(a) = g^+_{\delta}( 1 + \delta) = 0,
   \end{equation}
   together with the inequalities
 \begin{equation}
 \begin{array}{lcl}
   g^\pm_\delta(u )  > 0, \qquad u \in (-\infty, \pm \delta) \cup (a , 1 \pm \delta)
   \\[0.2cm]
   g^\pm_\delta(u )  < 0, \qquad u \in (\pm \delta , a ) \cup (1 \pm \delta, \infty).
   \\[0.2cm]
 \end{array}
 \end{equation}
}
\item[(iii)]{
We have the equalities
\begin{equation}
Dg^-_\delta( - \delta) = Dg(0) = Dg^+_\delta( + \delta),
\qquad Dg^-_\delta(1 - \delta) = Dg(1) = Dg^+_\delta( 1 + \delta).
\end{equation}
}
\item[(iv)]{
  The maps $(u, \nu) \mapsto g^\pm_{\nu^2}(u)$ are $C^1$-smooth.
  In addition, for fixed $\delta > 0$ the maps
  $u \mapsto D g^\pm_{\delta}(u)$ are locally Lipschitz.
}
\item[(v)]{
  There exists $\kappa_{\mathrm{dis}} > 0$ such that
  we have
  \begin{equation}
    g^-_{\delta}( u) - g(u) \le - \kappa_{\mathrm{dis}} \delta
  \end{equation}
  for any $-\delta \le u \le 0$, together with
  \begin{equation}
    g^+_{\delta}(u) - g(u) \ge \kappa_{\mathrm{dis}} \delta
  \end{equation}
  for any $1 \le u \le 1 + \delta$.
}
\end{itemize}
\end{lem}
\begin{proof}
Upon writing
\begin{equation}
g^-_\delta(u) = g\big( \tau^+( u, \sqrt{\delta})\big),
\qquad
g^+_\delta(u) = g\big( \tau^-( u, \sqrt{\delta}) \big),
\end{equation}
the properties (i) through (iv) follow immediately
from Lemma \ref{lem:prlm:exTauFncs}.

Addressing item (v), we note that there exists $\delta_u > 0$
such that $g'(u) < \frac{1}{2}g'(0) < 0$ for all $\abs{u} < \delta_u$.
In addition, since
\begin{equation}
Dg^-_{\delta}(u) = Dg\big(\tau^+(u, \sqrt{\delta})\big) \partial_u \tau^+(u, \sqrt{\delta})
\end{equation}
with $u \le \tau^+(u, \sqrt{\delta}) \le u + \delta$
and $\frac{1}{2} \le \partial_u \tau^+(u,\sqrt{\delta}) \le 1$,
we see
that we can pick $\kappa' > 0$ in such a way that
\begin{equation}
Dg^-_{\delta}(u) \le - \kappa', \qquad Dg(u) \le - \kappa',
\qquad -\delta \le u \le 0,
\end{equation}
possibly after restricting $\delta > 0$.
We now find, for any $-\delta \le u \le 0$,
\begin{equation}
g^-_{\delta}(u) - g(u)
= Dg^-_{\delta}(u_1)(u + \delta) - Dg(u_2) u,
\end{equation}
for some $-\delta < u_1 < u < u_2 < 0$.
In particular, if $-\frac{\delta}{2} \le u \le 0$,
we have
\begin{equation}
g^-_{\delta}(u) - g(u) \le Dg^-_{\delta}(u_1)(u + \delta) \le - \kappa' \frac{\delta}{2},
\end{equation}
while if $-\delta \le u \le -\frac{\delta}{2}$ we have
\begin{equation}
g^-_{\delta}(u) - g(u) \le -Dg(u_2) u \le - \kappa' \frac{\delta}{2}.
\end{equation}
The inequality for $g^+_{\delta}$ follows analogously.
\end{proof}

Linearizing the travelling wave MFDE \sref{eq:prlm:trvWaveMFDE} around the wave $(c, \Phi)$,
we arrive at the homogeneous MFDE
\begin{equation}
\label{eq:prlm:linTrvWaveMFDE}
c v'(\xi) = v(\xi + \sigma_h) + v(\xi - \sigma_h) + v(\xi + \sigma_v) + v(\xi - \sigma_v) - 4 v(\xi) + g'\big(\Phi(\xi)\big) v(\xi).
\end{equation}
Our analysis in the remainder of this section hinges upon
understanding solutions to \sref{eq:prlm:linTrvWaveMFDE}
that decay at specified exponential rates on half-lines.

In particular,
we choose four exponents $\eta^\pm_{\mathrm{fs}}$ and $\eta^\pm_{\mathrm{sl}}$
in such a way that
\begin{equation}
0 < \eta^\pm_{\mathrm{sl}} < \eta^\pm_{\Phi} < \eta^\pm_{\mathrm{fs}} < 2 \eta^\pm_{\mathrm{sl}},
\end{equation}
while all non-real roots of $\Delta^+(z) = 0$
have $\Re z \notin [ - \eta^+_{\mathrm{fs}},  0]$ and all non-real roots of $\Delta^-(z) = 0$
have $\Re z \notin [0, \eta^-_{\mathrm{fs}}]$.
Using these exponents,
we introduce
the constant
\begin{equation}
  \sigma = \max\{ \abs{\sigma_h} , \abs{\sigma_v} \}
\end{equation}
together with the function spaces
\begin{equation}
\begin{array}{lcl}
BC^+_{\mathrm{sl}} & = &
   \{ v \in C([0, \infty), \Real) :
     \norm{v}_{BC^+_{\mathrm{sl}}} :=
        \sup_{\xi \ge 0} e^{ \eta^+_{\mathrm{sl}} \abs{\xi} }  \abs{v(\xi)}
       < \infty \},
\\[0.2cm]
BC^-_{\mathrm{sl}} & = &
   \{ v \in C((-\infty, 0], \Real) :
     \norm{v}_{BC^-_{\mathrm{sl}}} :=
        \sup_{\xi \le 0} e^{ \eta^-_{\mathrm{sl}} \abs{\xi} }  \abs{v(\xi)}
       < \infty \},
\\[0.2cm]
BC^\oplus_{1,\mathrm{sl}} & = &
   \{ v \in C([-\sigma, \infty), \Real) :
     \norm{v}_{BC^\oplus_{1,\mathrm{sl}}} :=
        \sup_{\xi \ge -\sigma} e^{ \eta^+_{\mathrm{sl}} \abs{\xi} } [ \abs{v(\xi)} + \abs{v'(\xi)} ]
       < \infty \},
\\[0.2cm]
BC^\ominus_{1,\mathrm{sl}} & = &
   \{ v \in C((-\infty, \sigma], \Real) :
     \norm{v}_{BC^\ominus_{1,\mathrm{sl}}} :=
        \sup_{\xi \le \sigma} e^{ \eta^-_{\mathrm{sl}} \abs{\xi} } [ \abs{v(\xi)} + \abs{v'(\xi)} ]
       < \infty \},
\\[0.2cm]
\end{array}
\end{equation}
with similar definitions
for $BC^\pm_{\mathrm{fs}}$, $BC^\ominus_{1,\mathrm{fs}}$ and $BC^\oplus_{1,\mathrm{fs}}$.

%
Returning to \sref{eq:prlm:linTrvWaveMFDE},
we introduce the solution spaces
\begin{equation}
\label{eq:prlm:defSolSpaces}
\begin{array}{lcl}
\mathcal{P}_{\mathrm{fs}} & = & \big\{ v \in BC^\ominus_{1,\mathrm{fs}} :  [\mathcal{L}_0 v](\xi) = 0 \hbox{ for all } \xi \le 0 \big\}, \\[0.2cm]
\mathcal{Q}_{\mathrm{fs}} & = & \big\{ v \in BC^\oplus_{1,\mathrm{fs}} :  [\mathcal{L}_0 v](\xi) = 0 \hbox{ for all } \xi \ge 0 \big\}. \\
\end{array}
\end{equation}
We note that we are abusing notation here in the sense that
in addition to the definition \sref{eq:mr:defLomega},
we are interpreting $\mathcal{L}_0$ as an operator
from $BC^\oplus_{1,\mathrm{fs}} \to BC^+_{\mathrm{fs}}$
and also as an operator from $BC^\ominus_{1,\mathrm{fs}} \to BC^-_{\mathrm{fs}}$.

In order to capture the initial conditions associated to the functions in the
solution spaces \sref{eq:prlm:defSolSpaces},
we will use the notation $\mathrm{ev}_{\xi} u \in C([-\sigma, \sigma], \Real)$
to denote the state of a continuous function $u$ at $\xi$,
which is defined by
\begin{equation}
[\mathrm{ev}_{\xi}u](\vartheta) := u(\xi + \vartheta), \qquad \vartheta\in[-\sigma,\sigma].
\end{equation}
This allows us to define the segment spaces
\begin{equation}
\label{eq:prlm:defSegSpaces}
\begin{array}{lcl}
P_{\mathrm{fs}} & = & \big\{\phi \in C([-\sigma,\sigma], \Real) : \phi = \mathrm{ev}_0 v \hbox{ for some } v \in \mathcal{P}_{\mathrm{fs}} \big\}, \\[0.2cm]
Q_{\mathrm{fs}} & = & \big\{\phi \in C([-\sigma,\sigma], \Real) : \phi = \mathrm{ev}_0 v \hbox{ for some } v \in \mathcal{Q}_{\mathrm{fs}} \big\}, \\[0.2cm]
B & = & \mathrm{span} \{ \mathrm{ev}_0 \Phi' \}.
\end{array}
\end{equation}
In view of the asymptotics
\sref{eq:prlm:asymEsts:estOnMinus}-\sref{eq:prlm:asymEsts:estOnPlus},
the fact that the kernel of $\mathcal{L}_0$ is one-dimensional
implies that
\begin{equation}
P_{\mathrm{fs}} \cap Q_{\mathrm{fs}}  = \{ 0 \}, \qquad P_{\mathrm{fs}} \cap B = \{ 0 \}, \qquad Q_{\mathrm{fs}} \cap B = \{ 0 \}.
\end{equation}
In order to obtain a splitting for the state space $C([-\sigma,\sigma], \Real)$
involving the components \sref{eq:prlm:defSegSpaces},
we need to exploit the Hale inner product \cite{HLVL}.
In the current setting, this bilinear form is given by
\begin{equation}
  \langle \psi, \phi \rangle := c \psi(0) \phi(0)
     -  \int_0^\sigma \psi(\theta - \sigma) \phi(\theta) \,\rmd\theta
     - \int_0^{-\sigma} \psi(\theta + \sigma) \phi(\theta) \,\rmd\theta
\end{equation}
for any pair $\phi, \psi \in C([-\sigma, \sigma], \Real)$.
The Hale inner product is non-degenerate in the sense that if
$\langle \psi, \phi \rangle  = 0$ for all
$\psi \in C([-\sigma,\sigma], \Real)$, then necessarily $\phi = 0$ \cite{MPVL}.
As a consequence of \cite[Thm. 4.3]{MPVL}, we now have the characterization
\begin{equation}
P_{\mathrm{fs}} \oplus Q_{\mathrm{fs}} \oplus B = \{ \phi \in C([-\sigma, \sigma], \Real) \mid \langle \mathrm{ev}_0 \, \Psi , \phi \rangle = 0 \},
\end{equation}
in which we have recalled the function $\Psi$ defined by \sref{eq:mr:defPsi}.
Let us now pick a one-dimensional space $\Gamma \subset C([-\sigma,\sigma], \Real)$
that has the property that $\phi \in \Gamma$ satisfies $\phi = 0$ if and only if
$\langle \mathrm{ev}_0 \Psi , \phi \rangle = 0$.
We now see that
\begin{equation}
\label{eq:prlm:splittingX}
C([-\sigma,\sigma], \Real) = B\oplus Q_{\mathrm{fs}} \oplus P_{\mathrm{fs}} \oplus \Gamma.
\end{equation}

As customary, the solution $\Phi$ to the
travelling wave equation \sref{eq:prlm:trvWaveMFDE}
breaks when changing the parameters $(c,\sigma_h, \sigma_v)$.
The key ingredient we will use in this section
is that the arising gap can be captured in the finite dimensional space $\Gamma$.
As a consequence, the size of such gaps can be measured effectively by means of the Hale inner product.
This is particularly useful in view of the identity
\begin{equation}
\label{eq:dng:idHale}
  \begin{array}{lcl}
    \frac{\rmd}{\rmd\xi} \langle \mathrm{ev}_{\xi} \Psi, \mathrm{ev}_{\xi} v \rangle
       & = & \Psi(\xi) [\mathcal{L}_0 v](\xi),  \\
  \end{array}
\end{equation}
which holds for any pair $v \in C^1(\Real, \Real)$ and $\xi \in \Real$.

We now turn our attention to the perturbed linearization
\begin{equation}
\label{eq:prlm:linTrvWaveMFDE:perturbed}
c' v'(\xi) = v(\xi + \sigma'_h) + v(\xi - \sigma'_h) + v(\xi + \sigma'_v) + v(\xi - \sigma'_v) - 4 v(\xi) + g'\big(\Phi(\xi)\big) v(\xi).
\end{equation}
For convenience, we introduce the parameter $q' = (c', \sigma_h', \sigma_v')$
and the set
\begin{equation}
\mathcal{D}_q(\delta_q) = \{ (c', \sigma_h', \sigma_v') \in \Real^3 : \abs{c' - c} +
  \abs{\sigma_h'- \sigma_h} + \abs{\sigma_v' - \sigma_v} < \delta_q
  \hbox{ and } \sigma' = \sigma \},
\end{equation}
in which we have introduced the notation
\begin{equation}
\sigma' = \max\{ \abs{\sigma_h'} , \abs{\sigma_v'} \}.
\end{equation}
We note that the restriction $\sigma' = \sigma$ is a purely technical one
in order to ensure that the state space $C([-\sigma, \sigma], \Real)$
remains unaffected. In light of the fact that \sref{eq:prlm:trvWaveMFDE}
remains invariant under the transformations
\begin{equation}
\xi \mapsto \lambda \xi, \qquad (\sigma_h, \sigma_v) \to \lambda^{-1} (\sigma_h, \sigma_v), \qquad c \mapsto \lambda c,
\end{equation}
this restriction will not hinder our ability to describe waves travelling in arbitrary directions
sufficiently close to $(\sigma_h, \sigma_v)$.

For any $q' \in  \mathcal{D}_q(\delta_q)$,
we introduce the differential operator $\mathcal{L}(q')$ that acts as
\begin{equation}
[\mathcal{L}(q') v](\xi)
= -c' v'(\xi) + v(\xi + \sigma_h') + v(\xi +\sigma_v') + v(\xi -\sigma_h')  + v(\xi - \sigma_v') - 4 v(\xi) + g'\big(\Phi(\xi) \big) v(\xi).
\end{equation}
As above, this operator will be interpreted
as a linear map on both $BC^\oplus_{1,\mathrm{fs}}$ and $BC^\ominus_{1,\mathrm{fs}}$,
mapping into $BC^+_{\mathrm{fs}}$ and $BC^-_{\mathrm{fs}}$ respectively.

We now borrow some convenient results
from \cite{HJHLIN} that describe how $\mathcal{L}(q')$
and the spaces \sref{eq:prlm:defSegSpaces}
vary with $q'$. For explicitness, we write
\begin{equation}
q_* = (c , \sigma_h, \sigma_v).
\end{equation}

\begin{lem}[{see \cite[{\S}5]{HJHLIN}}]
Suppose that
$\textrm{(hg}\textrm{)}_{\textrm{\S\ref{sec:prlm}}}$
and $\textrm{(h}\Phi\textrm{)}_{\textrm{\S\ref{sec:prlm}}}$
are both satisfied
and pick $\delta_q > 0$ sufficiently small.
Then for any $q' = (c',\sigma_h', \sigma_v') \in \mathcal{D}_q(\delta_q)$,
there exist linear maps
\begin{equation}
u^*_{Q_{\mathrm{fs}} }(q') : Q_{\mathrm{fs}} \to BC^\oplus_{1,\mathrm{fs}}, \qquad u^*_{P_{\mathrm{fs}}}(q') : P_{\mathrm{fs}} \to BC^\ominus_{1,\mathrm{fs}}
\end{equation}
that satisfy the following properties.
\begin{itemize}
\item[(i)]{
 For any $(\phi_Q, \phi_P) \in Q_{\mathrm{fs}} \times P_{\mathrm{fs}} $
 and $q' \in \mathcal{D}_{q'}(\delta_q)$,
 the function $v^+ = u^*_{Q_{\mathrm{fs}}}(q') \phi_Q$
  satisfies $[\mathcal{L}(q') v^+ ](\xi) = 0$
  for all $\xi \ge 0$, while $v^- = u^*_{P_{\mathrm{fs}}}(q') \phi_P$
  satisfies $[\mathcal{L}(q') v^- ](\xi) = 0$
  for all $\xi \le 0$.
}
\item[(ii)]{
  Pick any $q' \in \mathcal{D}_{q}(\delta_q)$
  and consider any pair $(v^+, v^-) \in BC^\oplus_{1,\mathrm{fs}} \times BC^\ominus_{1,\mathrm{fs}}$
  for which $[\mathcal{L}(q') v^+ ](\xi) = 0$ for all $\xi \ge 0$
  and $[\mathcal{L}(q') v^- ](\xi) = 0$ for all $\xi \le 0$
  Then we must have
  \begin{equation}
    v^+ = u^*_{Q_{\mathrm{fs}}}(q') \Pi_{Q_{\mathrm{fs}}} \mathrm{ev}_0 v^+,
    \qquad
    v^- = u^*_{P_{\mathrm{fs}}}(q') \Pi_{P_{\mathrm{fs}}} \mathrm{ev}_0 v^-.
  \end{equation}
}
\item[(iii)]{
 For any $q' \in \mathcal{D}_{q}(\delta_q)$,
 we have the identities
 \begin{equation}
   \Pi_{Q_{\mathrm{fs}}} \mathrm{ev}_0 u^*_{Q_{\mathrm{fs}}}(q')= I,
   \qquad
   \Pi_{P_{\mathrm{fs}}} \mathrm{ev}_0 u^*_{P_{\mathrm{fs}}}(q')= I.
 \end{equation}
}
\item[(iv)]{
  The maps
\begin{equation}
q' \mapsto
\left\{ \begin{array}{l}
  u^*_{Q_{\mathrm{fs}}}(q') \in \mathcal{L}\big( Q_{\mathrm{fs}}, BC^\oplus_{1,\mathrm{fs}}    \big)
  \\[0.2cm]
  u^*_{P_{\mathrm{fs}}}(q') \in \mathcal{L}\big( P_{\mathrm{fs}}, BC^\ominus_{1,\mathrm{fs}}    \big)
  \end{array}
\right.
\end{equation}
are $C^1$-smooth.
}
\end{itemize}
\end{lem}
\begin{lem}[{see \cite[{\S}3]{HJHLIN}} ]
\label{lem:prlm:defInverseOnHalflines}
Suppose that
$\textrm{(hg}\textrm{)}_{\textrm{\S\ref{sec:prlm}}}$
and $\textrm{(h}\Phi\textrm{)}_{\textrm{\S\ref{sec:prlm}}}$
are both satisfied
and pick $\delta_q > 0$ sufficiently small.
Then for any $q' = (c',\sigma_h', \sigma_v') \in \mathcal{D}_q(\delta_q)$,
there exist linear maps
\begin{equation}
\mathcal{L}^+_{\mathrm{inv}}(q') : BC^+_{\mathrm{fs}} \to BC^\oplus_{1,\mathrm{fs}},
\qquad
\mathcal{L}^-_{\mathrm{inv}}(q') : BC^-_{\mathrm{fs}} \to BC^\ominus_{1,\mathrm{fs}}
\end{equation}
that satisfy the following properties.
\begin{itemize}
\item[(i)]{
  For every $f^\pm \in BC^\pm_{\mathrm{fs}}$ and $q' \in \mathcal{D}_{q}(\delta_q)$,
  the function $v^+ = \mathcal{L}^+_{\mathrm{inv}}(q') f^+$
  satisfies $[\mathcal{L}(q') v^+ ](\xi) = f^+(\xi)$
  for all $\xi \ge 0$, while $v^- = \mathcal{L}^-_{\mathrm{inv}}(q') f^-$
  satisfies $[\mathcal{L}(q') v^- ](\xi) = f^-(\xi)$
  for all $\xi \le 0$.
}
\item[(ii)]{
 For every $f^\pm \in BC^\pm_{\mathrm{fs}}$ and $q' \in \mathcal{D}_{q}(\delta_q)$,
 we have the identities
 \begin{equation}
   \begin{array}{lcl}
   \Pi_{Q_{\mathrm{fs}}} \mathrm{ev}_0 \mathcal{L}^+_{\mathrm{inv}}(q') f^+  & = & 0,
   \\[0.2cm]
   \Pi_{P_{\mathrm{fs}}} \mathrm{ev}_0 \mathcal{L}^-_{\mathrm{inv}}(q') f^-  & = & 0.
   \end{array}
 \end{equation}
}
\item[(iii)]{
The maps
\begin{equation}
q' \mapsto
\left\{ \begin{array}{l}
  \mathcal{L}^+_{\mathrm{inv}}(q') \in \mathcal{L}\big( BC^+_{\mathrm{fs}} , BC^\oplus_{1,\mathrm{fs}} \big)
  \\[0.2cm]
  \mathcal{L}^-_{\mathrm{inv}}(q') \in \mathcal{L}\big( BC^-_{\mathrm{fs}} , BC^\ominus_{1,\mathrm{fs}} \big)
  \end{array}
\right.
\end{equation}
are $C^1$-smooth.
}
\end{itemize}
\end{lem}
The next result can be seen as a continuation result
for the two halves of the wave profile $\Phi'$ upon varying $q$.
In particular, we construct two solution families
for the homogeneous MFDE
\sref{eq:prlm:linTrvWaveMFDE:perturbed} that
decay at the relevant slow exponential rate.
This will allow us to control the constants appearing in the
asymptotic expansions \sref{eq:prlm:asymEsts:estOnMinus}-\sref{eq:prlm:asymEsts:estOnPlus}.
\begin{lem}
Suppose that
$\textrm{(hg}\textrm{)}_{\textrm{\S\ref{sec:prlm}}}$
and $\textrm{(h}\Phi\textrm{)}_{\textrm{\S\ref{sec:prlm}}}$
are both satisfied
and pick $\delta_q > 0$ sufficiently small.
Then for any $q' = (c',\sigma_h', \sigma_v') \in \mathcal{D}_q(\delta_q)$,
there exist functions
\begin{equation}
b^+ = b^+(q') \in BC^\oplus_{1,\mathrm{sl}}, \qquad b^- = b^-(q') \in BC^\ominus_{1,\mathrm{sl}}
\end{equation}
that satisfy the following properties.
\begin{itemize}
\item[(i)]{
 We have $[\mathcal{L}(q') b^+(q')](\xi) = 0$ for all $\xi \ge 0$,
 together with $[\mathcal{L}(q') b^-(q')](\xi) = 0$ for all $\xi \le 0$.
}
\item[(ii)]{
The maps $q' \mapsto b^+(q') \in BC^\oplus_{1,\mathrm{sl}}$
and $q' \mapsto b^-(q') \in BC^\ominus_{1,\mathrm{sl}}$ are $C^1$-smooth,
with $b^\pm(q_*) = \Phi'$.
}
\item[(iii)]{
 For all $q' \in \mathcal{D}_{q}(\delta_q)$
 we have $\Pi_B \mathrm{ev}_0 b^\pm(q') = \mathrm{ev}_0 \Phi'$,
 together with
 \begin{equation}
   \Pi_{Q_{\mathrm{fs}}} \mathrm{ev}_0 b^+(q') = 0,
   \qquad
   \Pi_{P_{\mathrm{fs}}} \mathrm{ev}_0 b^-(q') = 0.
 \end{equation}
}
\item[(iv)]{
  Upon writing $\eta^+_{q'} > 0$
  for the exponent defined in Lemma \ref{lem:prlm:defSpatExps}
  applied to the characteristic equation
  \begin{equation}
    \Delta^+_{q'}(z) = c' z - (2 \cosh(\sigma_h' z) + 2 \cosh(\sigma_v' z) - 4) - g'(1)
  \end{equation}
  and similarly defining $\eta^-_{q'} > 0$,
  there exist constants $C^\pm_{q'} > 0$ and $K_2 > 1$ such that
  for all $q' \in \mathcal{D}_q(\delta_q)$ we have
  \begin{equation}
    \begin{array}{lcl}
      \abs{ b^+(q')(\xi) - C^+_{q'} e^{-\eta^+_{q'} \abs{\xi} } }  &\le & K_2 e^{-\eta^+_{\mathrm{fs}} \abs{\xi} },
        \qquad \xi \ge 0,
      \\[0.2cm]
      \abs{ b^-(q')(\xi) - C^-_{q'} e^{-\eta^-_{q'} \abs{\xi} } }  &\le & K_2 e^{-\eta^-_{\mathrm{fs}} \abs{\xi} },
        \qquad \xi \le 0.
    \end{array}
  \end{equation}
  In addition, the maps $q' \mapsto C^\pm_{q'}$ are $C^1$-smooth.
  }
\end{itemize}
\end{lem}
\begin{proof}
Since the maps $q' \mapsto \eta^+_{q'}$ are $C^1$-smooth, we can construct
a map $q' \mapsto b^+_0(q')$ that satisfies conditions (ii) - (iv)
simply by using
\begin{equation}
  b^+_0(q')(\xi) = \Phi'\big( [\eta^+_{q'} / \eta^+_{\Phi} ] \xi\big), \qquad \xi \gg 1
\end{equation}
and ensuring that $b^+_0(q')(\xi) = \Phi'(\xi)$ for $\xi \in [-\sigma, \sigma]$.

We now write $b^+(q') = b^+_0(q') + v(q')$
and find that (i) now requires that the function $v(q')$ satisfy
\begin{equation}
[\mathcal{L}(q') v(q')](\xi) = f_{q'}(\xi) := - [\mathcal{L}(q') b^+_0(q')](\xi),
  \qquad \xi \ge 0.
\end{equation}
By construction however, we see that $f_{q'} \in BC^+_{\mathrm{fs}}$,
which allows us to
write
\begin{equation}
v(q') = \mathcal{L}^+_{\mathrm{inv}}(q') f_{q'} \in BC^\oplus_{1,\mathrm{fs}},
\end{equation}
which depends $C^1$-smoothly on $q'$ and has $v(q_*) = 0$.
This ensures that (ii) and (iv) remain satisfied.
In addition, a simple multiplicative rescaling allows (iii) to be restored.
\end{proof}

Based on the ingredients above, we can follow the procedure
developed in \cite{HJHLIN} to implement a version of Lin's method.
In particular, we combine the two inverses $\mathcal{L}^\pm_{\mathrm{inv}}(q')$
to construct solutions to $\mathcal{L}(q')v = f$ up to a gap at zero,
which can be contained in the one-dimensional space $\Gamma$.
\begin{lem}[{see \cite[Lem. 5.10]{HJHFZHNGM}}]
Suppose that
$\textrm{(hg}\textrm{)}_{\textrm{\S\ref{sec:prlm}}}$
and $\textrm{(h}\Phi\textrm{)}_{\textrm{\S\ref{sec:prlm}}}$
are both satisfied
and pick $\delta_q > 0$ sufficiently small.
Then for any $q' = (c',\sigma_h', \sigma_v') \in \mathcal{D}_q(\delta_q)$
and any pair $(f^-, f^+) \in BC^-_{\mathrm{fs}} \times BC^+_{\mathrm{fs}}$,
there is a unique quadruplet
\begin{equation}
\big(v^-, \alpha^-, v^+, \alpha^+ \big)
\in  BC^\ominus_{1,\mathrm{fs}} \times \Real
\times  BC^\oplus_{1,\mathrm{fs}} \times \Real
\end{equation}
for which the pair
\begin{equation}
w^-(q') = v^-(q') + \alpha^-(q') b^-(q') \in BC^\ominus_{1,\mathrm{sl}},
\qquad
w^+(q') = v^+(q') + \alpha^+(q') b^+(q') \in BC^\oplus_{1,\mathrm{sl}}
\end{equation}
satisfies the following properties.
\begin{itemize}
\item[(i)]{
 For all $\xi \ge 0$ we have $[\mathcal{L}(q') v^-](\xi) = [\mathcal{L}(q') w^-](\xi) = f^-(\xi)$,
 while for all $\xi \ge 0$
 we have $[\mathcal{L}(q') v^+](\xi) = [\mathcal{L}(q') w^+](\xi) = f^+(\xi)$.
}
\item[(ii)]{
 We have the inclusions $\mathrm{ev}_0  w^\pm  \in P_{\mathrm{fs}} \oplus Q_{\mathrm{fs}} \oplus B$.
}
\item[(iii)]{
  The gap between $w^+$ and $w^-$ at zero satisfies
  $\mathrm{ev}_0 [w^+ - w^-] \in \Gamma$.
}
\end{itemize}
Upon writing
\begin{equation}
\big(v^-, \alpha^-, v^+, \alpha^+ \big) = L_3(q')(f^-, f^+)
\end{equation}
for the quadruplet described above, the map
\begin{equation}
q' \mapsto L_3(q') \in \mathcal{L} \Big(  BC^-_{\mathrm{fs}} \times BC^+_{\mathrm{fs}},
BC^\ominus_{1,\mathrm{fs}} \times \Real
\times  BC^\oplus_{1,\mathrm{fs}} \times \Real  \Big)
\end{equation}
is $C^1$-smooth.
In addition, the gap at zero satisfies
the identity
\begin{equation}
\label{eq:prlm:linsMethodGapExpr}
\langle \mathrm{ev}_0 \Psi, \mathrm{ev}_0 [w^+ - w^-] \rangle
= \int_{-\infty}^0 \Psi(\xi) [\mathcal{L}(q_*)w^-](\xi)   \, d \xi
+ \int_{0}^{\infty} \Psi(\xi) [\mathcal{L}(q_*) w^+](\xi) \, d \xi.
\end{equation}
\end{lem}

We are now ready to construct a solution to the MFDE
\begin{equation}
\label{eq:prlm:eqForTrvWaveW}
\begin{array}{lcl}
c' W'(\xi) & = &
  W(\xi + \sigma_h') + W(\xi + \sigma_v')
     + W(\xi - \sigma_h') + W(\xi - \sigma_v')
     -4 W(\xi)
\\[0.2cm]
& & \qquad + g^-_{\delta}\big( W(\xi) \big)
\end{array}
\end{equation}
on half-lines.
We note that the asymptotic
estimates \sref{eq:prlm:asymEsts:estOnMinus}-\sref{eq:prlm:asymEsts:estOnPlus}
allow us to write
\begin{equation}
\begin{array}{lclcl}
\Phi(\xi) & = & v^-_* + [\eta_\Phi^- ]^{-1} \Phi'(\xi), & & \xi \le \sigma,
\\[0.2cm]
\Phi(\xi) & = & 1 + v^+_* - [\eta_\Phi^+ ]^{-1} \Phi'(\xi), & & \xi \ge -\sigma,
\end{array}
\end{equation}
for some pair $(v^-_*, v^+_*) \in BC^\ominus_{1,\mathrm{fs}} \times BC^\oplus_{1,\mathrm{fs}}$.
Introducing the notation
\begin{equation}
h^- = (v^- , \alpha^-) \in BC^\ominus_{1,\mathrm{fs}} \times \Real,
\qquad h^+ = ( v^+, \alpha^+ ) \in BC^\oplus_{1,\mathrm{fs}} \times \Real,
\end{equation}
we fix $\delta > 0$ and define the functions
\begin{equation}
\begin{array}{lclcl}
[W^-(h^-)](\xi) & = &  - \delta + v^-_* + v^- + \big(\alpha^- + [\eta^-_\Phi]^{-1} \big) b^-(q')
 & & \xi \le \sigma,
 \\[0.2cm]
[W^+(h^+)](\xi) & = & 1 - \delta + v^+_* + v^+ + \big(\alpha^+ - [\eta^+_\Phi]^{-1} \big) b^+(q')
& & \xi \ge -\sigma.
\end{array}
\end{equation}
We intend to find a pair $(h^-, h^+)$ such that
\sref{eq:prlm:eqForTrvWaveW} with $W = W^+(h^+)$
is satisfied for $\xi \ge 0$,
while \sref{eq:prlm:eqForTrvWaveW} with $W = W^-(h^-)$
is satisfied for $\xi \le 0$.
Plugging this Ansatz into \sref{eq:prlm:eqForTrvWaveW},
we find
\begin{equation}
\begin{array}{lclcl}
- [\mathcal{L}(q') v^-](\xi)
& = & \mathcal{R}^-_{q',\delta}(h^- ; \xi),
& & \xi \le 0,
\\[0.2cm]
- [\mathcal{L}(q') v^+](\xi)
& = & \mathcal{R}^+_{q',\delta}(h^+ ; \xi),
& & \xi \ge 0,
\\[0.2cm]
\end{array}
\end{equation}
with nonlinear terms
\begin{equation}
\begin{array}{lcl}
\mathcal{R}^-_{q',\delta}( h^- ; \xi)
& = &
(c - c') [v^-_*]'(\xi)
+ v^-_*(\xi + \sigma'_h)  + v^-_*(\xi - \sigma'_h)
 - v^-_*(\xi + \sigma_h) - v^-_*(\xi - \sigma_h)
\\[0.2cm]
& & \qquad
+ v^-_*(\xi + \sigma'_v) + v^-_*(\xi - \sigma'_v)
 - v^-_*(\xi + \sigma_v)  - v^-_*(\xi - \sigma_v)
\\[0.2cm]
& & \qquad
+ g^-_{\delta}\big( W^-( h^-) \big) -  g\big( \delta +  W^-( h^-) \big)
\\[0.2cm]
& & \qquad
+ \mathcal{R}_0\Big( v^-(\xi) + \alpha^- b^-(q')(\xi)
   + [\eta^-_\Phi]^{-1} [b^-(q')(\xi) - \Phi'(\xi) ]   ; \xi \Big),
\\[0.2cm]
\mathcal{R}^+_{q',\delta}( h^+ ; \xi)
& = &
(c - c') [v^+_*]'(\xi)
 + v^+_*(\xi + \sigma'_h)  + v^+_*(\xi - \sigma'_h)
 - v^+_*(\xi + \sigma_h) - v^+_*(\xi - \sigma_h)
\\[0.2cm]
& & \qquad
+ v^+_*(\xi + \sigma'_v) + v^+_*(\xi - \sigma'_v)
- v^+_*(\xi + \sigma_v) - v^+_*(\xi - \sigma_v)
\\[0.2cm]
& & \qquad
+ g^-_{\delta}\big( W^+(h^+ ) \big) - g\big( \delta + W^+(h^+) \big)
\\[0.2cm]
& & \qquad
+ \mathcal{R}_0\Big( v^+(\xi) + \alpha^+ b^+(q')(\xi)
   - [\eta^+_\Phi]^{-1} [b^+(q')(\xi) - \Phi'(\xi) ]  ; \xi \Big),
\\[0.2cm]
\end{array}
\end{equation}
in which we have introduced the expression
\begin{equation}
 \mathcal{R}_0( v ; \xi)  = g\big(\Phi(\xi) + v\big)  - g\big(\Phi(\xi)\big) - g'\big(\Phi(\xi) \big) v.
\end{equation}
\begin{lem}
\label{lem:prlm:nlR}
Suppose that
$\textrm{(hg}\textrm{)}_{\textrm{\S\ref{sec:prlm}}}$
and $\textrm{(h}\Phi\textrm{)}_{\textrm{\S\ref{sec:prlm}}}$
are both satisfied,
pick sufficiently small constants $\delta_q > 0$ and $\delta_0 > 0$
and an arbitrary constant $M_4 > 1$.
Then there exists a constant $C_4 > 1$
and an exponent $\kappa_4 > 0 $
such that for any sets
\begin{equation}
(v^+, v^+_1, v^+_2) \in [BC^\oplus_{1,\mathrm{fs}}]^3,
\qquad
(\alpha^+, \alpha^+_1, \alpha^+_2) \in \Real^3
\end{equation}
that have
\begin{equation}
\norm{v^+}_{BC^\oplus_{1,\mathrm{fs}}} + \norm{v^+_1}_{BC^\oplus_{1,\mathrm{fs}}}
  + \norm{v^+_2}_{BC^\oplus_{1,\mathrm{fs}}}
\le M_4,
\qquad
\abs{\alpha^+} + \abs{\alpha^+_1} + \abs{\alpha^+_2}
\le M_4,
\end{equation}
any $0 \le \delta < \delta_0$
and any $q' \in \mathcal{D}_q(\delta_q)$,
we have the estimates
\begin{equation}
\label{eq:prlm:nlRBnd}
\begin{array}{lcl}
\norm{ \mathcal{R}_{q',\delta}^+(v^+, \alpha^+ ; \cdot ) }_{BC^+_{\mathrm{fs}}}
& \le & C_4 \delta^{\kappa_4} + C_4 \abs{q' - q_*}
+  C_4 \big[ \abs{\alpha^+} +  \norm{v^+}_{BC^\oplus_{1,\mathrm{fs}}} \big]^2,
\\[0.2cm]
\end{array}
\end{equation}
together with
\begin{equation}
\label{eq:prlm:nlRLipBnd}
\begin{array}{lcl}
\norm{ \mathcal{R}_{q',\delta}^+(v^+_1, \alpha^+_1 ; \cdot )
- \mathcal{R}_{q',\delta}^+(v^+_2, \alpha^+_2 ; \cdot ) }_{BC^+_{\mathrm{fs}}}
& \le &
C_4 \big[ \sqrt{\delta} + \norm{v^+_1}_{BC^\oplus_{1,\mathrm{fs}}}
  + \norm{v^+_2}_{BC^\oplus_{1,\mathrm{fs}}} + \abs{\alpha^+_1} + \abs{\alpha^+_2} \big]
\\[0.2cm]
& & \qquad \times \big[ \abs{\alpha^+_1 - \alpha^+_2} + \norm{v^+_1 -v^+_2}_{BC^\oplus_{1,\mathrm{fs}}} \big].
\\[0.2cm]
\end{array}
\end{equation}
Similar estimates hold for the nonlinearities $\mathcal{R}^-_{q',\delta}$.
\end{lem}
\begin{proof}
Notice first that for every $u  \in \Real$,
there exists $0 < \vartheta < 1$ so that
\begin{equation}
\tau^+(u - \delta , \sqrt{\delta} ) - u
= \big[ \partial_u \tau^+( - \delta + \vartheta u , \sqrt{\delta} )
  - \partial_u \tau^+(-\delta , \sqrt{\delta} )  \big] u.
\end{equation}
Using the Lipschitz property (v)
obtained in Lemma \ref{lem:prlm:exTauFncs} for $\partial_u \tau^+$,
together with the bound \sref{eq:prlm:defTauPlus:bndDerivative},
one finds that there exists $C'_1 > 1$ so that
\begin{equation}
\abs{ g^-_\delta(- \delta + u) - g(u) } \le C'_1 \min\{ \sqrt{\delta}, \abs{u} \} \abs{u}
\end{equation}
holds for all $u \in \Real$ and all $0\le \delta < \delta_0$.
In a similar fashion, for all such $u$ and $\delta$ we have
\begin{equation}
\abs{ g^-_\delta(- \delta + u) - g(u) } \le C'_1 \min\{ \sqrt{\delta}, \abs{1 -u} \} \abs{1 - u}.
\end{equation}
In particular, whenever $1 - u \in BC^+_{\mathrm{sl}}$, we can estimate
\begin{equation}
\norm{ g_{\delta}\big( - \delta + u(\cdot) \big) - g\big(u(\cdot) \big) }_{BC^+_{\mathrm{fs}} }
 \le C'_1 \delta^{\frac{2 \eta^+_{\mathrm{sl}} - \eta^+_{\mathrm{fs}}}{2 \eta^+_{\mathrm{sl}}} } \norm{1 - u}_{BC^+_{\mathrm{sl}}}.
\end{equation}
In addition, exploiting the Lipschitz continuity of $g'$,
for any $u \in BC^+_{\mathrm{sl}}$ one easily estimates
\begin{equation}
\norm{ \mathcal{R}_0(u(\cdot) ; \cdot ) }_{BC^+_{\mathrm{fs}}} \le C'_2 \norm{u}^2_{BC^+_{\mathrm{sl}}},
\end{equation}
for some $C'_2 > 1$, which suffices to establish \sref{eq:prlm:nlRBnd}.
The Lipschitz bound \sref{eq:prlm:nlRLipBnd} can be obtained using standard arguments,
again exploiting the Lipschitz continuity of $g'$ and $\partial_u \tau^+$.
\end{proof}

\begin{proof}[Proof of Proposition \ref{prp:prlm:subsupWithDelta}]
Fix $\delta_0 > 0$ and $\delta_q > 0$ sufficiently small.
On account of the estimates in Lemma \ref{lem:prlm:nlR},
a fixed point argument can be used to
construct for any $q' \in \mathcal{D}_q(\delta_q)$ and $0 \le \delta < \delta_0$,
a pair $(h^-_*, h^+_*) = (h^-_*, h^+_*)(q', \delta)$ with
\begin{equation}
(h^-_*, h^+_*) = L_3(q') \big( \mathcal{R}^-_{q', \delta}(h^-_* ; \cdot) , \mathcal{R}^+_{q', \delta}(h^+_* ; \cdot ) \big).
\end{equation}
In addition, the map $( q', \nu) \mapsto (h^-_*, h^+_*)(q', \nu^2)$ is $C^1$-smooth with $(h^-_*, h^+_*)(q, 0) = 0$.
The functions $W^+(h^+_*)$ and $W^-(h^-_*)$ together define a solution
to the travelling wave system \sref{eq:prlm:eqForTrvWaveW}
provided $\mathrm{ev}_0 [W^+(h^+_*) - W^-(h^-_*) ] =0 \in \Gamma$.
This one dimensional equation implicitly defines $c'$ as a
$C^1$-smooth function of $(\sqrt{\delta}, \sigma_h', \sigma_v')$,
which can be seen by
exploiting \sref{eq:prlm:linsMethodGapExpr}
and using the identity $\int_{\Real} \Psi(\xi) \Phi'(\xi) \, d \xi = 1$
to verify the conditions of the implicit function theorem.
Similar computations can be found in \cite{HJHSTBFHN}.
This concludes the construction of the pairs $(c^-_p, \Phi^-_p)$.

The pairs $(c^+_p, \Phi^+_p)$ can be constructed
in a similar fashion
and the differential inequalities \sref{eq:prp:prlm:subsupWithDelta:diffIneq}
now follow from the identities
\begin{equation}
\mathcal{J}^\pm_{ij}(t)  = g^\pm_\delta\big(W^\pm_{ij}(t) \big)  -  g\big(W^\pm_{ij}(t) \big),
\end{equation}
together with the inequalities \sref{eq:lem:prlm:nlGDeltaPM:ineqOnGPM}.
\end{proof}

\section{Spreading Speed}
\label{sec:expblob}

In this section we consider the homogeneous lattice
and set out to prove that large disturbances from the zero rest
state spread out to fill the entire lattice $\Wholes^2$,
provided the initial support of the disturbance is sufficiently large.
This is the analogue of the classic result \cite[Thm. 5.3]{AW78}
obtained by Aronson and Weinberger for bistable reaction-diffusion PDEs.

\begin{prop}
\label{prp:expblob:blbExp}
Consider the unobstructed LDE \sref{eq:mr:lde:hom},
suppose that $\textrm{(hg}\textrm{)}_{\textrm{\S\ref{sec:prlm}}}$
and (H$\Phi$) both hold and write
$c_* = \min_{\zeta \in [0, 2\pi] }\{ c_{\zeta} \}$.
Fix any $0 < c < c_*$. Then for any sufficiently small $\eta > 0$,
there exist $R = R(c, \eta) > 0$ and $T = T(c, \eta) > 0$ such that the
solution to the LDE \sref{eq:mr:lde:hom} with initial condition
\begin{equation}
u_{ij}(0) = \left\{
  \begin{array}{lcl}
     1 - \eta & & \sqrt{i^2 + j^2 } \le R
     \\[0.2cm]
     0 & & \abs{ i^2 + j^2 } > R
  \end{array}
  \right.
\end{equation}
satisfies $u_{ij}(t) \ge 1 - \eta$
for all $t \ge T$ and $\sqrt{i^2 + j^2} \le R + c( t- T)$.
\end{prop}

The main difficulty here is that it is no longer possible
to construct a radially symmetric expanding sub-solution,
because the wave profiles are angular dependent. In addition,
the technical trick \cite[(5.8)]{AW78},
which allowed a smooth flat core to be connected to an outwardly travelling
wave via a sharp interface, is no longer available.
Instead, we work here with suitably stretched versions of the angular dependent
wave profiles in order to construct a wide transition area between the core
and the outgoing waves.

\begin{lem}
\label{lem:expblb:prfls}
Consider the unobstructed LDE \sref{eq:mr:lde:hom} and
suppose that $\textrm{(hg}\textrm{)}_{\textrm{\S\ref{sec:prlm}}}$
and (H$\Phi$) both hold.
Pick a sufficiently small $\delta_0 > 0$
and recall the nonlinearities $g^-_{\delta}$ defined in Lemma \ref{lem:prlm:nlGDeltaPM}.
Then for all $0\le \delta< \delta_0$ and $\zeta \in [0, 2 \pi]$
there exists a wave speed $c_{\zeta;\delta}  > 0$
together with a wave profile $\Phi_{\zeta;\delta} \in C^1(\Real, \Real)$
that satisfies the MFDE
\begin{equation}
\label{eq:expblb:trvWaveMFDE}
\begin{array}{lcl}
c \Phi_{\zeta;\delta}'(\xi) & = &
\Phi_{\zeta;\delta}(\xi + \cos \zeta) + \Phi_{\zeta;\delta}(\xi + \sin \zeta)
 + \Phi_{\zeta;\delta}(\xi - \cos \zeta) + \Phi_{\zeta;\delta}(\xi - \sin \zeta)
 - 4 \Phi_{\zeta;\delta}(\xi)
\\[0.2cm]
& & \qquad
  + g^-_{\delta}\big( \Phi_{\zeta; \delta}(\xi) \big)
\end{array}
\end{equation}
and enjoys the limits
\begin{equation}
\lim_{\xi \to - \infty} \Phi_{\zeta;\delta}(\xi) = -\delta, \qquad \lim_{\xi \to + \infty} \Phi_{\zeta;\delta}(\xi) = 1 - \delta.
\end{equation}
In addition, there exists a constant $C_1 > 1$ so that
for all $0 \le \delta < \delta_0$ and $\zeta \in [0, 2\pi]$,
we have the uniform bound
\begin{equation}
\label{eq:expbl:defKShift}
\frac{ \abs{\Phi'_{\zeta;\delta}(\xi') } } { \abs{\Phi'_{\zeta ;\delta}(\xi) } } \le C_1,
\qquad \abs{\xi' - \xi} \le 2.
\end{equation}
\end{lem}
\begin{proof}
The statements follow directly from Proposition \ref{prp:prlm:subsupWithDelta}.
In particular, the uniform bound \sref{eq:expbl:defKShift} follows
from the continuity properties stated in item (v) of this result.
\end{proof}

In order to connect the waves defined above to a flat inner core,
we need to ensure that the wave profiles are all cut off at the same value.
In addition, we need to enforce a convexity condition in the transition area.
The next result handles these two requirements.

\begin{lem}
\label{lem:expblb:defShiftWaveProfiles}
Suppose that $\textrm{(hg}\textrm{)}_{\textrm{\S\ref{sec:prlm}}}$
and (H$\Phi$) both hold,
pick a sufficiently small $\delta_0 > 0$
and recall the wave profiles $\Phi_{\zeta;\delta}$ defined in Lemma \ref{lem:expblb:prfls}.
Then for any $h_{\infty} \ge 0$,
there exists a continuous map $\tau_{h_\infty}:[0, 2\pi] \times (0, \delta_0)  \to \Real$
and a continuous map $\Phi_{\infty;h_\infty}: (0, \delta_0) \to (1 - 2\delta_0, 1)$,
such that the shifted profiles
\begin{equation}
\Phi_{\zeta; \delta,h_\infty}(\xi) := \Phi_{\zeta; \delta}\big(\xi + \tau_{h_\infty}(\zeta,\delta) \big)
\end{equation}
satisfy the following properties.
\begin{itemize}
\item[(i)]{
For any $\delta \in (0, \delta_0)$, we have
\begin{equation}
\Phi''_{\zeta;\delta,h_\infty}(\xi) \le 0, \qquad \xi \ge 0, \qquad 0 \le \zeta \le 2 \pi.
\end{equation}
}
\item[(ii)]{
For any $\delta \in (0, \delta_0)$, we have
\begin{equation}
\Phi_{\zeta;\delta,h_\infty}(0) \ge 1 - 2 \delta, \qquad 0 \le \zeta \le 2 \pi.
\end{equation}
}
\item[(iii)]{
  For any $\delta \in (0, \delta_0)$ and $0 \le \zeta \le 2 \pi$, we have
  \begin{equation}
    \Phi_{\zeta; \delta, h_\infty}(h_\infty) = \Phi_{\infty;h_\infty}(\delta) \ge 1 - 2 \delta.
  \end{equation}
}
\item[(iv)]{
  For any fixed $0 < \delta < \delta_0$, the map $\zeta \mapsto \tau_{h_\infty}(\zeta, \delta)$
  is $C^1$-smooth.
}
\end{itemize}
\end{lem}
\begin{proof}
We note first that the asymptotic estimates in Proposition \ref{prp:prlm:asymEsts}
imply that the wave profiles $\Phi_{\zeta;\delta}(\xi)$ are convex down as $\xi \to + \infty$.
In particular, the continuity properties of the asymptotic coefficients
that are stated in item (v) of
Proposition \ref{prp:prlm:subsupWithDelta}
allow us to construct a continuous function $\tau_*: (0, \delta_0) \to \Real$
in such a way that for all $\zeta \in [0, 2\pi]$ and all $0 < \delta < \delta_0$,
we have the inequalities
\begin{equation}
\Phi''_{\zeta;\delta}\big(\xi + \tau_*(\delta)\big) \le 0, \qquad \xi \ge 0,
\end{equation}
together with
\begin{equation}
\Phi_{\zeta;\delta}\big( \tau_*(\delta) \big) \ge 1 - 2\delta.
\end{equation}

Now, for every $0 < \delta < \delta_0$ we define
\begin{equation}
\Phi_{\infty;h_{\infty}}(\delta) = \max_{\zeta \in [0, 2\pi]} \Phi_{\zeta;\delta}( \tau_*(\delta) + h_{\infty}),
\end{equation}
which depends continuously on $\delta$.
We can now pick
$\tau_{h_\infty}( \zeta,  \delta)$ in such a way that
\begin{equation}
\label{eq:expblob:defTauInfty}
\Phi_{\zeta;\delta}\big(\tau_{h_\infty}(\zeta,\delta) \big) =
 \Phi_{\infty; h_{\infty}}(\delta)
\end{equation}
holds for all $\zeta \in [0, 2\pi]$, which establishes (iii).
Since $\Phi'_{\zeta;\delta} > 0$, we necessarily have
$\tau_{h_\infty}(\delta, \zeta) \ge \tau_*(\delta)$,
which implies (i) and (ii).

The smoothness (iv) can be established by noting
that for any $\alpha$ in the range of $\Phi_{\zeta; \nu^2}$,
the implicit definition
\begin{equation}
\Phi_{\zeta;\nu^2}( \xi ) = \alpha
\end{equation}
locally defines a $C^1$-smooth function $\xi = \xi(\alpha, \zeta, \nu)$
on account of the fact that $\Phi'_{\zeta;\delta} > 0$.
This observation also implies that the continuity of the map $(\zeta, \delta) \mapsto \tau_{h_\infty}(\zeta, \delta)$
follows directly from the continuity of $\delta \mapsto \Phi_{\infty; h_{\infty}}(\delta)$.
\end{proof}

Our next result provides a bound on the angular derivatives
of the wave profiles.
We emphasize that the constants $K_2$ below cannot be taken to be uniform
across $0 \le \delta < \delta_0$, because the shifts $\tau_{h_\infty}$
defined above generically become unbounded as $\delta \to 0$.
\begin{cor}
\label{cor:expblob:bndOnPartialZeta}
Suppose that $\textrm{(hg}\textrm{)}_{\textrm{\S\ref{sec:prlm}}}$
and (H$\Phi$) both hold
and pick a sufficiently small $\delta_0 > 0$.
Then there exists an exponent $\eta_* > 0$
so that the following holds true.

For any $h_{\infty} \ge 0$ and any $0 < \delta < \delta_0$,
there exists a constant $K_2 = K_2(\delta;h_\infty) > 1$ such that
the estimates
\begin{equation}
\label{eq:expblb:cor:zetaDerivExpBnd}
\partial_{\zeta} \Phi_{\zeta;\delta, h_\infty}(\xi) \le K_2 e^{ - \eta_* \abs{\xi} }, \qquad \xi \in \Real, \qquad \zeta \in [0, 2\pi],
\end{equation}
hold for the wave profiles defined in Lemma \ref{lem:expblb:defShiftWaveProfiles}.
\end{cor}
\begin{proof}
The existence of the derivative with respect to $\zeta$ follows from item (iv) of Lemma \ref{lem:expblb:defShiftWaveProfiles},
together with the smoothness properties established in item (iv) of Proposition \ref{prp:prlm:subsupWithDelta}.
The exponential bound \sref{eq:expblb:cor:zetaDerivExpBnd}
follows from items (iv) and (v) of Proposition \ref{prp:prlm:subsupWithDelta}.
\end{proof}

We now set out to construct an expanding sub-solution
for the unobstructed LDE \sref{eq:mr:lde:hom}.
We introduce the set
\begin{equation}
\begin{array}{lcl}
\mathcal{D}_p = \{ p = ( \delta, c, \rho,  h_\infty, h) \in \Real^4 \times C^2(\Real, \Real) \hbox{ for which (i)}_p \hbox{ through (iii)}_p \hbox{ below hold} \},
\end{array}
\end{equation}
in which the three conditions are specified below.
\begin{itemize}
\item[$(i)_p$]{
  Recalling $\delta_0$ from Lemma \ref{lem:expblb:defShiftWaveProfiles}, we have $0 < \delta < \delta_0$.
}
\item[$(ii)_p$]{
  We have $\rho > 0$ and $0 < c < \min_{\zeta \in [0, 2\pi]} c_{\zeta; \delta}$.
}
\item[$(iii)_p$]{
  We have the inequalities $0 \le h'(\xi) \le 1$ for all $\xi \in \Real$,
  together with the bound $\norm{h''}_\infty < \infty$ and the identity $h(\xi) = h(0) = h_\infty \ge 0$ for all $\xi \ge 0$.
}
\end{itemize}
The variable $\rho \gg 1$ should be seen as the radius of the initial
inner core where
our sub-solution is close to one and constant,
while the function $h$ should be seen as a stretching function that smoothens the transition
between an outer region where the sub-solution follows the wave profiles and the inner region
where the sub-solution is constant.

For any $p = ( \delta, c, \rho,  h_\infty, h) \in \mathcal{D}_p$,
we now introduce the function $u: [0, \infty) \to \ell^{\infty}(\Wholes^2; \Real)$
given by
\begin{equation}
\label{eq:blob:defSubSol}
u_{ij}(t) = u_{ij ; p}(t) =  \Phi_{\zeta_{ij} ; \delta, h_\infty}( h\big(\rho + ct - R_{ij} ) \big),
\end{equation}
where the pair $(R_{ij}, \zeta_{ij})$ is defined in such a way that
\begin{equation}
 i = R_{ij}\cos(\zeta_{ij}), \quad j = R_{ij}\sin(\zeta_{ij}), \qquad R_{ij} \ge 0.
\end{equation}
In order to formulate conditions under which $u$ is in fact
a sub-solution for  \sref{eq:mr:lde:hom},
we define
\begin{equation}
\begin{array}{lcl}
\mathcal{J}^-_{ij}(t) & = & \mathcal{J}^-_{ij;p}(t)
\\[0.2cm]
& = & \dot{u}_{ij}(t) - [\Delta^+ u(t)]_{ij} - g\big(u_{ij}(t) \big).
\end{array}
\end{equation}
Upon introducing the shorthand $z_{ij} = z_{ij}(t) = \rho + ct - R_{ij}$,
we may compute
\begin{equation}
\begin{array}{lcl}
\mathcal{J}^-_{ij}(t)
& = & c h'\big(z_{ij}\big) \Phi'_{\zeta_{ij} ; \delta, h_\infty}\big(h(z_{ij}) \big)
\\[0.2cm]
& & \qquad - \big[ \Phi_{\zeta_{i+1, j} ; \delta, h_\infty}\big(h(z_{i+1,j})\big)
  + \Phi_{\zeta_{i-1 ,j} ; \delta, h_\infty}\big(h(z_{i-1,j})\big)
       -2 \Phi_{\zeta_{ij} ; \delta, h_\infty}\big(h(z_{ij}) \big)
  \big]
\\[0.2cm]
& & \qquad - \big[\Phi_{\zeta_{i, j+1} ; \delta, h_\infty}\big(h(z_{i,j+1})\big)
      + \Phi_{\zeta_{i ,j-1} ; \delta, h_\infty}\big(h(z_{i,j-1})\big)
  - 2 \Phi_{\zeta_{ij} ; \delta, h_\infty}\big(h(z_{ij}) \big)
  \big]
\\[0.2cm]
& & \qquad - g\Big(\Phi_{\zeta_{ij} ; \delta, h_\infty}\big( h(z_{ij}) \big) \Big).
\end{array}
\end{equation}
The wave profile equation \sref{eq:expblb:trvWaveMFDE} can be rephrased as
\begin{equation}
\begin{array}{lcl}
0 & = & -c_{\zeta_{ij};\delta} \Phi'_{\zeta_{ij} ; \delta, h_\infty}\big(h(z_{ij})\big)
\\[0.2cm]
& & \qquad
 + \big[
  \Phi_{\zeta_{ij} ; \delta, h_\infty}\big(h(z_{ij}) + \cos \zeta_{ij}\big)
  + \Phi_{\zeta_{ij} ; \delta, h_\infty}\big(h(z_{ij}) - \cos \zeta_{ij}\big)
  - 2 \Phi_{\zeta_{ij} ; \delta, h_\infty}\big(h(z_{ij})\big)
 \big]
\\[0.2cm]
& & \qquad
  + \big[
  \Phi_{\zeta_{ij} ; \delta, h_\infty}\big(h(z_{ij}) + \sin \zeta_{ij}\big)
  + \Phi_{\zeta_{ij} ; \delta, h_\infty}\big(h(z_{ij}) - \sin \zeta_{ij}\big)
  - 2 \Phi_{\zeta_{ij} ; \delta, h_\infty}\big(h(z_{ij})\big)
 \big]
\\[0.2cm]
& & \qquad
 + g^-_\delta\Big(\Phi_{\zeta_{ij} ; \delta, h_\infty}\big(h(z_{ij})\big) \Big).
\end{array}
\end{equation}
Upon introducing for any sequence $v: \Wholes^2 \to \Real$
the notation
\begin{equation}
\pi^+_{ij} v =
 \big( v_{i+1, j} , v_{i, j+1} , v_{i-1, j} , v_{i, j-1}, v_{ij} \big) \in \Real^5,
\end{equation}
we may write
\begin{equation}
\begin{array}{lcl}
\mathcal{J}^-_{ij;p}(t)  & = &
- \mathcal{H}_1\big(  \zeta_{ij},  z_{ij}(t) ; p \big)
  + \mathcal{R}_2\big( \pi^+_{ij} \zeta, \pi^+_{ij} z(t); p \big)
  + \mathcal{R}_3\big( \pi^+_{ij} \zeta, \pi^+_{ij} z(t); p \big),
\end{array}
\end{equation}
in which we have defined the expressions
\begin{equation}
\begin{array}{lcl}
\mathcal{H}_1(  \zeta_{ij},  z_{ij} ; p )
& = &
 \big(c_{\zeta_{ij} ; \delta} - ch'(z_{ij}) \big)
 \Phi'_{\zeta_{ij} ; \delta, h_\infty} \big(h(z_{ij})\big)
\\[0.2cm]
& & \qquad
- g^-_\delta\Big(\Phi_{\zeta_{ij} ; \delta, h_\infty}\big(h(z_{ij})\big) \Big)
+ g\Big(\Phi_{\zeta_{ij} ; \delta, h_\infty}\big(h(z_{ij})\big) \Big),
\\[0.2cm]
\mathcal{R}_2( \pi^+_{ij} \zeta, \pi^+_{ij} z ; p)
& = &
 \Phi_{\zeta_{ij};\delta, h_\infty}\big(h(z_{ij}) + \cos \zeta_{ij}\big)
      - \Phi_{\zeta_{i + 1,j};\delta, h_\infty}\big(h(z_{i + 1,j})\big)
\\[0.2cm]
& & \qquad
 + \Phi_{\zeta_{ij};\delta, h_\infty}\big(h(z_{ij}) - \cos \zeta_{ij}\big)
     - \Phi_{\zeta_{i - 1,j};\delta, h_\infty}\big(h(z_{i - 1,j})\big),
\\[0.2cm]
\mathcal{R}_3( \pi^+_{ij} \zeta, \pi^+_{ij} z ; p)
& = &
 \Phi_{\zeta_{ij};\delta, h_\infty}\big(h(z_{ij}) + \sin \zeta_{ij}\big)
   - \Phi_{\zeta_{i ,j + 1};\delta, h_\infty}\big(h(z_{i ,j + 1})\big)
\\[0.2cm]
& & \qquad
 + \Phi_{\zeta_{ij};\delta, h_\infty}\big(h(z_{ij}) - \sin \zeta_{ij}\big)
   - \Phi_{\zeta_{i ,j - 1};\delta, h_\infty}\big(h(z_{i ,j - 1})\big).
\\[0.2cm]
\end{array}
\end{equation}
Roughly speaking, our goal is to show that
$\mathcal{H}_1$ can be used to dominate $\mathcal{R}_2$ and $\mathcal{R}_3$.
The following series of results will focus on obtaining bounds for the first
two of these three expressions, which in view of the symmetry between
$\mathcal{R}_2$ and $\mathcal{R}_3$ will suffice for our purposes.
\begin{lem}
\label{lem:expblb:bndForH1}
Suppose that $\textrm{(hg}\textrm{)}_{\textrm{\S\ref{sec:prlm}}}$
and (H$\Phi$) both hold.
Then for any $p \in \mathcal{D}_p$, the bound
\begin{equation}
\mathcal{H}_1(  \zeta_{ij},  z_{ij} ; p ) \ge
   \big[ \min_{\zeta \in [0, 2\pi]} \{ c_{\zeta ; \delta} \} - c  \big]
   \Phi'_{\zeta_{ij}; \delta, h_\infty}\big(h(z_{ij})\big)
\end{equation}
holds for all pairs $(\zeta_{ij}, z_{ij}) \in \Real^2$.

In addition, there exists a constant $\kappa_3 > 0$
such that
we have the bound
\begin{equation}
\mathcal{H}_1(  \zeta_{ij},  z_{ij} ;p ) \ge  \kappa_3 \delta
\end{equation}
for any $p \in \mathcal{D}_p$ and $(\zeta_{ij}, z_{ij}) \in \Real^2$
for which $\Phi_{\zeta_{ij} ; \delta, h_\infty}\big(h(z_{ij}) \big) \le 0$.
\end{lem}
\begin{proof}
This follows directly from $(iii)_p$ together with item (v)
of Lemma \ref{lem:prlm:nlGDeltaPM}.
\end{proof}
For convenience, we split the expression $\mathcal{R}_2$
into three parts by introducing the expressions
\begin{equation}
\begin{array}{lcl}
\mathcal{R}_{2,A} ( \pi^+_{ij} \zeta, \pi^+_{ij} z ; p)
& = &
 \Phi_{\zeta_{ij}; \delta,h_\infty}(h(z_{ij}) + \cos\zeta_{ij})
     - \Phi_{\zeta_{ij}; \delta,h_\infty}\big(h(z_{ij} + \cos\zeta_{ij})\big)
\\[0.2cm]
& & \qquad
 + \Phi_{\zeta_{ij}; \delta,h_\infty}(h(z_{ij}) - \cos\zeta_{ij})
    - \Phi_{\zeta_{ij}; \delta,h_\infty}\big(h(z_{ij} - \cos\zeta_{ij})\big),
\\[0.2cm]
\mathcal{R}_{2,B} ( \pi^+_{ij} \zeta, \pi^+_{ij} z ; p )
& = &
\Phi_{\zeta_{ij}; \delta,h_\infty}\big(h(z_{ij} - \cos \zeta_{ij})\big)
   - \Phi_{\zeta_{ij}; \delta,h_\infty}\big(h(z_{i +1 ,j})\big)
\\[0.2cm]
& & \qquad
+ \Phi_{\zeta_{ij}; \delta,h_\infty}\big(h(z_{ij} + \cos \zeta_{ij})\big)
  - \Phi_{\zeta_{ij}; \delta,h_\infty}\big(h(z_{i - 1,j})\big),
\\[0.2cm]
\mathcal{R}_{2,C} ( \pi^+_{ij} \zeta, \pi^+_{ij} z ; p)
& = &
 \Phi_{\zeta_{ij}; \delta,h_\infty}\big(h(z_{i + 1,j})\big)
     - \Phi_{\zeta_{i + 1,j}; \delta,h_\infty}\big(h(z_{i + 1,j})\big)
\\[0.2cm]
& & \qquad
+ \Phi_{\zeta_{ij}; \delta,h_\infty}\big(h(z_{i -1,j})\big)
    - \Phi_{\zeta_{i - 1,j}; \delta,h_\infty}\big(h(z_{i - 1,j})\big).
\\[0.2cm]
\end{array}
\end{equation}
This allow us to write
\begin{equation}
\mathcal{R}_2( \pi^+_{ij} \zeta, \pi^+_{ij} z ; p)
 = \mathcal{R}_{2,A}( \pi^+_{ij} \zeta, \pi^+_{ij} z ; p)
 + \mathcal{R}_{2,B}( \pi^+_{ij} \zeta, \pi^+_{ij} z ; p)
 + \mathcal{R}_{2,C}( \pi^+_{ij} \zeta, \pi^+_{ij} z ; p)
\end{equation}
and we set out to bound each of the components separately.
\begin{lem}
\label{lem:expblb:termn2aa}
Suppose that $\textrm{(hg}\textrm{)}_{\textrm{\S\ref{sec:prlm}}}$
and (H$\Phi$) both hold.
Then there exists a constant $K_4 > 1$
such that the following holds true.

Consider any $p \in \mathcal{D}_p$,
any pair $(i,j) \in \Wholes^2$
and any $t \ge 0$
and suppose that at least one of the two following conditions is satisfied:
\begin{itemize}
\item[(a)]{
For all $z' \in \Real$ with $\abs{z' - z_{ij}(t)} \le 1$, we have $h'(z') = 1$.
}
\item[(b)]{
For all $\xi \in \Real$ with $\abs{\xi - h\big(z_{ij}(t)\big) } \le 2 $, we have
$\Phi''(\xi) \le 0$.
}
\end{itemize}
Then one has the estimate
\begin{equation}
\mathcal{R}_{2,A}\big( \pi^+_{ij} \zeta, \pi^+_{ij} z(t) ; p\big)
\le K_4 \Phi'_{\zeta_{ij} ; \delta, h_\infty}\Big(h\big(z_{ij}(t)\big) \Big) \norm{h''}_\infty.
\end{equation}
\end{lem}
\begin{proof}
We rewrite $\mathcal{R}_{2,A}$ as the difference
\begin{equation}
\begin{array}{lcl}
\mathcal{R}_{2,A}( \pi^+_{ij} \zeta, \pi^+_{ij} z ;p )
& = &
 \Phi_{\zeta_{ij}; \delta, h_\infty}\Big( h\big(z_{ij} + (\sigma-\tau)\cos\zeta_{ij}\big)
   + (\sigma + \tau - 1)\cos\zeta_{ij}  \Big)
     \big|_{\tau = 0}^{\tau = 1} \big|_{\sigma = 0}^{\sigma = 1},
\end{array}
\end{equation}
which allows us to apply the fundamental theorem of calculus to
write
\begin{equation}
\begin{array}{lcl}
\overline{z}_{ij}(\sigma, \tau) & = & z_{ij} + (\sigma - \tau) \cos \zeta_{ij},
\\[0.2cm]
\xi_{ij}(\sigma, \tau) & = & h\big( \overline{z}_{ij}(\sigma, \tau) \big) + (\sigma + \tau - 1)\cos\zeta_{ij}
\end{array}
\end{equation}
and obtain
\begin{equation}
\label{eq:expblob:lem:n2aa:id}
\begin{array}{lcl}
\mathcal{R}_{2,A}( \pi^+_{ij} \zeta, \pi^+_{ij} z ; p)& = &
 \int_0^1 \int_0^1 \Phi''\big(\xi_{ij}(\sigma, \tau) \big)
 \Big[1 - \Big(h'\big( \overline{z}_{ij}(\sigma, \tau)\big)\Big)^2 \Big]
     \cos^2(\zeta_{ij}) \, d\sigma d\tau
\\[0.2cm]
& &
\qquad - \int_0^1\int_0^1 \Phi'\big(  \xi_{ij}(\sigma, \tau) \big)
  h''\big(\overline{z}_{ij}(\sigma, \tau) \big)\cos^2(\zeta_{ij}) \, d\sigma d\tau.
\\[0.2cm]
\end{array}
\end{equation}
Note that for all $0 \le \sigma \le 1$
and $0 \le \tau \le 1$,
we have $\abs{\overline{z}_{ij}(\sigma, \tau) - z_{ij}} \le 1$ and
$\abs{\xi_{ij} - h(z_{ij} ) } \le 2$.
In particular,
if either (a) or (b) holds, the
first term in \sref{eq:expblob:lem:n2aa:id} is non-positive,
which allows us to apply \sref{eq:expbl:defKShift}
and obtain the desired bound.
\end{proof}

\begin{lem}
\label{lem:expblb:termn2bb}
Suppose that $\textrm{(hg}\textrm{)}_{\textrm{\S\ref{sec:prlm}}}$
and (H$\Phi$) both hold.
Then there exists a constant $K_5 > 1$
such that for any $p \in \mathcal{D}_p$,
any $(i,j) \in \Wholes^2$
and any $t \ge 0$, we have the bound
\begin{equation}
\mathcal{R}_{2,B}( \pi^+_{ij} \zeta, \pi^+_{ij} z(t) ; p)
\le K_5 \Phi'_{\zeta_{ij}; \delta, h_\infty}\Big(h\big(z_{ij}(t)\big)\Big) \frac{1}{1 + R_{ij}},
\end{equation}
which can be sharpened to
\begin{equation}
\label{eq:lem:expblb:termn2bb}
\mathcal{R}_{2,B}\big( \pi^+_{ij} \zeta, \pi^+_{ij} z(t) ; p\big)  = 0
\end{equation}
whenever $z_{ij}(t) \ge 2$.
\end{lem}
\begin{proof}
The triangle inequality
\begin{equation}
\label{eq:expblob:bndN2bb:triangle}
\abs{R_{ij} - R_{i + 1 ,j} } \le 1
\end{equation}
gives the bound $\abs{z_{ij} - z_{i+ 1, j} } \le 1$
for all $(i,j) \in \Wholes^2$.
An application of the mean value theorem  implies
\begin{equation}
\Phi_{\zeta_{ij}; \delta, h_\infty}\big(h(z_{ij} - \cos \zeta_{ij})\big)
 - \Phi_{\zeta_{ij};\delta, h_\infty}\big(h(z_{i +1 ,j})\big)
 = \Phi'_{\zeta_{ij};\delta}\big( h(z_*) \big) h'(z_*)
    \big[ z_{ij} - \cos \zeta_{ij} - z_{i+1, j} \big]
\end{equation}
for some $z_*$ that has $\abs{z_* - z_{ij} } \le 2$.
The identity \sref{eq:lem:expblb:termn2bb} for $z_{ij} \ge 2$
is now immediate.

Assume for the moment that $(i,j) \neq (0, 0)$.
Exploiting the elementary observation
\begin{equation}
(R_{i+1, j} - R_{ij} )( R_{i+1, j} + R_{ij} ) = R_{i+1, j}^2 - R_{ij}^2 = 1 +  2i,
\end{equation}
a little algebra leads to the identity
\begin{equation}
\begin{array}{lcl}
 R_{i +  1,j} - R_{ij} - \cos \zeta_{ij}
  & =  &
     \frac{1}{2R_{ij}}\left(1 + \frac{(1 + 2i)(R_{ij}-R_{i+1,j})}{R_{ij} + R_{i+1,j}}\right).
  \\[0.2cm]
  \end{array}
\end{equation}
Again applying \sref{eq:expblob:bndN2bb:triangle},
we find
\begin{equation}
\begin{array}{lcl}
 \abs{ R_{i +  1,j} - R_{ij} - \cos \zeta_{ij} }
  & \le  &\frac{1}{2R_{ij}}
    \left(1 + \frac{ \abs{ i} + \abs{i + 1}}{R_{ij} + R_{i + 1,j}}\right)
  \\[0.2cm]
  & \le &  \frac{1}{R_{ij}}.
  \end{array}
\end{equation}
In particular, for any $(i,j) \in \Wholes^2$ we may write
\begin{equation}
\abs{ R_{i + 1,j} - R_{ij} - \cos \zeta_{ij} } \le \frac{2}{1 + R_{ij} }.
\end{equation}
Observing the identity $z_{i + 1, j} - z_{ij} = R_{ij} - R_{i + 1,j}$
and noting the estimate $\abs{h(z_*) - h(z)} \le \abs{z_* - z} \le 2$,
one can now exploit the uniform bound \sref{eq:expbl:defKShift}
to complete the proof.
\end{proof}

\begin{lem}
\label{lem:expblb:estN2C}
Suppose that $\textrm{(hg}\textrm{)}_{\textrm{\S\ref{sec:prlm}}}$
and (H$\Phi$) both hold and recall the
constant $\eta_*$ defined in Corollary
\ref{cor:expblob:bndOnPartialZeta}.
Then for any $p \in \mathcal{D}_p$,
there exists a constant $K_6 = K_6(\delta, h_\infty) > 1$
so that for any pair $(i,j) \in \Wholes^2$
and any $t \ge 0$,
we have the bound
\begin{equation}
\label{eq:lem:expblb:estN2C:bnd1}
\mathcal{R}_{2,C}( \pi^+_{ij} \zeta, \pi^+_{ij} z(t) ; p )
\le K_6 e^{ - \eta_* \abs{ h(z_{ij}(t) ) } } (1 + R_{ij})^{-1},
\end{equation}
which can be sharpened to the identity
\begin{equation}
\label{eq:lem:expblb:estN2C:id2}
\mathcal{R}_{2,C}\big( \pi^+_{ij} \zeta, \pi^+_{ij} z(t) \big)  = 0
\end{equation}
whenever $z_{ij}(t) \ge  1$.
\end{lem}
\begin{proof}
First of all, there exists $C' > 0$ for which the geometric bound
\begin{equation}
\label{eq:expblb:lem:estn2c:thetaGeomBnd}
\abs{\zeta_{i +1, j} - \zeta_{ij} } \le C' (1 + R_{ij})^{-1}
\end{equation}
holds for all $(i,j) \in \Wholes^2$.
We now exploit the mean value theorem
to write
\begin{equation}
\begin{array}{lcl}
\Phi_{\zeta_{ij} ; \delta, h_\infty}\big( h(z_{i+1, j}) \big)
- \Phi_{\zeta_{i + 1, j} ; \delta, h_\infty}\big( h(z_{i + 1, j} ) \big)
&= & \partial_\zeta \Phi_{ \zeta_* ; \delta, h_\infty}\big(h(z_{i+1,j})\big)
      (\zeta_{i+1,j}-\zeta_{ij})
\end{array}
\end{equation}
for some $\zeta_* \in [0, 2 \pi ]$.
The estimate \sref{eq:lem:expblb:estN2C:bnd1}
now follows upon combining \sref{eq:expblb:lem:estn2c:thetaGeomBnd}
with \sref{eq:expblb:cor:zetaDerivExpBnd}.

Finally, the identity \sref{eq:lem:expblb:estN2C:id2}
follows from the fact that $z_{i+1,j} \ge 0$,
which implies that $h(z_{i+1, j} ) =h_\infty$
and hence
\begin{equation}
\partial_\zeta \Phi_{ \zeta_* ; \delta, h_\infty}\big(h(z_{i+1,j})\big) =
\partial_\zeta \Phi_{ \zeta_* ; \delta, h_\infty}(h_\infty)
= \partial_\zeta \Phi_{\infty; h_\infty}(\delta) = 0.
\end{equation}
\end{proof}

We now show how to construct the stretching function $h$,
reflecting the specific criteria
that arise in the statement of
Lemma \ref{lem:expblb:termn2aa}.
This stretching function is then used to confirm
the sub-solution status of the functions \sref{eq:blob:defSubSol}.

\begin{lem}
\label{lem:expblb:consH}
For any $\delta_{h}  > 0$, there exist constants $L = L(\delta_h) > 3$
and $h_{\infty} = h_{\infty}(\delta_h) > 3$
together with a $C^2$-smooth function $h = h(\delta_h): \Real \to \Real$ that satisfies
the following properties.
\begin{itemize}
\item[(i)]{
  We have $h(-L) = 0$ and $h(\xi) = h_\infty$ for all $\xi \ge 0$.
}
\item[(ii)]{
  For all $\xi \in \Real$ we have
  the bounds $0 \le h'(\xi) \le 1$ and $-\delta_h \le  h''(\xi) \le 0$.
}
\item[(iii)]{
  For all $\xi \le -L + 3$ we have $h'(\xi) = 1$.
}
\end{itemize}
\end{lem}
\begin{proof}
First of all, write $\ell = L - 3$
and consider the polynomial
\begin{equation}
P(\xi) = \frac{\ell}{2} + \frac{1}{2} \ell^{-3}(\xi^4 + 2 \ell \xi^3).
\end{equation}
It is easy to verify that
\begin{equation}
P(-\ell) = 0, \qquad P'(-\ell) = 1, \qquad P''(-\ell) = 0,
\end{equation}
while also
\begin{equation}
P(0) = \frac{1}{2} \ell, \qquad P'(0) = 0, \qquad P''(0) = 0.
\end{equation}
In addition, for all $-\ell \le \xi \le 0$ we have
\begin{equation}
0 \le P'(\xi) = \ell^{-3}\xi^2( 2 \xi + 3 \ell ) \le 1
\end{equation}
and
\begin{equation}
0 \ge P''(\xi) = 6 \ell^{-3}(\xi^2 + \ell\xi ) \ge P''(-\frac{1}{2} \ell) = -\frac{3}{2} \ell^{-1}.
\end{equation}
We can hence find the desired function $h$ by
picking $\ell \gg 1$ and writing
$h(\xi) = 3 + \xi + \ell$ for $\xi \le -\ell$,
together with $h(\xi) = 3 +  P(\xi)$ for $-\ell \le \xi \le 0$
and $h(\xi) = 3 + \frac{1}{2} \ell$
for $\xi \ge 0$.
\end{proof}

\begin{lem}
\label{lem:expblob:subsol}
Suppose that $\textrm{(hg}\textrm{)}_{\textrm{\S\ref{sec:prlm}}}$
and (H$\Phi$) both hold and
pick any $0 < c < \min_{\zeta \in [0, 2\pi]  c_{\zeta}}$.
Then there exist constants $h_{\infty} = h_\infty(c) > 0$ and $\delta_0 = \delta_0(c) > 0$
together with a $C^2$-function $h = h(c): \Real \to \Real$,
so that the following holds true.

For any $0 < \delta < \delta_0$,
there exists a constant $\rho = \rho(c,\delta) \gg 1$ such that
the function
\begin{equation}
u_{ij}(t) = \Phi_{\zeta_{ij} ; \delta, h_\infty} \big( h ( \rho + ct  - R_{ij} ) \big)
\end{equation}
satisfies the following properties.
\begin{itemize}
\item[(i)]{
  For all $t \ge 0$ and $(i,j) \in \Wholes^2$ we have the differential inequality
  \begin{equation}
    \dot{u}_{ij}(t) \le  [\Delta^+ u(t)]_{ij}(t)  + g\big(u_{ij}(t) \big).
  \end{equation}
}
\item[(ii)]{
  For all $(i,j) \in \Wholes^2$ and $t \ge 0$ for which $R_{ij} \le \rho + ct$, we have
  the inequalities
  \begin{equation}
    1 - 2 \delta < u_{ij}(t) = \Phi_{\infty ; h_\infty}(\delta) < 1 - \delta.
  \end{equation}
}
\item[(iii)]{
  For all $t \ge 0$ we have the spatial limits
  \begin{equation}
    \lim_{\abs{i } + \abs{j} \to \infty} u_{ij}(t) = - \delta.
  \end{equation}
}
\end{itemize}
\end{lem}
\begin{proof}
First, pick $\kappa' > 0$ and $\delta_0 > 0$ sufficiently small to ensure that
for all $0 < \delta < \delta_0$
we have
\begin{equation}
\min_{\zeta \in [0, 2\pi] } [c_{\theta;\delta} - c ] > \kappa'.
\end{equation}
Now, pick $\delta_h$ in such a way that
\begin{equation}
K_4 \delta_h < \frac{1}{6} \kappa'
\end{equation}
and recall the function $h$ and constants $L$ and $h_\infty$ defined in Lemma
\ref{lem:expblb:consH}.

Consider now the conditions (a) and (b) in Lemma \ref{lem:expblb:termn2aa}.
Our choice of the stretching function $h$
implies that we have $h'(z) =1$ whenever $z \le -L + 3$, which implies that (a) is
satisfied for $z_{ij} \le -L+2$. For $z_{ij} \ge -L + 2$, we note
that $h(z_{ij}) \ge 2$.
Since $\Phi''_{\theta_{ij};\delta,h_\infty}(\xi) \le 0$ for $\xi \ge 0$,
we see that (b) is satisfied in this case.
In particular, we may conclude that
for all $0 < \delta < \delta_0$ and any $t \ge 0$ we have
\begin{equation}
\mathcal{R}_{2,A}( \pi^+_{ij} \zeta, \pi^+_{ij} z(t) ; p )
 \le \frac{1}{6} \kappa' \Phi'_{\zeta_{ij} ; \delta, h_{\infty} }
      \Big(h\big(z_{ij}(t)\big) \Big)
\le  \frac{1}{6} \mathcal{H}_1(\zeta_{ij}, z_{ij}(t) ; p ).
\end{equation}

Turning our attention to $\mathcal{R}_{2, B}$,
we note that Lemma \ref{lem:expblb:termn2bb}
implies that
\begin{equation}
\mathcal{R}_{2, B}\big(\pi^+_{ij} \zeta, \pi^+_{ij} z(t) ; p \big) = 0, \qquad z_{ij}(t) \ge 2.
\end{equation}
On the other hand,
for $z_{ij} \le 2$
we see that
\begin{equation}
R_{ij} = \rho + ct - z_{ij} \ge \rho  - 2.
\end{equation}
In particular, by choosing $\rho \gg 1 $
in such a way that
\begin{equation}
K_5 (1 + [\rho - 2] )^{-1} \le \frac{1}{6} \kappa',
\end{equation}
we can ensure that for all $0 < \delta < \delta_0$ we have
\begin{equation}
\mathcal{R}_{2, B}\big(\pi^+_{ij} \zeta, \pi^+_{ij} z(t) ; p \big)
\le  \frac{1}{6} \mathcal{H}_1(\zeta_{ij}, z_{ij}(t) ; p ).
\end{equation}

It remains to consider the term $\mathcal{R}_{2, C}$
in the range $z_{ij} \le 1$,
for which we know $R_{ij} \ge \rho - 1$.
To this end,
fix a value for $0 < \delta < \delta_0$,
recall the setting of Lemma \ref{lem:expblb:estN2C}
and choose $h_* \gg 1$ in such a way that
\begin{equation}
K_6(\delta, h_\infty) e^{ - \eta_* h_*} \le \frac{1}{6} \kappa_3 \delta
\end{equation}
holds, together with
\begin{equation}
\Phi_{\zeta; \delta, h_\infty}( - h_* ) \le 0, \qquad \zeta \in [0, 2\pi].
\end{equation}
Lemma \ref{lem:expblb:bndForH1}
implies that
\begin{equation}
\label{eq:lem:expblb:subsol:bndForN2C}
\mathcal{R}_{2, C}\big(\pi^+_{ij} \zeta, \pi^+_{ij} z(t) ; p \big)
\le  \frac{1}{6} \mathcal{H}_1\big( \zeta_{ij}, z_{ij}(t) ; p \big)
\end{equation}
whenever $h(z_{ij}(t)) \le - h_*$,
which is equivalent to $z_{ij}(t) \le z_*$ for some $z_* \in \Real$.
By possibly increasing $\rho \gg 1$, we can now ensure that
\begin{equation}
K_{6} e^{ - \eta_* \abs{h(z_{ij})} } \rho^{-1} \le
\frac{1}{6} \kappa' \Phi'_{\zeta;\delta,h_{\infty}}\big(h(z_{ij}) \big)
\end{equation}
holds for all $z_* \le z_{ij} \le 1$.
This in fact implies that \sref{eq:lem:expblb:subsol:bndForN2C}
holds for all $z_{ij}(t) \in \Real$.

The terms present in $\mathcal{R}_3$ can be recovered
from the terms in $\mathcal{R}_2$ by the symmetry
$(i,j,\theta) \mapsto (j,i,\theta + \pi/2)$. In particular,
we can now write
\begin{equation}
\begin{array}{lcl}
\mathcal{J}^-_{ij;p}(t)
& \le & - \mathcal{H}_1(\zeta_{ij}, z_{ij}(t); p)
+ \frac{1}{2} \mathcal{H}_1(\zeta_{ij}, z_{ij}(t); p)
+ \frac{1}{2} \mathcal{H}_1(\zeta_{ij}, z_{ij}(t); p)
\\[0.2cm]
& \le & 0,
\end{array}
\end{equation}
which establishes (i). The remaining properties (ii) and (iii)
follow directly from properties of the profiles $\Phi_{\zeta; \delta, h_\infty}$.
\end{proof}

\begin{proof}[Proof of Proposition \ref{prp:expblob:blbExp}]
Pick $\delta_0 = \delta_0(c)$ and $h_\infty = h_\infty(c)$
from Lemma \ref{lem:expblob:subsol} above.
For any small $\eta > 0$, we can pick $0 < \delta < \delta_0$ such that
$\Phi_{\infty, h_\infty}(\delta) = 1 - \eta$
by continuity and the limit
\begin{equation}
\lim_{ \delta \downarrow 0 } \Phi_{\infty,h_\infty}(\delta) = 1.
\end{equation}
The result now follows directly from Lemma \ref{lem:expblob:subsol}.
\end{proof}

\section{Large Disturbances}
\label{sec:oblq:subsup}

In this section we show how large but localized
disturbances to planar travelling waves can be controlled
by suitably constructed sub and super-solutions. In particular,
we show that such disturbances eventually die out, showing
that the planar waves are extremely robust.

Our focus will be on planar waves that travel in the
rational direction $(\sigma_h, \sigma_v) \in \Wholes^2 \setminus \{0 , 0 \}$.
To ease our notation, we introduce a new coordinate system
that reflects the geometry of the wave.
In particular, we write
\begin{equation}
\label{eq:cds:directions}
\begin{array}{lcl}
n & = & i \sigma_h + j \sigma_v, \\[0.2cm]
l & = & i \sigma_v - j \sigma_h.
\end{array}
\end{equation}
The first of these coordinates represents the direction parallel to the propagation of the wave,
while the second coordinate represents the direction perpendicular to wave motion.
In the sequel we often refer to $n$ as the wave coordinate and
$l$ as the transverse coordinate.

We emphasize that we always have $(n,l) \in \Wholes^2$
due to our assumption that $(\sigma_h, \sigma_v) \in \Wholes^2$. However,
the inverse of the transformation \sref{eq:cds:directions}
is given by
\begin{equation}
\begin{array}{lcl}
i & = & [\sigma_h^2 + \sigma_v^2]^{-1}\big( n \sigma_h + l \sigma_v \big), \\[0.2cm]
j & = & [\sigma_h^2 + \sigma_v^2]^{-1}\big( n \sigma_v - l \sigma_h \big), \\[0.2cm]
\end{array}
\end{equation}
which means that the pairs $(n,l)$ in the range of the transformation \sref{eq:cds:directions}
represent only a sublattice of $\Wholes^2$. We choose to ignore
this issue in the current paper, simply taking $(n,l) \in \Wholes^2$.
One can think of this choice as simply solving a number of independent
systems simultaneously.

Rewriting the
homogeneous LDE \sref{eq:mr:lde:hom}
in terms of our new coordinates, we obtain the system
\begin{equation}
\dot{u}_{nl}(t) = [\Delta^\times u(t)]_{nl} + g \big( u_{nl}(t) \big), \qquad (n,l) \in \Wholes^2.
\label{eq:lde:newCoords}
\end{equation}
Here we have introduced
the notation
\begin{equation}
[\Delta^\times u]_{nl} = \sum_{(n', l') \in \mathcal{N}^\times_{\Wholes^2}(n,l)} [ u_{n' l'} - u_{ij} ],
\end{equation}
for any $u \in \ell^\infty(\Wholes^2 ; \Real)$,
in which the neighbour set
\begin{equation}
\label{eq:oblq:newTimesCoordsNeighbours}
\mathcal{N}^\times_{\Wholes^2}(n,l) =
\{  ( n + \sigma_h, l + \sigma_v ),
  (n + \sigma_v, l - \sigma_h ),
  (n - \sigma_h, l - \sigma_v ),
  (n - \sigma_v, l + \sigma_h )
\} \subset \Wholes^2
\end{equation}
encodes the geometry of the new coordinate system.

The planar travelling wave solutions
\sref{eq:prlm:trvWaveAnsatz}
can now be written as
\begin{equation}
\label{eq:oblq:trvWaveAnsatz}
u_{nl}(t) = \Phi(n + ct). 
\end{equation}
The main result of this section constructs sub and super-solutions
for \sref{eq:lde:newCoords} that converge to
shifted versions of \sref{eq:oblq:trvWaveAnsatz}.
Our version is rather technical as we intend it
to be strong enough to allow the effects of the obstacle to be included later on.
The properties (vi) - (viii) together with the fact that we do not prescribe
a specific choice for $z$ should be seen in this light.
On the other hand, the algebraic decay properties
imposed on $z$ and stated in (iii)
can be seen as direct consequences of
the discrete nature of the lattice.

\begin{prop}
\label{prp:hom:oblq:mr:sub:sup}
Consider any angle $\zeta_*$ with $\tan \zeta_* \in \mathbb{Q}$
and suppose that (Hg) and $(HS)_{\zeta_*}$
both hold. Pick $(\sigma_h, \sigma_v) \in \Wholes^2 \setminus \{(0 , 0)\}$
with the property that
\begin{equation}
  \sqrt{\sigma_h^2 + \sigma_v^2}(\cos \zeta_*, \sin \zeta_*) = (\sigma_h, \sigma_v),
  \qquad
  \mathrm{gcd}(\sigma_h, \sigma_v) = 1
\end{equation}
and suppose that $\textrm{(h}\Phi\textrm{)}_{\textrm{\S\ref{sec:prlm}}}$
holds for this pair $(\sigma_h, \sigma_v)$ with $c > 0$.

Then there exist constants
\begin{equation}
\delta_\epsilon > 0,
\qquad  \eta_z > 0,
\qquad K_Z > 1,
\qquad K_{\mathcal{N}} > 1,
\qquad \eta_{\mathcal{N}} > 0,
\end{equation}
such that for any triplet $(\epsilon_1, \epsilon_2, \epsilon_3)$
that has
\begin{equation}
0 < 2 \epsilon_2 < \epsilon_1 \le \delta_{\epsilon},
\qquad 0 < \epsilon_3 \le \epsilon_1
\end{equation}
and any  pair $\Omega_\perp > 0$, $\Omega_{\mathrm{phase}} > 0$,
there exists a function $\theta: [0, \infty) \to \ell^{\infty}(\Wholes; \Real)$
so that the following holds true.

Consider any phase shift $\vartheta \in \Real$ and any function $z: [0, \infty) \to \Real$
that satisfies the conditions
\begin{itemize}
\item[$(i)_z$]{
   We have $z'(t) \ge - \eta_z z(t)$ for all $t \ge 0$.
}
\item[$(ii)_z$]{
  We have $0 < z(t) \le z(0) = \epsilon_1$ for all $t \ge 0$.
}
\item[$(iii)_z$]{
   We have $z(t) \ge \epsilon_3 (1 + t)^{-3/2}$.
}
\end{itemize}
There exist functions $W^\pm: [0, \infty) \to \ell^{\infty}(\Wholes^2 ; \Real)$
and $\xi^\pm: [0, \infty) \to \ell^{\infty}(\Wholes^2; \Real)$
that satisfy the following properties.
\begin{itemize}
\item[(i)]{
The quantities
\begin{equation}
\begin{array}{lcl}
\mathcal{J}^-_{nl}(t) & = &  \dot{W}^-_{nl}(t)
  - [\Delta^\times W^-(t)]_{nl}
  - g\big( W^-_{nl}(t) \big),
\\[0.2cm]
\mathcal{J}^+_{nl}(t) & = & \dot{W}^+_{nl}(t)
- [\Delta^\times W^+(t)]_{nl} - g\big( W^+_{nl}(t) \big),
\end{array}
\end{equation}
satisfy the bounds
\begin{equation}
\label{eq:prp:oblq:subsub:termsToSpare}
\begin{array}{lcl}
\mathcal{J}^-_{nl}(t) & \le & - \frac{1}{2} \eta_z z(t),
\\[0.2cm]
\mathcal{J}^+_{nl}(t) & \ge & + \frac{1}{2} \eta_z z(t),
\end{array}
\end{equation}
for all $t \ge 0$ and $(n,l) \in \Wholes^2$.
}
\item[(ii)]{
  For $\abs{l} \le \Omega_\perp$, we have
  \begin{equation}
    W^-_{nl}(0) \le \Phi(n + ct + \vartheta - \Omega_{\mathrm{phase}} ),
    \qquad W^+_{nl}(0) \ge \Phi(n + ct + \vartheta + \Omega_{\mathrm{phase}} ).
  \end{equation}
}
\item[(iii)]{
  For every $t \ge 0$ and every $(n,l) \in \Wholes^2$, we have the bounds
  \begin{equation}
    \begin{array}{lclcl}
    \Phi\big(\xi^-_{nl}(t) \big) - \epsilon_2 (1 + t)^{-1/2}
    & \le & W^-_{nl}(t) + z(t)
    & \le & \Phi\big(\xi^-_{nl}(t) \big) + \epsilon_2 (1 + t)^{-1/2}
    \\[0.2cm]
    \Phi\big(\xi^+_{ nl}(t) \big) - \epsilon_2 (1 + t)^{-1/2}
    & \le & W^+_{nl}(t) - z(t)
    & \le & \Phi\big(\xi^+_{nl}(t) \big) + \epsilon_2 (1 + t)^{-1/2}
    \\[0.2cm]
    \end{array}
  \end{equation}

}
\item[(iv)]{
  We have $\theta_l(t) \ge 0$ for all $t \ge 0$ and $l \in \Wholes$, together with the
  uniform limit
  \begin{equation}
    \lim_{t \to \infty} [ \sup_{l \in \Wholes} \theta_l(t) ] = 0.
  \end{equation}
}
\item[(v)]{
  Introducing the function
  \begin{equation}
    Z(t)= K_Z \int_{0}^{t} z(t') \, d t',
  \end{equation}
  we have the identities
  \begin{equation}
    \begin{array}{lcl}
      \xi^-_{nl}(t) & = & n + ct + \vartheta - \theta_l(t) - Z(t), \\[0.2cm]
      \xi^+_{nl}(t) & = & n + ct + \vartheta + \theta_l(t) + Z(t). \\[0.2cm]
    \end{array}
  \end{equation}
}
\item[(vi)]{
  Consider any bounded set $S \subset \Wholes^2$. Upon writing
  \begin{equation}
   \mathrm{diam}(S) = \sup_{ (n,l) \in S, (n',l') \in S} \big[\abs{n - n'} + \abs{l -l'} \big],
  \end{equation}
  we have the uniform bounds
  \begin{equation}
    \begin{array}{lcl}
      \max_{(n,l)\in S} \xi^-_{nl}(t) - \min_{(n,l) \in S}\xi^-_{nl}(t)
        & \le & 1 + \mathrm{diam}(S), \\[0.2cm]
      \max_{(n,l)\in S} \xi^+_{ nl}(t) - \min_{(n,l) \in S}\xi^+_{ nl}(t)
        & \le & 1 + \mathrm{diam}(S), \\[0.2cm]
    \end{array}
  \end{equation}
  for every $t \ge 0$.
}

\item[(vii)]{
  For any $(n,l) \in \Wholes^2$ with $\abs{l} \le \Omega_\perp$ and $t \ge 0$, we have
  \begin{equation}
   \dot{\xi}^\pm_{nl}(t) \ge \frac{c}{2}.
  \end{equation}
}

\item[(viii)]{
  For any pairs $(n,l) \in \Wholes^2$ and $(n', l') \in \mathcal{N}^\times_{\Wholes^2}(n,l)$,
  we have the bounds
  \begin{equation}
    \abs{ W^\pm_{nl}(t) - W^\pm_{n'l'}(t) } \le K_{\mathcal{N}} e^{ - \eta_{\mathcal{N}}  \abs{ \xi^\pm_{nl}(t)  } }.
  \end{equation}
}
%
\end{itemize}
\end{prop}
Notice that a direct consequence of (iii), (iv) and (v) is that
for all $(n,l) \in \Wholes^2$, we have
\begin{equation}
  W^-_{nl}(0) \le \Phi(n + ct + \vartheta ) - \frac{1}{2} \epsilon_1,
  \qquad W^+_{nl}(0) \ge \Phi(n + ct + \vartheta) + \frac{1}{2} \epsilon_1,
\end{equation}
which shows that we can indeed interpret
the result above as a mechanism for turning small global additive
perturbations into small phase shifts, as customary in one-dimensional
results of this nature. The extra feature in two dimensions
is that we can also include large localized phase shifts
in the initial perturbation.

In order to assist the reader in interpreting
the result above, we conclude this subsection
by using it to establish
the nonlinear stability
of the travelling wave \sref{eq:oblq:trvWaveAnsatz},
as stated in Theorem \ref{thm:mr:unobstructed:stb}.
As a preparation, we construct a template function $z_{\mathrm{hom}}$
that satisfies the requirements $(i)_z$ through $(iii)_z$.
\begin{lem}
\label{lem:oblq:defZHom}
Fix any $0 < \eta_z < 1$. Then there exists constants
$\mathcal{I}_{\mathrm{hom}} = \mathcal{I}_{\mathrm{hom}}(\eta_z) > 1$
and $\kappa_{\mathrm{hom}} = \kappa_{\mathrm{hom}}(\eta_z) > 0$
together with a $C^1$-smooth function
$z_{\mathrm{hom}}: [0, \infty) \to \Real$
that satisfies the following properties.
\begin{itemize}
\item[(i)]{
  We have $z'_{\mathrm{hom}}(t) \ge -\eta_z z_{\mathrm{hom}}(t)$ for all $t \ge 0$.
}
\item[(ii)]{
  We have $\kappa_{\mathrm{hom}} (1 + t )^{-3/2} \le z_{\mathrm{hom}}(t ) \le z_{\mathrm{hom}}(0) = 1$ for all $ t \ge 0$.
}
\item[(iii)]{
  We have $\int_{0}^{\infty} z_{\mathrm{hom}}(t) \, d t < \mathcal{I}_{\mathrm{hom}}$.
}
\end{itemize}
\end{lem}
\begin{proof}
For $0 \le t \le \frac{3}{2} \eta_z^{-1} - 1$ we write
\begin{equation}
z_{\mathrm{hom}}(t) = e^{ - \eta_z t},
\end{equation}
while for $t \ge \frac{3}{2}\eta_z^{-1} - 1$ we write
\begin{equation}
z_{\mathrm{hom}}(t) = \eta_z^{-3/2} (\frac{3}{2})^{3/2} e^{ \eta_z - \frac{3}{2} }(1 + t)^{-3/2}.
\end{equation}
One can readily verify that $z_{\mathrm{hom}}$ is $C^1$-smooth and that (i) and (ii)
are satisfied. Property (iii) follows from the identity
\begin{equation}
\int_{0}^{\infty} z_{\mathrm{hom}}(t) \, d t = \eta_z^{-1} \big[ 2 e^{\eta_z - \frac{3}{2} } + 1 \big].
\end{equation}
\end{proof}

\begin{proof}[Proof of Theorem \ref{thm:mr:unobstructed:stb}]
Pick any $\delta_* > 0$. We restrict ourselves here to showing that
\begin{equation}
\label{eq:oblq:mr:finLimitToShow}
\liminf_{t \to \infty} \inf_{(n,l) \in \Wholes^2} [ U_{nl}(t) - \Phi(n + ct ) ] \ge - \delta_*,
\end{equation}
noting that the companion bound
\begin{equation}
\limsup_{t \to \infty} \sup_{(n,l) \in \Wholes^2} [ U_{nl}(t) - \Phi(n + ct ) ] \le + \delta_*,
\end{equation}
can be obtained in a similar fashion.

Pick $\epsilon_1 > 0$ in such a way that
\begin{equation}
\label{eq:prp:unobs:bndOnEps1}
\epsilon_1 K_Z \mathcal{I}_{\mathrm{hom}} \norm{\Phi'}_\infty \le \delta_*
\end{equation}
and write $\epsilon_2 = \frac{1}{2} \epsilon_1$,
$\epsilon_3 = \epsilon_1 \kappa_{\mathrm{hom}}$, $\vartheta = 0$
and $z(t) = \epsilon_1 z_{\mathrm{hom}}(t)$.
There exists a finite set $S_{\epsilon_2} \subset \Wholes^2$
for which
\begin{equation}
\abs{U_{nl}(0) - \Phi( n  ) } \le \epsilon_2
\end{equation}
holds for all $(n,l) \in \Wholes^2 \setminus S_{\epsilon_2}$.
In particular, by picking $\Omega_\perp$ and $\Omega_{\mathrm{phase}}$
appropriately, we can ensure that
\begin{equation}
W^-_{nl}(0) \le U_{nl}(0)
\end{equation}
holds for all $(n,l) \in \Wholes^2$.
Since
\begin{equation}
\lim_{t \to \infty} \sup_{(n,l) \in \Wholes^2}
  \big[ W^-_{nl}(t) - \Phi\big(n + ct  - Z(t) \big)\big]  = 0,
\end{equation}
while also
\begin{equation}
\abs{\Phi(n + ct ) - \Phi\big(n + ct - Z(t) \big)}
\le \norm{\Phi'}_\infty \abs{Z(t)} \le \epsilon_1 \norm{\Phi'}_\infty K_Z \mathcal{I}_{\mathrm{hom}}
 \le \delta_*,
\end{equation}
the comparison principle directly implies \sref{eq:oblq:mr:finLimitToShow}.
\end{proof}

\subsection{Notation}
\label{sec:oblq:notation}
In this subsection we set up the notation
that will be used throughout \S\ref{sec:oblq:subsup}.
In addition, we perform some preliminary computations
that will aid us in the construction
of the sub and super-solutions described
in Proposition \ref{prp:hom:oblq:mr:sub:sup}.

First of all, we introduce
for any $u \in \ell^\infty(\Wholes^2; \Real)$
and any $(n,l) \in \Wholes^2$, the vector
\begin{equation}
\pi^\times_{nl} u = \big( u_{n + \sigma_h, l + \sigma_v} , u_{n + \sigma_v, l - \sigma_h}, u_{n - \sigma_h, l -\sigma_v},
 u_{n - \sigma_v, l + \sigma_h} , u_{nl} \big) \in \Real^5,
\end{equation}
which can be seen as evaluating $u$
on a stencil of grid points that consists
of $(n,l)$ and its nearest neighbours $\mathcal{N}^\times_{\Wholes^2}(n,l)$.

Upon introducing the vector
\begin{equation}
L^\times = ( 1, 1, 1, 1, -4) \in \Real^5,
\end{equation}
we can now rewrite \sref{eq:lde:newCoords} in the form
\begin{equation}
\label{eq:lde:locForm}
\dot{u}_{nl}(t) = L^\times  \pi_{nl}^{\times} u(t)  + g\big( u_{nl}(t) \big).
\end{equation}
To avoid clutter, we also introduce the operator
\begin{equation}
\pi^\times: \ell^{\infty}(\Wholes^2; \Real) \to \ell^\infty(\Z^2,\Real^5)
\end{equation}
that acts as
\begin{equation}
[\pi^\times u]_{nl} = \pi^\times_{nl} u.
\end{equation}
This allows us to restate \sref{eq:lde:newCoords} as
\begin{equation}
\label{eq:lde:glbForm}
\dot{u}(t) = L^\times  \pi^{\times} u(t)  + g\big( u(t) \big),
\end{equation}
in which the nonlinearity $g$ is interpreted to act componentwise.
We will refer to equations such as \sref{eq:lde:glbForm},
where the dependence on $(n,l)$ has been dropped,
as equations in global form. Similarly,
equations such as \sref{eq:lde:locForm} are called
equations in local form.

\subsubsection*{Summation Convention}
Throughout the sequel
we will use greek indices
\begin{equation}
\mu, \mu', \mu'' \in \{1, 2, 3, 4, 5\},
\qquad \nu, \nu', \nu'' \in \{1, 2, 3, 4,5 \}
\end{equation}
with  the following summation convention.
The indices $\mu, \mu', \mu''$ will only appear
on the right side of identities and any term involving $n \ge 1$
\textbf{distinct} greek indices needs to be summed over
all $5^n$ combinations of these indices.
On the other hand, the indices $\nu, \nu', \nu''$ may appear on both sides
of an identity and do \textbf{not} require a summation.

We refer to individual components of $\pi^\times$ and $L^\times$
by writing
\begin{equation}
\pi^\times_{nl}  = \big( \pi^\times_{nl; 1} , \ldots , \pi^\times_{nl; 5} \big),
\qquad L^\times = \big( L^\times_{1} , \ldots , L^\times_{ 5} \big),
\end{equation}
together with
\begin{equation}
\pi^\times = \big( \pi^\times_{; 1}, \ldots , \pi^\times_{; 5} \big).
\end{equation}
In particular, the local form \sref{eq:lde:locForm}
can be written as
\begin{equation}
\dot{u}_{nl}(t) = L^\times_{\mu}  \pi_{nl; \mu}^{\times} u(t)  + g\big( u_{nl}(t) \big),
\end{equation}
while the global form \sref{eq:lde:glbForm} can be written as
\begin{equation}
\dot{u}(t) = L^\times_{\mu}  \pi^{\times}_{; \mu} u(t)  + g\big( u(t) \big).
\end{equation}

Let us now fix a
speed $c \in \Real$ together with a
$C^1$-smooth function $\theta: [0, \infty) \to \ell^\infty(\Wholes; \Real)$
and a $C^1$-smooth function $Z: [0, \infty) \to \Real$.
In what follows, a crucial role will be played by the
related quantities
\begin{equation}
\xi_{nl}(t) := n + ct - \theta_l(t) - Z(t), \qquad (n,l) \in \Wholes^2, \qquad t \ge 0.
\end{equation}
In particular, let us consider any $C^2$-smooth function $h: \Real \to \Real$
for which $h$, $h'$ and $h''$ are all bounded.
Upon introducing the notation
\begin{equation}
\j_{nl}(h; t) = h\big(\xi_{nl}(t) \big)
\end{equation}
and writing
\begin{equation}
\j(h ;t) \in \ell^{\infty}(\Wholes; \Real^2)
\end{equation}
for the sequence that has
\begin{equation}
[\j(h ;t)]_{nl} = \j_{nl}(h ; t),
\end{equation}
we observe that the map $t \mapsto \j(h ; t)$
is a $C^1$-smooth map from $[0, \infty)$
into $\ell^\infty(\Wholes; \Real^2)$.
In particular, by evaluating at $\xi_{nl}(t)$
the scalar function $h$ has been transformed
into a smooth sequence-valued function.
We lose an order of smoothness here because of
the uniform continuity requirements arising from the $\ell^{\infty}$
norm.

Let us now consider two functions
\begin{equation}
p, q \in C^1 \Big( [0, \infty) , \ell^{\infty}(\Wholes ; \Real) \Big).
\end{equation}
We introduce the notation
\begin{equation}
\j_{nl}(q, h; t) = q_l(t) h\big(\xi_{nl}(t) \big),
\qquad \j_{nl}(p, q, h ; t) = p_l(t) q_l(t) h\big(\xi_{nl}(t) \big)
\end{equation}
and write
\begin{equation}
\j(q, h ;t) \in \ell^{\infty}(\Wholes; \Real^2)
\qquad
\j(p, q, h ;t) \in \ell^{\infty}(\Wholes; \Real^2)
\end{equation}
for the sequences that have
\begin{equation}
[\j(q,h ;t)]_{nl} = \j_{nl}(q,h ; t),
\qquad
[\j(p,q,h ;t)]_{nl} = \j_{nl}(p,q,h ; t).
\end{equation}
As before,
the maps $t \mapsto \j(q,h ; t)$
and $t \mapsto \j(p,q,h;t)$
are $C^1$-smooth maps from $[0, \infty)$
into $\ell^\infty(\Wholes; \Real^2)$.

Turning our attention to sequences $\theta \in \ell^{\infty}(\Wholes; \Real)$,
we need to introduce a number of difference operators.
To this end, we define the shifts
\begin{equation}
  (\sigma_1, \ldots, \sigma_5) = \big( \sigma_v, -\sigma_h , - \sigma_v, \sigma_h, 0 \big)
\end{equation}
and introduce the notation
\begin{equation}
\pi^\diamond_l \theta = \big(\pi^\diamond_{l ; 1} \theta, \ldots , \pi^{\diamond}_{l ; 5} \theta  \big) \in \Real^5
\end{equation}
for first difference operators $\pi^\diamond_{l; \nu}$ that act as
\begin{equation}
\pi^{\diamond}_{l ; \nu} \theta = \theta_{l + \sigma_\nu} - \theta_l, \qquad 1 \le \nu \le 5.
\end{equation}
In global form, we write
\begin{equation}
\pi^\diamond \theta = \{ \pi^\diamond_l \theta \}_{l \in \Wholes} \in \ell^{\infty}(\Wholes; \Real^5)
\end{equation}
and refer to the individual components as
\begin{equation}
\pi^\diamond_{ ; \nu} \theta = \{ \pi^\diamond_{l ; \nu} \theta \}_{l \in \Wholes} \in \ell^{\infty}(\Wholes; \Real).
\end{equation}

In a similar fashion,
we introduce the notation
\begin{equation}
\pi^{\diamond \diamond}_l \theta = \big( \pi^{\diamond \diamond}_{l ; \nu \nu'} \theta \big)_{(\nu, \nu') \in \{1, \ldots 5 \}^2 } \in \Real^{5 \times 5},
\end{equation}
with second difference operators that act as
\begin{equation}
\pi^{\diamond\diamond}_{l ; \nu \nu'} = \big(\theta_{l + \sigma_\nu + \sigma_\nu'} - \theta_{l + \sigma_\nu'} \big)
- \big(\theta_{l + \sigma_\nu} - \theta_l \big).
\end{equation}
In other words, we have
\begin{equation}
\pi^{\diamond\diamond}_{l ; \nu \nu'} = \pi^{\diamond}_{l;\nu'} \pi^\diamond_{; \nu} \theta.
\end{equation}
In global form, we write
\begin{equation}
\pi^{\diamond\diamond} \theta = \{ \pi^{\diamond \diamond}_l \theta \}_{l \in \Wholes} \in \ell^{\infty}(\Wholes; \Real^{5 \times 5}),
\qquad
\pi^{\diamond \diamond}_{ ; \nu \nu'} \theta = \{ \pi^{\diamond\diamond}_{l ; \nu \nu'} \theta \}_{l \in \Wholes} \in \ell^{\infty}(\Wholes; \Real).
\end{equation}
Naturally, this allows us to define
third differences
\begin{equation}
\pi^{\diamond \diamond \diamond} \theta \in \ell^{\infty}(\Wholes; \Real^{5\times 5\times 5} )
\end{equation}
by means of the components
\begin{equation}
\pi^{\diamond \diamond \diamond}_{l ; \nu \nu' \nu''} \theta = \pi^\diamond_{l ; \nu''} \pi^{\diamond \diamond}_{; \nu \nu'} \theta
=\pi^\diamond_{l ; \nu''} \pi^{\diamond}_{ ; \nu'} \pi^{\diamond}_{; \nu } \theta.
\end{equation}

We also need to consider a second transformation
of the function $h$. To this end, we define the five constants
\begin{equation}
  (\tau_1, \ldots, \tau_5) = \big( \sigma_h, \sigma_v , - \sigma_h, -\sigma_v, 0 \big)
\end{equation}
and write $\tau h \in C^2(\Real, \Real^5)$ for the function that has
\begin{equation}
\begin{array}{lcl}
[\tau h](\xi) 
& = & \big(h( \xi + \tau_1), \ldots,   h(\xi + \tau_5) \big) \in \Real^5.
\end{array}
\end{equation}
Abusing notation, we often write
\begin{equation}
[\tau h](\xi)  = \big( [\tau_1 h](\xi), \ldots, [\tau_5 h](\xi) \big)
\end{equation}
for the five components of $\tau h$. We note that the pairing of the constants
$\sigma_\mu$ and $\tau_\mu$ comes directly from
the form of the neighbour set $\mathcal{N}^\times_{\Wholes^2}$
defined in \sref{eq:oblq:newTimesCoordsNeighbours}.

For any $(n,l) \in \Wholes^2$ and $t \ge 0$,
we write
\begin{equation}
\j_{nl}( L ; t) = L^\times  + \Big(0 , 0, 0, 0, g'\big(\Phi(\xi_{nl}(t) ) \big) \Big) \in \Real^5.
\end{equation}
In particular, for any sequence $u \in \ell^\infty(\Wholes^2; \Real)$
we have
\begin{equation}
\j_{nl}(L; t) \pi^\times_{nl} u = L^\times_{\mu} \pi^\times_{nl ; \mu} u
+ g'\Big(\Phi\big(\xi_{nl}(t) \big) \Big) u_{nl}.
\end{equation}
In global form, we shorten this to
\begin{equation}
\j(L; t) \pi^\times u = L^\times_{\mu} \pi^\times_{ ; \mu} u
+ g'\Big(\Phi\big(\xi(t) \big) \Big) u.
\end{equation}

For convenience, we often use the shorthand
\begin{equation}
\begin{array}{lcl}
\j_{nl}( L \tau  h ; t) & = & \j_{nl}(L ; t) \j_{nl}( \tau h; t)
\\[0.2cm]
& = & L^\times_{\mu} [\tau_\mu h]\big(\xi_{nl}(t) \big)
+ g'\Big(\Phi\big(\xi_{nl}(t) \big)\Big) h\big(\xi_{nl}(t) \big)
\\[0.2cm]
& = & L^\times_{\mu} h\big(\xi_{nl}(t) + \tau_\mu \big)
+ g'\Big(\Phi\big(\xi_{nl}(t) \big)\Big) h\big(\xi_{nl}(t) \big),
\end{array}
\end{equation}
together with
\begin{equation}
\begin{array}{lcl}
\j_{nl}( q, L \tau  h ; t) & = & q_l(t) \j_{nl}(L ; t) \j_{nl}( \tau h; t)
\\[0.2cm]
& = & q_l(t) \j_{nl}(L \tau h; t),
\\[0.2cm]
\j_{nl}( p, q, L \tau  h ; t) & = & p_l(t) q_l(t) \j_{nl}(L ; t) \j_{nl}( \tau h; t)
\\[0.2cm]
& = & p_l(t) q_l(t) \j_{nl}( L \tau h ; t).
\end{array}
\end{equation}
In global form, we write
\begin{equation}
\begin{array}{lcl}
\j( L \tau h; t) & = & \j(L; t) \j( \tau h; t),
\\[0.2cm]
\j(q,  L \tau h; t) & = & q(t) \j( L \tau h; t),
\\[0.2cm]
\j(p, q,  L \tau h; t) & = & p(t) q(t) \j( L \tau h; t).
\end{array}
\end{equation}

\subsection{Preliminary Computations}
\label{sec:oblq:prlm}
In this subsection we set out to derive a number of
tractable expressions for the quantities
\begin{equation}
\j(L ; t) \pi^\times \j( h; t),
\qquad \j(L ; t) \pi^\times \j(q, h; t),
\qquad \j(L ; t) \pi^\times \j( p ,q, h; t),
\end{equation}
since these play a crucial role in the verification
of the relevant differential inequalities for our sub-solution.
For use in the sequel when discussing obstacle problems,
we also consider the quantities
\begin{equation}
[\pi^\times_{; \nu} - \pi^\times_{; 5} ] \j( h; t),
\qquad [\pi^\times_{; \nu} - \pi^\times_{; 5} ] \j(q, h; t),
\qquad [\pi^\times_{; \nu} - \pi^\times_{; 5} ] \j( p ,q, h; t).
\end{equation}
As in \S\ref{sec:oblq:notation},
the function $h: \Real \to \Real$
is assumed to be $C^2$-smooth
with uniform bounds for $h$, $h'$ and $h''$,
while $p$ and $q$ are assumed to be
two $C^1$-smooth functions
mapping $[0, \infty)$ into $\ell^{\infty}(\Wholes; \Real)$.

We start by writing
\begin{equation}
\label{sec:oblq:prlm:pijh}
\begin{array}{lcl}
\pi_{nl; \nu}^\times \, \j(h ; t)
& = &  \j_{nl}(\tau_\nu h;t)  + \mathcal{M}_{h, 1; \nu}\big( \xi_{nl}(t) , \pi^{\diamond}_{l;\nu} \theta(t)  \big)
\\[0.2cm]
& = &  \j_{nl}(\tau_\nu h;t)   - \pi^{\diamond}_{l; \nu} \theta(t) \,  \j_{nl}( \tau_\nu h'; t)
 + \mathcal{M}_{h, 2; \nu}( \xi_{nl}(t), \pi^{\diamond}_{l;\nu} \theta(t)\big),
\\[0.2cm]
\end{array}
\end{equation}
which should be seen as  implicit definitions for the expressions $\mathcal{M}_{h, 1 ; \nu}$
and $\mathcal{M}_{h, 2 ; \nu}$. The mean value theorem
implies the identities
\begin{equation}
\begin{array}{lcl}
\mathcal{M}_{h, 1; \nu}( \xi_{nl},  \pi^\diamond_{l;\nu} \theta  )  & = &
 h'( \xi_{nl} + \tau_\nu + \vartheta_1 [\theta_l - \theta_{l + \sigma_\nu}] \big)
 [\theta_l - \theta_{l + \sigma_\nu} ],
\\[0.2cm]
\mathcal{M}_{h, 2; \nu}( \xi_{nl}, \pi^\diamond_{l;\nu} \theta )
& = & \frac{1}{2} h''( \xi_{nl} + \tau_\nu + \vartheta_2
  [\theta_l - \theta_{l + \sigma_\nu}] \big) [\theta_l - \theta_{l + \sigma_\nu} ]^2,
\end{array}
\end{equation}
for some pair $0 < \vartheta_1 < 1$ and $0 < \vartheta_2 < 1$
that depends on $\xi_{nl} \in \Real$ and $\pi^\diamond_{l; \nu} \theta \in \Real^5$.
In particular, the uniform bounds on $h'$ and $h''$
imply that there exists $C > 0$ such that
\begin{equation}
\begin{array}{lcl}
\mathcal{M}_{h, 1; \nu}( \xi_{nl} , \pi^\diamond_{l; \nu} \theta ) & \le & C \abs{\pi^\diamond_{l;\nu} \theta },
\\[0.2cm]
\mathcal{M}_{h, 2; \nu}(\xi_{nl} , \pi^\diamond_{l;\nu} \theta  ) & \le & C \abs{\pi^\diamond_{l; \nu} \theta }^2,
\end{array}
\end{equation}
for any $\xi_{nl} \in \Real$, any $\pi^\diamond_{l;\nu} \theta \in \Real$
and any integer  $1 \le \nu \le 5$.
For convenience, for any $t \ge 0$ we introduce the global form expressions
\begin{equation}
\mathcal{N}_{h, 1; \nu}( \pi^\diamond_{; \nu} \theta ; t) \in \ell^\infty(\Wholes^2; \Real),
\qquad
\mathcal{N}_{h, 2; \nu}( \pi^\diamond_{; \nu} \theta ; t) \in \ell^\infty(\Wholes^2; \Real),
\end{equation}
that are given by
\begin{equation}
[\mathcal{N}_{h, 1; \nu}( \pi^\diamond_{; \nu} \theta ; t)]_{nl}
  = \mathcal{M}_{h, 1; \nu}\big( \xi_{nl}(t), \pi^\diamond_{l; \nu} \theta(t) \big),
\qquad
[\mathcal{N}_{h, 2; \nu}( \pi^\diamond_{; \nu} \theta ; t)]_{nl}
  = \mathcal{M}_{h, 2; \nu}\big( \xi_{nl}(t), \pi^\diamond_{l; \nu} \theta(t) \big).
\end{equation}

Notice in particular, that for any $\nu \in \{1 , \ldots, 4\}$ we have
\begin{equation}
\begin{array}{lcl}
[\pi^\times_{nl;\nu} - \pi^\times_{nl; 5} ] \, \j(h; t)
& = & \j_{nl}(\tau_{\nu} h; t ) - \j_{nl}( h; t) + \mathcal{M}_{h, 1 ; \nu}\big(\xi_{nl}(t), \pi^{\diamond}_{l;\nu} \theta(t)  \big).
\\[0.2cm]
\end{array}
\end{equation}
In global form, we write this as
\begin{equation}
\begin{array}{lcl}
[\pi^\times_{;\nu} - \pi^\times_{; 5} ] \, \j(h; t)
& = & \j(\tau_{\nu} h; t ) - \j( h; t) + \mathcal{N}_{h, 1 ; \nu}\big( \pi^{\diamond}_{;\nu} \theta ;t  \big).
\\[0.2cm]
\end{array}
\end{equation}

Moving on, we use \sref{sec:oblq:prlm:pijh} to compute
\begin{equation}
\begin{array}{lcl}
\j_{nl}(L ; t) \, \pi^\times_{nl} \, \j(h; t)
& = & L^\times_\mu \j_{nl}(\tau_\mu h ; t)
 + g'\Big(\Phi\big(   \xi_{nl}(t) \big) \Big) \j_{nl}( h; t)
+ L^\times_{\mu} \mathcal{M}_{h, 1 ; \mu}\big(\xi_{nl}(t), \pi^{\diamond}_{l;\mu} \theta(t)  \big)
\\[0.2cm]
& = &
L^\times_\mu \j_{nl}( \tau_\mu h ; t)
  + g'\Big(\Phi\big( \xi_{nl}(t) \big)\Big) \j_{nl}(h ;t)
- L^\times_\mu \pi^{\diamond}_{l; \mu} \theta(t) \, \j_{nl}( \tau_\mu h' ; t)
\\[0.2cm]
& & \qquad
+ L^\times_\mu \mathcal{M}_{h, 2; \mu}\big(\xi_{nl}(t); \pi^{\diamond}_{l; \mu} \theta(t) \big).
\end{array}
\end{equation}
We remind the reader that according to our summation convention,
all three terms featuring $\mu$ in the final identity
come with an implicit $\sum_{\mu=1}^5$ summation in front.
Exploiting the fact that $\pi^\diamond_{;5} = 0$
and $L^\times_{\nu} = 1$ for $1 \le \nu \le 4$,
we can now write
\begin{equation}
\j(L, t) \pi^\times \j(h ; t)
= \j( L  \tau h ; t) - \j( \pi^\diamond_{; \mu} \theta , \tau_\mu h' ; t)
+  \mathcal{N}_{h, 2;\mu}(\pi^\diamond_{; \mu} \theta; t).
\end{equation}

We now focus on expressions involving $\j(q, h;t)$.
First of all, a short computation shows that
\begin{equation}
\begin{array}{lcl}
\pi^{\times}_{nl ; \nu} \, \j(q, h; t)  & = &
q_{l + \sigma_\nu}(t) h\big(\xi_{n + \tau_\nu, l + \sigma_\nu}(t) \big)
\\[0.2cm]
& = &
 \pi^\diamond_{l ; \nu} q(t) \pi^{\times}_{nl ; \nu} \, \j( h; t)
  + q_l(t) \pi^{\times}_{nl ; \nu} \, \j(h ; t).
\end{array}
\end{equation}
Using the expression \sref{sec:oblq:prlm:pijh} above,
we expand this as
\begin{equation}
\begin{array}{lcl}
\pi^{\times}_{nl ; \nu} \, \j(q, h; t)  & = &
\pi^\diamond_{l; \nu} q(t) \j_{nl}( \tau_\nu h ; t)
+ \pi^\diamond_{l; \nu} q(t) \mathcal{M}_{h, 1; \nu}
       \big( \xi_{nl}(t) ,  \pi^\diamond_{l;\nu} \theta(t) \big)
\\[0.2cm]
& & \qquad
+ q_l(t) \j_{nl}(\tau_\nu h; t)
- q_l(t) \pi^\diamond_{l; \nu} \theta(t) \j_{nl}(\tau_\nu h' ; t)
\\[0.2cm]
& & \qquad \qquad
+ q_l(t) \mathcal{M}_{h, 2;\nu}\big(\xi_{nl}(t), \pi^\diamond_{l; \nu} \theta(t) \big).
\end{array}
\end{equation}
In global form, this is
\begin{equation}
\begin{array}{lcl}
\pi^{\times}_{; \nu} \, \j(q, h; t)  & = &
 \j( \pi^\diamond_{; \nu} q,  \tau_\nu h ; t)
+ \pi^\diamond_{; \nu} q(t) \mathcal{N}_{h, 1; \nu}
     \big( \pi^\diamond_{;\nu} \theta ; t \big)
\\[0.2cm]
& & \qquad
+ \j(q , \tau_\nu h; t)
- \j(q, \pi^\diamond_{;\nu} \theta,  \tau_\nu h' ; t)
\\[0.2cm]
& & \qquad \qquad
+ q(t) \mathcal{N}_{h, 2;\nu}\big(\pi^\diamond_{; \nu} \theta ;t \big).
\end{array}
\end{equation}
At times, it suffices to use the cruder version
\begin{equation}
\label{eq:oblq:prlm:jqhcrude}
\begin{array}{lcl}
\pi^{\times}_{; \nu} \, \j(q, h; t)  & = &
 \j( \pi^\diamond_{; \nu} q,  \tau_\nu h ; t)
+ \pi^\diamond_{; \nu} q(t) \mathcal{N}_{h, 1; \nu}
     \big( \pi^\diamond_{;\nu} \theta ; t \big)
\\[0.2cm]
& & \qquad
+ \j(q , \tau_\nu h; t)
+ q(t) \mathcal{N}_{h, 1;\nu}\big(\pi^\diamond_{; \nu} \theta ;t \big).
\end{array}
\end{equation}
In particular, exploiting the crude identity \sref{eq:oblq:prlm:jqhcrude},
we obtain
\begin{equation}
\begin{array}{lcl}
[\pi^\times_{;\nu} - \pi^\times_{; 5} ] \, \j(q,h; t)
& = &
\j(\pi^\diamond_{; \nu} q, \tau_\nu h; t) + \pi^\diamond_{; \nu} q(t) \mathcal{N}_{h,1;\nu}
     \big( \pi^\diamond_{;\nu} \theta ; t \big)
\\[0.2cm]
& & \qquad
+  \j( q, \tau_\nu h; t) - \j(q ,  h; t)
   + q(t) \mathcal{N}_{h,1;\nu} \big(  \pi^\diamond_{;\nu} \theta ; t \big).
\\[0.2cm]
\end{array}
\end{equation}

Moving on, we compute
\begin{equation}
\begin{array}{lcl}
\j_{nl}(L ; t) \pi^\times_{nl} \, \j(q, h; t)
& = &
L^\times_\mu \pi^{\diamond}_{l; \mu} q(t) \,
\j_{nl}( \tau_\mu h; t)
 + L^\times_\mu \pi^{\diamond}_{l; \mu} q(t) \mathcal{M}_{h, 1 ; \mu}\big(\xi_{nl}(t), \pi^{\diamond}_{l;\mu} \theta(t) \big)
\\[0.2cm]
& & \qquad
+ q_l(t) \j_{nl}(L; t) \j_{nl}( \tau h ; t)
 - q_l(t) L^\times_\mu \pi^{\diamond}_{l; \mu} \theta(t) \j_{nl}( \tau_\mu h'; t)
\\[0.2cm]
& & \qquad
 + q_l(t) L^\times_\mu \mathcal{M}_{h, 2 ; \mu}\big(\xi_{nl}(t),  \pi^{\diamond}_{l;\mu} \theta(t)  \big),
\end{array}
\end{equation}
which as before can be simplified to
\begin{equation}
\begin{array}{lcl}
\j_{nl}(L; t) \pi^{\times}_{nl} \, \j(q, h; t)
& = &
\pi^{\diamond}_{l; \mu} q(t) \,
\j_{nl}( \tau_\mu h; t)
 + \pi^{\diamond}_{l; \mu} q(t) \mathcal{M}_{h, 1 ; \mu}\big(\xi_{nl}(t), \pi^{\diamond}_{l;\mu} \theta(t) \big)
\\[0.2cm]
& & \qquad
+ q_l(t) \j_{nl}(L ; t) \j_{nl}( \tau h ; t)
- q_l(t) \pi^\diamond_{l; \mu} \theta(t) \j_{nl}(\tau_\mu h' ;t )
\\[0.2cm]
& & \qquad
 + q_l(t) \mathcal{M}_{h, 2 ; \mu}\big(\xi_{nl}(t),  \pi^{\diamond}_{l;\mu} \theta(t)  \big).
\end{array}
\end{equation}
In global form, we hence have
\begin{equation}
\begin{array}{lcl}
\j(L; t) \pi^\times \j(q,h;t) & = &
\j( \pi^\diamond_{;\mu} q , \tau_\mu h ; t)
+ \pi^{\diamond}_{;\mu} q(t) \mathcal{N}_{h, 1; \mu}( \pi^\diamond_{;\mu} \theta ; t)
\\[0.2cm]
& & \qquad
+ \j(q , L \tau h ; t)
- \j\big(q, \pi^\diamond_{\mu} \theta, \tau_\mu h' ; t \big)
+ q(t) \mathcal{N}_{h, 2;\mu}\big( \pi^\diamond_{; \mu} \theta ; t \big),
\end{array}
\end{equation}
which if desired can be simplified to
\begin{equation}
\begin{array}{lcl}
\j(L; t) \pi^\times \j(q,h;t) & = &
\j( \pi^\diamond_{;\mu} q , \tau_\mu h ; t)
+ \pi^{\diamond}_{;\mu} q(t) \mathcal{N}_{h, 1; \mu}( \pi^\diamond_{; \mu} \theta ; t)
\\[0.2cm]
& & \qquad
+ \j(q , L  \tau h ; t) + q(t) \mathcal{N}_{h, 1;\mu}\big( \pi^\diamond_{; \mu} \theta ; t \big).
\end{array}
\end{equation}

Finally, we discuss the terms involving $\j(p,q,h;t)$.
A short computation shows
\begin{equation}
\begin{array}{lcl}
\pi^\times_{nl; \nu} \, \j(p, q, h; t) & = &p_{l + \sigma_\nu}(t) q_{l + \sigma_\nu}(t) h\big(\xi_{n + \tau_\nu, l + \sigma_\nu}(t) \big)
\\[0.2cm]
& = & \pi^\diamond_{l; \nu} p(t) \pi^\times_{nl; \nu} \, \j( q , h ;t )
   + p_l(t) \pi^\times_{nl; \nu} \, \j( q , h; t),
\\[0.2cm]
\end{array}
\end{equation}
which using \sref{eq:oblq:prlm:jqhcrude} expands as
\begin{equation}
\begin{array}{lcl}
\pi^\times_{nl; \nu} \, \j(p, q, h; t)
& = &
\pi^\diamond_{l; \nu} p(t) \j_{nl}( \pi^\diamond_{; \nu} q, \tau_\nu h; t)
+ \pi^\diamond_{l; \nu} p(t) \pi^\diamond_{l; \nu} q(t) \mathcal{M}_{h, 1; \nu}\big( \xi_{nl}(t),  \pi^\diamond_{l;\nu} \theta(t) \big)
\\[0.2cm]
& & \qquad
+ \pi^\diamond_{l; \nu} p(t) \j_{nl}( q, \tau_{\nu} h ; t) + \pi^\diamond_{l; \nu}p(t) q_l(t) \mathcal{M}_{h,1;\nu}\big( \xi_{nl}(t), \pi^\diamond_{l; \nu} \theta(t) \big)
\\[0.2cm]
& &
 + p_l(t) \j_{nl}( \pi^\diamond_{; \nu} q, \tau_\nu h; t)
+ p_l(t) \pi^\diamond_{l; \nu} q(t) \mathcal{M}_{h, 1; \nu}\big( \xi_{nl}(t), \pi^\diamond_{l;\nu} \theta(t) \big)
\\[0.2cm]
& & \qquad
 + p_l(t) \j( q, \tau_{\nu} h ; t)
 + p_l(t) q_l(t) \mathcal{M}_{h,1;\nu}\big(\xi_{nl}(t), \pi^\diamond_{l; \nu} \theta(t) \big).
\\[0.2cm]
\end{array}
\end{equation}

Inspection of this expression readily yields
\begin{equation}
\begin{array}{lcl}
[\pi^\times_{; \nu} - \pi^\times_{; 5} ] \, \j(p, q, h; t)
& = &
 \j(\pi^\diamond_{;\nu} p,  \pi^\diamond_{; \nu} q, \tau_\nu h; t)
+ \pi^\diamond_{;\nu} p(t) \pi^\diamond_{; \nu} q(t)
  \mathcal{N}_{h, 1; \nu}\big( \pi^\diamond_{;\nu} \theta ; t \big)
\\[0.2cm]
& & \qquad
+  \j( \pi^\diamond_{; \nu} p(t), q, \tau_{\nu} h ; t)
 + \pi^\diamond_{; \nu}p(t) q(t)
   \mathcal{N}_{h,1;\nu}\big(\pi^\diamond_{; \nu} \theta ; t \big)
\\[0.2cm]
& &
 + \j(p, \pi^\diamond_{; \nu} q, \tau_\nu h; t)
+ p(t) \pi^\diamond_{; \nu} q(t)
   \mathcal{N}_{h, 1; \nu}\big( \pi^\diamond_{;\nu} \theta ; t \big)
\\[0.2cm]
& & \qquad
 + \j(p, q, \tau_\nu h ; t) - \j(p, q, h; t)
 + p(t) q(t) \mathcal{N}_{h,1;\nu}\big(\pi^\diamond_{; \nu} \theta ; t \big).
\\[0.2cm]
\end{array}
\end{equation}
In addition, we can compute
\begin{equation}
\begin{array}{lcl}
\j_{nl}(L ; t) \pi^\times_{nl} \, \j(p, q, h; t)
& = &
\pi^\diamond_{l; \mu} p(t) \j_{nl}( \pi^\diamond_{; \mu} q, \tau_\mu h; t)
+ \pi^\diamond_{l; \mu} p(t) \pi^\diamond_{l; \mu} q(t)
   \mathcal{M}_{h, 1; \mu}\big(\xi_{nl}(t),  \pi^\diamond_{l;\mu} \theta(t)  \big)
\\[0.2cm]
& & \qquad
+ \pi^\diamond_{l; \mu} p(t) \j_{nl}( q, \tau_{\mu} h ; t) + \pi^\diamond_{l; \mu}p(t) q_l(t)
   \mathcal{M}_{h,1;\mu}\big(\xi_{nl}(t),  \pi^\diamond_{l; \mu} \theta(t) \big)
\\[0.2cm]
& &
 + p_l(t) \j_{nl}( \pi^\diamond_{; \mu} q, \tau_\mu h; t)
+ p_l(t) \pi^\diamond_{l; \mu} q(t)
  \mathcal{M}_{h, 1; \mu}\big(\xi_{nl}(t),  \pi^\diamond_{l;\mu} \theta(t)  \big)
\\[0.2cm]
& & \qquad
 + p_l(t) q_l(t) \j_{nl}( L \tau h ; t)
 + p_l(t) q_l(t)
   \mathcal{M}_{h,1;\mu}\big(\xi_{nl}(t), \pi^\diamond_{l; \mu} \theta(t)  \big),
\\[0.2cm]
\end{array}
\end{equation}
which can be rewritten in global form as
\begin{equation}
\begin{array}{lcl}
\j(L ; t) \pi^\times \j(p, q, h; t)
& = &
 \j(\pi^\diamond_{;\mu} p(t),  \pi^\diamond_{; \mu} q, \tau_\mu h; t)
+ \pi^\diamond_{;\mu} p(t) \pi^\diamond_{; \mu} q(t)
   \mathcal{N}_{h, 1; \mu}\big( \pi^\diamond_{;\mu} \theta ; t \big)
\\[0.2cm]
& & \qquad
+  \j( \pi^\diamond_{; \mu} p(t), q, \tau_{\mu} h ; t)
 + \pi^\diamond_{; \mu}p(t) q(t)
   \mathcal{N}_{h,1;\mu}\big(\pi^\diamond_{; \mu} \theta ; t \big)
\\[0.2cm]
& &
 + \j(p, \pi^\diamond_{; \mu} q, \tau_\mu h; t)
+ p(t) \pi^\diamond_{; \mu} q(t)
   \mathcal{N}_{h, 1; \mu}\big( \pi^\diamond_{;\mu} \theta ; t \big)
\\[0.2cm]
& & \qquad
 + \j(p, q, L \tau h ; t)
 + p(t) q(t) \mathcal{N}_{h,1;\mu}\big(\pi^\diamond_{; \mu} \theta ; t \big).
\\[0.2cm]
\end{array}
\end{equation}

\subsection{The Ansatz}
In this subsection we introduce the basic form of the sub-solution
that we will analyze and perform some preliminary computations
pertaining to the differential inequality that sub-solutions must satisfy.
In particular, throughout this subsection we fix three external functions
\begin{equation}
\theta \in C^1\big([0, \infty), \ell^{\infty}(\Wholes; \Real) \big),
\qquad z \in C^1\big([0, \infty), \Real\big),
\qquad Z \in C^1\big( [0, \infty), \Real \big)
\end{equation}
and consider fifty-five auxilliary functions
\begin{equation}
\label{eq:oblq:ansatz:diamondfncs}
p^\diamond_{\nu} \in BC^2(\Real, \Real),
\qquad p^{\diamond\diamond}_{\nu \nu'} \in BC^2(\Real, \Real),
\qquad q^{\diamond\diamond}_{\nu \nu'} \in BC^2(\Real, \Real)
\end{equation}
that will be determined later on in this subsection.

Our Ansatz can be written as
\begin{equation}
\begin{array}{lcl}
u^-_{nl}(t) & = & \Phi\big(n + ct - \theta_l(t) - Z(t) \big)
 + \pi^\diamond_{l;\mu} \theta(t)  p^{\diamond}_\mu\big(n + ct - \theta_l(t) - Z(t)\big)
\\[0.2cm]
& & \qquad
 + \pi^{\diamond \diamond}_{l; \mu \mu'} \theta(t) p^{\diamond \diamond}_{\mu \mu'} \big(n + ct - \theta_l(t) - Z(t) \big)
\\[0.2cm]
& & \qquad
+ \pi^{\diamond}_{l; \mu} \theta(t) \pi^{\diamond}_{l; \mu'} \theta(t) q^{\diamond \diamond}_{\mu \mu'} \big(n + ct - \theta_l(t) - Z(t) \big)
 - z(t).
\end{array}
\end{equation}
Upon writing
\begin{equation}
\xi_{nl}(t) = n + ct - \theta_l(t) - Z(t),
\end{equation}
our Ansatz can be rephrased in the global form
\begin{equation}
\label{eq:oblq:subSolGlobForm}
\begin{array}{lcl}
u^-(t) & = &
\j(\Phi ; t) + \j( \pi^\diamond_{;\mu} \theta, p^{\diamond}_\mu ; t)
+ \j( \pi^{\diamond \diamond}_{;\mu \mu'} \theta, p^{\diamond \diamond}_{\mu \mu'} ; t)
+ \j( \pi^\diamond_{; \mu} \theta, \pi^\diamond_{; \mu'} \theta , q^{\diamond\diamond}_{\mu \mu'} ; t )
- z(t).
\end{array}
\end{equation}
We note that $\j(\Phi ; t) - z(t)$
can be seen as the direct lifting of the PDE sub-solution used in \cite{BHM}
to the discrete setting.
The terms $\j( \pi^\diamond_{;\mu} \theta, p^{\diamond}_\mu ; t)$
correspond to those that were explicitly discussed in \S\ref{sec:int},
which allowed a factor $\Phi'(\xi)$ to be pulled off from
all first differences in $\theta$ appearing in the residual.
The remaining terms in \sref{eq:oblq:subSolGlobForm}
are designed to allow a similar factorization for all second order
differences and certain problematic products of first order differences.

As a first preparation,
we introduce the nonlinear expression
\begin{equation}
\label{eq:oblq:anstz:defrN}
\begin{array}{lcl}
\mathcal{R}_{\mathcal{N};\nu}( \pi^\diamond \theta, \pi^{\diamond\diamond} \theta ;t )
& = &[ \pi^\times_{;\nu} - \pi^\times_{; 5} ] u^-(t)
\\[0.2cm]
& = & \mathcal{R}_{\mathcal{N}; \Phi; \nu} (\pi^\diamond \theta ; t)
+ \mathcal{R}_{\mathcal{N} ; p^\diamond ; \nu}( \pi^\diamond \theta, \pi^{\diamond \diamond} \theta  ; t)
\\[0.2cm]
& & \qquad
+ \mathcal{R}_{\mathcal{N} ; p^{\diamond\diamond} ; \nu}( \pi^\diamond \theta, \pi^{\diamond \diamond} \theta  ; t)
+ \mathcal{R}_{\mathcal{N} ; q^{\diamond\diamond} ; \nu}( \pi^\diamond \theta, \pi^{\diamond \diamond} \theta  ; t),
\end{array}
\end{equation}
in which we have defined
\begin{equation}
\begin{array}{lcl}
\mathcal{R}_{\mathcal{N}; \Phi; \nu} (\pi^\diamond \theta ; t)
& = &
[ \pi^\times_{;\nu} - \pi^\times_{; 5} ]
  \j(\Phi; t),
\\[0.2cm]
\mathcal{R}_{\mathcal{N} ; p^\diamond ; \nu}( \pi^\diamond \theta, \pi^{\diamond \diamond} \theta  ; t)
& = &
[ \pi^\times_{;\nu} - \pi^\times_{; 5} ]
  \j(\pi^\diamond_{; \mu'} \theta, p^\diamond_{\mu'} ; t ),
\\[0.2cm]
\mathcal{R}_{\mathcal{N} ; p^{\diamond\diamond} ; \nu}( \pi^\diamond \theta, \pi^{\diamond \diamond} \theta  ; t)
& = &
[ \pi^\times_{;\nu} - \pi^\times_{; 5} ]
  \j(\pi^{\diamond\diamond}_{; \mu'\mu''} \theta, p^{\diamond\diamond}_{\mu'\mu''} ; t ),
\\[0.2cm]
\mathcal{R}_{\mathcal{N} ; q^{\diamond\diamond} ; \nu}( \pi^\diamond \theta, \pi^{\diamond \diamond} \theta  ; t)
& = &
[ \pi^\times_{;\nu} - \pi^\times_{; 5} ]
  \j(\pi^\diamond_{; \mu'} \theta , \pi^\diamond_{;\mu''} \theta, q^\diamond_{\mu'\mu''} ; t ).
\\[0.2cm]
\end{array}
\end{equation}
Using the expressions obtained in \S\ref{sec:oblq:prlm},
we now compute
\begin{equation}
\begin{array}{lcl}
\mathcal{R}_{\mathcal{N}; \Phi; \nu} (\pi^\diamond \theta ; t)
& = &
\j(\tau_\nu \Phi; t) - \j(\Phi ; t)
+ \mathcal{N}_{\Phi,1;\nu}\big( \pi^\diamond_{; \nu} \theta(t) ; t \big)
\\[0.2cm]
\mathcal{R}_{\mathcal{N} ; p^\diamond ; \nu}( \pi^\diamond \theta, \pi^{\diamond \diamond} \theta  ; t)
& = &
 \j\big( \pi^{\diamond\diamond}_{; \mu' \nu} \theta, p^\diamond_{\mu'} ; t)
+\pi^{\diamond\diamond}_{\mu' \nu} \theta(t) \mathcal{N}_{p^\diamond_{\mu'},1;\nu}\big( \pi^\diamond_{; \nu} \theta(t) ; t \big)
\\[0.2cm]
& & \qquad
+  \j(\pi^\diamond_{\mu'}\theta, \tau_\nu p^\diamond_{\mu'} ; t) - \j(\pi^\diamond_{\mu'}\theta, p^\diamond_{\mu'} ; t)
\\[0.2cm]
& & \qquad
+ \pi^\diamond_{\mu'}\theta(t) \mathcal{N}_{p^\diamond_{\mu'}, 1 ; \nu}\big( \pi^\diamond_{; \nu} \theta(t) \big),
\\[0.2cm]
\end{array}
\end{equation}
together with
\begin{equation}
\begin{array}{lcl}
\mathcal{R}_{\mathcal{N} ; p^{\diamond\diamond} ; \nu}( \pi^\diamond \theta, \pi^{\diamond \diamond} \theta  ; t)
& = &
 \j\big( \pi^{\diamond\diamond\diamond}_{; \mu' \mu'' \nu} \theta, p^{\diamond\diamond}_{\mu' \mu''} ; t)
+\pi^{\diamond\diamond\diamond}_{\mu' \mu'' \nu} \theta(t) \mathcal{N}_{p^{\diamond\diamond}_{\mu'\mu''},1;\nu}\big( \pi^\diamond_{; \nu} \theta(t) ; t \big)
\\[0.2cm]
& & \qquad
+  \j( \pi^{\diamond\diamond}_{\mu'\mu''}\theta, \tau_\nu p^{\diamond\diamond}_{\mu'\mu''} ; t)
            - \j( \pi^{\diamond\diamond}_{\mu'\mu''}\theta, p^{\diamond\diamond}_{\mu'\mu''} ; t)
\\[0.2cm]
& & \qquad
+ \pi^{\diamond\diamond}_{\mu'\mu''}\theta(t) \mathcal{N}_{p^{\diamond\diamond}_{\mu'\mu''}, 1 ; \nu}\big( \pi^\diamond_{; \nu} \theta(t) \big),
\\[0.2cm]
\mathcal{R}_{\mathcal{N} ; q^{\diamond\diamond} ; \nu}( \pi^\diamond \theta, \pi^{\diamond \diamond} \theta  ; t)
& = &
 \j(\pi^{\diamond\diamond}_{\mu' \nu} \theta, \pi^{\diamond\diamond}_{\mu'' \nu} \theta, \tau_{\nu} q^{\diamond\diamond}_{\mu' \mu''}; t)
+ \pi^{\diamond\diamond}_{\mu' \nu} \theta(t) \pi^{\diamond\diamond}_{\mu'' \nu} \theta(t)
  \mathcal{N}_{q^{\diamond\diamond}_{\mu' \mu''}, 1; \nu}\big( \pi^\diamond_{;\nu} \theta ; t \big)
\\[0.2cm]
& & \qquad
+  \j(  \pi^{\diamond\diamond}_{\mu' \nu} \theta,  \pi^{\diamond}_{\mu''} \theta, \tau_{\nu} q^{\diamond\diamond}_{\mu' \mu''} ; t)
 + \pi^{\diamond\diamond}_{\mu' \nu} \theta(t)  \pi^{\diamond}_{\mu''} \theta(t)
    \mathcal{N}_{q^{\diamond\diamond}_{\mu' \mu''},1;\nu}\big(\pi^\diamond_{; \nu} \theta ; t \big)
\\[0.2cm]
& & \qquad
  + \j(\pi^{\diamond}_{\mu'} \theta,  \pi^{\diamond\diamond}_{\mu'' \nu} \theta, \tau_{\nu} q^{\diamond\diamond}_{\mu' \mu''}; t)
+ \pi^{\diamond}_{\mu'} \theta(t)   \pi^{\diamond\diamond}_{\mu'' \nu} \theta(t) \mathcal{N}_{q^{\diamond\diamond}_{\mu' \mu''}, 1; \nu}\big( \pi^\diamond_{;\nu} \theta ; t \big)
\\[0.2cm]
& & \qquad
 + \j(\pi^{\diamond}_{\mu'} \theta,  \pi^{\diamond}_{\mu''} \theta, \tau_\nu q^{\diamond\diamond}_{\mu' \mu''} ; t)
 - \j(\pi^{\diamond}_{\mu'} \theta,  \pi^{\diamond}_{\mu''} \theta, q^{\diamond\diamond}_{\mu' \mu''} ; t)
\\[0.2cm]
& & \qquad
 + \pi^{\diamond}_{\mu'} \theta(t)  \pi^{\diamond}_{\mu''} \theta(t) \mathcal{N}_{q^{\diamond\diamond}_{\mu' \mu''},1;\nu}\big(\pi^\diamond_{; \nu} \theta ; t \big).
\\[0.2cm]
\end{array}
\end{equation}

We now turn to the main task in this subsection,
which is to consider the quantity
\begin{equation}
\begin{array}{lcl}
\mathcal{J}^-_{nl}(t) & = & \dot{u}^-_{nl}(t) - [\Delta^\times u^-(t)]_{nl} - g\big(u^-_{nl}(t)\big)
\end{array}
\end{equation}
and determine suitable choices for the functions \sref{eq:oblq:ansatz:diamondfncs}.
To this end, we decompose $\mathcal{J}^-_{nl}(t)$
as
\begin{equation}
\begin{array}{lcl}
\mathcal{J}^-_{nl}(t)
& = & \dot{u}^-_{nl}(t) -
L^\times \pi^\times_{nl} u^-(t) - g\big( u^-_{nl}(t) \big)
\\[0.2cm]
& = & \dot{u}^-_{nl}(t) - \j_{nl}(L ; t ) \pi^\times_{nl} u^-(t)
  + g'\Big(\Phi\big(\xi_{nl}(t)\big)\Big) u^-_{nl}(t) - g\big(u^-_{nl}(t)\big)
\\[0.2cm]
& = & \dot{u}^-_{nl}(t) - \j_{nl}(L; t))  \pi^\times_{nl}  u^-(t)
 + g'\Big(\Phi\big(\xi_{nl}(t)\big)\Big) \big[u^-_{nl}(t) - \Phi\big(\xi_{nl}(t)\big) \big]
    - g\big(u^-_{nl}(t)\big)
\\[0.2cm]
& & \qquad
 + g'\Big(\Phi\big(\xi_{nl}(t)\big)\Big) \Phi\big(\xi_{nl}(t)\big).
\end{array}
\end{equation}
For convenience, we rephrase this as
\begin{equation}
\label{eq:oblq:anstz:defnJMinus}
\begin{array}{lcl}
\mathcal{J}^-(t) & = &
\dot{u}^-(t) - \j(L; t) \pi^\times  u^-(t)
 + \j\big(g'(\Phi) ; t\big) \big[u^-(t) - \j(\Phi; t) \big] - g\big(u^-(t)\big)
\\[0.2cm]
& & \qquad
 + \j(L  \tau \Phi; t) - L^\times \j( \tau \Phi ; t)
\end{array}
\end{equation}
and carefully study each of the terms.

First of all, a short computation shows that
\begin{equation}
\begin{array}{lcl}
\dot{u}^-(t) & = &
c \j( \Phi'; t) - \j( \dot{\theta} + \dot{Z} , \Phi'; t) - \dot{z}(t)
\\[0.2cm]
& & \qquad + \j( \pi^{\diamond}_{; \mu} \dot{\theta} , p^{\diamond}_{\mu} ; t)
 + c \j( \pi^{\diamond}_{; \mu} \theta , D p^{\diamond}_{\mu} ; t)
  - \j( \dot{\theta} + \dot{Z}, \pi^{\diamond}_{; \mu} \theta, D p^{\diamond}_{\mu} ; t )
\\[0.2cm]
& & \qquad + \j(\pi^{\diamond \diamond}_{; \mu \mu'} \dot{\theta} , p^{\diamond \diamond}_{\mu \mu'} ; t)
 + c \j( \pi^{\diamond \diamond}_{; \mu \mu'} \theta , D p^{\diamond\diamond}_{\mu\mu'} ; t)
  - \j( \dot{\theta} + \dot{Z}, \pi^{\diamond\diamond}_{; \mu\mu'} \theta, D p^{\diamond\diamond}_{\mu\mu'} ; t )
\\[0.2cm]
& & \qquad + \j(\pi^{\diamond}_{; \mu } \dot{\theta} , \pi^\diamond_{; \mu'} \theta, q^{\diamond \diamond}_{\mu \mu'} ; t)
  + \j(\pi^{\diamond}_{; \mu } \theta , \pi^\diamond_{; \mu'} \dot{ \theta }, q^{\diamond \diamond}_{\mu \mu'} ; t)
\\[0.2cm]
& & \qquad \qquad
 + c \j( \pi^{\diamond}_{; \mu } \theta , \pi^\diamond_{; \mu'}  \theta , D q^{\diamond\diamond}_{\mu\mu'} ; t)
  - \j( \dot{\theta} + \dot{Z},  \pi^{\diamond}_{; \mu } \theta , \pi^\diamond_{; \mu'}  \theta, D q^{\diamond\diamond}_{\mu\mu'} ; t ).
\\[0.2cm]
%
\end{array}
\end{equation}
For later use, we introduce the fifty-five functions
\begin{equation}
f^\diamond_{p; \nu} \in BC^1(\Real, \Real),
\qquad f^{\diamond\diamond}_{p; \nu \nu'} \in BC^1(\Real, \Real),
\qquad f^{\diamond \diamond}_{q; \nu \nu'} \in BC^1(\Real, \Real)
\end{equation}
that are defined by
\begin{equation}
\label{eq:oblq:anstz:defFDiams}
f^{\diamond}_{p; \nu} = \mathcal{L}_0 p^{\diamond},
\qquad f^{\diamond \diamond}_{p; \nu \nu'} =  \mathcal{L}_0 p^{\diamond\diamond}_{\nu \nu'},
\qquad f^{\diamond \diamond}_{q; \nu \nu'} =  \mathcal{L}_0 q^{\diamond\diamond}_{\nu \nu'},
\end{equation}
in which we have recalled the linear operator
\begin{equation}
\begin{array}{lcl}
[\mathcal{L}_0 v](\xi) & = & - c v'(\xi) + v(\xi + \sigma_h) + v(\xi + \sigma_v)
+ v(\xi - \sigma_h) + v(\xi - \sigma_v)  - 4 v(\xi)
\\[0.2cm]
& & \qquad + g'\big(\Phi(\xi)\big) v(\xi).
\end{array}
\end{equation}
Notice in particular that if $\mathcal{L}_0 v = f$, then
\begin{equation}
c \j_{nl}( v' ; t) =
\j_{nl}(L ; t) \j_{nl}( \tau v ; t) - f\big(\xi_{nl}(t)\big),
\end{equation}
which allows us to write
\begin{equation}
\begin{array}{lcl}
c \j( v' ; t)  & = &
\j(L \tau v ; t) - \j(f ; t).
\end{array}
\end{equation}
In addition, the wave profile equation implies that for any $t \ge 0$
we have
\begin{equation}
c\Phi'(\xi_{nl}(t)) = L^\times_\mu [\tau_\mu \Phi]\big(\xi_{nl}(t)\big)
  + g\Big( \Phi\big(\xi_{nl}(t)\big) \Big),
\end{equation}
which implies that
\begin{equation}
\begin{array}{lcl}
c \j(\Phi' ; t) & = & L^\times_\mu \j(\tau_\mu \Phi ; t) + \j\big(g(\Phi) ; t\big)
\\[0.2cm]
& = & L^\times \j(\tau \Phi ; t) + \j\big(g(\Phi) ; t\big).
\end{array}
\end{equation}
With this hand, we can expand $\dot{u}^-(t)$ as
\begin{equation}
\label{eq:oblq:prlm:expForDotU}
\begin{array}{lcl}
\dot{u}^-(t)  & = &
 L^\times \j(\tau \Phi ; t) + \j(g(\Phi) ; t)
 - \j( \dot{\theta} + \dot{Z} , \Phi'; t)
 - \dot{z}(t)
\\[0.2cm]
& & \qquad
+ \j( \pi^{\diamond}_{; \mu} \dot{\theta} , p^{\diamond}_{\mu} ; t)
 - \j( \dot{\theta} + \dot{Z}, \pi^{\diamond}_{; \mu} \theta, D p^{\diamond}_{\mu} ; t )
\\[0.2cm]
& & \qquad \qquad
 +  \j( \pi^{\diamond}_{; \mu} \theta, L \tau p^{\diamond}_{\mu} ; t)
 -\j(\pi^{\diamond}_{; \mu} \theta, f^\diamond_{p;\mu} ; t)
\\[0.2cm]
& & \qquad
  + \j(\pi^{\diamond \diamond}_{; \mu \mu'} \dot{\theta} , p^{\diamond \diamond}_{\mu \mu'} ; t)
  - \j( \dot{\theta} + \dot{Z}, \pi^{\diamond\diamond}_{; \mu\mu'} \theta, D p^{\diamond\diamond}_{\mu\mu'} ; t )
\\[0.2cm]
& & \qquad \qquad
 + \j( \pi^{\diamond \diamond}_{; \mu \mu'} \theta, L  \tau p^{\diamond \diamond}_{\mu \mu'} ; t )
  - \j( \pi^{\diamond \diamond}_{; \mu \mu'} \theta , f^{\diamond \diamond}_{p;\mu \mu'} ; t)
\\[0.2cm]
& & \qquad + \j(\pi^{\diamond}_{; \mu } \dot{\theta} , \pi^\diamond_{; \mu'} \theta, q^{\diamond \diamond}_{\mu \mu'} ; t)
  + \j(\pi^{\diamond}_{; \mu } \theta , \pi^\diamond_{; \mu'} \dot{ \theta }, q^{\diamond \diamond}_{\mu \mu'} ; t)
\\[0.2cm]
& & \qquad \qquad
- \j( \dot{\theta} + \dot{Z},  \pi^{\diamond}_{; \mu } \theta , \pi^\diamond_{; \mu'}  \theta, D q^{\diamond\diamond}_{\mu\mu'} ; t )
\\[0.2cm]
& & \qquad \qquad
 +  \j( \pi^{\diamond}_{; \mu } \theta , \pi^\diamond_{; \mu'}  \theta , L  \tau q^{\diamond\diamond}_{\mu\mu'} ; t)
 -  \j( \pi^{\diamond}_{; \mu } \theta , \pi^\diamond_{; \mu'}  \theta , f^{\diamond\diamond}_{q;\mu\mu'} ; t).
\\[0.2cm]
%
\end{array}
\end{equation}

Moving on to the second term in \sref{eq:oblq:anstz:defnJMinus},
we use the computations
in \S\ref{sec:oblq:prlm} to compute
\begin{equation}
\label{eq:oblq:prlm:expForLPi}
\begin{array}{lcl}
\j(L ;t) \pi^\times u^-(t)& = &
  \j(L, \tau \Phi ; t)
 - \j( \pi^{\diamond}_{\mu} \theta, \tau_\mu \Phi' ; t )
 + \mathcal{N}_{\Phi, 2;\mu}(\pi^{\diamond}_{; \mu} \theta ; t)
 - \j(L ; t) z(t)
\\[0.2cm]
& &
\qquad
+ \j( \pi^{\diamond\diamond}_{;\mu \mu'} \theta, \tau_{\mu'} p^{\diamond}_{\mu} ; t)
+ \pi^{\diamond\diamond}_{; \mu \mu'}\theta(t) \mathcal{N}_{p^{\diamond}_{\mu} , 1; \mu'}(\pi^{\diamond}_{; \mu'} \theta ; t)
\\[0.2cm]
& &
\qquad \qquad
+ \j(\pi^\diamond_{; \mu} \theta , L \tau p^{\diamond}_{\mu} ; t)
- \j(\pi^\diamond_{; \mu} \theta,  \pi^\diamond_{; \mu'} \theta, \tau_{\mu'} D p^{\diamond}_{\mu} ; t )
\\[0.2cm]
& & \qquad \qquad
+ \pi^{\diamond}_{; \mu} \theta(t) \mathcal{N}_{p^{\diamond}_\mu , 2; \mu'}\big(\pi^\diamond_{;\mu'} \theta ; t \big)
\\[0.2cm]
& &
\qquad
+ \j(\pi^{\diamond\diamond\diamond}_{;\mu \mu' \mu''} \theta, \tau_{\mu''} p^{\diamond\diamond}_{\mu\mu'} ; t)
+ \pi^{\diamond\diamond\diamond}_{; \mu\mu'\mu''}\theta(t) \mathcal{N}_{p^{\diamond}_{\mu \mu'} , 1; \mu''}(\pi^{\diamond}_{; \mu''} \theta ; t)
\\[0.2cm]
& &
\qquad \qquad
+ \j(\pi^{\diamond\diamond}_{; \mu \mu'} \theta , L \tau p^{\diamond}_{\mu \mu'} ; t)
+ \pi^{\diamond\diamond}_{; \mu \mu'} \theta(t) \mathcal{N}_{p^{\diamond\diamond}_{\mu\mu'} , 1; \mu''}\big(\pi^\diamond_{;\mu''} \theta ; t \big)
\\[0.2cm]
& & \qquad
 + \j(\pi^{\diamond\diamond}_{\mu \mu''} \theta, \pi^{\diamond\diamond}_{\mu' \mu''} \theta, \tau_{\mu''} q^{\diamond\diamond}_{\mu \mu'}; t)
+ \pi^{\diamond\diamond}_{\mu \mu''} \theta(t) \pi^{\diamond\diamond}_{\mu' \mu''} \theta(t) \mathcal{N}_{q^{\diamond\diamond}_{\mu \mu'}, 1; \mu''}\big( \pi^\diamond_{;\mu''} \theta ; t \big)
\\[0.2cm]
& & \qquad \qquad
+  \j(  \pi^{\diamond\diamond}_{\mu \mu''} \theta,  \pi^{\diamond}_{\mu'} \theta, \tau_{\mu''} q^{\diamond\diamond}_{\mu \mu'} ; t)
 + \pi^{\diamond\diamond}_{\mu \mu''} \theta(t)  \pi^{\diamond}_{\mu'} \theta(t) \mathcal{N}_{q^{\diamond\diamond}_{\mu \mu'},1;\mu''}\big(\pi^\diamond_{; \mu''} \theta ; t \big)
\\[0.2cm]
& & \qquad \qquad
  + \j(\pi^{\diamond}_{\mu} \theta,  \pi^{\diamond\diamond}_{\mu' \mu''} \theta, \tau_{\mu''} q^{\diamond\diamond}_{\mu \mu'}; t)
+ \pi^{\diamond}_{\mu} \theta(t)   \pi^{\diamond\diamond}_{\mu' \mu''} \theta(t) \mathcal{N}_{q^{\diamond\diamond}_{\mu \mu'}, 1; \mu''}\big( \pi^\diamond_{;\mu''} \theta ; t \big)
\\[0.2cm]
& & \qquad \qquad
 + \j(\pi^{\diamond}_{\mu} \theta,  \pi^{\diamond}_{\mu'} \theta, L \tau q^{\diamond\diamond}_{\mu \mu'} ; t)
 + \pi^{\diamond}_{\mu} \theta(t)  \pi^{\diamond}_{\mu'} \theta(t) \mathcal{N}_{q^{\diamond\diamond}_{\mu \mu'},1;\mu''}\big(\pi^\diamond_{; \mu''} \theta ; t \big).
\\[0.2cm]
%
%
\end{array}
\end{equation}

Comparing
\sref{eq:oblq:anstz:defnJMinus}, \sref{eq:oblq:prlm:expForDotU} and
\sref{eq:oblq:prlm:expForLPi},
we see that a fair number of terms cancel. In order to organize
the remaining terms, we first introduce
for $v \in \ell^\infty(\Wholes^2; \Real)$ the nonlinear expression
\begin{equation}
\begin{array}{lcl}
\mathcal{G}_0\big(  v ; t)
&= &
 - g\big( \j(\Phi ; t)  + v\big) + g\big( \j(\Phi ; t) \big) + g'\big(\j(\Phi ; t)\big) v
 \qquad \in \ell^{\infty}(\Wholes^2; \Real),
\\[0.2cm]
\end{array}
\end{equation}
which measures the purely nonlinear part of $g$ near
the wave $\Phi\big(\xi_{nl}(t) \big)$.
Utilizing $\mathcal{G}_0$,
we now define the nonlinear expression
\begin{equation}
\label{eq:oblq:anstz:defr1}
\begin{array}{lcl}
\mathcal{R}_1\big( \pi^{\diamond} \theta, \pi^{\diamond\diamond} \theta , z ; t \big)
& = & \mathcal{G}_0 \Big(
  \j( \pi^{\diamond}_{; \mu} \theta , p^{\diamond}_\mu ; t)
  + \j( \pi^{\diamond \diamond}_{\mu \mu'} \theta, p^{\diamond\diamond}_{\mu \mu'} ; t)
  + \j( \pi^{\diamond}_{\mu} \theta, \pi^{\diamond}_{\mu'} \theta, q^{\diamond\diamond}_{\mu \mu'} ; t)
  - z(t)  ; t
  \Big)
\\[0.2cm]
& & \qquad
- \mathcal{N}_{\Phi, 2;\mu}(\pi^{\diamond}_{; \mu} \theta ; t).
\\[0.2cm]
\end{array}
\end{equation}
In addition, we define the nonlinear expressions
\begin{equation}
\label{eq:oblq:anstz:defr24}
\begin{array}{lcl}
\mathcal{R}_2 \big( \dot{Z}, \pi^\diamond \theta, \pi^{\diamond \diamond} \theta ; t \big) & = &
  - \j ( \dot{Z},  \pi^{\diamond}_{; \mu} \theta , D p^{\diamond}_{\mu} ; t)
 -\j ( \dot{Z} ,  \pi^{\diamond\diamond}_{; \mu \mu'} \theta , D p^{\diamond \diamond}_{\mu \mu'} ; t),
\\[0.2cm]
\mathcal{R}_3 \big(\dot{\theta},  \pi^{\diamond \diamond} \theta ; t \big) & = &
 -\j ( \dot{\theta} ,  \pi^{\diamond\diamond}_{; \mu \mu'} \theta , D p^{\diamond \diamond}_{\mu \mu'} ; t),
\\[0.2cm]
\mathcal{R}_4\big(\pi^\diamond \dot{\theta} , \pi^\diamond \theta ; t \big) & = &
 \j(\pi^{\diamond}_{; \mu } \dot{\theta} , \pi^\diamond_{; \mu'} \theta, q^{\diamond \diamond}_{\mu \mu'} ; t)
  + \j(\pi^{\diamond}_{; \mu } \theta , \pi^\diamond_{; \mu'} \dot{ \theta }, q^{\diamond \diamond}_{\mu \mu'} ; t),
\\[0.2cm]
\end{array}
\end{equation}
together with
\begin{equation}
\label{eq:oblq:anstz:defr5}
\begin{array}{lcl}
\mathcal{R}_5 \big( \pi^{\diamond} \theta, \pi^{\diamond \diamond} \theta ,
 \pi^{\diamond \diamond \diamond} \theta ; t \big) & = &
- \pi^{\diamond\diamond}_{; \mu \mu'}\theta(t) \mathcal{N}_{p^{\diamond}_{\mu} , 1; \mu'}(\pi^{\diamond}_{; \mu'} \theta ; t)
- \pi^{\diamond}_{; \mu} \theta(t) \mathcal{N}_{p^{\diamond}_\mu , 2; \mu'}\big(\pi^\diamond_{;\mu'} \theta ; t \big)
\\[0.2cm]
& & \qquad - \pi^{\diamond\diamond\diamond}_{; \mu\mu'\mu''}\theta(t) \mathcal{N}_{p^{\diamond}_{\mu \mu'} , 1; \mu''}(\pi^{\diamond}_{; \mu''} \theta ; t)
- \pi^{\diamond\diamond}_{; \mu \mu'} \theta(t) \mathcal{N}_{p^{\diamond\diamond}_{\mu\mu'} , 1; \mu''}\big(\pi^\diamond_{;\mu''} \theta ; t \big)
\\[0.2cm]
& & \qquad
 - \j(\pi^{\diamond\diamond}_{\mu \mu''} \theta, \pi^{\diamond\diamond}_{\mu' \mu''} \theta, \tau_{\mu''} q^{\diamond\diamond}_{\mu \mu'}; t)
- \pi^{\diamond\diamond}_{\mu \mu''} \theta(t) \pi^{\diamond\diamond}_{\mu' \mu''} \theta(t) \mathcal{N}_{q^{\diamond\diamond}_{\mu \mu'}, 1; \mu''}\big( \pi^\diamond_{;\mu''} \theta ; t \big)
\\[0.2cm]
& & \qquad \qquad
-  \j(  \pi^{\diamond\diamond}_{\mu \mu''} \theta,  \pi^{\diamond}_{\mu'} \theta, \tau_{\mu''} q^{\diamond\diamond}_{\mu \mu'} ; t)
 - \pi^{\diamond\diamond}_{\mu \mu''} \theta(t)  \pi^{\diamond}_{\mu'} \theta(t) \mathcal{N}_{q^{\diamond\diamond}_{\mu \mu'},1;\mu''}\big(\pi^\diamond_{; \mu''} \theta ; t \big)
\\[0.2cm]
& & \qquad \qquad
  - \j(\pi^{\diamond}_{\mu} \theta,  \pi^{\diamond\diamond}_{\mu' \mu''} \theta, \tau_{\mu''} q^{\diamond\diamond}_{\mu \mu'}; t)
- \pi^{\diamond}_{\mu} \theta(t)   \pi^{\diamond\diamond}_{\mu' \mu''} \theta(t) \mathcal{N}_{q^{\diamond\diamond}_{\mu \mu'}, 1; \mu''}\big( \pi^\diamond_{;\mu''} \theta ; t \big)
\\[0.2cm]
& & \qquad \qquad
 - \pi^{\diamond}_{\mu} \theta(t)  \pi^{\diamond}_{\mu'} \theta(t) \mathcal{N}_{q^{\diamond\diamond}_{\mu \mu'},1;\mu''}\big(\pi^\diamond_{; \mu''} \theta ; t \big).
\\[0.2cm]
\end{array}
\end{equation}
Finally, we define the two linear expressions
\begin{equation}
\label{eq:oblq:anstz:defe12}
\begin{array}{lcl}
\mathcal{E}_1 \big( \pi^{\diamond \diamond} \dot{\theta} ; t \big)
& = &  \j ( \pi^{\diamond \diamond}_{; \mu \mu'} \dot{\theta}
 p^{\diamond \diamond}_{\mu \mu'} ; t ),
\\[0.2cm]
\mathcal{E}_2\big(  \pi^{\diamond \diamond \diamond} \theta ; t \big)
& = &
- \j( \pi^{\diamond \diamond \diamond}_{;\mu \mu' \mu'' } \theta ,
         \tau_{\mu''} p^{\diamond \diamond}_{\mu \mu'} ; t).
\end{array}
\end{equation}
Together, these expressions allow us to write
\begin{equation}
\label{eq:oblq:anstz:jminBeforeInhom}
\begin{array}{lcl}
\mathcal{J}^-(t)
& = &
 - \j( \dot{Z} , \Phi'; t)   - \dot{z}(t)  + \j(L ; t) z(t)
\\[0.2cm]
& & \qquad
  - \j( \dot{\theta}  , \Phi'; t)
  + \j( \pi^{\diamond}_{\mu} \theta, \tau_\mu \Phi' ; t )
  -\j(\pi^{\diamond}_{; \mu} \theta, f^\diamond_{p;\mu} ; t)
\\[0.2cm]
& & \qquad \qquad
  + \j( \pi^{\diamond}_{; \mu} \dot{\theta} , p^{\diamond}_{\mu} ; t)
  -  \j( \pi^{\diamond\diamond}_{;\mu \mu'} \theta, \tau_{\mu'} p^{\diamond}_{\mu} ; t)
  - \j( \pi^{\diamond \diamond}_{; \mu \mu'} \theta , f^{\diamond \diamond}_{p;\mu \mu'} ; t)
\\[0.2cm]
& & \qquad
  - \j( \dot{\theta} , \pi^{\diamond}_{; \mu} \theta, D p^{\diamond}_{\mu} ; t )
  + \j(\pi^\diamond_{; \mu} \theta,  \pi^\diamond_{; \mu'} \theta, \tau_{\mu'} D p^{\diamond}_{\mu} ; t )
  -  \j( \pi^{\diamond}_{; \mu } \theta , \pi^\diamond_{; \mu'}  \theta , f^{\diamond\diamond}_{q;\mu\mu'} ; t)
\\[0.2cm]
%
%
%
& & \qquad
+ \mathcal{R}_1 \big( \pi^{\diamond} \theta, \pi^{\diamond\diamond} \theta , z ; t \big)
+ \mathcal{R}_2 \big( \dot{Z}, \pi^\diamond \theta, \pi^{\diamond \diamond} \theta ; t \big)
+ \mathcal{R}_3 \big(\dot{\theta},  \pi^{\diamond \diamond} \theta ; t \big)
\\[0.2cm]
& & \qquad \qquad
+ \mathcal{R}_4 \big(\pi^\diamond \dot{\theta} , \pi^\diamond \theta ; t \big)
+ \mathcal{R}_5 \big( \pi^{\diamond} \theta, \pi^{\diamond \diamond} \theta ,
 \pi^{\diamond \diamond \diamond} \theta ; t \big)
\\[0.2cm]
& & \qquad
+ \mathcal{E}_1\big( \pi^{\diamond \diamond}\dot{\theta} ; t \big)
+ \mathcal{E}_2\big( \pi^{\diamond \diamond \diamond} \theta ; t \big).
\end{array}
\end{equation}
We emphasize at this point that in the sequel all the terms
contained in numbered expressions will be bounded
by the terms that are kept in their explicit form.

Our task is to choose the inhomogeneities \sref{eq:oblq:anstz:defFDiams}
in such a way that all the explicit terms in \sref{eq:oblq:anstz:jminBeforeInhom}
that do not involve the function $z$ are
of the form $\j(* , \Phi' ; t)$.
This way, the dependence of these terms on $\xi_{nl}(t)$ can be factored out,
leaving only terms that merely depend on the transverse coordinate $l$.
To this end, we start by writing
\begin{equation}
f_{p;\nu}^\diamond(\xi) = [\tau_\nu \Phi'](\xi) - \alpha^\diamond_{p;\nu} \Phi'(\xi),
\end{equation}
in which the choice
\begin{equation}
\alpha^\diamond_{p;\nu} = \int_\Real \Psi(\xi) [\tau_\nu \Phi'](\xi) \, d \xi,
\end{equation}
ensures by the characterization \sref{eq:mr:defRangeL0}
that one can find functions $p^{\diamond}_{\nu} \in BC^1(\Real, \Real)$
for which the relevant identity in \sref{eq:oblq:anstz:defFDiams} holds.

Incorporating the definitions above
into \sref{eq:oblq:anstz:jminBeforeInhom},
we arrive at the identity
\begin{equation}
\begin{array}{lcl}
\mathcal{J}^-(t)
& = &
 - \j( \dot{Z} , \Phi'; t)   - \dot{z}(t)  + \j(L ; t) z(t)
\\[0.2cm]
& & \qquad
  - \j( \dot{\theta}  , \Phi'; t)
  + \j( \alpha^\diamond_{p; \mu} \pi^{\diamond}_{\mu} \theta, \Phi' ; t )
\\[0.2cm]
& & \qquad \qquad
  + \j( \alpha^\diamond_{p;\mu} \pi^{\diamond}_{; \mu \mu'} \theta , p^{\diamond}_{\mu'} ; t)
  -  \j( \pi^{\diamond\diamond}_{;\mu \mu'} \theta, \tau_{\mu'} p^{\diamond}_{\mu} ; t)
  - \j( \pi^{\diamond \diamond}_{; \mu \mu'} \theta , f^{\diamond \diamond}_{p;\mu \mu'} ; t)
\\[0.2cm]
& & \qquad
  - \j( \alpha^\diamond_{p;\mu} \pi^\diamond_{\mu} \theta, \pi^{\diamond}_{; \mu'} \theta, D p^{\diamond}_{\mu'} ; t )
  + \j(\pi^\diamond_{; \mu} \theta,  \pi^\diamond_{; \mu'} \theta, \tau_{\mu'} D p^{\diamond}_{\mu} ; t )
  -  \j( \pi^{\diamond}_{; \mu } \theta , \pi^\diamond_{; \mu'}  \theta , f^{\diamond\diamond}_{q;\mu\mu'} ; t)
\\[0.2cm]
%
%
%
& & \qquad
+ \mathcal{R}_1 \big( \pi^{\diamond} \theta, \pi^{\diamond\diamond} \theta , z ; t \big)
+ \mathcal{R}_2 \big( \dot{Z}, \pi^\diamond \theta, \pi^{\diamond \diamond} \theta ; t \big)
+ \mathcal{R}_3 \big(\dot{\theta},  \pi^{\diamond \diamond} \theta ; t \big)
\\[0.2cm]
& & \qquad \qquad
+ \mathcal{R}_4 \big(\pi^\diamond \dot{\theta} , \pi^\diamond \theta ; t \big)
+ \mathcal{R}_5 \big( \pi^{\diamond} \theta, \pi^{\diamond \diamond} \theta ,
 \pi^{\diamond \diamond \diamond} \theta ; t \big)
+ \mathcal{R}_6 \big(\dot{\theta},  \pi^{\diamond} \theta ; t)
\\[0.2cm]
& & \qquad
+ \mathcal{E}_1\big( \pi^{\diamond \diamond}\dot{\theta} ; t \big)
+ \mathcal{E}_2\big( \pi^{\diamond \diamond \diamond} \theta ; t \big)
+ \mathcal{E}_3\big(  \pi^{\diamond} \dot{\theta}, \pi^{\diamond\diamond} \theta ; t \big),
\end{array}
\end{equation}
in which we have introduced the extra nonlinear expression
\begin{equation}
\label{eq:oblq:anstz:defr6}
\begin{array}{lcl}
\mathcal{R}_6 \big(\dot{\theta},  \pi^{\diamond} \theta ; t)
& = &  - \j( \dot{\theta} - \alpha^\diamond_{\mu} \pi^\diamond_{\mu} \theta , \pi^{\diamond}_{; \mu'} \theta, D p^{\diamond}_{\mu'} ; t ),
\\[0.2cm]
\end{array}
\end{equation}
together with the extra linear expression
\begin{equation}
\label{eq:oblq:anstz:defe3}
\begin{array}{lcl}
\mathcal{E}_3 \big(  \pi^\diamond \dot{\theta} , \pi^{\diamond\diamond} \theta ; t)
& = & \j( \pi^\diamond_{; \mu'} [ \dot{\theta} - \alpha^{\diamond}_{\mu} \pi^{\diamond}_{\mu} \theta ] , p^{\diamond}_{\mu'} ; t).
\\[0.2cm]
\end{array}
\end{equation}

In order to satisfy the goal mentioned above, it now suffices
to pick
\begin{equation}
\begin{array}{lcl}
f^{\diamond\diamond}_{p; \nu \nu'}(\xi) & = &
\alpha^\diamond_{p;\nu} p^\diamond_{\nu'}(\xi) - [\tau_{\nu'} p^\diamond_\nu](\xi)
 - \alpha^{\diamond\diamond}_{p;\nu \nu'} \Phi'(\xi),
\\[0.2cm]
f^{\diamond\diamond}_{q; \nu \nu'}(\xi)  & = &
- \alpha^\diamond_{p;\nu} D p^\diamond_{\nu'}(\xi) + [\tau_{\nu'} D p^\diamond_\nu](\xi)
 - \alpha^{\diamond\diamond}_{q;\nu \nu'} \Phi'(\xi),
\end{array}
\end{equation}
in which we have introduced the constants
\begin{equation}
\begin{array}{lcl}
\alpha^{\diamond\diamond}_{p; \nu \nu'} & = &
\int_\Real \Psi(\xi) \big[ \alpha^\diamond_{p;\nu} p^\diamond_{\nu'}(\xi) - [\tau_{\nu'} p^\diamond_\nu](\xi) \big] \, d \xi,
\\[0.2cm]
\alpha^{\diamond\diamond}_{q;\nu \nu'} & = &
\int_\Real \Psi(\xi) \big[ - \alpha^\diamond_{p;\nu} D p^\diamond_{\nu'}(\xi) + [\tau_{\nu'} Dp^\diamond_\nu](\xi) \big] \, d \xi.
\end{array}
\end{equation}
As above, these choices ensure that
one can find functions $p^{\diamond\diamond}_{\nu\nu'}, q^{\diamond\diamond}_{\nu \nu'} \in BC^1(\Real, \Real)$
for which the relevant identities in \sref{eq:oblq:anstz:defFDiams} hold.

Upon introducing our final nonlinear expression
\begin{equation}
\label{eq:oblq:anstz:defr7}
\mathcal{R}_7 \big( \pi^\diamond \theta ; t\big)
= \alpha^{\diamond\diamond}_{q; \mu \mu'} \j( \pi^{\diamond}_{; \mu } \theta , \pi^\diamond_{; \mu'}  \theta , \Phi' ; t),
\end{equation}
we can summarize our computations in the following result.
\begin{lem}
\label{lem:oblq:anstz:obsv}
Pick any pair $(\sigma_h, \sigma_v) \in \Wholes^2 \setminus \{(0 , 0)\}$
and suppose that (Hg) and $\textrm{(h}\Phi\textrm{)}_{\textrm{\S\ref{sec:prlm}}}$
both hold.
Then for every $(\nu, \nu') \in \{1, \ldots 5\}^2$,
there exist functions
\begin{equation}
p^\diamond_{\nu}, \,
p^{\diamond\diamond}_{\nu \nu'}, \,
q^{\diamond\diamond}_{\nu \nu'} \in BC^2(\Real, \Real)
\end{equation}
that satisfy the identities
\begin{equation}
\label{eq:lem:oblq:defpqdiams}
\begin{array}{lcl}
[\mathcal{L}_0 p^\diamond_{\nu}](\xi) & = &
[\tau_\nu \Phi'](\xi) - \alpha^{\diamond\diamond}_{p;\nu} \Phi'(\xi),
\\[0.2cm]
[\mathcal{L}_0 p^{\diamond \diamond}_{\nu \nu'}](\xi) & = &
\alpha^\diamond_{p;\nu} p^\diamond_{\nu'}(\xi) - [\tau_{\nu'} p^\diamond_\nu](\xi)
 - \alpha^{\diamond\diamond}_{p;\nu \nu'} \Phi'(\xi),
\\[0.2cm]
[\mathcal{L}_0 q^{\diamond \diamond}_{\nu \nu'}](\xi) & = &
- \alpha^\diamond_{p;\nu} D p^\diamond_{\nu'}(\xi) + [\tau_{\nu'} D p^\diamond_\nu](\xi)
 - \alpha^\diamond_{q;\nu \nu'} \Phi'(\xi),
\\[0.2cm]
\end{array}
\end{equation}
with coefficients
\begin{equation}
\label{eq:lem:oblq:defalphadiams}
\begin{array}{lcl}
\alpha^\diamond_{p, \nu} & = &
\int_\Real \Psi(\xi) [\tau_\nu \Phi'](\xi) \, d \xi,
\\[0.2cm]
\alpha^{\diamond\diamond}_{p; \nu \nu'} & = &
\int_\Real \Psi(\xi) \big[ \alpha^\diamond_{p;\nu} p^\diamond_{\nu'}(\xi) - [\tau_{\nu'} p^\diamond_\nu](\xi) \big] \, d \xi,
\\[0.2cm]
\alpha^{\diamond\diamond}_{q; \nu \nu'} & = &
\int_\Real \Psi(\xi) \big[ - \alpha^\diamond_{p;\nu} D p^\diamond_{\nu'}(\xi) + [\tau_{\nu'} Dp^\diamond_\nu](\xi) \big] \, d \xi.
\\[0.2cm]
\end{array}
\end{equation}
In addition,
for every triplet of $C^1$-smooth functions
\begin{equation}
z: [0, \infty) \to \Real, \qquad Z: [0, \infty) \to \Real,
\qquad \theta: [0, \infty) \to \ell^{\infty}(\Wholes; \Real),
\end{equation}
the function $u^-: [0, \infty) \to \ell^{\infty}(\Wholes^2; \Real)$,
defined by
\begin{equation}
\begin{array}{lcl}
u^-(t) & = &
\j(\Phi ; t) + \j( \pi^\diamond_{;\mu} \theta, p^{\diamond}_\mu ; t)
+ \j( \pi^{\diamond \diamond}_{;\mu \mu'} \theta, p^{\diamond \diamond}_{\mu \mu'} ; t)
+ \j( \pi^\diamond_{; \mu} \theta, \pi^\diamond_{; \mu'} \theta , q^{\diamond\diamond}_{\mu \mu'} ; t )
- z(t)
\end{array}
\end{equation}
with $\xi_{nl}(t) = n + ct - \theta_l(t) - Z(t)$,
admits the identity
\begin{equation}
\label{eq:lem:oblq:defJminusFinal}
\begin{array}{lcl}
\mathcal{J}^-(t)
&:=&
 \dot{u}^-(t) - \Delta^\times u^-(t) - g\big(u^-(t)\big)
\\[0.2cm]
& = &
 - \j( \dot{Z} , \Phi'; t)   - \dot{z}(t)  + \j(L ; t) z(t)
\\[0.2cm]
& & \qquad
  - \j( \dot{\theta}  , \Phi'; t)
  +\alpha^\diamond_{p; \mu} \j(  \pi^{\diamond}_{\mu} \theta, \Phi' ; t )
  + \alpha^{\diamond\diamond}_{p; \mu\mu'} \j( \pi^{\diamond\diamond}_{\mu\mu'} \theta, \Phi' ; t )
\\[0.2cm]
%
%
%
& & \qquad
+ \mathcal{R}_1 \big( \pi^{\diamond} \theta, \pi^{\diamond\diamond} \theta , z ; t \big)
+ \mathcal{R}_2 \big( \dot{Z}, \pi^\diamond \theta, \pi^{\diamond \diamond} \theta ; t \big)
+ \mathcal{R}_3 \big(\dot{\theta},  \pi^{\diamond \diamond} \theta ; t \big)
\\[0.2cm]
& & \qquad \qquad
+ \mathcal{R}_4 \big(\pi^\diamond \dot{\theta} , \pi^\diamond \theta ; t \big)
+ \mathcal{R}_5 \big( \pi^{\diamond} \theta, \pi^{\diamond \diamond} \theta ,
 \pi^{\diamond \diamond \diamond} \theta ; t \big)
+ \mathcal{R}_6 \big(\dot{\theta},  \pi^{\diamond} \theta ; t)
+ \mathcal{R}_7 \big( \pi^\diamond \theta ; t\big)
\\[0.2cm]
& & \qquad
+ \mathcal{E}_1\big( \pi^{\diamond \diamond}\dot{\theta} ; t \big)
+ \mathcal{E}_2\big( \pi^{\diamond \diamond \diamond} \theta ; t \big)
+ \mathcal{E}_3\big(  \pi^{\diamond} \dot{\theta}, \pi^{\diamond\diamond} \theta ; t \big).
\end{array}
\end{equation}
Here the nonlinear expressions $\mathcal{R}_1$ through $\mathcal{R}_7$
are defined in \sref{eq:oblq:anstz:defr1}, \sref{eq:oblq:anstz:defr24},
\sref{eq:oblq:anstz:defr5},
\sref{eq:oblq:anstz:defr6} and \sref{eq:oblq:anstz:defr7},
while the linear expressions $\mathcal{E}_1$ through
$\mathcal{E}_3$ are defined in \sref{eq:oblq:anstz:defe12} and \sref{eq:oblq:anstz:defe3}.

Finally, for any $\nu \in \{1 , \ldots 4 \}$ we have
\begin{equation}
[\pi^\times_{; \nu} - \pi^\times_{; 5}] u^-(t) = \mathcal{R}_{\mathcal{N};\nu}( \pi^\diamond \theta, \pi^{\diamond\diamond} \theta ; t ),
\end{equation}
with $\mathcal{R}_{\mathcal{N}; \nu}$ defined in \sref{eq:oblq:anstz:defrN}.
\end{lem}

In the remainder of this subsection we derive
some useful estimates on the numbered terms appearing in \sref{eq:lem:oblq:defJminusFinal}.
In addition, we establish a critical relation between the
expressions \sref{eq:lem:oblq:defpqdiams}-\sref{eq:lem:oblq:defalphadiams}
and spectral properties of the operators $\mathcal{L}_{\omega}$
discussed in Proposition \ref{prp:mr:melnikov}.

\begin{lem}
\label{lem:oblq:anstz:bndOnPDiams}
Consider the setting of Lemma \ref{lem:oblq:anstz:obsv}.
There exist constants $K_\diamond > 1$ and $\eta_\diamond > 0$
such that we have the bounds
\begin{equation}
\label{eq:lem:oblq:anstz:boundsOnPDiams}
\begin{array}{lcl}
\abs{ p^\diamond(\xi) } +
\abs{ p^{\diamond \diamond}(\xi) }
+ \abs{ q^{\diamond \diamond}(\xi) }
& \le & K_\diamond e^{ - \eta_\diamond \abs{\xi} },
\\[0.2cm]
\abs{ D p^\diamond(\xi) } +
\abs{ Dp^{\diamond \diamond}(\xi) }
+ \abs{ Dq^{\diamond \diamond}(\xi) }
& \le & K_\diamond e^{ - \eta_\diamond \abs{\xi} },
\\[0.2cm]
\end{array}
\end{equation}
for all $\xi \in \Real$.
In addition, we have the bounds
\begin{equation}
\label{eq:lem:oblq:anstz:quadBoundsOnPDiams}
\begin{array}{lcl}
\big[ \abs{ p^\diamond(\xi) } +
\abs{ p^{\diamond \diamond}(\xi) }
+ \abs{ q^{\diamond \diamond}(\xi) } \big]^2
& \le & K_\diamond \Phi'(\xi),
\end{array}
\end{equation}
again for all $\xi \in \Real$.
\end{lem}
\begin{proof}
Upon recalling the constants $\eta^\pm_\Phi$
from Lemma \ref{lem:prlm:defSpatExps},
let us write $\eta^\pm_\diamond = \frac{3}{4} \eta^\pm_\Phi$.
Upon writing
\begin{equation}
\Lambda_{\mathrm{inv}}^+:  BC_{- \eta^+_\diamond}([0, \infty), \Real) \to BC^1_{-\eta^+_\diamond}( [-\sigma, \infty), \Real)
\end{equation}
for the inverse of $\mathcal{L}_0$ that has properties
analogous to those stated in Lemma \ref{lem:prlm:defInverseOnHalflines},
we can write
\begin{equation}
[p^\diamond_{\nu}]_{\mid [-\sigma, \infty) } = \Lambda_{\mathrm{inv}}^+ [\tau_\nu \Phi' - \alpha^\diamond_{p;\nu} \Phi' ].
\end{equation}
Similar properties hold on $(-\infty, \sigma]$
and the desired estimates now follow directly from these observations.
\end{proof}

\begin{lem}
\label{lem:oblq:specProps}
Consider the setting of Lemma \ref{lem:oblq:anstz:obsv}.
We have the identities
\begin{equation}
\begin{array}{lcl}
[ \frac{d}{d \omega} \lambda_{\omega} ]_{\omega = 0} & = &
  i\sigma_\mu \alpha^\diamond_\mu,
\\[0.2cm]
[ \frac{d^2}{d \omega^2} \lambda_{\omega} ]_{\omega = 0} & = &
  - \sigma_\mu^2 \alpha^\diamond_{p; \mu} - 2 \sigma_\mu \sigma_{\mu'} \alpha^{\diamond\diamond}_{p; \mu \mu'}.
\end{array}
\end{equation}
\end{lem}
\begin{proof}
Throughout this proof we use the shorthand $D_{\omega} = \frac{d}{d \omega}$.
Taylor expanding the quantities
$\lambda_\omega$ and $\phi_\omega$ for $\omega \approx 0$
in the expression
$(\mathcal{L}_\omega - \lambda_\omega)\phi_\omega = 0$
and remembering that $\phi_0 = \Phi'$, we obtain
\begin{equation}
\label{eq:mel1}
\begin{array}{lcl}
 O(\omega^3) & = &
\mathcal{L}_0\Phi' + \omega\Big(  \mathcal{L}_0 [D_{\omega}\phi_\omega]_{\omega = 0}
+ [D_{\omega}\mathcal{L}_\omega- D_{\omega} \lambda_\omega]_{\omega = 0} \Phi' \Big)
\\[0.2cm]
& & \qquad
+ \frac{\omega^2}{2}\Big( \mathcal{L}_0 [D_\omega^2 \phi_\omega]_{\omega = 0}
+ 2[D_\omega \mathcal{L}_\omega - D_\omega \lambda_\omega]_{\omega = 0} [D_\omega \phi_\omega]_{\omega = 0}
+ [D_\omega^2\mathcal{L}_\omega - D_\omega^2 \lambda_\omega]_{\omega = 0} \Phi' \Big).
\end{array}
\end{equation}
Taking the inner product against $\Psi$ and recalling that $\mathcal{L}_0^* \Psi= 0$ leaves
\begin{equation}
 [D_{\omega}\lambda_\omega]_{\omega = 0} = \int_{-\infty}^\infty \Psi(\xi)
   \big[[D_{\omega}\mathcal{L}_\omega]_{\omega = 0} \Phi'\big](\xi)  \, d \xi.
\end{equation}
Recalling the identity
\begin{equation}
[\mathcal{L}_{\omega} v](\xi)
=
- cv'(\xi) + L^\times_\mu \exp ( i \sigma_\mu \omega ) [\tau_\mu v](\xi)
   + g'\big(\Phi(\xi) \big) v,
\end{equation}
we may compute
\begin{equation}
\begin{array}{lcl}
[D_\omega \mathcal{L}_{\omega} v](\xi)
& = &
 i \sigma_\mu L^\times_\mu \exp ( i \sigma_\mu \omega ) [\tau_\mu v](\xi),
\\[0.2cm]
& = &
 i \sigma_\mu \exp ( i \sigma_\mu \omega ) [\tau_\mu v](\xi),
\\[0.2cm]
[D^2_\omega \mathcal{L}_{\omega} v](\xi)
& = &
- \sigma^2_\mu \exp ( i \sigma_\mu \omega ) [\tau_\mu v](\xi).
\end{array}
\end{equation}
We hence see that
\begin{equation}
\begin{array}{lcl}
[D_{\omega}\mathcal{L}_\omega]_{\omega = 0} \Phi'
& = &
 i \sigma_\mu \tau_\mu \Phi',
\end{array}
\end{equation}
which implies
\begin{equation}
\begin{array}{lcl}
[D_{\omega}\lambda_\omega]_{\omega = 0}
& = &  i \sigma_\mu \int_{-\infty}^\infty \Psi(\xi) [\tau_\mu \Phi' ](\xi)
 \, d \xi
\\[0.2cm]
& = &  i \sigma_\mu \alpha^\diamond_\mu,
\end{array}
\end{equation}
as desired.

Moving on, we see that
\begin{equation}
\begin{array}{lcl}
\mathcal{L}_0 [D_{\omega} \phi_{\omega}]_{\omega = 0}
& = &  i \sigma_\mu \alpha^{\diamond}_\mu \Phi' - i \sigma_\mu \tau_\mu \Phi'
\\[0.2cm]
& = & - i \sigma_\mu f^\diamond_\mu.
\end{array}
\end{equation}
In particular, up to multiples of $\Phi'$ we may write
\begin{equation}
[D_{\omega} \phi_{\omega}]_{\omega = 0} = - i \sigma_\mu p^\diamond_\mu.
\end{equation}
We now integrate the $O(\omega^2)$ term in \sref{eq:mel1} against $\Psi$, yielding
\begin{equation}
\begin{array}{lcl}
[D^2_{\omega} \lambda_{\omega} ]_{\omega = 0} & = &
 \int_{-\infty}^\infty \Psi(\xi)  \big[ [D_{\omega}^2 \mathcal{L}_\omega]_{\omega = 0}
    \Phi' \big](\xi)  \, d \xi \\[0.2cm]
& & \qquad + 2\int_{-\infty}^\infty  \Psi(\xi)
      \Big[ [D_\omega \mathcal{L}_\omega -
            D_\omega \lambda_\omega]_{\omega = 0}
      [D_\omega\phi_\omega]_{\omega = 0} \Big](\xi)  \, d \xi.
\label{eq:mel2}
\end{array}
\end{equation}
Plugging in the identities obtained above, we obtain
\begin{equation}
\begin{array}{lcl}
[D^2_{\omega} \lambda_{\omega} ]_{\omega = 0} & = &
 \int_{-\infty}^\infty  \Psi(\xi)  \big[ - \sigma_\mu^2 \tau_\mu \Phi'   \big](\xi)  \, d \xi \\[0.2cm]
& & \qquad + 2\int_{-\infty}^\infty  \Psi(\xi)
  \Big[ [ i \sigma_\mu \tau_\mu - i \sigma_\mu \alpha^\diamond_\mu ]
    [- i \sigma_{\mu'} p^\diamond_{\mu'} ] \Big](\xi)  \, d \xi
\\[0.2cm]
& = &
- \sigma_\mu^2 \alpha^\diamond_\mu
+ 2 \sigma_\mu \sigma_{\mu'}
\int_{-\infty}^{\infty} \Psi(\xi) [\tau_\mu p^\diamond_{\mu'} - \alpha^\diamond_\mu p^\diamond_{\mu'} ](\xi) \, d \xi
\\[0.2cm]
& = &
- \sigma_\mu^2 \alpha^\diamond_\mu
+ 2 \sigma_\mu \sigma_{\mu'}
\int_{-\infty}^{\infty} \Psi(\xi)
[\tau_{\mu'} p^\diamond_\mu - \alpha^\diamond_\mu p^\diamond_{\mu'} ](\xi) \, d \xi
\\[0.2cm]
& = &
- \sigma_\mu^2 \alpha^\diamond_\mu - 2 \sigma_\mu \sigma_{\mu'} \alpha^{\diamond\diamond}_{\mu \mu'},
\\[0.2cm]
\end{array}
\end{equation}
as desired.
\end{proof}

\begin{lem}
\label{lem:oblq:anstz:bndOnR17}
Fix any pair $(\sigma_h, \sigma_v) \in \Wholes^2 \setminus \{(0 , 0)\}$,
suppose that (Hg) and $\textrm{(h}\Phi\textrm{)}_{\textrm{\S\ref{sec:prlm}}}$
both hold and pick any $M_1 > 0$.
Then there exists a constant  $C_1 = C_1(M_1) > 1$
such that for every pair of $C^1$-smooth functions
\begin{equation}
\theta: [0, \infty) \to \ell^{\infty}(\Wholes; \Real),
\qquad z: [0, \infty) \to \Real
\end{equation}
that satisfy the bound
\begin{equation}
\abs{z(t)} + \norm{\pi^\diamond \theta(t)}_{\ell^\infty(\Wholes; \Real^5)}  <M_1, \qquad t \ge 0,
\end{equation}
we have the estimates
\begin{equation}
\begin{array}{lcl}
\abs{ [\mathcal{R}_1
   \big( \pi^{\diamond} \theta, \pi^{\diamond\diamond} \theta , z ; t \big)
 ]_{nl}
}
&\le &
C_1 \abs{z(t)}\big( \abs{z(t)}   + \abs{\pi^{\diamond}_l \theta(t)}
+ \abs{\pi^{\diamond \diamond}_l \theta(t)} \big)
\\[0.2cm]
& & \qquad
+
C_1\big(  \abs{\pi^{\diamond}_l \theta(t)}
      + \abs{\pi^{\diamond \diamond}_l \theta(t)} \big)^2 \Phi'\big(\xi_{nl}(t) \big),
\\[0.2cm]
\abs{ [\mathcal{R}_7
  \big( \pi^\diamond \theta ; t \big)
]_{nl}
}
& \le &
C_1 \abs{ \pi^\diamond_l \theta(t) }^2 \Phi'\big(\xi_{nl}(t) \big)
\\[0.2cm]
\end{array}
\end{equation}
for every $(n,l) \in \Wholes^2$ and  $t \ge 0$.
\end{lem}
\begin{proof}
On account of the a-priori bound $\abs{\pi^\diamond_l \theta} < M_1$,
we can invoke Corollary \ref{cor:prlm:shifts}
to obtain
\begin{equation}
\abs{\mathcal{M}_{\Phi, 2, \mu}\big( \xi_{nl}(t) , \pi^\diamond_{l;\mu} \theta(t)\big)}
\le C'_1 \Phi'\big(\xi_{nl}(t) \big) \abs{ \pi^\diamond_{l; \mu} \theta(t) }^2
\end{equation}
for some $C'_1 = C'_1(M_1)> 1$. In addition, the a-priori bound on
$u^-(t) - \j( \Phi ; t)$ allows us to write
\begin{equation}
\abs{ \big[ \mathcal{G}_0\big( u^-(t) - \j(\Phi ; t) \big) \big]_{nl} } \le C'_2 \abs{u^-_{nl}(t) - \j_{nl}(\Phi ; t)}^2
\end{equation}
for some $C'_2 > 0$.
The desired estimates follow directly from these observations
upon utilizing the bound \sref{eq:lem:oblq:anstz:quadBoundsOnPDiams}.
\end{proof}

\begin{lem}
\label{lem:hom:oblq:ansz:bndN2Plus}
Fix any pair $(\sigma_h, \sigma_v) \in \Wholes^2 \setminus \{(0 , 0)\}$,
suppose that (Hg) and $\textrm{(h}\Phi\textrm{)}_{\textrm{\S\ref{sec:prlm}}}$
both hold and pick any $M_1 > 0$.
Then there exists a constant  $C_1 = C_1(M_1) > 1$
such that for every pair of $C^1$-smooth functions
\begin{equation}
\theta: [0, \infty) \to \ell^{\infty}(\Wholes; \Real),
\qquad Z: [0, \infty) \to \Real
\end{equation}
that satisfy the bound
\begin{equation}
\norm{\pi^\diamond \theta(t)}_{\ell^\infty(\Wholes; \Real^5) }  <M_1, \qquad t \ge 0,
\end{equation}
we have the estimates
\begin{equation}
\begin{array}{lcl}
\abs{ [ \mathcal{R}_2
      \big( \dot{Z}, \pi^\diamond \theta, \pi^{\diamond \diamond} \theta ; t \big)
  ]_{nl} }
&\le &
C_1 \abs{\dot{Z}(t) } \big[ \abs{ \pi^\diamond_l \theta(t) }
   + \abs{\pi^{\diamond \diamond}_l \theta(t)}  \big],
\\[0.2cm]
\abs{ [ \mathcal{R}_3
   \big(\dot{\theta},  \pi^{\diamond \diamond} \theta ; t \big)
  ]_{nl} }
&\le &
C_1 \abs{ \dot{\theta}_l(t) } \abs{ \pi^{\diamond\diamond}_l \theta(t) },
\\[0.2cm]
\abs{ [ \mathcal{R}_4
 \big(\pi^\diamond \dot{\theta} , \pi^\diamond \theta ; t \big)
]_{nl} }
& \le &  C_1 \abs{ \pi^\diamond_l \dot{\theta}(t) } \abs{ \pi^\diamond_l \theta(t) },
\\[0.2cm]
\abs{ [ \mathcal{R}_5
 \big( \pi^{\diamond} \theta, \pi^{\diamond \diamond} \theta ,
 \pi^{\diamond \diamond \diamond} \theta ; t \big)
]_{nl} }
& \le &
\abs{\pi^\diamond_l \theta(t) }
\big[ \abs{\pi^{\diamond \diamond}_l \theta(t) } + \abs{\pi^\diamond_l \theta(t) }^2
+ \abs{\pi^{\diamond \diamond\diamond}_l \theta(t) } \big]
\\[0.2cm]
& & \qquad
+ \abs{ \pi^{\diamond\diamond}_l \theta(t) }^2,
\\[0.2cm]
\abs{ [ \mathcal{R}_6
 \big(\dot{\theta},  \pi^{\diamond} \theta ; t)
]_{nl} }
& \le & C_1 \abs{ \dot{\theta}_l(t) - \alpha^\diamond_\mu \pi^\diamond_{l; \mu} \theta(t) } \abs{ \pi^\diamond_l \theta(t) },
\\[0.2cm]
\abs{ [ \mathcal{R}_{\mathcal{N}; \nu} \big(
   \pi^\diamond \theta, \pi^{\diamond \diamond} \theta , \pi^{\diamond \diamond \diamond} \theta ;t
  \big)
  ]_{nl} }
&\le & C_1 e^{ - \kappa_\diamond \abs{ \xi_{nl}(t) } }
\end{array}
\end{equation}
for every $(n,l) \in \Wholes^2$ and  $t \ge 0$.
\end{lem}
\begin{proof}
These bounds follow directly from
the definitions
\sref{eq:oblq:anstz:defrN},
\sref{eq:oblq:anstz:defr1}, \sref{eq:oblq:anstz:defr24},
\sref{eq:oblq:anstz:defr5}, \sref{eq:oblq:anstz:defr6} and \sref{eq:oblq:anstz:defr7},
together with
the bounds \sref{eq:lem:oblq:anstz:boundsOnPDiams}.
\end{proof}

\begin{lem}
\label{lem:hom:oblq:ansz:bndE13}
Fix any pair $(\sigma_h, \sigma_v) \in \Wholes^2 \setminus \{(0 , 0)\}$
and suppose that (Hg) and $\textrm{(h}\Phi\textrm{)}_{\textrm{\S\ref{sec:prlm}}}$
both hold. Then there exists a constant  $C_1 > 1$
such that for every $C^1$-smooth function
\begin{equation}
\theta: [0, \infty) \to \ell^{\infty}(\Wholes; \Real),
\end{equation}
we have the bounds
\begin{equation}
\begin{array}{lcl}
\abs{ [\mathcal{E}_1\big(
  \pi^{\diamond\diamond} \dot{\theta} ; t
  \big) ]_{nl}
}
&\le & C_1  \abs{\pi^{\diamond\diamond}_l  \dot{\theta}(t) },
\\[0.2cm]
\abs{ [ \mathcal{E}_2\big(
   \pi^{\diamond \diamond \diamond} \theta ;t
  \big)
  ]_{nl} }
&\le & C_1 \abs{  \pi^{\diamond \diamond \diamond}_l \theta(t) },
\\[0.2cm]
\abs{ [ \mathcal{E}_3\big(
   \pi^{\diamond} \dot{\theta}, \pi^{\diamond \diamond} \theta  ;t
  \big)
  ]_{nl} }
&\le & C_1 \abs{ \pi^{\diamond}_l
   \big[\dot{\theta}(t) - \alpha^\diamond_\mu \pi^\diamond_{l;\mu} \theta(t) \big] }
\end{array}
\end{equation}
for all $(n,l) \in \Wholes^2$ and $t \ge 0$.
\end{lem}
\begin{proof}
These estimates follow directly from
the definitions \sref{eq:oblq:anstz:defe12} and \sref{eq:oblq:anstz:defe3}
together with the bounds \sref{eq:lem:oblq:anstz:boundsOnPDiams}.
\end{proof}

\subsection{The expanding plateau }
\label{sec:oblq:plt}
In this subsection we construct a function $v: [1, \infty) \to \ell^{\infty}(\Wholes; \Real)$
that can be thought of as an expanding and convecting plateau.
Later on, we will obtain the function $\theta$ by
multiplying $v$ by a global prefactor that decays in time.

For convenience, we define the quantities
\begin{equation}
\begin{array}{lcl}
\nu_1  & = & \frac{1}{i} [\frac{d}{d \omega} \lambda_{\omega} ]_{\omega = 0},
\\[0.2cm]
\nu_2 & = &
- \frac{1}{2}[\frac{d^2}{d \omega^2} \lambda_{\omega} ]_{\omega = 0}.
\end{array}
\end{equation}
We assume throughout this section that $\nu_2 > 0$,
noting that this is a consequence of the Melnikov condition $(HS)_{\zeta_*}$.
For any $\gamma \ge 1$ and $t \ge 1$, we define
\begin{equation}
\label{eq:oblq:plt:defPlateauV}
v_{l; \gamma}(t) =
(\gamma t)^{1/2} \int_{-\infty}^\infty \exp[ i \omega ( l  + \nu_1 t )  ]
  \exp[ -  \nu_2 \omega^2 \gamma t ] \,  d \omega,
\end{equation}
which can be explicitly evaluated as
\begin{equation}
v_{l; \gamma}(t) = \sqrt{\frac{\pi}{\nu_2}} \exp\big[ - \frac{ (l +  \nu_1 t)^2 } { 4 \nu_2 \gamma t} \big].
\end{equation}
We also introduce  the functions
\begin{equation}
\mathcal{V}^{(k)}_{l; \gamma}(t) =
(\gamma t)^{1/2} \int_{-\infty}^\infty \omega^k \exp[ i \omega \big( l  + \nu_1 t \big)  ]
  \exp[ -  \nu_2 \omega^2 \gamma t ] \,  d \omega
\end{equation}
for integers $k \ge 1$, which can be evaluated as
\begin{equation}
\label{eq:oblq:plt:bndVk}
\begin{array}{lcl}
\mathcal{V}^{(1)}_{l ; \gamma}(t) & = &
\frac{1}{2} i \frac{l + \nu_1 t}{ \nu_2 \gamma t} v_{l ; \gamma}(t),
\\[0.2cm]
\mathcal{V}^{(2)}_{l ; \gamma}(t) & = &
\big[ - \frac{1}{4} \frac{ (l + \nu_1 t)^2 }{ (\nu_2 \gamma t)^2} + \frac{1}{2} \frac{1}{ \nu_2 \gamma t } \big] v_{l ; \gamma}(t),
\\[0.2cm]
\mathcal{V}^{(3)}_{l ; \gamma}(t) & = &
\big[ - \frac{i}{8} \frac{ ( l + \nu_1 t)^3} { (\nu_2 \gamma t)^3  } + \frac{3i}{4}
\frac{ ( l + \nu_1 t)}{(\nu_2 \gamma t)} ( \nu_2 \gamma t)^{-1} \big]
v_{l ; \gamma}(t).
\\[0.2cm]
\end{array}
\end{equation}
To prevent cumbersome notation, we introduce the shorthand
\begin{equation}
\rho = \rho(l, t; \gamma) =  \frac{l + \nu_1 t}{ 2 \nu_2 \gamma t},
\end{equation}
which allows us to reduce the expressions above to the compact form
\begin{equation}
\label{eq:oblq:plt:defVK}
\begin{array}{lcl}
\mathcal{V}^{(1)}_{l ; \gamma}(t) & = &
i \rho v_{l ; \gamma}(t),
\\[0.2cm]
\mathcal{V}^{(2)}_{l ; \gamma}(t) & = &
\big[ - \rho^2 + \frac{1}{2} ( \nu_2 \gamma t )^{-1} \big] v_{l ; \gamma}(t),
\\[0.2cm]
\mathcal{V}^{(3)}_{l ; \gamma}(t) & = &
\big[ -i \rho^3 + \frac{3i}{2} \rho ( \nu_2 \gamma t)^{-1} \big]
v_{l ; \gamma}(t).
\\[0.2cm]
\end{array}
\end{equation}
Our first result studies
the effects of non-polynomial Fourier multipliers
applied to $v_{; \gamma}(t)$.
\begin{lem}
\label{lem:oblq:plt:bndOnW}
Pick a sufficiently small $\delta_{\omega} > 0$
and any $M > 1$. Consider any analytic function
$\wp: \Complex \to \Complex$ that
has
\begin{equation}
\label{eq:lem:oblq:plt:bndOnW:absBoundWP}
\abs{\wp(\omega)} \le M( 1 + \omega^4), \qquad \omega \in \Real
\end{equation}
and satisfies the bound
\begin{equation}
\label{eq:lem:oblq:plt:bndOnW:absBoundWPZero}
\abs{\wp(\omega)} \le M \abs{\omega}^s, \qquad - \delta_{\omega} \le \Re \omega \le \delta_{\omega}, \qquad
- \delta_{\omega} \le \Im \delta \le + \delta_{\omega}
\end{equation}
for some $s \in \{0, 2, 4\}$.
In addition, consider the expression
\begin{equation}
\mathcal{W}^{(\wp)}_{l; \gamma}(t) =
(\gamma t)^{1/2} \int_{-\infty}^{\infty} \wp(\omega)  \exp[ i \omega \big( l  + \nu_1 t \big)  ]
  \exp[ -  \nu_2 \omega^2 \gamma t ] \,  d \omega.
\end{equation}
Then for any triplet $(l, t,\gamma)$
with $\gamma \ge 1$, $t \ge 1$ and $\abs{\rho(l, t; \gamma)} \le \delta_{\omega}$,
we have the bounds
\begin{equation}
\label{eq:lem:oblq:plt:bndWM}
\abs{\mathcal{W}^{(\wp)}_{l; \gamma}(t) }
\le
\left\{ \begin{array}{lcl}
   M v_{l; \gamma}(t)
   \\[0.2cm] \qquad + M \big[(\gamma t)^{1/2} + 16\nu_2^{-1}\delta_{\omega}^{-1}
     [ \frac{3}{4} (\nu_2 \delta_{\omega})^{-4} + 1 ]  \big] e^{ - \nu_2 \delta_{\omega}^2 \gamma t},
   & & \hbox{for } s = 0,
   \\[0.2cm]
   2 M v_{l; \gamma}(t) [ \rho^2 + \frac{1}{2}( \nu_2 \gamma t)^{-1} ]
   \\[0.2cm] \qquad +  M \big[(\gamma t)^{1/2} + 16\nu_2^{-1}\delta_{\omega}^{-1}[ \frac{3}{4} (\nu_2 \delta_{\omega})^{-4} + 1 ]  \big] e^{ - \nu_2 \delta_{\omega}^2 \gamma t},
   & & \hbox{for } s = 2,
   \\[0.2cm]
   8 M v_{l; \gamma}(t) [ \rho^4 + \frac{3}{4}( \nu_2 \gamma t)^{-2} ]
   \\[0.2cm] \qquad +  M \big[(\gamma t)^{1/2} + 16\nu_2^{-1}\delta_{\omega}^{-1}[ \frac{3}{4} (\nu_2 \delta_{\omega})^{-4} + 1 ]  \big] e^{ - \nu_2 \delta_{\omega}^2 \gamma t},
   & & \hbox{for } s = 4,
   \\[0.2cm]
  \end{array}
\right.
\end{equation}
with $\rho = \rho(l, t; \gamma)$.
In addition, for any $(l, t,\gamma)$ with
$\gamma \ge 1$, $t \ge 1$ and $\abs{ \rho(l, t; \gamma) } \ge \delta_{\omega}$,
we have
\begin{equation}
\label{eq:lem:oblq:plt:bndWMRhoLarge}
\abs{\mathcal{W}^{(\wp)}_{l; \gamma}(t) }
\le  M \big[2(\gamma t)^{1/2} + 16\nu_2^{-1}\delta_{\omega}^{-1}[ \frac{3}{4} (\nu_2 \delta_{\omega})^{-4} + 1 ]  \big] e^{ - \nu_2 \delta_{\omega}^2 \gamma t}.
\end{equation}
\end{lem}
\begin{proof}
Upon introducing the expressions
\begin{equation}
p(\omega) = - \nu_2 \omega^2 \gamma t, \qquad q(\omega) = \omega ( l + \nu_1 t),
\end{equation}
we compute, for any $y \in \Real$,
\begin{equation}
\begin{array}{lcl}
p(\omega + i y) & = & - \nu_2 ( - y^2 + 2 i y \omega + \omega^2) \gamma t
\\[0.2cm]
& = &
\nu_2 y^2 \gamma t - i \omega y(2 \nu_2 \gamma t) - \nu_2 \omega^2 \gamma t,
\\[0.2cm]
q(\omega + i y) & = & i y (l + \nu_1 t) + \omega(l + \nu_1 t)
\\[0.2cm]
& = & i y \rho (2 \nu_2 \gamma t) + \omega \rho (2 \nu_2 \gamma t),
\\[0.2cm]
\end{array}
\end{equation}
where $\rho = \rho(l, t ;\gamma)$.
In other words, we have
\begin{equation}
\label{eq:lem:hom:oblq:techBondsOnInt:defRePart}
\Re \, [p(\omega + i y) + iq(\omega + i y)]   =
[ y^2 - 2y \rho] \nu_2 \gamma t  - \omega^2 \nu_2  \gamma t.
\end{equation}

For convenience, let us assume from now on that $\rho \ge 0$.
On the interval $y \in [0, \rho]$, we have
\begin{equation}
\label{eq:lem:hom:oblq:techBondsOnInt:bndsOnRePart}
\Re \, \big[p(\omega + i y) + iq(\omega + i y) \big]   \le
 - \omega^2 \nu_2  \gamma t.
\end{equation}
For each fixed $\omega$, the expression
\sref{eq:lem:hom:oblq:techBondsOnInt:defRePart} is minimized
on the domain $y \in [0, \delta_{\omega}]$
upon choosing $y = y_* = \min\{ \rho, \delta_{\omega} \}$.
In the case $y_* = \rho$,
we have
\begin{equation}
p(\omega + i \rho) + iq(\omega + i \rho)   =
 -  \rho^2 \nu_2 \gamma t - \nu_2 \omega^2 \gamma t,
\end{equation}
while in the case $y_* = \delta_{\omega}$, we have
\begin{equation}
\Re \, \big[p(\omega + i y_*) + iq(\omega + i y_*) \big]
\le -\nu_2 \delta_{\omega}^2 \gamma t - \omega^2 \nu_2 \gamma t.
\end{equation}

Upon introducing the five line segments
\begin{equation}
\begin{array}{ll}
\Gamma_1 = (-\infty, - \delta_{\omega}], &
\Gamma_2 = [-\delta_{\omega} ,-\delta_{\omega} +  i y_*],
\qquad
\Gamma_3 = [-\delta_{\omega} + i y_*, + \delta_{\omega} + i y_*],
\\[0.2cm]
\Gamma_4 =[ \delta_{\omega} + i y_* , +\delta_{\omega} ], &
\Gamma_5 = [\delta_{\omega}, \infty),
\end{array}
\end{equation}
we define the separate integrals
\begin{equation}
\label{eq:lem:hom:oblq:techBondsOnInt:indInts}
\mathcal{W}^{(\wp)}_{l; \gamma; \Gamma_i}(t) = (\gamma t)^{1/2} \int_{\Gamma_i} \wp(\omega) \exp[p(\omega) + i q (\omega)] \, d\omega
\end{equation}
and note that Cauchy's theorem implies that
\begin{equation}
\mathcal{W}^{(\wp)}_{l; \gamma}(t) = \sum_{i=1}^5 \mathcal{W}^{(\wp)}_{l; \gamma; \Gamma_i}(t).
\end{equation}

Setting out to bound each of the integrals \sref{eq:lem:hom:oblq:techBondsOnInt:indInts}
separately, we start by computing
\begin{equation}
\begin{array}{lcl}
\int_{\delta_{\omega}}^{\infty} (1 +  \omega^4 ) e^{ - \nu_2 \omega^2 \gamma t } d \omega
& = &
e^{ - \nu_2 \delta_{\omega}^2 \gamma t } \int_{\delta_{\omega}}^\infty
  (1 + \omega^4) e^{ - \nu_2 (\omega^2 - \delta_{\omega}^2) \gamma t } d\omega
\\[0.2cm]
& = &
e^{ - \nu_2 \delta_{\omega}^2 \gamma t }
\int_{\delta_{\omega}}^\infty
  (1 +  \omega^4 ) e^{ - \nu_2 (\omega - \delta_{\omega})(\omega + \delta_{\omega}) \gamma t } d\omega
\\[0.2cm]
& \le &
8 e^{ - \nu_2 \delta_{\omega}^2 \gamma t }
\int_{\delta_{\omega}}^\infty
  [(\omega-\delta_{\omega})^4 + 1 + \delta_{\omega}^4 ] e^{ - \nu_2 (\omega - \delta_{\omega}) 2 \delta_{\omega} \gamma t } d\omega
\\[0.2cm]
& = &
8\big[ 24(2 \nu_2 \delta_{\omega} \gamma t)^{-5} + (1+  \delta_\omega^4)(2 \nu_2 \delta_{\omega} \gamma t)^{-1}  \big]e^{ - \nu_2 \delta_{\omega}^2 \gamma t }.
\end{array}
\end{equation}
In particular, imposing the restriction $0 < \delta_{\omega} \le \frac{1}{2}$
and remembering $\gamma t \ge 1$, we see that
\begin{equation}
\begin{array}{lcl}
\abs{\mathcal{W}^{(\wp)}_{l; \gamma; \Gamma_1} (t) }
+ \abs{\mathcal{W}^{(\wp)}_{l; \gamma; \Gamma_5} (t) }
& \le &
 16 M (\nu_2 \delta_{\omega})^{-1}[ \frac{3}{4} (\nu_2 \delta_{\omega})^{-4} + 1 ].
\end{array}
\end{equation}

Moving on,
\sref{eq:lem:hom:oblq:techBondsOnInt:bndsOnRePart}
implies that
for all $\omega \in \Gamma_2 \cup \Gamma_4$ we have
\begin{equation}
\Re \, \big[ p(\omega) + i q(\omega) \big] \le
  - \nu_2 \delta_{\omega}^2 \gamma t.
\end{equation}
Remembering that $0 \le y_* \le \delta_{\omega}$
and imposing the  restriction $0 < \delta_{\omega} \le \frac{1}{2}$,
we obtain
\begin{equation}
\begin{array}{lcl}
 \abs{ \mathcal{W}^{(m)}_{l; \gamma; \Gamma_2}(t) }
 + \abs{ \mathcal{W}^{(m)}_{l; \gamma; \Gamma_4}(t) }
 & \le &
 2 M \delta_{\omega} (\gamma t)^{1/2} \abs{2 \delta_{\omega} }^{s} e^{ - \nu_2 \delta_{\omega}^2 \gamma t}
 \\[0.2cm]
 & \le &   M (\gamma t)^{1/2}  e^{ - \nu_2 \delta_{\omega}^2 \gamma t}.
 \end{array}
\end{equation}
In addition, for $0 \le \rho \le \delta_{\omega}$ we have
\begin{equation}
\begin{array}{lcl}
(\gamma t)^{-1/2} \abs{ \mathcal{W}^{(\wp)}_{l; \gamma; \Gamma_3}(t) }
& \le & \exp[  - y_*^2 \nu_2 \gamma t]
 \int_{-\delta_{\omega}}^{\delta_{\omega}}
  M ( \abs{\omega} + \abs{ y_*} )^s e^{ - \nu_2 \omega^2 \gamma t } \, d\omega
\\[0.2cm]
& \le &
\exp[  -y_*^2 \nu_2 \gamma t]
 \int_{-\infty}^{+ \infty}
  M ( \abs{\omega} + \abs{ y_*} )^s e^{ - \nu_2 \omega^2 \gamma t } \, d\omega
\\[0.2cm]
& \le &
\exp[  - y_*^2 \nu_2 \gamma t]
 \int_{-\infty}^{+ \infty}
  M 2^{s -1} ( \abs{\omega}^s + \abs{ y_*}^s )
   e^{ - \nu_2 \omega^2 \gamma t } \, d\omega.
\end{array}
\end{equation}
The desired expressions \sref{eq:lem:oblq:plt:bndWM}
for $0 \le \rho \le \delta_{\omega}$ now follow from the identities
\begin{equation}
\begin{array}{lcl}
\int_{-\infty}^{+ \infty} e^{ - \nu_2 \omega^2 \gamma t } \, d\omega & = &
\sqrt{\frac{\pi}{ \nu_2 \gamma t} },
\\[0.2cm]
\int_{-\infty}^{+ \infty} \omega^2 e^{ - \nu_2 \omega^2 \gamma t } \, d\omega & = &
\frac{1}{2}(\nu_2 \gamma t)^{-1} \sqrt{\frac{\pi}{ \nu_2 \gamma t} },
\\[0.2cm]
\int_{-\infty}^{+ \infty} \omega^4 e^{ - \nu_2 \omega^2 \gamma t } \, d\omega & = &
\frac{3}{4}(\nu_2 \gamma t)^{-2} \sqrt{\frac{\pi}{ \nu_2 \gamma t} }.
\end{array}
\end{equation}
On the other hand,
for $\rho \ge \delta_{\omega}$, we compute
\begin{equation}
\begin{array}{lcl}
 \abs{ \mathcal{W}^{(\wp)}_{l; \gamma; \Gamma_3}(t) }
& \le & (\gamma t)^{1/2} \exp[  - \delta_{\omega}^2 \nu_2 \gamma t]
 \int_{-\delta_{\omega}}^{\delta_{\omega}}
  M (  \abs{2 \delta_{\omega} }^s )  e^{ - \nu_2 \omega^2 \gamma t } \, d\omega
\\[0.2cm]
& \le &
 (\gamma t)^{1/2} \exp[  -\delta_{\omega}^2 \nu_2 \gamma t]
 \big(2 \delta_{\omega} M \big)
\\[0.2cm]
& \le & M (\gamma t)^{1/2} \exp[  -\delta_{\omega}^2 \nu_2 \gamma t],
\end{array}
\end{equation}
which suffices to establish \sref{eq:lem:oblq:plt:bndWMRhoLarge}.
\end{proof}

We remark that the bounds \sref{eq:oblq:plt:bndVk}
and \sref{eq:lem:oblq:plt:bndWM}
all involve the quantity $v_{l; \gamma}(t)$.
We augment these results with
the following uniform estimates.
\begin{lem}
\label{lem:oblq:plt:unifBnds}
For all $\gamma \ge 1$, $t \ge 1$ and $l \in \Wholes$,
we have the uniform bounds
\begin{equation}
\label{eq:lem:oblq:plt:unifBnds}
\begin{array}{lcl}
\abs{ v_{\gamma;l}(t) } & \le &
  \sqrt{ \frac{\pi}{\nu_2} },
\\[0.2cm]
\abs{ \rho  v_{\gamma; l}(t) } & \le &
  \sqrt{ \frac{\pi}{\nu_2} }
  e^{-1/2} \sqrt{ \frac{1}{2 \nu_2 \gamma t} },
\\[0.2cm]
\abs{ \rho^2 v_{\gamma; l}(t)} & \le &
  \sqrt{ \frac{\pi}{\nu_2} }
  e^{-1} \frac{1}{\nu_2 \gamma t},
\\[0.2cm]
\abs{ \rho^3 v_{\gamma; l}(t) }& \le &
  \sqrt{ \frac{\pi}{\nu_2} }
  e^{-3/2}  \big( \frac{3}{2 \nu_2 \gamma t} \big)^{3/2},
\\[0.2cm]
\end{array}
\end{equation}
again with $\rho= \rho(l, t; \gamma)$.
\end{lem}
\begin{proof}
Observe first that for all $x \in \Real$ and $\alpha > 0$ we have the bounds
\begin{equation}
\abs{ x } \exp[ -\alpha x^2  ]  \le e^{-1/2} \sqrt{\frac{1}{2 \alpha}},
\qquad
 x^2 \exp[ -\alpha x^2  ]  \le \alpha^{-1} e^{-1},
\qquad
\abs{ x^3 } \exp[ - \alpha x^2 ]  \le \big( \frac{3}{2 \alpha} \big)^{3/2} e^{ - 3/2 }.
\end{equation}
Using the expression
\begin{equation}
v_{\gamma; l }(t) = \sqrt{\frac{\pi}{\nu_2}} \exp[ - \nu_2 \rho^2 \gamma t],
\end{equation}
the desired estimates \sref{eq:lem:oblq:plt:unifBnds} now follow immediately.
\end{proof}

Taking a time derivative
in \sref{eq:oblq:plt:defPlateauV}
yields the expression
\begin{equation}
\dot{v}_{l;\gamma}(t) =
(\gamma t)^{1/2} \int_{-\infty}^\infty
\big[\frac{1}{2} t^{-1} + i \omega \nu_1 - \gamma \nu_2 \omega^2 ] \exp[ i \omega \big( l  + \nu_1 t \big)  ]
  \exp[ -  \nu_2 \omega^2 \gamma t ] \,  d \omega,
\end{equation}
which in view of
the identities \sref{eq:oblq:plt:defVK}
can be written as
\begin{equation}
\label{eq:oblq:plt:cmpDotV}
\begin{array}{lcl}
\dot{v}_{\gamma;l}(t) & = &
\big[
\frac{1}{2}t^{-1}  - \nu_1 \rho(l, t; \gamma) + \gamma \nu_2 \rho^2(l, t; \gamma)
- \frac{1}{2} \gamma (\gamma t)^{-1}
\big]
v_{l; \gamma}(t)
\\[0.2cm]
& = &
\big[
  - \nu_1 \rho + \gamma \nu_2 \rho^2
\big]
v_{l; \gamma}(t),
\end{array}
\end{equation}
again with $\rho = \rho(l, t ; \gamma)$.
For convenience, we note
that Lemma \ref{lem:oblq:plt:unifBnds} implies the uniform bound
\begin{equation}
\label{eq:hom:oblq:unifBNdDotV}
\abs{\dot{v}_{l; \gamma}(t) } \le
\nu_1 \sqrt{\frac{\pi}{\nu_2}} e^{-1/2} \sqrt{\frac{1}{2 \nu_2 \gamma t} }
+ \frac{1}{t} e^{-1}\sqrt{\frac{\pi}{\nu_2}}
\end{equation}
for all $\gamma \ge 1$, $t \ge 1$ and $l \in \Wholes$.

For any sequence $v \in \ell^\infty(\Wholes; \Real)$,
let us define the quantity
\begin{equation}
\label{eq:oblq:plt:defCalK}
[\mathcal{K} v]_l :=
  \alpha^\diamond_\mu \pi^\diamond_{l ; \mu} v
  + \alpha^{\diamond\diamond}_{\mu \mu'} \pi^{\diamond\diamond}_{l; \mu \mu'} v.
\end{equation}
A short computation using Lemma \ref{lem:oblq:specProps}
shows that for some $\kappa_3 \in \Real$, we have
\begin{equation}
\begin{array}{lcl}
[\mathcal{K} v_{; \gamma}(t)]_l
&=&
(\gamma t)^{1/2} \int_{-\infty}^\infty
\big[
\alpha^\diamond_\mu (e^{ i \sigma_\mu \omega} - 1)
+ \alpha^{\diamond\diamond}_{\mu \mu'} (e^{i \sigma_\mu \omega} -1) (e^{i \sigma_{\mu'} \omega} - 1 )
\big]
\\[0.2cm]
& & \qquad \qquad \qquad \times
\exp[ i \omega ( l  + \nu_1 t )  ]
  \exp[ -  \nu_2 \omega^2 \gamma t ] \,  d \omega
\\[0.2cm]
& = &
(\gamma t)^{1/2} \int_{-\infty}^\infty
\big[ i \nu_1 \omega - \nu_2 \omega^2 + i \kappa_3 \omega^3 +  \mathcal{O}_4(\omega)
\big]
\exp[ i \omega \big( l  + \nu_1 t \big)  ]
  \exp[ -  \nu_2 \omega^2 \gamma t ] \,  d \omega.
\end{array}
\end{equation}
Here for any $s \in \{0, 2 , 4\}$, we have introduced the notation $\mathcal{O}_s(\omega)$
to denote a function that satisfies the conditions
\sref{eq:lem:oblq:plt:bndOnW:absBoundWP} - \sref{eq:lem:oblq:plt:bndOnW:absBoundWPZero}
for some $M > 1$.

An important role in the sequel will be reserved for the quantity
\begin{equation}
\label{eq:oblq:plt:defCalS}
\mathcal{S}_{l; \gamma}(t) = \dot{v}_{l; \gamma}(t) - \frac{1}{2} (\gamma t)^{-1} v_{l; \gamma}(t) - [\mathcal{K} v_{; \gamma}(t)]_l,
\end{equation}
for which we can compute
\begin{equation}
\label{eq:oblq:plt:frmCalS}
\begin{array}{lcl}
\mathcal{S}_{l; \gamma}(t)
& = & (\gamma t)^{1/2} \int_{-\infty}^\infty
\big[
  \frac{1}{2} t^{-1} + i \omega \nu_1 - \gamma \nu_2 \omega^2
 - \frac{1}{2} (\gamma t)^{-1} -i \nu_1 \omega + \nu_2 \omega^2 - i \kappa_3 \omega^3 - \mathcal{O}_4(\omega)
\big]
\\[0.2cm]
& & \qquad \qquad \qquad \times
 \exp[ i \omega ( l  + \nu_1 t )  ]
  \exp[ -  \nu_2 \omega^2 \gamma t ] \,  d \omega
\\[0.2cm]
& = &
(\gamma t)^{1/2} \int_{-\infty}^\infty
\big[\frac{1}{2}(\gamma - 1) (\gamma t)^{-1} - (\gamma - 1) \nu_2 \omega^2 - i \kappa_3 \omega^3
 - \mathcal{O}_4(\omega)
\big]
\\[0.2cm]
& & \qquad \qquad \qquad \times
 \exp[ i \omega ( l  + \nu_1 t )  ]
  \exp[ -  \nu_2 \omega^2 \gamma t ] \,  d \omega.
\end{array}
\end{equation}

We conclude this subsection by obtaining near field ($\abs{\rho} \le \delta_{\omega}$)
and far field ($\abs{\rho} \ge \delta_{\omega}$) bounds
on the various quantities introduced here,
which will allow us to obtain useful bounds
on the numbered expressions in \sref{eq:lem:oblq:defJminusFinal}.

\begin{lem}
\label{lem:hom:oblq:bndOnCalS}
Pick $\delta_{\omega} > 0$ sufficiently small.
There exists a constant $C_2 = C_2(\delta_{\omega}) > 1$ such that for all $\gamma \ge 1$,
$t \ge 1$ and all $l \in \Wholes$ for which $\abs{\rho(l, t; \gamma)} \le \delta_{\omega}$,
we have the bound
\begin{equation}
\begin{array}{lcl}
\abs{ \mathcal{S}_{l; \gamma}(t) - (\gamma - 1) \nu_2 \rho^2 v_{l; \gamma}(t) }
&\le&
 C_2 \big[  \abs{\rho}^3  + \abs{\rho} (\gamma t)^{-1} + (\gamma t)^{-2}\big] v_{l; \gamma}(t)
\\[0.2cm]
& &
\qquad +C_2 (\gamma t)^{1/2} e^{ - \nu_2 \delta_{\omega}^2 \gamma t },
\end{array}
\end{equation}
in which $\rho = \rho(l, \gamma ;t )$.
\end{lem}
\begin{proof}
This estimate follows directly from \sref{eq:oblq:plt:frmCalS},
using the identities \sref{eq:oblq:plt:defVK}
and the bounds \sref{eq:lem:oblq:plt:bndWM}.
\end{proof}

\begin{lem}
\label{lem:hom:sub:obl:bndsVDiamonds}
Pick $\delta_{\omega} > 0$ sufficiently small.
There exists a constant $C_2 = C_2(\delta_{\omega}) > 1$ such that for all $\gamma \ge 1$,
$t \ge 1$ and all $l \in \Wholes$ for which $\abs{\rho(l, t; \gamma)} \le \delta_{\omega}$,
we have the bounds
\begin{equation}
\begin{array}{lcl}
\abs{ \pi^\diamond_l v_{;\gamma}(t) }
& \le &
C_2 \big[ \abs{ \rho } + ( \gamma t)^{-1} \big]  v_{l; \gamma}(t)
+ C_2 (\gamma t)^{1/2} e^{ - \nu_2 \delta_{\omega}^2 \gamma t },
\\[0.2cm]
\abs{ \pi^{\diamond\diamond}_l v_{;\gamma}(t) }
& \le &
C_2 \big[ \rho^2 + ( \gamma t)^{-1} \big] v_{l; \gamma}(t)
 + C_2 (\gamma t)^{1/2} e^{ - \nu_2 \delta_{\omega}^2 \gamma t },
\\[0.2cm]
\abs{ \pi^{\diamond \diamond\diamond}_l v_{;\gamma}(t) }
& \le &
C_2 \big[ \abs{\rho}^3 +  \abs{\rho} ( \gamma t)^{-1}  + (\gamma t)^{-2} \big]  v_{l; \gamma}(t)
 + C_2 (\gamma t)^{1/2} e^{ - \nu_2 \delta_{\omega}^2 \gamma t },
\\[0.2cm]
\end{array}
\end{equation}
in which $\rho = \rho(l, \gamma ;t )$.
\end{lem}
\begin{proof}
Observing that
\begin{equation}
\begin{array}{lcl}
\pi^\diamond_{l;\mu} v_{;\gamma}(t)
& = &
(\gamma t)^{1/2} \int_{-\infty}^\infty
 (e^{ i \sigma_\mu \omega} - 1)
\exp[ i \omega ( l  + \nu_1 t )  ]
  \exp[ -  \nu_2 \omega^2 \gamma t ] \,  d \omega
\\[0.2cm]
& = &
(\gamma t)^{1/2} \int_{-\infty}^\infty
 (i \sigma_\mu \omega + \mathcal{O}_2(\omega) )
\exp[ i \omega ( l  + \nu_1 t )  ]
  \exp[ -  \nu_2 \omega^2 \gamma t ] \,  d \omega,
\end{array}
\end{equation}
the identities \sref{eq:oblq:plt:defVK}
and the bounds \sref{eq:lem:oblq:plt:bndWM}
suffice to obtain the desired estimate on $\pi^{\diamond}_{l} v_{; \gamma}(t)$.
The other estimates can be obtained in a similar fashion.
\end{proof}

\begin{lem}
\label{lem:hom:sub:obl:bndsDotTheta:light}
Pick $\delta_{\omega} > 0$ sufficiently small.
There exists a constant $C_2 = C_2(\delta_{\omega}) > 1$ such that for all $\gamma \ge 1$,
$t \ge 1$ and all $l \in \Wholes$ for which $\abs{\rho(l, t; \gamma)} \le \delta_{\omega}$,
we have the bounds
\begin{equation}
\begin{array}{lcl}
\abs{ \dot{v}_{l; \gamma}(t) } & \le &
C_2 [ \rho + \gamma \rho^2 ]   v_{l; \gamma}(t),
\\[0.2cm]
%
\abs{ \pi^\diamond_l
  \dot{v}_{;\gamma}(t)  }
& \le &
C_2 \big[ \gamma (\gamma t)^{-1}  \abs{\rho} +  \rho^2 + (\gamma t)^{-1}  +
\gamma  \abs{ \rho }^3    \big]
   v_{l; \gamma}(t)
\\[0.2cm]
& &
\qquad + \gamma C_2 (\gamma t)^{1/2} e^{ - \nu_2 \delta_{\omega}^2 \gamma t },
\\[0.2cm]
\abs{ \pi^{\diamond\diamond}_l \dot{v}_{;\gamma}(t) }
& \le &
C_2 \big[ (\gamma t)^{-1} \gamma \rho^2 + \gamma (\gamma t)^{-2} + \abs{\rho}^3 + (\gamma t)^{-1} \abs{\rho} + \gamma \rho^4 \big]
  v_{l; \gamma}(t)
\\[0.2cm]
& &
\qquad + \gamma C_2 (\gamma t)^{1/2} e^{ - \nu_2 \delta_{\omega}^2 \gamma t },
\\[0.2cm]
\end{array}
\end{equation}
in which $\rho = \rho(l, \gamma ;t )$.
\end{lem}
\begin{proof}
For convenience, we assume throughout this proof that $\rho \ge 0$.
The bound for $\dot{v}_{l; \gamma}(t)$ follows directly from \sref{eq:oblq:plt:cmpDotV}.
Moving on to the second estimate,
we write
\begin{equation}
\begin{array}{lcl}
\pi^{\diamond}_{l; \mu }\dot{v}_{; \gamma}(t)
& = &
(\gamma t)^{1/2}
 \int_{-\infty}^{\infty}
 \big[
   \big( i \sigma_{\mu} \omega + \mathcal{O}_2(\omega)  \big)
   \big(
     \frac{1}{2} t^{-1} +i \omega \nu_1  - \gamma \nu_2 \omega^2
   \big)
 \big]
\\[0.2cm]
& & \qquad \qquad \qquad \times
 \exp[i \omega(l + \nu_1 t)] \exp[ - \nu_2 \omega^2 \gamma t] \, d\omega
\\[0.2cm]
& = &
(\gamma t)^{1/2}
 \int_{-\infty}^{\infty}
 \big[
      \frac{1}{2} t^{-1} i \sigma_\mu \omega
      + \frac{1}{2} t^{-1} \mathcal{O}_2(\omega)
      + \mathcal{O}_2(\omega)  - i  \sigma_{\mu} \gamma \nu_2 \omega^3
      + \gamma \mathcal{O}_4(\omega)
 \big]
\\[0.2cm]
& & \qquad \qquad \qquad \times
 \exp[i \omega(l + \nu_1 t)] \exp[ - \nu_2 \omega^2 \gamma t] \, d\omega.
\\[0.2cm]
\end{array}
\end{equation}
Using \sref{eq:oblq:plt:defVK}
and \sref{eq:lem:oblq:plt:bndWM},
we hence see that there exists $C_2' > 1$ for which
\begin{equation}
\begin{array}{lcl}
\abs{\pi^{\diamond}_{l; \mu }\dot{v}_{; \gamma}(t)}
& \le &
C_2' [ t^{-1} ( \rho + \rho^2 + (\gamma t)^{-1} ) + \rho^2 + (\gamma t)^{-1}  +
\gamma  ( \rho^3 + \rho (\gamma t)^{-1} + \rho^4 +  (\gamma t)^{-2} ) \big]
\\[0.2cm]
& & \qquad +  \gamma C_2' (\gamma t)^{1/2} e^{ - \nu_2 \delta_{\omega}^2 \gamma t },
\end{array}
\end{equation}
from which the desired estimate immediately follows.

To obtain the third estimate, we note that
there exists $\kappa_3 \in \Real$
for which
\begin{equation}
\begin{array}{lcl}
\pi^{\diamond\diamond}_{l; \mu \mu'}\dot{v}_{; \gamma}(t)
& = &
(\gamma t)^{1/2}
 \int_{-\infty}^{\infty}
 \big[
   \big( - \sigma_{\mu} \sigma_{\mu'} \omega^2 + i \kappa_3 \omega^3 +  \mathcal{O}_4(\omega)  \big)
   \big(
     \frac{1}{2} t^{-1} +i \omega \nu_1  - \gamma \nu_2 \omega^2
   \big)
 \big]
\\[0.2cm]
& & \qquad \qquad \qquad \times
 \exp[i \omega(l + \nu_1 t)] \exp[ - \nu_2 \omega^2 \gamma t] \, d\omega
\\[0.2cm]
& = &
(\gamma t)^{1/2}
 \int_{-\infty}^{\infty}
 \big[
     \frac{1}{2} t^{-1} \mathcal{O}_2(\omega)
        - i \omega^3 \nu_1 \sigma_{\mu} \sigma_{\mu'} + \mathcal{O}_4(\omega)
          + \gamma \mathcal{O}_{4}(\omega)
 \big]
\\[0.2cm]
& & \qquad \qquad \qquad \times
 \exp[i \omega(l + \nu_1 t)] \exp[ - \nu_2 \omega^2 \gamma t] \, d\omega.
\\[0.2cm]
\end{array}
\end{equation}
Using \sref{eq:oblq:plt:defVK}
and \sref{eq:lem:oblq:plt:bndWM},
we hence see that there exists $C_2' > 1$ for which
\begin{equation}
\begin{array}{lcl}
\abs{\pi^{\diamond\diamond}_{l; \mu \mu'}\dot{v}_{; \gamma}(t)}
& \le &
C_2' \big[ t^{-1} (\rho^2 + (\gamma t)^{-1} )+ \rho^3 + (\gamma t)^{-1} \rho +
(\gamma +1 ) (\rho^4 + (\gamma t)^{-2} ) \big]
\\[0.2cm]
& & \qquad
 + \gamma C_2'  (\gamma t)^{1/2} e^{ - \nu_2 \delta_{\omega}^2 \gamma t },
\end{array}
\end{equation}
from which the desired estimate immediately follows.
\end{proof}

\begin{lem}
\label{lem:hom:sub:obl:bndsDotTheta}
Pick $\delta_{\omega} > 0$ sufficiently small.
There exists a constant $C_2 = C_2(\delta_{\omega}) > 1$ such that for all $\gamma \ge 1$,
$t \ge 1$ and all $l \in \Wholes$ for which $\abs{\rho(l, t; \gamma)} \le \delta_{\omega}$,
we have the bounds
\begin{equation}
\begin{array}{lcl}
\abs{\dot{v}_{;\gamma}(t) - \alpha^\diamond_\mu \pi^\diamond_{;\mu} v_{; \gamma}(t) }
& \le &
C_2 \big[ \rho^2 + (\gamma t)^{-1}  \big]  v_{l; \gamma}(t)
\\[0.2cm]
& &
\qquad + C_2 (\gamma t)^{1/2} e^{ - \nu_2 \delta_{\omega}^2 \gamma t },
\\[0.2cm]
\abs{ \pi^\diamond_l
  [\dot{v}_{l;\gamma}(t) - \alpha^\diamond_\mu \pi^\diamond_{l;\mu} v_{; \gamma}(t) ]
  }
& \le &
C_2 \big[ (\gamma t)^{-1} \gamma  \abs{\rho} + \gamma (\gamma t)^{-2} + \gamma \abs{\rho}^3 \big]
   v_{l; \gamma}(t)
\\[0.2cm]
& &
\qquad + \gamma C_2 (\gamma t)^{1/2} e^{ - \nu_2 \delta_{\omega}^2 \gamma t },
\\[0.2cm]
\end{array}
\end{equation}
in which $\rho = \rho(l, \gamma ;t )$.
\end{lem}
\begin{proof}
As before, we restrict ourselves to the setting $\rho \ge 0$.
First of all, we write
\begin{equation}
w_{l; \gamma}(t) = \dot{v}_{l;\gamma}(t) - \alpha^\diamond_\mu \pi^\diamond_{l;\mu} v_{; \gamma}(t)
\end{equation}
and note that there exist $\kappa_2, \kappa_3 \in \Real$
for which we can compute
\begin{equation}
\begin{array}{lcl}
w_{l; \gamma}(t)
& = &
(\gamma t)^{1/2}
 \int_{-\infty}^{\infty}
 \big[
   \frac{1}{2} t^{-1} + i \omega \nu_1 - \gamma \nu_2 \omega^2
   -\alpha^\diamond_{\mu}( e^{i \sigma_\mu \omega} - 1)
 \big]
\\[0.2cm]
& & \qquad \qquad \qquad \times
 \exp[i \omega(l + \nu_1 t)] \exp[ - \nu_2 \omega^2 \gamma t] \, d\omega
\\[0.2cm]
& = &
(\gamma t)^{1/2}
 \int_{-\infty}^{\infty}
 \big[
   \frac{1}{2} t^{-1} + i \omega \nu_1 - \gamma \nu_2 \omega^2
   - i \omega \nu_1 + \kappa_2 \omega^2 + i \kappa_3 \omega^3 +  \mathcal{O}(\omega)
 \big]
\\[0.2cm]
& & \qquad \qquad \qquad \times
 \exp[i \omega(l + \nu_1 t)] \exp[ - \nu_2 \omega^2 \gamma t] \, d\omega
\\[0.2cm]
& = &
(\gamma t)^{1/2}
 \int_{-\infty}^{\infty}
 \big[
   \frac{1}{2} t^{-1} - \gamma \nu_2 \omega^2
   + \kappa_2 \omega^2 + i \kappa_3 \omega^3 + \mathcal{O}_4(\omega)
 \big]
\\[0.2cm]
& & \qquad \qquad \qquad \times
 \exp[i \omega(l + \nu_1 t)] \exp[ - \nu_2 \omega^2 \gamma t] \, d\omega.
\end{array}
\end{equation}
In particular,
invoking \sref{eq:oblq:plt:defVK}
and \sref{eq:lem:oblq:plt:bndWM}
and exploiting a partial cancellation in the first two terms,
we see that there exists $C_2' > 1$ for which
\begin{equation}
\begin{array}{lcl}
\abs{ w_{l ; \gamma}(t) }
& \le &
C_2' \big[  \rho^2 + (\gamma t)^{-1}  + \rho^3 +
  \rho (\gamma t)^{-1} + \rho^4 +  (\gamma t)^{-2}  \big] v_{l; \gamma}(t)
\\[0.2cm]
& & \qquad
+ C_2' (\gamma t)^{1/2} e^{ - \nu_2 \delta_{\omega}^2 \gamma t },
\end{array}
\end{equation}
which suffices to obtain the first stated estimate.

Moving on to the second estimate,
we compute
\begin{equation}
\begin{array}{lcl}
\pi^\diamond_{l; \mu}w_{; \gamma}(t)
& = &
(\gamma t)^{1/2}
 \int_{-\infty}^{\infty}
 \big[
   \big( i \sigma_{\mu} \omega + \mathcal{O}_2(\omega)  \big)
   \big( \frac{1}{2} t^{-1}  - \gamma \nu_2 \omega^2
    + \kappa_2 \omega^2 + i \kappa_3 \omega^3 +  \mathcal{O}_4(\omega) \big)
 \big]
\\[0.2cm]
& & \qquad \qquad \qquad \times
 \exp[i \omega(l + \nu_1 t)] \exp[ - \nu_2 \omega^2 \gamma t] \, d\omega
\\[0.2cm]
& = &
(\gamma t)^{1/2}
 \int_{-\infty}^{\infty}
 \big[
  \frac{1}{2} t^{-1} \big( i \sigma_{\mu} \omega + \mathcal{O}_2(\omega) \big)
  - \gamma i \sigma_{\mu} \nu_2 \omega^3  + i \sigma_\mu \kappa_2 \omega^3
  + (1 + \gamma) \mathcal{O}_4(\omega)
 \big]
\\[0.2cm]
& & \qquad \qquad \qquad \times
 \exp[i \omega(l + \nu_1 t)] \exp[ - \nu_2 \omega^2 \gamma t] \, d\omega.
\end{array}
\end{equation}
As before, there exists $C_2' > 1$ for which
\begin{equation}
\begin{array}{lcl}
\abs{ \pi^\diamond_{l; \mu}w_{; \gamma}(t) }
& \le &
C_2' \big[t^{-1} \big( \rho + (\gamma t)^{-1} \big)
  + (\gamma + 1) \big(\rho^3 +  \rho (\gamma t)^{-1} + (\gamma t)^{-2} \big) \big] v_{l ; \gamma}(t)
\\[0.2cm]
 & & \qquad
 + \gamma C_2' (\gamma t)^{1/2} e^{ - \nu_2 \delta_{\omega}^2 \gamma t },
\end{array}
\end{equation}
which suffices to complete the proof.
\end{proof}
\begin{lem}
\label{lem:oblq:plt:farfieldBnds}
Pick a sufficiently small $\delta_{\rho} > 0$.
There exists a constant
$C_3 = C_3(\delta_{\rho}) > 1$
so that for every $\gamma \ge 1$,
every $t \ge 1$ and every $l \in \Wholes$
for which $\abs{\rho(l, t; \gamma)} \ge \delta_\rho$,
we have the bounds
\begin{equation}
\begin{array}{lcl}
\abs{ v_{\gamma ; l}(t) }
+\abs{ \pi^{\diamond}_l v_{; \gamma}(t) }
+ \abs{ \pi^{\diamond\diamond}_l v_{; \gamma}(t) }
+\abs{ \pi^{\diamond\diamond\diamond}_l v_{; \gamma}(t) }
& \le &
 C_3 (\gamma t)^{1/2} \exp[ - \nu_2 \delta_\rho^2 \gamma t],
\\[0.2cm]
\abs{ \dot{v}_{\gamma;l}(t) }
+ \abs{ \pi^{\diamond}_l \dot{v}_{; \gamma}(t) }
+ \abs{ \pi^{\diamond\diamond}_l \dot{v}_{; \gamma}(t) }
& \le & \gamma  C_3 (\gamma t)^{1/2} \exp[ - \nu_2 \delta_\rho^2 \gamma t],
\\[0.2cm]
\abs{\mathcal{S}_{l; \gamma}(t) }
& \le & \gamma  C_3 (\gamma t)^{1/2} \exp[ - \nu_2 \delta_\rho^2 \gamma t].
\end{array}
\end{equation}
\end{lem}
\begin{proof}
These estimates follow directly from \sref{eq:lem:oblq:plt:bndWMRhoLarge}
together with the identities
\sref{eq:oblq:plt:cmpDotV} and \sref{eq:oblq:plt:frmCalS}.
\end{proof}

\subsection{The function $\theta$}
In this subsection we scale the plateau function $v_{; \gamma}$
defined in \sref{eq:oblq:plt:defPlateauV} by a
global factor that decays very slowly in time.
The resulting function $\theta$ controls
the phase-shifts of the sub-solution $u^-$ in the direction transverse
to the wave propagation.

Throughout this subsection, we write
\begin{equation}
\alpha = \alpha(\gamma) = \frac{1}{4 \gamma}.
\end{equation}
For every $\beta \ge 1$ and $\gamma \ge 1$, we
now define the $C^1$-smooth function
\begin{equation}
\theta_{;\beta, \gamma}: [1, \infty) \to \ell^\infty(\Wholes; \Real)
\end{equation}
that acts as
\begin{equation}
\label{eq:oblq:theta:defTheta}
\theta_{l; \beta,  \gamma}(t) = \beta t^{- \alpha} v_{l; \gamma}(t)
\end{equation}
and can be differentiated as
\begin{equation}
\label{eq:hom:sub:obl:compDotTheta}
\begin{array}{lcl}
\dot{\theta}_{l; \beta, \gamma}(t)
& = &
 \beta t^{- \alpha} \big[ - \frac{1}{4} (\gamma t)^{-1} v_{l; \gamma}(t) + \dot{v}_{l; \gamma}(t) \big].
\end{array}
\end{equation}

An important role in the sequel will be played by the quantities
\begin{equation}
\begin{array}{lcl}
\mathcal{T}_{l; \beta, \gamma }(t)
& = &  \dot{\theta}_{l; \beta, \gamma}(t)
 - \alpha^\diamond_\mu \pi^\diamond_{l;\mu} \theta_{; \beta, \gamma}(t)
 - \alpha^{\diamond\diamond}_{\mu\mu'} \pi^{\diamond\diamond}_{l;\mu\mu'} \theta_{; \beta, \gamma}(t).
\\[0.2cm]
\end{array}
\end{equation}
Recalling the linear operator $\mathcal{K}$
defined in \sref{eq:oblq:plt:defCalK},
a short computation shows that
\begin{equation}
\begin{array}{lcl}
\mathcal{T}_{l; \beta, \gamma }(t)
& = &
\beta t^{- \alpha}
\big[
  \dot{v}_{l; \gamma}(t) - \frac{1}{4}(\gamma t)^{-1} v_{l; \gamma} - [\mathcal{K}v_{; \gamma}(t)]_l
\big]
\\[0.2cm]
& = &
\beta t^{- \alpha}
\big[
  \frac{1}{4}(\gamma t)^{-1} v_{l; \gamma}(t) + \mathcal{S}_{l; \gamma}(t)
\big].
\end{array}
\end{equation}
In order to obtain a useful bound
on $\mathcal{T}_{l; \beta, \gamma}(t)$,
we introduce the strictly positive expressions
\begin{equation}
\label{eq:oblq:theta:defCalQ}
\mathcal{Q}_{l; \beta, \gamma}(t)
=   \frac{1}{8}\beta t^{- \alpha}
  [\nu_2 \gamma \rho^2 + (\gamma t)^{-1} ] v_{l; \gamma}(t),
\end{equation}
with $\rho = \rho(l, t; \gamma)$.
\begin{lem}
\label{lem:oblq:theta:locBndCalT}
Pick a sufficiently small $\delta_{\omega} > 0$. There exists constants
$C_{\mathcal{T}} = C_{\mathcal{T}}(\delta_\omega) > 0$ and $\gamma_* =\gamma_*(\delta_\omega) \ge 1$
such that for all $\gamma \ge \gamma_*$, all $\beta \ge 1$,
all $t \ge 1$
and all $l \in \Wholes$ for which $\abs{\rho(l,t;\gamma)} \le \delta_{\omega}$,
we have
\begin{equation}
\mathcal{T}_{l; \beta, \gamma}(t) \ge
 \mathcal{Q}_{l; \beta, \gamma}(t)
   - \beta C_{\mathcal{T}}
   (\gamma t)^{1/2} e^{ - \nu_2 \delta_{\omega}^2 \gamma t}.
\end{equation}
\end{lem}
\begin{proof}
Notice first that $\gamma_* \ge 2 $ ensures $\gamma_* -1 \ge \frac{1}{2} \gamma_*$.
In addition, using the elementary estimate
\begin{equation}
\begin{array}{lcl}
\rho^3  + \rho (\gamma t)^{-1} + (\gamma t)^{-2}
& \le & \rho^3 + \rho^2 + 2 ( \gamma t)^{-2}
\\[0.2cm]
&=&
\nu_2 \gamma \rho^2 [  \nu_2^{-1} \gamma^{-1} \rho +  \nu_2^{-1} \gamma^{-1} ] + (\gamma t)^{-1} (2  \frac{1}{ \gamma t}),
\end{array}
\end{equation}
Lemma \ref{lem:hom:oblq:bndOnCalS}
implies that it suffices to pick $\gamma_* \ge 2$
in such a way that
\begin{equation}
C_2 \big[ \nu_2^{-1} \gamma_*^{-1} \delta_{\omega} +  \nu_2^{-1} \gamma_*^{-1} \big] \le \frac{1}{8},
\qquad 2 C_2  \gamma_*^{-1} \le \frac{1}{8}.
\end{equation}
\end{proof}

In the remainder of this subsection we obtain near field ($\abs{\rho} \le \delta_{\omega}$),
far field ($\abs{\rho} \ge \delta_{\omega}$)
and global bounds
on various terms involving $\theta_{;\beta, \gamma}(t)$,
which will allow us to obtain useful bounds
on the numbered expressions in \sref{eq:lem:oblq:defJminusFinal}.

\begin{lem}
\label{lem:hom:sub:obl:bndsThetaDiamonds}
Pick $\delta_{\omega} > 0$ sufficiently small.
There exists a constant $C_4 = C_4(\delta_{\omega}) > 1$ such that for all $\beta \ge 1$,
all $\gamma \ge 1$, all $t \ge 1$ and all $l \in \Wholes$
for which  $\abs{\rho(l, t; \gamma)} \le \delta_{\omega}$,
we have the bounds
\begin{equation}
\begin{array}{lcl}
\abs{ \pi^\diamond_l \theta_{;\beta,\gamma}(t) }
& \le &
\beta C_4 t^{-\alpha} \big[ \abs{ \rho } + ( \gamma t)^{-1} \big]  v_{l; \gamma}(t)
\\[0.2cm]
& &
\qquad + \beta C_4  (\gamma t)^{1/2} e^{ - \nu_2 \delta_{\omega}^2 \gamma t },
\\[0.2cm]
\abs{ \pi^{\diamond\diamond}_l \theta_{;\beta,\gamma}(t) }
& \le &
\beta C_4 t^{-\alpha} \big[ \rho^2 + ( \gamma t)^{-1} \big] v_{l; \gamma}(t)
\\[0.2cm]
& &
\qquad + \beta C_4 (\gamma t)^{1/2} e^{ - \nu_2 \delta_{\omega}^2 \gamma t },
\\[0.2cm]
\abs{ \pi^{\diamond \diamond\diamond}_l \theta_{;\beta,\gamma}(t) }
& \le &
\beta C_4 t^{-\alpha}
\big[ \abs{\rho}^3 + \abs{\rho} ( \gamma t)^{-1}  + (\gamma t)^{-2} \big] v_{l; \gamma}(t)
\\[0.2cm]
& &
\qquad + \beta C_4 (\gamma t)^{1/2} e^{ - \nu_2 \delta_{\omega}^2 \gamma t },
\\[0.2cm]
\end{array}
\end{equation}
with $\rho = \rho(l, t; \gamma)$.
\end{lem}
\begin{proof}
In view of the identity
\sref{eq:oblq:theta:defTheta},
these bounds follow immediately
from Lemma \ref{lem:hom:sub:obl:bndsVDiamonds}.
\end{proof}
\begin{lem}
\label{lem:oblq:theta:nearFieldDotTheta:light}
Pick $\delta_{\omega} > 0$ sufficiently small.
There exists a constant $C_4 = C_4(\delta_{\omega}) > 1$ such that for all $\beta \ge 1$,
all $\gamma \ge 1$, all $t \ge 1$ and all $l \in \Wholes$
for which  $\abs{\rho(l, t; \gamma)} \le \delta_{\omega}$,
we have the bounds
\begin{equation}
\begin{array}{lcl}
\abs{ \dot{\theta}_{l;\beta,  \gamma}(t) } & \le &
\beta C_4 t^{- \alpha}
  \big[ \rho + \gamma \rho^2  + (\gamma t)^{-1} \big]
  v_{l; \gamma}(t),
\\[0.2cm]
\abs{ \pi^\diamond_l \dot{\theta}_{; \beta, \gamma}(t) }
& \le &
\beta C_4 t^{- \alpha}
\big[ \gamma (\gamma t)^{-1} \abs{ \rho } + \rho^2 + (\gamma t)^{-1} + \gamma \abs{\rho}^3 \big]
v_{l; \gamma}(t)
\\[0.2cm]
& & \qquad
  + \beta \gamma C_4 (\gamma t)^{1/2} e^{ - \nu_2 \delta_{\omega}^2 \gamma t },
\\[0.2cm]
\abs{ \pi^{\diamond\diamond}_l \dot{\theta}_{;\beta, \gamma}(t) }
& \le &
\beta C_4 t^{- \alpha}
  \big[ (\gamma t)^{-1} \gamma \rho^2 + \gamma (\gamma t)^{-2} + \abs{\rho}^3
     + (\gamma t)^{-1} \abs{ \rho } + \gamma \rho^4 \big]
  v_{l; \gamma}(t)
\\[0.2cm]
& & \qquad
  + \beta \gamma C_4 (\gamma t)^{1/2} e^{ - \nu_2 \delta_{\omega}^2 \gamma t },
\\[0.2cm]
\end{array}
\end{equation}
with $\rho = \rho(l,t ; \gamma)$.
\end{lem}
\begin{proof}
Combining the results from Lemmas
\ref{lem:hom:sub:obl:bndsVDiamonds} and \ref{lem:hom:sub:obl:bndsDotTheta:light},
these estimates follow directly from \sref{eq:hom:sub:obl:compDotTheta}.
\end{proof}
\begin{lem}
\label{lem:oblq:theta:nearFieldDotTheta}
Pick $\delta_{\omega} > 0$ sufficiently small.
There exists a constant $C_4 = C_4(\delta_{\omega}) > 1$ such that for all $\beta \ge 1$,
all $\gamma \ge 1$, all $t \ge 1$ and all $l \in \Wholes$
for which  $\abs{\rho(l, t; \gamma)} \le \delta_{\omega}$,
we have the bounds
\begin{equation}
\begin{array}{lcl}
\abs{   \dot{\theta}_{l;\beta, \gamma}(t)
    - \alpha^\diamond_\mu \pi^\diamond_{;\mu} \theta_{;\beta,  \gamma}(t)
  }
& \le &
\beta C_4 t^{- \alpha}
  \big[ \rho^2 + (\gamma t)^{-1} \big]
  v_{l; \gamma}(t)
\\[0.2cm]
& & \qquad
  + \beta C_4 (\gamma t)^{1/2} e^{ - \nu_2 \delta_{\omega}^2 \gamma t },
\\[0.2cm]
\abs{ \pi^\diamond_l
  [\dot{\theta}_{;\beta, \gamma, t_*}(t)
    - \alpha^\diamond_\mu \pi^\diamond_{;\mu} \theta_{;\beta,  \gamma}(t) ]
  }
& \le &
\beta C_4 t^{- \alpha}
\big[ (\gamma t)^{-1} \gamma \abs{\rho}
  + \gamma (\gamma t)^{-2} + \gamma \abs{\rho}^3 \big]
  v_{l; \gamma}(t)
\\[0.2cm]
& & \qquad
  + \beta \gamma C_4 (\gamma t)^{1/2} e^{ - \nu_2 \delta_{\omega}^2 \gamma t },
\\[0.2cm]
\end{array}
\end{equation}
with $\rho = \rho(l,t;\gamma)$.
\end{lem}
\begin{proof}
Combining the results from Lemmas
\ref{lem:hom:sub:obl:bndsVDiamonds} and \ref{lem:hom:sub:obl:bndsDotTheta},
these estimates follow directly from \sref{eq:hom:sub:obl:compDotTheta}.
\end{proof}
\begin{lem}
\label{lem:oblq:theta:farfieldBnds}
Pick a sufficiently small $\delta_{\rho} > 0$.
There exists a constant
$C_5 = C_5(\delta_{\rho}) > 1$
so that for all $\beta \ge 1$,
all $\gamma \ge 1$,
all $t \ge 1$ and all $l \in \Wholes$
for which $\abs{\rho(l, t; \gamma)} \ge \delta_\rho$,
we have the bounds
\begin{equation}
\begin{array}{lcl}
\abs{ \theta_{l; \beta, \gamma }(t) }
+\abs{ \pi^{\diamond}_l \theta_{; \beta,\gamma}(t) }
+ \abs{ \pi^{\diamond\diamond}_l \theta_{; \beta, \gamma}(t) }
+\abs{ \pi^{\diamond\diamond\diamond}_l \theta_{;\beta, \gamma}(t) }
& \le &
 \beta C_5 t^{-\alpha} (\gamma t)^{1/2} \exp[ - \nu_2 \delta_\rho^2 \gamma t],
\\[0.2cm]
\abs{ \dot{\theta}_{\gamma;l}(t) }
+ \abs{ \pi^{\diamond}_l \dot{\theta}_{; \beta,\gamma}(t) }
+ \abs{ \pi^{\diamond\diamond}_l \dot{\theta}_{;\beta, \gamma}(t) }
& \le &  \beta \gamma C_5 t^{-\alpha} (\gamma t)^{1/2} \exp[ - \nu_2 \delta_\rho^2 \gamma t],
\\[0.2cm]
\abs{\mathcal{T}_{l; \beta, \gamma}(t) }
& \le &  \beta  \gamma C_5 t^{-\alpha} (\gamma t)^{1/2} \exp[ - \nu_2 \delta_\rho^2 \gamma t].
\end{array}
\end{equation}
\end{lem}
\begin{proof}
Using Lemma \ref{lem:oblq:plt:farfieldBnds},
these estimates follow directly
from \sref{eq:oblq:theta:defTheta} and  \sref{eq:hom:sub:obl:compDotTheta}.
\end{proof}
\begin{lem}
\label{eq:oblq:theta:glbFndsPiThetas}
There exist
constants $C_6 > 1$ and $\gamma_* \ge 1$ such that for all $\gamma \ge \gamma_*$,
all $\beta \ge 1$, all $t \ge 1$ and all $l \in \Wholes$, we have the bounds
\begin{equation}
\begin{array}{lcl}
\abs{ \pi^\diamond_l \theta_{;\beta,\gamma}(t) }
& \le &
\beta C_6 t^{-\alpha}  (\gamma t)^{-1/2},
\\[0.2cm]
\abs{ \pi^{\diamond\diamond}_l \theta_{;\beta,\gamma}(t) }
& \le &
\beta C_6 t^{-\alpha}  ( \gamma t)^{-1},
\\[0.2cm]
\abs{ \pi^{\diamond \diamond\diamond}_l \theta_{;\beta,\gamma}(t) }
& \le &
\beta C_6 t^{-\alpha}   (\gamma t)^{-3/2}.
\\[0.2cm]
\end{array}
\end{equation}
\end{lem}
\begin{proof}
Note first that for every fixed
$\delta_{\omega} > 0$,
we can obtain
\begin{equation}
\gamma (\gamma t)^{1/2} e^{ - \nu_2 \delta_{\omega}^2 \gamma t } \le (\gamma t)^{-3/2}
\end{equation}
for all $t \ge 1$ and all $\gamma \ge \gamma_*$ by picking
$\gamma_*$ to be sufficiently large.
The desired bounds
now follow from Lemma \ref{lem:hom:sub:obl:bndsThetaDiamonds}
and Lemma \ref{lem:oblq:theta:farfieldBnds},
upon using the global estimates
obtained in Lemma \ref{lem:oblq:plt:unifBnds}.
\end{proof}
\begin{lem}
\label{lem:oblq:theta:glbBndDotTheta}
There exist
constants $C_6 > 1$ and $\gamma_* \ge 1$ such that for all $\gamma \ge \gamma_*$,
all $\beta \ge 1$, all $t \ge 1$ and all $l \in \Wholes$, we have the bounds
\begin{equation}
\begin{array}{lcl}
\abs{ \dot{\theta}_{l;\beta,  \gamma}(t) } & \le &
\beta C_6t^{- \alpha} \big[ (\gamma t)^{-1/2} + \gamma (\gamma t)^{-1} \big],
\\[0.2cm]
\abs{ \pi^\diamond_l \dot{\theta}_{; \beta, \gamma}(t) }
& \le &
\beta C_6 t^{- \alpha}
\big[ \gamma (\gamma t)^{-3/2}  + (\gamma t)^{-1} \big],
\\[0.2cm]
\abs{ \pi^{\diamond\diamond}_l \dot{\theta}_{;\beta, \gamma*}(t) }
& \le &
\beta C_6 t^{- \alpha}
  \big[ \gamma (\gamma t)^{-2}  + (\gamma t)^{-3/2}    \big].
\\[0.2cm]
\end{array}
\end{equation}
\end{lem}
\begin{proof}
Arguing similarly as in the proof of Lemma \ref{eq:oblq:theta:glbFndsPiThetas},
these estimates follow from
Lemmas \ref{lem:oblq:theta:nearFieldDotTheta:light} and \ref{lem:oblq:theta:farfieldBnds}.
\end{proof}
\begin{lem}
\label{lem:oblq:theta:glbBndDotTheta:heavy}
There exist
constants $C_6 > 1$ and $\gamma_* \ge 1$ such that for all $\gamma \ge \gamma_*$,
all $\beta \ge 1$, all $t \ge 1$ and all $l \in \Wholes$, we have the bounds
\begin{equation}
\begin{array}{lcl}
\abs{   \dot{\theta}_{l;\beta, \gamma}(t)
    - \alpha^\diamond_\mu \pi^\diamond_{;\mu} \theta_{;\beta,  \gamma}(t)
  }
& \le &
\beta C_6 t^{- \alpha}
  (\gamma t)^{-1},
\\[0.2cm]
\abs{ \pi^\diamond_l
  \big[\dot{\theta}_{;\beta, \gamma}(t)
    - \alpha^\diamond_\mu \pi^\diamond_{;\mu} \theta_{;\beta,  \gamma}(t) \big]
  }
& \le &
\beta C_6 t^{- \alpha}
 \gamma (\gamma t)^{-3/2}.
\\[0.2cm]
\end{array}
\end{equation}
\end{lem}
\begin{proof}
Arguing similarly as in the proof of Lemma \ref{eq:oblq:theta:glbFndsPiThetas},
these estimates follow from
Lemmas \ref{lem:oblq:theta:nearFieldDotTheta} and \ref{lem:oblq:theta:farfieldBnds}.
\end{proof}

\subsection{Construction of sub-solution}
In this subsection we finally provide
the proof of Proposition \ref{prp:hom:oblq:mr:sub:sup}.
In fact, we set out prove the following result,
which is more closely related to the notation we have developed in this section.
\begin{prop}
\label{prp:hom:oblq:subSolnDetailed}
Consider any angle $\zeta_*$ with $\tan \zeta_* \in \mathbb{Q}$
and suppose that (Hg) and $(HS)_{\zeta_*}$
both hold. Pick $(\sigma_h, \sigma_v) \in \Wholes^2 \setminus \{(0 , 0)\}$
with the property that
\begin{equation}
  \sqrt{\sigma_h^2 + \sigma_v^2}(\cos \zeta_*, \sin \zeta_*) = (\sigma_h, \sigma_v),
  \qquad
  \mathrm{gcd}(\sigma_h, \sigma_v) = 1,
\end{equation}
suppose that $\textrm{(h}\Phi\textrm{)}_{\textrm{\S\ref{sec:prlm}}}$
holds for this pair $(\sigma_h, \sigma_v)$ with $c > 0$
and recall the setting of Lemma \ref{lem:oblq:anstz:obsv}.

Then there exist constants $\delta_{z} > 0$, $\eta_z > 0$,
$K_Z > 0$, $\eta_{\mathcal{N}} > 0$ and $K_{\mathcal{N}} > 0$
so that the following holds true.
Pick any $\epsilon_2 > 0$, any $\beta > 1$,
any $\Omega_\perp > 0$
and any $C^1$-smooth function
$z:[1, \infty) \to (0, \delta_{z}]$
that has
\begin{equation}
\dot{z}(t) \ge - \eta_z z(t)
\end{equation}
for all $t \ge 1,$ together with
\begin{equation}
\epsilon_3 := \inf_{t \ge 1} t^{3/2} z(t) > 0.
\end{equation}
Then
there exists $\gamma = \gamma( \epsilon_2, \epsilon_3, \beta, \Omega_\perp) \ge 1$
such that
the function  $u^-: [1, \infty) \to \ell^{\infty}(\Wholes^2; \Real)$
defined by
\begin{equation}
\begin{array}{lcl}
u^-(t) & = & \j(\Phi; t) + \j( \pi^\diamond_{; \mu} \theta, p^\diamond_\mu ; t)
+ \j ( \pi^{\diamond\diamond}_{\mu \mu'} \theta, p^{\diamond\diamond}_{\mu \mu'} ; t)
+ \j ( \pi^{\diamond}_{\mu} \theta, \pi^\diamond_{\mu'} \theta,  q^{\diamond\diamond}_{\mu \mu'} ; t)
- z(t)
\\[0.2cm]
\end{array}
\end{equation}
with
\begin{equation}
\begin{array}{lcl}
Z(t) & = & K_Z \int_{1}^t z(s) \, ds,
\\[0.2cm]
\xi_{nl}(t) & = &  n + ct - \theta_{l; \beta, \gamma}(t) - Z(t)
\end{array}
\end{equation}
satisfies the differential inequality
\begin{equation}
\label{eq:oblq:cnstr:diffInEq}
\mathcal{J}^-_{nl}(t) \le - \frac{1}{2} \eta_z z(t), \qquad (n,l) \in \Wholes^2, \qquad t \ge 1
\end{equation}
and admits the bound
\begin{equation}
\label{eq:prop:hom:oblq:subSolDiffFromWave}
\begin{array}{lcl}
\abs{ u^-(t) - \j(\Phi ; t) + z(t) }  & \le &  \epsilon_2 t^{-1/2}, \qquad t \ge 1.
\end{array}
\end{equation}
In addition, for every $(n,l) \in \Wholes^2$ and $t \ge 1$
we have the inequality
\begin{equation}
\label{eq:prop:hom:oblq:subSolPhaseDiff}
\abs{ \theta_{l + 1; \beta,\gamma}(t) - \theta_{l;\beta,\gamma}(t) } \le 1,
\end{equation}
which for $\abs{l} \le \Omega_\perp$ can be augmented by
\begin{equation}
\label{eq:prop:hom:oblq:subSolPhaseDrv}
\dot{\xi}_{nl}(t) \ge \frac{c}{2}.
\end{equation}
Finally, for any $\nu \in \{1, \ldots, 4\}$,
we have the bound
\begin{equation}
\label{eq:prop:hom:obql:subsol:det:bndOnLaplaceElms}
\abs{ [\pi^\times_{nl; \nu} - \pi^\times_{nl; 5}] u^-(t) } \le
 K_{\mathcal{N}} e^{-\eta_{\mathcal{N}} \abs{\xi_{nl}(t) } },
 \qquad (n,l) \in \Wholes^2, \qquad t \ge 1.
\end{equation}
\end{prop}

We start our analysis by looking at the
auxilliary estimates \sref{eq:prop:hom:oblq:subSolDiffFromWave},
\sref{eq:prop:hom:oblq:subSolPhaseDiff},
\sref{eq:prop:hom:oblq:subSolPhaseDrv}
and \sref{eq:prop:hom:obql:subsol:det:bndOnLaplaceElms}.
First of all,
Lemma \ref{lem:oblq:anstz:bndOnPDiams}
implies that there exists $C_1' > 1$ for which
\begin{equation}
\begin{array}{lcl}
\abs{ u^-_{nl}(t)  + z(t) - \j_{nl}(\Phi ; t) }
& \le &
\abs{\j_{nl}( \pi^\diamond_{; \mu} \theta , p^\diamond_\mu ; t) }
+ \abs{\j_{nl}( \pi^{\diamond\diamond}_{; \mu\mu'} \theta , p^\diamond_{\mu\mu'} ; t) }
+ \abs{\j_{nl}( \pi^{\diamond}_{\mu} \theta, \pi^{\diamond}_{\mu'} \theta, q^\diamond_{\mu\mu'} ; t) }
\\[0.2cm]
& \le & C_1'
\big[ \abs{\pi^\diamond_l \theta(t) } + \abs{\pi^{\diamond \diamond}_l \theta(t) }
 + \abs{ \pi^\diamond_l \theta(t) }^2 \big].
\end{array}
\end{equation}
In addition, the fact that $\gcd(\sigma_h, \sigma_v ) = 1$
implies that $\kappa_h \sigma_h + \kappa_v \sigma_v = 1$
for some pair $(\kappa_h , \kappa_v) \in \Wholes^2$. This means
that
\begin{equation}
\abs{ \theta_{l + 1 ; \beta, \gamma}(t) - \theta_{l; \beta, \gamma}(t) }
\le  [\abs{\kappa_h} + \abs{\kappa_v} ] \norm{ \pi^\diamond \theta_{; \beta, \gamma}(t) }_{\ell^\infty(\Wholes; \Real^5)}.
\end{equation}
In particular, the following result
suffices to obtain \sref{eq:prop:hom:oblq:subSolDiffFromWave}
and \sref{eq:prop:hom:oblq:subSolPhaseDiff}.
\begin{lem}
\label{lem:hom:oblq:aprior:bnds}
Pick any $\kappa > 0$.
Then for any $\beta > 1$,
there exists
$\gamma_* =\gamma_*(\kappa,  \beta) \ge 1$,
such that for any $\gamma \ge \gamma_*$, any $t \ge 1$
and any $l \in \Wholes$
we have the bounds
\begin{equation}
\label{eq:lem:hom:oblq:aprior:bnds:thetaDiamond}
\abs{\pi^\diamond_l \theta_{; \beta, \gamma}(t)} + \abs{\pi^{\diamond\diamond}_l \theta_{; \beta, \gamma}(t)} \le \kappa t^{-1/2}.
\end{equation}
\end{lem}
\begin{proof}
This follows directly from the global bounds in Lemma
\ref{eq:oblq:theta:glbFndsPiThetas},
choosing $\gamma_* \gg \beta^2$.
\end{proof}
In view of the identity
\begin{equation}
\dot{\xi}_{nl}(t) = c - \dot{\theta}_{l; \beta, \gamma}(t) - K_Z z(t),
\end{equation}
the following result can be used
to establish \sref{eq:prop:hom:oblq:subSolPhaseDrv}
provided that $\abs{z(t)} \le \frac{c}{4 K_Z}$ for all $t \ge 1$.
\begin{lem}
\label{lem:hom:oblq:aprior:bnds:dotTheta}
For any $\beta > 1$ and $\Omega_\perp > 0$,
there exist
$\gamma_* =\gamma_*(\beta, \Omega_\perp) \ge 1$,
such that for any $\gamma \ge \gamma_*$, any $t \ge 1$
and any $l \in \Wholes$ for which $\abs{l} \le \Omega_\perp$,
we have the bound
\begin{equation}
\label{eq:lem:hom:oblq:aprior:bnds:dotTheta}
\abs{\dot{\theta}_{l; \beta, \gamma} } < \frac{c}{4}.
\end{equation}
\end{lem}
\begin{proof}
The uniform bound
for $\dot{\theta}$
obtained in Lemma \ref{lem:oblq:theta:glbBndDotTheta}
implies that there exists $t_0 \ge 1$
such that for every $\gamma \ge 1$,
$t \ge t_0$ and $l \in \Wholes$
we have the bound
\sref{eq:lem:hom:oblq:aprior:bnds:dotTheta}.

It hence remains to consider
the regime $1 \le t \le t_0$ and $\abs{l} \le \Omega_\perp$,
for which we may estimate
\begin{equation}
\abs{ \rho(l, t ; \gamma) } = \frac{1}{2} \abs{ \frac{l + \nu_1 t}{\nu_2 \gamma t } }
\le \frac{1}{2} (\nu_2 \gamma)^{-1} [\Omega_\perp + \abs{\nu_1} t_0 ] \le C_2' \gamma^{-1}
\end{equation}
for some $C_2' > 1$.
Using \sref{eq:hom:sub:obl:compDotTheta},
we write
\begin{equation}
\dot{\theta}_{l; \beta, \gamma}(t) = \beta t^{-\alpha}
\big[ - \frac{1}{4} (\gamma t)^{-1} - \nu_1 \rho + \gamma \nu_2 \rho^2 \big]
 v_{l ; \gamma}(t),
\end{equation}
which shows that
\begin{equation}
\abs{ \dot{\theta}_{l; \beta, \gamma}(t)}
\le \beta [  \frac{1}{4} \gamma ^{-1}  + \abs{\nu_1} C_2' \gamma^{-1}  + \gamma \nu_2 (C_2')^2 \gamma^{-2}]
\end{equation}
whenever  $1 \le t \le t_0$ and $\abs{l} \le \Omega_\perp$.
Picking $\gamma_* \ge 1$ sufficiently large hence establishes
the desired bound \sref{eq:lem:hom:oblq:aprior:bnds:dotTheta}.
\end{proof}
We note that the final auxilliary estimate
\sref{eq:prop:hom:obql:subsol:det:bndOnLaplaceElms}
can be obtained by noting
that the bound for
the nonlinear expressions $\mathcal{R}_{\mathcal{N}; \nu}$
obtained in Lemma \ref{lem:hom:oblq:ansz:bndN2Plus}
merely requires an a-priori estimate on $\pi^\diamond \theta_{; \beta, \gamma}(t)$.
Such an estimate can easily be obtained using Lemma \ref{lem:hom:oblq:aprior:bnds}
to restrict $\gamma \ge \gamma_* \ge 1$.


It now remains to establish the differential inequality
\sref{eq:oblq:cnstr:diffInEq}.
To this end,
we introduce the expressions
\begin{equation}
\begin{array}{lcl}
\Theta_{nl; \beta, \gamma}(t) & = &
\j_{nl}(\dot{\theta}_{; \beta, \gamma}, \Phi' ; t)
- \j_{nl} (\alpha^\diamond_\mu \pi^\diamond_\mu \theta_{; \beta, \gamma} , \Phi'; t)
- \j_{nl} (\alpha^{\diamond\diamond}_{\mu} \pi^{\diamond\diamond}_{\mu \mu'} \theta_{; \beta, \gamma}, \Phi' ; t)
\\[0.2cm]
& = &
  \mathcal{T}_{l; \beta, \gamma}(t) \Phi'\big(\xi_{nl}(t) \big),
\\[0.2cm]
\mathcal{Z}_{nl; \beta, \gamma , K_Z}(t)
& = & \j_{nl}(K_Z z, \Phi'; t) + \dot{z}(t) - \j_{nl}(L; t) z(t)
\\[0.2cm]
& = & K_Z z(t) \Phi'\big(\xi_{nl}(t) \big) + \dot{z}(t)
- g'\Big(\Phi\big(\xi_{nl}(t) \big) \Big) z(t).
\end{array}
\end{equation}
Remembering the choice $\dot{Z}(t) = K_Z z(t)$
and suppressing the dependence on $\beta$, $\gamma$ and $K_Z$,
the expression \sref{eq:lem:oblq:defJminusFinal} can now be written as
\begin{equation}
\label{eq:oblq:sup:defJMinusForSubFinal}
\begin{array}{lcl}
\mathcal{J}^-(t)
& = &
%
- \Theta(t) - \mathcal{Z}(t)
\\[0.2cm]
%
%
& & \qquad
+ \mathcal{R}_1 \big( \pi^{\diamond} \theta, \pi^{\diamond\diamond} \theta , z ; t \big)
+ \mathcal{R}_2 \big( \dot{Z}, \pi^\diamond \theta, \pi^{\diamond \diamond} \theta ; t \big)
+ \mathcal{R}_3 \big(\dot{\theta},  \pi^{\diamond \diamond} \theta ; t \big)
\\[0.2cm]
& & \qquad \qquad
+ \mathcal{R}_4 \big(\pi^\diamond \dot{\theta} , \pi^\diamond \theta ; t \big)
+ \mathcal{R}_5 \big( \pi^{\diamond} \theta, \pi^{\diamond \diamond} \theta ,
 \pi^{\diamond \diamond \diamond} \theta ; t \big)
+ \mathcal{R}_6 \big(\dot{\theta},  \pi^{\diamond} \theta ; t)
+ \mathcal{R}_7 \big( \pi^{\diamond} \theta ; t)
\\[0.2cm]
& & \qquad
+ \mathcal{E}_1\big( \pi^{\diamond \diamond}\dot{\theta} ; t \big)
+ \mathcal{E}_2\big( \pi^{\diamond \diamond \diamond} \theta ; t \big)
+ \mathcal{E}_3\big(  \pi^{\diamond} \dot{\theta}, \pi^{\diamond\diamond} \theta ; t \big).
\end{array}
\end{equation}

The following two results concern the terms $\Theta(t)$
and $\mathcal{Z}(t)$, which are the ones for which
a definite sign is (almost) available.
We note that the strictly positive function $\mathcal{Q}(t)$
was defined in \sref{eq:oblq:theta:defCalQ}.
\begin{lem}
\label{lem:oblq:sub:estTheta}
Pick any $\delta_{\omega} > 0$, any $\beta > 1$
and any $\epsilon_3 > 0$.
Then there exist $\gamma_* = \gamma_*(\delta_{\omega}, \beta, \epsilon_3 ) \ge 1$
such that for any $\gamma \ge \gamma_*$, any $t \ge 1$
and any $l \in \Wholes$,
the bound
\begin{equation}
\begin{array}{lcl}
\abs{ \mathcal{\Theta}_{l; \beta , \gamma}(t) - \mathcal{Q}_{l; \beta, \gamma}(t) \Phi'\big(\xi_{nl}(t) \big) }
& \le & \frac{1}{26} \eta_z \epsilon_3 t^{-3/2}
\end{array}
\end{equation}
holds provided $\abs{ \rho(l, \gamma ; t) } \le \delta_\omega$,
while
\begin{equation}
\begin{array}{lcl}
\abs{ \mathcal{\Theta}_{l; \beta , \gamma}(t) }
& \le & \frac{1}{26} \eta_z \epsilon_3 t^{-3/2}
\end{array}
\end{equation}
holds provided $\abs{\rho(l, \gamma ; t) } \ge \delta_{\omega}$.
\end{lem}
\begin{proof}
These bounds follow directly from
Lemmas \ref{lem:oblq:theta:locBndCalT}
and \ref{lem:oblq:theta:farfieldBnds},
noting that for any $C' > 1$
we can choose $\gamma_* \ge 1$
in such a way that
\begin{equation}
\label{eq:oblq:sup:bndThetaTerm:estWithExponentials}
\beta \gamma C' (\gamma t)^{1/2} e^{ - \nu_2 \delta^2_{\omega} \gamma t}
\le  \beta C' (\gamma t)^{-3/2} \le \epsilon_3 t^{-3/2}
\end{equation}
holds for any $t \ge 1$ and any $\gamma \ge \gamma_*$.
\end{proof}

\begin{lem}
There exist $\eta_z > 0$ and $K_Z > 1$ such that
for every $C^1$-smooth function $z: [1, \infty) \to \Real$ that
has $z(t) > 0$ and $\dot{z}(t) \ge -\eta_z z(t)$ for all $t \ge 1$,
we have
\begin{equation}
\mathcal{Z}_{nl; \beta, \gamma, K_Z}(t) \ge \eta_z  z(t)
\end{equation}
for all $\beta \ge 1$, all $\gamma \ge 1$,
all $(n,l) \in \Wholes^2$ and all $t \ge 1$.
\end{lem}
\begin{proof}
Since
\begin{equation}
\mathcal{Z}_{nl; \beta, \gamma, K_Z} \ge \big[ K_Z \Phi'\big(\xi_{nl}(t) \big)
  - \eta_z - g'\Big(\Phi\big(\xi_{nl}(t)\big)\Big) \big] z(t),
\end{equation}
it suffices to choose $\eta_z > 0$ and $K_Z \gg 1$ in such a way that
\begin{equation}
K_Z \Phi'(\xi) - \eta_z - D\big(\Phi(\xi)\big) \ge \eta_z
\end{equation}
holds for all $\xi \in \Real$. This is possible because
of the limits $\Phi(-\infty) = 0$, $\Phi(+ \infty) = 1$,
the inequalities
$g'(0) < 0$ and $g'(1) < 0$ and the fact that $\Phi'(\xi) > 0$ for all $\xi \in \Real$.
\end{proof}

We are now ready to estimate
the numbered terms appearing in \sref{eq:oblq:sup:defJMinusForSubFinal}.
The terms $\mathcal{R}_1$ and $\mathcal{R}_7$
need to be considered separately in the near-field and far-field
regimes, but the remaining numbered terms can be handled
using global bounds.
\begin{lem}
\label{lem:oblq:sub:estR17}
There exists $\delta_z > 0$ such that the following is true.
Pick any $\delta_{\omega} > 0$, any $\beta > 1$
and any $\epsilon_3 > 0$.
Then there exist $\gamma_* = \gamma_*(\delta_{\omega}, \beta, \epsilon_3 ) \ge 1$
such that for
any function $z: [1,\infty) \to (0, \delta_z]$,
any $\gamma \ge \gamma_*$, any $t \ge 1$
and any $l \in \Wholes$,
we have the bounds
\begin{equation}
\begin{array}{lcl}
\abs{ [\mathcal{R}_1\big(
  \pi^\diamond \theta, \pi^{\diamond \diamond} \theta , z ; t
  \big) ]_{nl}
}
&\le &   \frac{1}{26} \eta_z \abs{z(t)} + \frac{1}{26} \eta_z \epsilon_3  t^{-3/2}
\\[0.2cm]
& & \qquad
+ \frac{1}{2} \mathcal{Q}_{l; \beta, \gamma, t_*}(t) \Phi'\big(\xi_{nl}(t) \big),
\\[0.2cm]
\abs{ [\mathcal{R}_7
  \big( \pi^\diamond \theta ; t \big)
]_{nl}
}
& \le &
\frac{1}{2} \mathcal{Q}_{l; \beta, \gamma, t_*}(t) \Phi'\big(\xi_{nl}(t) \big)
  + \frac{1}{26} \eta_z \epsilon_3  t^{-3/2},
\\[0.2cm]
\end{array}
\end{equation}
provided $\abs{ \rho(l, \gamma ; t) } \le \delta_\omega$,
together with the bounds
\begin{equation}
\begin{array}{lcl}
\abs{ [\mathcal{R}_1\big(
  \pi^\diamond \theta, \pi^{\diamond \diamond} \theta , z ; t
  \big) ]_{nl}
}
&\le &   \frac{1}{26} \eta_z \abs{z(t)}
  + \frac{1}{26} \eta_z \epsilon_3  t^{-3/2},
\\[0.2cm]
\abs{ [\mathcal{R}_7
  \big( \pi^\diamond \theta ; t \big)
]_{nl}
}
& \le &
  \frac{1}{26} \eta_z \epsilon_3  t^{-3/2},
\\[0.2cm]
\end{array}
\end{equation}
provided $\abs{\rho(l, t, \gamma)} \ge \delta_{\omega}$.
Here we have used $\theta = \theta_{; \beta, \gamma}$.
\end{lem}
\begin{proof}
First of all, pick any $0 < \kappa' \le 1$.
Possibly decreasing $\delta_z > 0$
and increasing $\gamma_* \ge 1$,
Lemma \ref{lem:hom:oblq:aprior:bnds}
shows that we can arrange for
\begin{equation}
\delta_z + \abs{\pi^\diamond_l \theta_{; \beta, \gamma}(t) } + \abs{ \pi^{\diamond \diamond}_l \theta_{; \beta, \gamma}(t) } \le \kappa' \le 1,
\qquad \gamma \ge \gamma_*, \qquad l \in \Wholes, \qquad t \ge 1.
\end{equation}
Recalling the constant $C_1 = C_1(1)$
defined in Lemma \ref{lem:oblq:anstz:bndOnR17},
we see that
in the near field $\abs{\rho} \le \delta_{\omega}$
we have
\begin{equation}
\begin{array}{lcl}
\abs{ [\mathcal{R}_1\big(
  \pi^\diamond \theta, \pi^{\diamond \diamond} \theta , z ; t
  \big) ]_{nl}
}
&\le & C_1 \abs{z(t)}\big[\abs{z(t)} + \pi^\diamond_l \theta(t)
   + \abs{\pi^{\diamond\diamond}_l \theta(t) } \big]
\\[0.2cm]
& & \qquad + C_1\big( \abs{\pi^\diamond_l \theta(t)}
   + \abs{\pi^{\diamond\diamond}_l \theta(t) }\big)^2 \Phi'\big(\xi_{nl}(t)\big)
\\[0.2cm]
& \le &
C_1 \kappa' \abs{z(t)}
\\[0.2cm]
& & \qquad  + 16 \beta^2 t^{-2\alpha} C_1 C_4^2
  \big[
      \big(\rho^2 + (\gamma t)^{-2} \big) v^2_{l; \gamma}(t)
     + (\gamma t) e^{ -2 \nu_2 \delta_{\omega}^2 \gamma t}
  \big]\Phi'\big(\xi_{nl}(t) \big).
\\[0.2cm]
\end{array}
\end{equation}
In view of the definition \sref{eq:oblq:theta:defCalQ}
for $\mathcal{Q}_{l; \beta, \gamma}(t)$,
it hence suffices to pick $\kappa' > 0$
sufficiently small and  $\gamma_* \ge 1$ sufficiently large
to ensure that
\begin{equation}
C_1 \kappa' \le \frac{1}{26} \eta_z,
\qquad
16 C_1 \beta C_4^2 \le \frac{1}{16} \nu_2 \gamma_* ,
\qquad
16 C_1 \beta C_4^2 \gamma_*^{-1} \le \frac{1}{16}
\end{equation}
all hold,
together with
an exponential estimate similar to \sref{eq:oblq:sup:bndThetaTerm:estWithExponentials}.

The far field case $\abs{\rho(l,t;\gamma)} \ge \delta_{\omega}$
can be treated in a similar fashion.
Finally,
inspection of \sref{lem:oblq:anstz:bndOnR17}
shows that $\mathcal{R}_7$
only contains terms that are also present in $\mathcal{R}_1$.
\end{proof}

\begin{lem}
\label{lem:oblq:sub:estR26}
Pick any $\beta > 1$ and $\epsilon_3 > 0$.
Then there exists $\gamma_* = \gamma_*(\beta, \epsilon_3) \ge 1$
such that for
any function $z: [0, \infty) \to (0, 1]$,
any $\gamma \ge \gamma_*$
and any $t \ge 1$,
the following bounds hold
for every $(n,l) \in \Wholes^2$,
\begin{equation}
\begin{array}{lcl}
\abs{ [ \mathcal{R}_2
      \big( K_Z z, \pi^\diamond \theta, \pi^{\diamond \diamond} \theta ; t \big)
  ]_{nl} }
&\le &
\frac{1}{26} \eta_z z(t),
\\[0.2cm]
\abs{ [ \mathcal{R}_3
   \big(\dot{\theta},  \pi^{\diamond \diamond} \theta ; t \big)
  ]_{nl} }
&\le &
 \frac{1}{26} \eta_z \epsilon_3 t^{-3/2},
\\[0.2cm]
\abs{ [ \mathcal{R}_4
 \big(\pi^\diamond \dot{\theta} , \pi^\diamond \theta ; t \big)
]_{nl} }
& \le &   \frac{1}{26} \eta_z \epsilon_3 t^{-3/2},
\\[0.2cm]
\abs{ [ \mathcal{R}_5
 \big( \pi^{\diamond} \theta, \pi^{\diamond \diamond} \theta ,
 \pi^{\diamond \diamond \diamond} \theta ; t \big)
]_{nl} }
& \le &
 \frac{1}{26} \eta_z \epsilon_3 t^{-3/2},
\\[0.2cm]
\abs{ [ \mathcal{R}_6
 \big(\dot{\theta},  \pi^{\diamond} \theta ; t)
]_{nl} }
& \le &  \frac{1}{26} \eta_z \epsilon_3 t^{-3/2},
\\[0.2cm]
\end{array}
\end{equation}
in which we used $\theta = \theta_{; \beta, \gamma}$.
\end{lem}
\begin{proof}
First note that the global bounds
in Lemma \ref{eq:oblq:theta:glbFndsPiThetas}
imply that
\begin{equation}
\begin{array}{lcl}
\abs{ [ \mathcal{R}_2\big(
  K_Z z , \pi^\diamond \theta, \pi^{\diamond \diamond} \theta ;t
  \big)
  ]_{nl} }
&\le & \beta t^{-\alpha} C_1 K_Z \abs{z(t)} \big[
    C_6 ( \gamma t )^{-1/2} + C_6 ( \gamma t)^{-1}
\big],
\\[0.2cm]
\end{array}
\end{equation}
so for the first inequality it suffices to pick $\gamma_* \ge 1$ in such a way that
\begin{equation}
\beta C_1 K_Z C_6  ( \gamma_*^{-1/2} + \gamma_*^{-1} ) \le \frac{1}{26} \eta_z.
\end{equation}

In view of the global bounds
obtained in Lemmas \ref{eq:oblq:theta:glbFndsPiThetas},
\ref{lem:oblq:theta:glbBndDotTheta}
and \ref{lem:oblq:theta:glbBndDotTheta:heavy},
the remaining estimates follow from the computations
\begin{equation}
\begin{array}{lcl}
\abs{ [ \mathcal{R}_3
   \big(\dot{\theta},  \pi^{\diamond \diamond} \theta ; t \big)
  ]_{nl} }
&\le &
C_1 \abs{ \dot{\theta}_l(t) } \abs{ \pi^{\diamond\diamond}_l \theta(t) }
\\[0.2cm]
& \le & C_1 \beta^2 C_6^2 \big[ (\gamma t )^{-1/2} + \gamma (\gamma t)^{-1} \big]
  ( \gamma t)^{-1} ,
\\[0.2cm]
\abs{ [ \mathcal{R}_4
 \big(\pi^\diamond \dot{\theta} , \pi^\diamond \theta ; t \big)
]_{nl} }
& \le &  C_1 \abs{ \pi^\diamond_l \dot{\theta}(t) } \abs{ \pi^\diamond_l \theta(t) }
\\[0.2cm]
& \le & C_1 \beta^2 C_6^2 \big[ \gamma (\gamma t)^{-3/2} + (\gamma t)^{-1}  \big]
  (\gamma t)^{-1/2},
\\[0.2cm]
\abs{ [ \mathcal{R}_5
 \big( \pi^{\diamond} \theta, \pi^{\diamond \diamond} \theta ,
 \pi^{\diamond \diamond \diamond} \theta ; t \big)
]_{nl} }
& \le &
\abs{\pi^\diamond_l \theta(t) }
\big[ \abs{\pi^{\diamond \diamond}_l \theta(t) } + \abs{\pi^\diamond_l \theta(t) }^2
+ \abs{\pi^{\diamond \diamond\diamond}_l \theta(t) } \big]
+ \abs{ \pi^{\diamond\diamond}_l \theta(t) }^2
\\[0.2cm]
& \le &
C_1 \beta^2 C_6^2  (\gamma t)^{-1/2} \big[ (\gamma t )^{-1} + (\gamma t)^{-1} + (\gamma t)^{-3/2}   \big]
\\[0.2cm]
& & \qquad
+ C_1 \beta^2 C_6^2  (\gamma t)^{-2},
\\[0.2cm]
\abs{ [ \mathcal{R}_6
 \big(\dot{\theta},  \pi^{\diamond} \theta ; t)
]_{nl} }
& \le & C_1 \abs{ \dot{\theta}_l(t) - \alpha^\diamond_\mu \pi^\diamond_{l; \mu} \theta(t) } \abs{ \pi^\diamond_l \theta(t) }
\\[0.2cm]
& \le &
C_1 \beta^2 C_6^2
 (\gamma t)^{-1}       (\gamma t)^{-1/2}.
\\[0.2cm]
\end{array}
\end{equation}
Indeed, the worst of these terms is given by
$C_1 \beta^2 C_6^2 [ (\gamma t)^{-3/2} + \gamma (\gamma t)^{-2} ]$,
which can easily be estimated by $\epsilon_3 t^{-3/2}$
for all $t\ge 1$ by picking $\gamma_*$ sufficiently large.
\end{proof}

\begin{lem}
\label{lem:oblq:sub:estE123}
Pick any $\beta > 1$ and $\epsilon_3 > 0$.
Then there exists $\gamma_* =\gamma_*( \beta, \epsilon_3)$,
such that for any $\gamma \ge \gamma_*$, any $t \ge 1$
and any $(n,l) \in \Wholes^2$,
we have the bounds
\begin{equation}
\begin{array}{lcl}
\abs{ [\mathcal{E}_1\big(
  \pi^{\diamond\diamond} \dot{\theta} ; t
  \big) ]_{nl}
}
&\le &
  \frac{1}{26} \eta_z  \epsilon_3 t^{-3/2},
\\[0.2cm]
\abs{ [ \mathcal{E}_2\big(
   \pi^{\diamond \diamond \diamond} \theta ;t
  \big)
  ]_{nl} }
&\le &
  \frac{1}{26} \eta_z \epsilon_3 t^{-3/2},
\\[0.2cm]
\abs{ [ \mathcal{E}_3\big(
   \pi^{\diamond} \dot{\theta}, \pi^{\diamond \diamond} \theta  ;t
  \big)
  ]_{nl} }
&\le &
  \frac{1}{26} \eta_z \epsilon_3 t^{-3/2}.
\\[0.2cm]
\end{array}
\end{equation}
\end{lem}
\begin{proof}
Combining the global bounds
obtained in
Lemmas \ref{eq:oblq:theta:glbFndsPiThetas},
\ref{lem:oblq:theta:glbBndDotTheta}
and \ref{lem:oblq:theta:glbBndDotTheta:heavy}
with the estimates in Lemma \ref{lem:hom:oblq:ansz:bndE13},
we compute
\begin{equation}
\begin{array}{lcl}
\abs{ [\mathcal{E}_1\big(
  \pi^{\diamond\diamond} \dot{\theta} ; t
  \big) ]_{nl}
}
&\le & C_1  \abs{\pi^{\diamond\diamond}_l  \dot{\theta}(t) }
\\[0.2cm]
& \le &
C_1 C_6 \beta \big[ \gamma (\gamma t)^{-2} + (\gamma t)^{-3/2} \big],
\\[0.2cm]
\abs{ [ \mathcal{E}_2\big(
   \pi^{\diamond \diamond \diamond} \theta ;t
  \big)
  ]_{nl} }
&\le & C_1 \abs{  \pi^{\diamond \diamond \diamond}_l \theta(t) }
\\[0.2cm]
& \le &
C_1 C_6 \beta  (\gamma t)^{-3/2} ,
\\[0.2cm]
\abs{ [ \mathcal{E}_3\big(
   \pi^{\diamond} \dot{\theta}, \pi^{\diamond \diamond} \theta  ;t
  \big)
  ]_{nl} }
&\le & C_1 \abs{ \pi^{\diamond}_l [\dot{\theta}(t) - \alpha^\diamond_\mu \pi^\diamond_{l;\mu} \theta(t) ] }
\\[0.2cm]
& \le &
C_1 C_6 \beta  \gamma (\gamma t)^{-3/2}.
\\[0.2cm]
\end{array}
\end{equation}
The worst term is hence $C_1 C_6 \beta \gamma (\gamma t)^{-3/2}$, so it suffices
to pick $\gamma_*$ in such a way that
\begin{equation}
2 C_1 C_6 \beta \gamma_*^{-1/2} \le \frac{1}{26} \eta_z \epsilon_3.
\end{equation}
\end{proof}

\begin{proof}[Proof of Proposition \ref{prp:hom:oblq:subSolnDetailed}]
The statements follow from the discussion above,
picking $\delta_z > 0$ and $\gamma \ge \gamma_* \ge 1$ in such a way that
the statements in Lemmas
\ref{lem:oblq:sub:estTheta},
\ref{lem:oblq:sub:estR17},
\ref{lem:oblq:sub:estR26}
and
\ref{lem:oblq:sub:estE123}
all hold. In particular,
these estimates imply that for any $(n,l) \in \Wholes^2$ and any $t \ge 1$ we have
\begin{equation}
\mathcal{J}^-_{nl}(t)
 \le  -\frac{1}{2} \eta_z z(t) < 0,
\end{equation}
as desired.
\end{proof}

\begin{proof}[Proof of Proposition \ref{prp:hom:oblq:mr:sub:sup}]
The statements concerning $W^-$ follow directly from Proposition
\ref{prp:hom:oblq:subSolnDetailed}, upon
choosing $\beta > 1$ to be sufficiently large
and writing
\begin{equation}
W^-_{nl}(t) = u^-_{nl}(t - 1).
\end{equation}
We stress that the function $\theta_{; \beta, \gamma}$
depends only on the parameters $\epsilon_1$, $\epsilon_2$,
$\epsilon_3$, $\Omega_\perp$ and $\Omega_{\mathrm{phase}}$
and not on the specific form of $z(t)$.
The super-solution $W^+$ can be constructed analogously.
\end{proof}

\section{The Entire Solution}
\label{sec:ent}

Throughout the remainder of this paper we focus our attention
on the obstructed LDE \sref{eq:ent:main:lde:nl:coords}.
The purpose of this section is to establish the existence
of an entire solution to \sref{eq:ent:main:lde:nl:coords}
that converges to a planar travelling wave as $t \to - \infty$.
The ideas here closely follow the presentation in \cite[{\S}2-{\S}3]{BHM},
but the discreteness of the lattice requires certain technical
adjustments.

Throughout this section we fix a pair $(\sigma_h, \sigma_v) \in \Wholes^2 \setminus \{0, 0\}$
with $\gcd(\sigma_h, \sigma_v) = 1$ and recall the notation $\sigma = \max\{ \abs{\sigma_h}, \abs{\sigma_v} \}$.
For any $S \subset \Wholes^2$, we introduce the neighbour set
\begin{equation}
\mathcal{N}^\times_S(n,l) =
\{ ( n + \sigma_h, l + \sigma_v), (n + \sigma_v, l - \sigma_h), (n - \sigma_h, l - \sigma_v) , (n - \sigma_v, l + \sigma_h) \} \cap S,
\end{equation}
together with the associated punctured Laplacian
\begin{equation}
[\Delta^\times_{S}v ]_{nl} = \sum_{(n',l') \in \mathcal{N}^\times_{S}(n,l) } [v_{n'l'} - v_{nl} ]
\end{equation}
and the boundary
\begin{equation}
\label{sec:ent:defPartialTimes}
\partial_\times S = \{ (n,l) \in S \mid \mathcal{N}^\times_{S}(n,l) \neq \mathcal{N}^\times_{\Wholes^2}(n,l) \}.
\end{equation}

Upon writing
\begin{equation}
K_{\mathrm{obs}}^\times = \{ (n,l) \in \Wholes^2 \hbox{ for which } (n,l) =
\big( i \sigma_h + j \sigma_v, i \sigma_v - j \sigma_h \big)
\hbox{ for some } (i,j) \in K_{\mathrm{obs}} \},
\end{equation}
together with
\begin{equation}
\Lambda^\times = \Wholes^2 \setminus K_{\mathrm{obs}}^\times,
\end{equation}
we see that the obstructed LDE \sref{eq:mr:ldeWithObs}
is transformed into
\begin{equation}
\label{eq:ent:main:lde:nl:coords}
\dot{u}_{nl}(t) = [ \Delta^\times_{\Lambda^\times} u(t) ]_{nl} + g\big(u_{nl}(t) \big),
\qquad (n,l) \in \Lambda^\times.
\end{equation}

\begin{prop}
\label{prp:ent:entSol}
Consider the obstructed LDE \sref{eq:ent:main:lde:nl:coords}
and assume that (Hg), (HK1)
and $\textrm{(h}\Phi\textrm{)}_{\textrm{\S\ref{sec:prlm}}}$ with $c > 0$
all hold.
Then there exists a $C^1$-smooth function $U: \Real \to \ell^{\infty}(\Lambda^\times; \Real)$
that satisfies the obstructed LDE \sref{eq:ent:main:lde:nl:coords}
for all $t \in \Real$,
admits the uniform limit
\begin{equation}
\label{eq:lem:ent:exst:limit}
\sup_{ (n,l) \in \Lambda^\times } \abs{ U_{nl}(t) - \Phi(n + ct) } \to 0, \qquad  t \to -\infty
\end{equation}
and enjoys the estimates
\begin{equation}
\label{eq:lem:ent:exst:ineqls}
0 < U_{nl}(t) < 1, \qquad \dot{U}_{nl}(t) > 0
\end{equation}
for all $(n,l) \in \Lambda^\times$ and $t \in \Real$.
In addition, any $C^1$-smooth function
$V: \Real \to \ell^{\infty}(\Lambda^\times; \Real)$
that satisfies \sref{eq:ent:main:lde:nl:coords} for all $t \in \Real$
together with \sref{eq:lem:ent:exst:limit} must also have $V = U$.
\end{prop}

By relabelling our coordinate system and shifting the wave
profile $\Phi$,
we can arrange for the following two conditions to hold.
\begin{itemize}
\item[$\textrm{(hK}\textrm{)}_{\textrm{\S\ref{sec:ent}}}$]
{
  The obstacle satisfies $K^\times_{\mathrm{obs}} \subset \{ n <  - 2 \sigma \} \subset \Wholes^2$.
}
\item[$\textrm{(h}\Phi\textrm{)}_{\textrm{\S\ref{sec:ent}}}$]{
  Recalling the inequality $g(\xi) \le 0$ for $0 \le \xi \le a$,
  we have $\Phi(0) \le a$. In addition, we have
  $\Phi''(\xi) > 0$ for all $\xi \le 0$, together with $c > 0$.
}
\end{itemize}
Upon recalling the exponents $(\kappa_\Phi, \eta^\pm_\Phi)$
defined in Lemma \ref{prp:prlm:asymEsts},
we fix the exponent
\begin{equation}
\eta_0 =  \min\{ \eta^-_\Phi, \kappa_\Phi \}.
\end{equation}
For any $M_0 > 1$, we now introduce the function
\begin{equation}
\Xi(t)= \Xi_{M_0}(t)
\end{equation}
that is uniquely defined
by the initial value problem
\begin{equation}
\dot{\Xi}(t) = M_0 e^{\eta_0 \big(ct + \Xi(t)\big) }, \qquad \Xi(-\infty) = 0.
\end{equation}
We note that $\Xi(t)$ is defined on the interval $(-\infty, -T_0)$,
for some $T_0 = T_0(M_0) \gg 1$.

Our main task in this section is to
show that the two $C^1$-smooth functions
$u^\pm: (-\infty, -T_0(M_0) ) \to \ell^\infty(\Wholes^2; \Real)$
defined by
\begin{equation}
\begin{array}{lcl}
u^-_{nl}(t) & = & \left\{
  \begin{array}{lcl}
    \Phi\big(n + ct - \Xi(t)\big) - \Phi\big(- n + ct - \Xi(t) \big) & & n \ge 0,
    \\[0.2cm]
    0 & & n < 0,
  \end{array}
  \right.
\\[0.6cm]
u^+_{nl}(t) & = & \left\{
  \begin{array}{lcl}
    \Phi\big(n + ct + \Xi(t) \big) + \Phi\big(- n + ct + \Xi(t) \big) & & n \ge 0,
    \\[0.2cm]
    2 \Phi\big(ct + \Xi(t) \big) & & n < 0,
  \end{array}
  \right.
\end{array}
\end{equation}
are sub respectively super-solutions for \sref{eq:ent:main:lde:nl:coords}.
The form of these two functions is precisely the same
as that of their counterparts from \cite{BHM}, but
the non-local terms in our LDE require special care
in our analysis because the $\{n < 0 \}$ and $\{ n \ge 0 \}$ regimes
interact with each other across the $\{n = 0\}$ boundary.

In order to understand these non-local terms,
we first notice that $\textrm{(hK}\textrm{)}_{\textrm{\S\ref{sec:ent}}}$
implies that
\begin{equation}
\label{eq:ent:laplacians}
\Delta^\times_{\Lambda^\times} u^\pm(t) = \Delta^\times u^\pm(t)
\end{equation}
for all $t \le -T_0(M_0)$.
In addition, we introduce the notation
\begin{equation}
\begin{array}{lcl}
[\mathcal{I}_\Delta^-(t)]_n & = &
c\Phi'\big(n + ct - \Xi(t)\big) - c \Phi'\big(- n + ct - \Xi(t)\big)
\\[0.2cm]
& & \qquad
  - g\Big(\Phi\big(n + ct - \Xi(t)\big)\Big) + g\Big(\Phi\big( - n + ct - \Xi(t)\big) \Big),
\\[0.2cm]
[\mathcal{I}_\Delta^+(t)]_n & = & c\Phi'\big(n + ct + \Xi(t)\big)
  + c \Phi'\big(- n + ct + \Xi(t)\big)
\\[0.2cm]
& & \qquad - g\Big(\Phi\big(n + ct + \Xi(t)\big)\Big)
   - g\Big(\Phi\big( - n + ct + \Xi(t)\big) \Big).
\\[0.2cm]
\end{array}
\end{equation}
Before we state our first two technical results concerning
the discrete Laplacians \sref{eq:ent:laplacians},
we recall the notation
\begin{equation}
  (\sigma_1, \ldots, \sigma_5) = \big( \sigma_v, -\sigma_h , - \sigma_v, \sigma_h, 0 \big),
\end{equation}
together with the vector
\begin{equation}
(L^\times_1, \ldots, L^\times_5) = ( 1, 1, 1, 1, -4) \in \Real^5
\end{equation}
and the summation convention introduced in \S\ref{sec:oblq:notation}.

\begin{lem}
\label{lem:ent:discLaplaceUMinus}
Consider the
obstructed LDE \sref{eq:ent:main:lde:nl:coords},
suppose that
(Hg), $\textrm{(hK}\textrm{)}_{\textrm{\S\ref{sec:ent}}}$,
$\textrm{(h}\Phi\textrm{)}_{\textrm{\S\ref{sec:prlm}}}$
and $\textrm{(h}\Phi\textrm{)}_{\textrm{\S\ref{sec:ent}}}$
all hold and pick any $M_0 > 1$.
Then for any $t \le -T_0(M_0)$, we have the inequalities
\begin{equation}
\begin{array}{lclcl}
[\Delta^\times u^-(t)]_{nl}
& \ge &  [\mathcal{I}^-_{\Delta}(t)]_n, & & n > 0,
\\[0.2cm]
[\Delta^\times u^-(t)]_{nl} & \ge &  0, & & n \le 0.
\end{array}
\end{equation}
\end{lem}
\begin{proof}
For $n \ge \sigma$,
we may compute
\begin{equation}
\begin{array}{lcl}
[\Delta^\times u^-(t)]_{nl} & = & L^\times_{\mu} \Phi\big(n + \sigma_\mu + ct - \Xi(t) \big)
- L^\times_\mu \Phi\big(-n - \sigma_\mu + ct - \Xi(t) \big)
\\[0.2cm]
& = & [\mathcal{I}^-_{\Delta}(t)]_n.
\end{array}
\end{equation}
For $0 < n < \sigma$, we have
\begin{equation}
\begin{array}{lcl}
[\Delta^\times u^-(t)]_{nl} -  [\mathcal{I}^-_{\Delta}(t)]_n
&=&
 \sum_{n + \sigma_\mu < 0}\big[
  \Phi\big( - n - \sigma_\mu + ct - \Xi(t) \big)
  -  \Phi\big(n + \sigma_\mu + ct - \Xi(t)\big)
  \big]
\\[0.2cm]
& \ge & 0,
\end{array}
\end{equation}
since $\Phi$ is strictly increasing.
Similarly, for $- \sigma < n \le 0$ we have
\begin{equation}
\begin{array}{lcl}
[\Delta^\times u^-(t)]_{nl}
& = & \sum_{n + \sigma_\mu > 0}
\big[ \Phi\big(n + \sigma_\mu + ct - \Xi(t) \big)
 - \Phi\big(-n - \sigma_\mu + ct - \Xi(t)\big)  \big]
\\[0.2cm]
& \ge &  0,
\end{array}
\end{equation}
while for $n \le - \sigma$ we have
$[\Delta^\times u^-(t)]_{nl} = 0$.
\end{proof}

\begin{lem}
\label{lem:ent:discLaplaceUPlus}
Consider the
LDE \sref{eq:ent:main:lde:nl:coords}
and suppose that (Hg),
$\textrm{(hK}\textrm{)}_{\textrm{\S\ref{sec:ent}}}$,
$\textrm{(h}\Phi\textrm{)}_{\textrm{\S\ref{sec:prlm}}}$
and $\textrm{(h}\Phi\textrm{)}_{\textrm{\S\ref{sec:ent}}}$
all hold.
There exists a constant $C_1 > 1$
so that for any $M_0 > 1$
and any $t \le -T_0(M_0)$
for which $ct + \Xi(t) \le - \sigma$,
the estimate
\begin{equation}
\label{lem:ent:disLaplaceUPlus:est1}
\begin{array}{lcl}
[\Delta^\times u^+(t)]_{nl} - [\mathcal{I}^+_{\Delta}(t)]_n
& \le &  C_1 e^{ -   (\eta^-_\Phi + \kappa_\Phi) \abs{ ct + \Xi(t) } } \mathbf{1}_{n \in [0, \sigma)}
\end{array}
\end{equation}
holds whenever $n \ge 0$,
while
\begin{equation}
\label{lem:ent:disLaplaceUPlus:est2}
\begin{array}{lcl}
[\Delta^\times u^+]_{nl} - c\Phi'\big(ct + \Xi(t)\big) + g'(0) \Phi\big(ct + \Xi(t ) \big)
& \le &  C_1 e^{ -   (\eta^-_\Phi + \kappa_\Phi) \abs{ ct + \Xi(t) } }
\end{array}
\end{equation}
holds whenever $n < 0$.
\end{lem}
\begin{proof}
For convenience, we introduce the shorthand $\xi = ct + \Xi(t)$.
For $n \ge \sigma$, we have
\begin{equation}
\begin{array}{lcl}
[\Delta^\times u^+(t)]_{nl} & = &
L^\times_\mu \Phi\big( n + \sigma_\mu + \xi \big)
+ L^\times_\mu \Phi\big(- n - \sigma_\mu + \xi \big)
\\[0.2cm]
& = & [\mathcal{I}^+_{\Delta}(t)]_n.
\end{array}
\end{equation}
For $0 \le n < \sigma$, we may
use the asymptotics in Proposition \ref{prp:prlm:asymEsts}
to compute
\begin{equation}
\begin{array}{lcl}
[\Delta^\times u^+(t)]_{nl}
- [\mathcal{I}^+_{\Delta}(t)]_n &=&
 \sum_{n + \sigma_\mu < 0}\big[
 2 \Phi(\xi)
 -  \Phi(n + \sigma_\mu + \xi)
 -\Phi( - n - \sigma_\mu + \xi ) \big]
\\[0.2cm]
& \le &
\sum_{n + \sigma_{\mu} < 0}
 C^-_\Phi e^{  \eta^-_\Phi \xi }
 \big[ 2 - e^{ \eta^-_\Phi (n + \sigma_{\mu}) } - e^{ - \eta^-_\Phi(n + \sigma_{\mu})} ]
 \\[0.2cm]
& & \qquad
 + \sum_{n + \sigma_{\mu} < 0}
 K_\Phi e^{  (\eta^-_\Phi + \kappa_\Phi) \xi }
 \big[ 2 + e^{ (\eta^-_\Phi + \kappa_\Phi) (  n + \sigma_{\mu} ) }
         + e^{ - (\eta^-_\Phi + \kappa_\Phi) (  n + \sigma_{\mu} ) } \big]
\\[0.2cm]
& \le &
\sum_{n + \sigma_{\mu} < 0}
 K_\Phi e^{  (\eta^-_\Phi + \kappa_\Phi) \xi }
 \big[ 2 + 2 e^{ (\eta^-_\Phi + \kappa_\Phi) \sigma ) } ],
\end{array}
\end{equation}
since $2 - 2\cosh\big(\eta^-_\Phi( n + \sigma_{\mu})\big) \le 0$.
These two observations readily yield the first estimate \sref{lem:ent:disLaplaceUPlus:est1}.

%

For $- \sigma \le n < 0$ we
obtain
\begin{equation}
[\Delta^\times u^+(t)]_{nl}
= \sum_{n + \sigma_\mu \ge 0}
\big[ \Phi(n + \sigma_\mu + \xi ) + \Phi(-n - \sigma_\mu + \xi) - 2 \Phi( \xi) \big].
\end{equation}
In particular, we may
write
\begin{equation}
Q_{nl}(t) = [\Delta^\times u^+(t)]_{nl} - c\Phi'(\xi) + g'(0) \Phi(\xi)
\end{equation}
and compute
\begin{equation}
\begin{array}{lcl}
Q_{nl}(t)
& \le &
C^-_\Phi e^{ \eta^-_\Phi \xi}
\sum_{n + \sigma_\mu \ge 0}
[e^{\eta^-_\Phi (n + \sigma_{\mu})} + e^{- \eta^-_\Phi ( n + \sigma_{\mu}) } - 2 ]
\\[0.2cm]
& & \qquad
- c C^-_\Phi \eta^-_\Phi e^{\eta^-_\Phi \xi} + g'(0) C^-_\Phi e^{ \eta^-_\Phi \xi}
\\[0.2cm]
& & \qquad
+ K_\Phi e^{ ( \eta^-_\Phi + \kappa_\Phi) \xi }
\sum_{n + \sigma_\mu \ge 0}
[e^{ (\eta^-_\Phi + \kappa_\Phi)(n + \sigma_{\mu})}
 + e^{ -(\eta^-_\Phi + \kappa_\Phi)(n + \sigma_{\mu}) } + 2]
\\[0.2cm]
& & \qquad
+ K_\Phi e^{ (\eta^-_\Phi + \kappa_\Phi) \xi} [ c + \abs{g'(0)} ]
\\[0.2cm]
& \le &
C^-_\Phi e^{ \eta^-_\Phi \xi}
\big[ 2 \cosh( \sigma_h \eta^-_\Phi ) + 2 \cosh( \sigma_v \eta^-_\Phi) - 4
- c \eta^-_\Phi + g'(0)
\big]
\\[0.2cm]
& & \qquad
+ K_\Phi e^{ ( \eta^-_\Phi + \kappa_\Phi) \xi }
  [ c + \abs{g'(0)} + 2 + 2 e^{ (\eta^-_\Phi + \kappa_\Phi) \sigma ) }  ]
\\[0.2cm]
& = &
K_\Phi e^{ ( \eta^-_\Phi + \kappa_\Phi) \xi }
  [ c + \abs{g'(0)} +  2 + 2 e^{ (\eta^-_\Phi + \kappa_\Phi) \sigma ) }],
\end{array}
\end{equation}
where the last equality follows from \sref{eq:lem:prlm:defSpatExps:idForEta}.
Finally, for $n < - \sigma$ we have
$[\Delta^\times u^+(t)]_{nl} = 0$, which establishes
\sref{lem:ent:disLaplaceUPlus:est2} and concludes the proof.
\end{proof}

The specific forms for $u^\pm$ suggest
that it is worthwhile to introduce the two auxilliary functions
\begin{equation}
\begin{array}{lcl}
G(n,\xi) &=&
g\big( \Phi(n + \xi) \big) + g\big( \Phi(-n + \xi) \big)
- g\big( \Phi(n + \xi) + \Phi(-n + \xi) \big),
\\[0.2cm]
H(n, \xi) &=& g\big(\Phi(  n + \xi ) \big) - g\big(\Phi(- n + \xi) \big)
- g\big(  \Phi( n + \xi) - \Phi( - n + \xi)  \big).
\end{array}
\end{equation}
The next result collects some useful properties for $G$ and $H$.

\begin{lem}
\label{lem:ent:bndOnGH}
Consider the
LDE \sref{eq:ent:main:lde:nl:coords}
and suppose that (Hg),
$\textrm{(h}\Phi\textrm{)}_{\textrm{\S\ref{sec:prlm}}}$
and $\textrm{(h}\Phi\textrm{)}_{\textrm{\S\ref{sec:ent}}}$
all hold.
Then there exists $C_2 > 1$ such that for every
$\xi \in \Real$ and $n \in \Wholes$, we have the
inequalities
\begin{equation}
\label{eq:lem:ent:bndOnGH:glb}
\begin{array}{lcl}
\abs{ G(n, \xi) } & \le &
C_2 \Phi(n + \xi) \Phi( - n + \xi),
\\[0.2cm]
\abs{ H(n, \xi) } & \le &
C_2 \Phi( - n + \xi) \big(\Phi( n + \xi) - \Phi( - n + \xi) \big).
\\[0.2cm]
\end{array}
\end{equation}
In addition, if
$\eta^-_\Phi \le \eta^+_\Phi$,
there exist constants $L_2 > 1$ and $\kappa_2 > 0$
such that
the inequalities
\begin{equation}
\label{eq:lem:ent:bndOnGH:special}
\begin{array}{lcl}
G(n, \xi) & \ge & + \kappa_2 \Phi(-n + \xi),
\\[0.2cm]
H(n, \xi) & \le & - \kappa_2 \Phi(-n + \xi)
\\[0.2cm]
\end{array}
\end{equation}
hold for all $\xi \le 0$ and $n \ge L_2 - \xi$.
\end{lem}
\begin{proof}
For any pair $(u,v) \in \Real^2$ we have
\begin{equation}
g(u + v) - g(u) - g(v) = uv \int_0^1 \int_0^1 g''(su + tv) \, ds dt.
\end{equation}
In addition, writing $u =\Phi( n + \xi)$
and $v = \Phi(- n + \xi)$, we have
\begin{equation}
\begin{array}{lcl}
H(n, \xi) & = & g(u) - g(v) - g(u - v)
\\[0.2cm]
&  = & g\big( (u - v) + v\big) - g(v) - g( u - v)
\\[0.2cm]
& = & v ( u - v) \int_0^1 g''\big(s(u-v) + tv\big) \, ds dt.
\end{array}
\end{equation}
These observations directly imply the estimate \sref{eq:lem:ent:bndOnGH:glb}.

We note that the inequalities $c > 0$ and $\eta^-_\Phi \le \eta^+_\Phi$
directly imply that $g'(0) > g'(1)$.
Upon introducing the quantity
\begin{equation}
\mathcal{I}_G(u, v) =g(v) -g(0) + g(u) - g(u + v) +  [g'(1) - g'(0) ]v,
\end{equation}
we may compute
\begin{equation}
\begin{array}{lcl}
\mathcal{I}_G(u,v)
& = &  v \int_{t=0}^1 [g'(tv) - g'(0)] \, dt - v \int_{t=0}^1 [g'(u + tv) - g'(1) ] \, dt
\\[0.2cm]
& = &  v \int_{t=0}^1 [g'(tv) - g'(0)] \, dt
\\[0.2cm]
& & \qquad - v \int_{t=0}^1 [g'(u + tv) - g'(u) ] \, dt
- v \int_{t=0}^1 [ g'(u) - g'(1) ] \, dt.
\end{array}
\end{equation}
In particular, there exists $C' > 1$ for which
the bound
\begin{equation}
 \abs{ \mathcal{I}_G(u,v) } \le C' \abs{v}\big[\abs{v} + \abs{1 - u} \big]
\end{equation}
holds whenever $\abs{u} \le 2 $ and $\abs{v} \le 2$.
Since $g'(0) > g'(1)$, we hence see that there exists $\delta' > 0$
for which $G(n, \xi) \ge \frac{1}{2}[g'(0) - g'(1)] \Phi(-n + \xi)$
holds whenever
\begin{equation}
\label{eq:lem:ent:bndOnGH:cond}
\abs{ \Phi(-n + \xi) } + \abs{1 - \Phi(n + \xi)} < \delta'.
\end{equation}
Noting that the conditions on $(n, \xi)$ in the statement of this result
imply that $-n + \xi \le -L_2$ and $n + \xi \ge L_2$,
one can guarantee \sref{eq:lem:ent:bndOnGH:cond} by picking $L_2$ sufficiently large.

Similarly,
we write
\begin{equation}
\begin{array}{lcl}
\mathcal{I}_H(u, v) & = & g(u) - g(v) - g(u - v) +  [g'(0) - g'(1) ]v
\\[0.2cm]
& = & - \mathcal{I}_G(u - v, v),
\\[0.2cm]
\end{array}
\end{equation}
which implies
\begin{equation}
\abs{ \mathcal{I}_H(u,v) } \le C' \abs{v} \big[2 \abs{v} + \abs{1 - u } \big].
\end{equation}
Arguing as above, the estimates \sref{eq:lem:ent:bndOnGH:special} now easily follow.
\end{proof}

\begin{lem}
\label{lem:ent:estOnDisplPhis}
Consider the travelling wave MFDE \sref{eq:prlm:trvWaveMFDE}
and suppose that $(Hg)$, $\textrm{(h}\Phi\textrm{)}_{\textrm{\S\ref{sec:prlm}}}$
and $\textrm{(h}\Phi\textrm{)}_{\textrm{\S\ref{sec:ent}}}$ all
hold. Then there exists
a constant $\kappa_3 > 0$ such that we have
\begin{equation}
\frac{\Phi'(\xi_1) - \Phi'(\xi_2) }{\Phi(\xi_1) - \Phi(\xi_2) } \ge \kappa_3, \qquad \xi_2 <  \xi_1 \le 0.
\end{equation}
\end{lem}
\begin{proof}
Pick any $L' >1 $
and consider any pair $(\xi_1, \xi_2) \in \Real^2$ for which
\begin{equation}
\label{eq:lem:ent:estOnDisplPhis:smallDiff}
\xi_1 - L' \le \xi_2 < \xi_1 \le 0.
\end{equation}
The mean value theorem now gives
\begin{equation}
Q(\xi_1, \xi_2) :=  \frac{\Phi'(\xi_1) - \Phi'(\xi_2) }{\Phi(\xi_1) - \Phi(\xi_2) }
 = \frac{(\xi_1 - \xi_2) \Phi''(\theta_1) }{(\xi_1 - \xi_2) \Phi'(\theta_2) }
 = \frac{ \Phi''(\theta_1) }{ \Phi'(\theta_2) } > 0
\end{equation}
for some non-negative pair $(\theta_1, \theta_2) \in \Real^2$
with $\abs{\theta_1 - \theta_2} \le L'$. The asymptotics
\sref{eq:prlm:asymEsts:estOnMinus} now imply
that there exists a constant $\kappa' = \kappa'(L') > 0$
so that $Q(\xi_1, \xi_2) \ge \kappa'(L') $
whenever \sref{eq:lem:ent:estOnDisplPhis:smallDiff} holds.

On the other hand,
the exponential decay of $\Phi'$
stated in \sref{eq:prlm:asymEsts:estOnMinus} ensures that
for $L'$ sufficiently large we have
\begin{equation}
\Phi'(\xi_2) \le \frac{1}{2} \Phi'(\xi_1)
\end{equation}
whenever $\xi_2 + L' \le \xi_1 < 0$. The
desired lower bound now follows from the estimate
\begin{equation}
\frac{\Phi'(\xi_1) - \Phi'(\xi_2) }{\Phi(\xi_1) - \Phi(\xi_2) }
\ge \frac{   \frac{1}{2}  \Phi'(\xi_1) }{ \Phi(\xi_1) } \ge \kappa''
\end{equation}
for some $\kappa'' > 0$, again using
the asymptotics \sref{eq:prlm:asymEsts:estOnMinus} to obtain the second inequality.
\end{proof}

We are now ready to verify that the two functions
$u^\pm$ are a sub and super-solution for \sref{eq:ent:main:lde:nl:coords}.
In view of the preparations above,
the following two results can be established
almost exactly as in \cite[{\S}2.3]{BHM}.
\begin{lem}
\label{lem:ent:subsol}
Consider the
LDE \sref{eq:ent:main:lde:nl:coords}
and suppose that
(Hg),
$\textrm{(hK}\textrm{)}_{\textrm{\S\ref{sec:ent}}}$,
$\textrm{(h}\Phi\textrm{)}_{\textrm{\S\ref{sec:prlm}}}$
and $\textrm{(h}\Phi\textrm{)}_{\textrm{\S\ref{sec:ent}}}$
all hold.

Then there exist $M_0 > 1$ and $T_* \ge T_0(M_0)$ so that
the function $\mathcal{J}^-: (-\infty, -T_*] \to \ell^\infty(\Wholes^2, \Real)$
defined by
\begin{equation}
\mathcal{J}^-_{nl}(t) =
 \dot{u}^-_{nl}(t) - [\Delta^\times u^-(t)]_{nl}
 - g\big(u^-_{nl}(t) \big)
\end{equation}
satisfies the estimate
\begin{equation}
\mathcal{J}^-_{nl}(t) \le 0, \qquad (n,l) \in \Wholes^2, \qquad t \le - T_*.
\end{equation}
\end{lem}
\begin{proof}
For $n \le - \sigma$, we automatically have $\mathcal{J}^-_{nl}(t) = 0$.
For $-\sigma < n \le 0$, we
have
\begin{equation}
\mathcal{J}^-_{nl}(t) = - [\Delta^\times u^-(t)]_{nl} \le 0
\end{equation}
by Lemma \ref{lem:ent:discLaplaceUMinus}.

For $n > 0$ we write
\begin{equation}
\xi = ct - \Xi(t)
\end{equation}
and assume without loss that $\xi < -\sigma$.
We compute
\begin{equation}
\begin{array}{lcl}
\mathcal{J}^-_{nl}(t)
& = & \big(c - \dot{\Xi}(t) \big) [ \Phi'(\xi + n) - \Phi'( \xi - n) ]
 - [\Delta^\times u^-(t)]_{nl}
 - g\big(\Phi(\xi + n) - \Phi(\xi - n) \big)
\\[0.2cm]
& = &
-\dot{\Xi}(t) [\Phi'(\xi + n) - \Phi'(\xi - n)]
+ [\mathcal{I}^-_{\Delta}(t)]_n  - [\Delta^\times u^-(t)]_{nl}
+ H(n, \xi)
\\[0.2cm]
& \le &
-\dot{\Xi}(t) [\Phi'(\xi + n) - \Phi'(\xi - n)]
+ H(n, \xi).
\end{array}
\end{equation}

First, let us consider the case $0 < n \le - \xi$,
for which we obviously have $\xi \pm n \le 0$.
In particular, recalling the
estimates stated in Corollary \ref{cor:prlm:estOnWave},
we may compute
\begin{equation}
\begin{array}{lcl}
\mathcal{J}^-_{nl}(t)
& \le & -M_0 e^{ \eta_0 \xi } e^{ 2 \eta_0 \Xi(t) } \kappa_3 [\Phi(\xi + n) - \Phi(\xi - n)] + C_2 \Phi(\xi - n) [ \Phi(\xi + n) - \Phi(\xi - n) ]
\\[0.2cm]
& \le &
\big[ -M_0 e^{ \eta_0 \xi } e^{ 2 \eta_0 \Xi(t) } \kappa_3
  + C_2 \beta^-_{\mathrm{up}} e^{ - \eta^-_\Phi \abs{\xi  - n } } \big] [ \Phi(\xi + n) - \Phi(\xi - n) ]
\\[0.2cm]
& = &
\big[ -M_0 e^{ \eta_0 \xi } e^{ 2 \eta_0 \Xi(t) } \kappa_3
  + C_2  \beta^-_{\mathrm{up}} e^{\eta^-_\Phi \xi} e^{ - \eta^-_\Phi n } \big] [ \Phi(\xi + n) - \Phi(\xi - n) ]
\\[0.2cm]
& \le &
e^{ \eta_0 \xi} \big[ -M_0 \kappa_3 + C_2  \beta^-_{\mathrm{up}} \big]   [ \Phi(\xi + n) - \Phi(\xi - n) ].
\end{array}
\end{equation}
By picking $M_0 \gg 1$ to be sufficiently large we can hence arrange for $\mathcal{J}^-_{nl}(t) \le 0$
to hold in this regime.

We now study the situation that $n > - \xi \ge \sigma$.
In this case we
have $\xi + n > 0$ and $\xi - n < 0$,
which allows us to estimate
\begin{equation}
\label{eq:ent:lem:subsol:mainCompJ}
\begin{array}{lcl}
\mathcal{J}^-_{nl}(t)
& \le &
-\dot{\Xi}(t) [\Phi'(\xi + n) - \Phi'(\xi - n)]
+ C_2 \Phi(\xi - n) [ \Phi(\xi + n) - \Phi(\xi - n)]
\\[0.2cm]
& \le &
-M_0 e^{ \eta_0 \xi } e^{ 2 \eta_0 \Xi(t) } \big( \alpha^+_{\mathrm{low}} e^{- \eta^+_\Phi \abs{\xi + n}} - \alpha^-_{\mathrm{up}} e^{ - \eta^-_\Phi \abs{ \xi - n} } \big)
 + C_2 \beta^-_{\mathrm{up}} e^{ - \eta^-_\Phi \abs{ \xi - n} }
\\[0.2cm]
& = &
-M_0 e^{\eta_0 \xi}  e^{ 2 \eta_0 \Xi(t) } e^{ - \eta^-_\Phi n}
\big( \alpha^+_{\mathrm{low}} e^{  (\eta^-_\Phi - \eta^+_\Phi) n } e^{ - \eta^+_\Phi \xi}
- \alpha^-_{\mathrm{up}} e^{ \eta^-_\Phi \xi} - C_2 \beta^-_{\mathrm{up}} M_0^{-1} e^{ (\eta^-_\Phi -\eta_0) \xi} e^{ - 2 \eta_0 \Xi(t)} \big)
\\[0.2cm]
& \le &
-M_0 e^{\eta_0 \xi}  e^{ 2 \eta_0 \Xi(t) } e^{ - \eta^-_\Phi n}
\big( \alpha^+_{\mathrm{low}} e^{  (\eta^-_\Phi - \eta^+_\Phi) n } e^{ - \eta^+_\Phi \xi}
- \alpha^-_{\mathrm{up}} e^{ \eta^-_\Phi \xi} - C_2 \beta^-_{\mathrm{up}} M_0^{-1}  \big).
\\[0.2cm]
\end{array}
\end{equation}

If $\eta^-_\Phi \ge \eta^+_\Phi$,
we obtain the bound
\begin{equation}
\begin{array}{lcl}
\mathcal{J}^-_{nl}(t)
& \le &
-M_0 e^{\eta_0 \xi}  e^{ 2 \eta_0 \Xi(t) } e^{ - \eta^-_\Phi n}
\big( \alpha^+_{\mathrm{low}} e^{ - \eta^+_\Phi \xi} - \alpha^-_{\mathrm{up}} e^{ \eta^-_\Phi \xi} - C_2 \beta^-_{\mathrm{up}} M_0^{-1} \big).
\end{array}
\end{equation}
In particular, whenever $\xi \ll -1$ is sufficiently negative to ensure that
\begin{equation}
\alpha^+_{\mathrm{low}} e^{ - \eta^+_\Phi \xi} - \alpha^-_{\mathrm{up}} e^{ \eta^-_\Phi \xi} - C_4 \beta^-_{\mathrm{up}} M_0^{-1} > 0,
\end{equation}
we have $\mathcal{J}^-_{nl}(t) \le 0$. This restriction on $\xi$ can be achieved by choosing $T_*$ to be sufficiently large.

On the other hand, if $\eta^-_\Phi < \eta^+_\Phi$,
we consider two separate cases for $n$.
In particular, recalling the constants $L_2 > 1$ and $\kappa_2$ from Lemma \ref{lem:ent:bndOnGH},
we note that for $n \ge - \xi + L_2$ we have
\begin{equation}
\begin{array}{lcl}
\mathcal{J}^-_{nl}(t)
& \le &  \dot{\Xi}(t)\Phi'(\xi - n) -\kappa_2 \Phi(\xi - n) \\[0.2cm]
& \le &  M_0 e^{ \eta_0 \xi} e^{ 2 \eta_0 \Xi(t)} \alpha^-_{\mathrm{up}}
   e^{ - \eta^-_\Phi \abs{\xi - n} }
 - \kappa_2 \beta^-_{\mathrm{low}} e^{ - \eta^-_\Phi \abs{\xi - n} }
\\[0.2cm]
& = & e^{ - \eta^-_\Phi \abs{\xi - n} } [ M_0 e^{\eta_0\big( ct + \Xi(t)\big) }  \alpha^-_{\mathrm{up}} - \kappa_2 \beta^-_{\mathrm{low}} ].
\end{array}
\end{equation}
In this case, we have $\mathcal{J}^-_{nl}(t) \le 0$
provided $T_*$ is chosen to be sufficiently large to guarantee
that $ct + \Xi(t) \ll - 1$ is always sufficiently negative to have
\begin{equation}
M_0 e^{\eta_0\big( ct + \Xi(t)\big) }  \alpha^-_{\mathrm{up}} - \kappa_2 \beta^-_{\mathrm{low}} \le 0, \qquad t \le -T_*.
\end{equation}
Finally, for $-\xi < n < - \xi + L_2$,
we see from \sref{eq:ent:lem:subsol:mainCompJ}
that
\begin{equation}
\begin{array}{lcl}
\mathcal{J}^-_{nl}(t)
& \le &
-M_0 e^{\eta_0 \xi}  e^{ 2 \eta_0 \Xi(t) } e^{ - \eta^-_\Phi n}
\big( \alpha^+_{\mathrm{low}} e^{  (\eta^-_\Phi - \eta^+_\Phi) n } e^{ - \eta^+_\Phi \xi}
- \alpha^-_{\mathrm{up}} e^{ \eta^-_\Phi \xi} - C_2 \beta^-_{\mathrm{up}} M_0^{-1}  \big)
\\[0.2cm]
& \le &
-M_0 e^{\eta_0 \xi}  e^{ 2 \eta_0 \Xi(t) } e^{ - \eta^-_\Phi n}
\big( \alpha^+_{\mathrm{low}}  e^{ (\eta^-_\Phi - \eta^+_\Phi) (L_2 - \xi) }
                      e^{ - \eta^+_\Phi \xi}
- \alpha^-_{\mathrm{up}} e^{ \eta^-_\Phi \xi} - C_2 \beta^-_{\mathrm{up}} M_0^{-1}  \big)
\\[0.2cm]
& = &
-M_0 e^{\eta_0 \xi}  e^{ 2 \eta_0 \Xi(t) } e^{ - \eta^-_\Phi n}
\big( \alpha^+_{\mathrm{low}}  e^{ (\eta^-_\Phi - \eta^+_\Phi) L_2 }
                      e^{ - \eta^-_\Phi \xi}
- \alpha^-_{\mathrm{up}} e^{ \eta^-_\Phi \xi} - C_2 \beta^-_{\mathrm{up}} M_0^{-1}  \big).
\end{array}
\end{equation}
In particular, whenever $\xi \ll -1$ is sufficiently negative to ensure that
\begin{equation}
\alpha^+_{\mathrm{low}}  e^{ (\eta^-_\Phi - \eta^+_\Phi) L_2 }
                      e^{ - \eta^-_\Phi \xi}
- \alpha^-_{\mathrm{up}} e^{ \eta^-_\Phi \xi} - C_2 \beta^-_{\mathrm{up}} M_0^{-1} \ge 0,
\end{equation}
we have $\mathcal{J}^-_{nl}(t) \le 0$. As before, this restriction on $\xi$ can be achieved by
choosing $T_*$ to be sufficiently large.
\end{proof}

\begin{lem}
\label{lem:ent:supsol}
Consider the
LDE \sref{eq:ent:main:lde:nl:coords}
and suppose that
(Hg),
$\textrm{(hK}\textrm{)}_{\textrm{\S\ref{sec:ent}}}$,
$\textrm{(h}\Phi\textrm{)}_{\textrm{\S\ref{sec:prlm}}}$
and $\textrm{(h}\Phi\textrm{)}_{\textrm{\S\ref{sec:ent}}}$
all hold.

Then there exist $M_0 > 1$ and $T_* \ge T_0(M_0)$ so that
the function $\mathcal{J}^+: (-\infty, -T_*] \to \ell^\infty(\Wholes^2, \Real)$
defined by
\begin{equation}
\mathcal{J}^+_{nl}(t) =
 \dot{u}^+_{nl}(t) - [\Delta^\times u^+(t)]_{nl}
 - g\big(u^+_{nl}(t) \big)
\end{equation}
satisfies the estimate
\begin{equation}
\mathcal{J}^+_{nl}(t) \ge 0, \qquad (n,l) \in \Wholes^2, \qquad t \le - T_*.
\end{equation}
\end{lem}
\begin{proof}
For convenience, we write
\begin{equation}
\xi = ct + \Xi(t)
\end{equation}
and assume that
\begin{equation}
\xi < - \sigma, \qquad \Phi( \xi ) < \frac{a}{2},
\end{equation}
which for every $M_0 > 1$ can be arranged by picking $T_* = T_*(M_0)$
to be sufficiently large. Note in particular
that $g( 2 \Phi(\xi) ) < 0$.

Choose $C' > 1$ in such a way that $g''(u) \le C'$ for $u \le 0 \le 1$.
Remembering that $g(0) = 0$
and $g'(0) < 0$,
we pick $n < 0$
and use Lemma \ref{lem:ent:discLaplaceUPlus}
to obtain the bound
\begin{equation}
\begin{array}{lcl}
\mathcal{J}^+_{nl}(t)
& = &
2\big(c + \dot{\Xi}(t) \big)\Phi'(\xi ) - [\Delta^\times u^+(t)]_{nl}
- g\big( 2 \Phi( \xi ) \big)
\\[0.2cm]
&\ge&
 \big(c + 2 \dot{\Xi}(t) \big) \Phi'( \xi )
+g'(0) \Phi( \xi ) - g\big(2 \Phi( \xi ) \big)
- C_1 e^{ (\eta^-_\Phi + \kappa_\Phi) \xi }
\\[0.2cm]
& \ge &
 \big(c + 2 \dot{\Xi}(t) \big) \Phi'( \xi )
+2 g'(0) \Phi( \xi ) - g\big(2 \Phi( \xi ) \big)
- C_1 e^{ (\eta^-_\Phi + \kappa_\Phi) \xi }
\\[0.2cm]
& \ge &
 \big(c + 2 \dot{\Xi}(t) \big) \Phi'( \xi )
- 2 C' \Phi(\xi)^2
- C_1 e^{ (\eta^-_\Phi + \kappa_\Phi) \xi }
\\[0.2cm]
& \ge &
\big(c + 2 \dot{\Xi}(t) \big) \Phi'( \xi )
- 2 C' [\beta^-_{\mathrm{up}}]^2 e^{ 2 \eta^-_\Phi \xi }
- C_1 e^{ (\eta^-_\Phi + \kappa_\Phi) \xi }
\\[0.2cm]
& \ge &
2 \alpha^-_{\mathrm{low}} M_0 e^{ (\eta_0 + \eta^-_\Phi) \xi}
- 2 C'[\beta^-_{\mathrm{up}}]^2 e^{ 2 \eta^-_\Phi \xi }
- C_1 e^{ (\eta^-_\Phi + \kappa_\Phi) \xi }.
\end{array}
\end{equation}
By picking $M_0 \gg 1$ to be sufficiently small
we can hence guarantee $\mathcal{J}^+_{nl}(t) \ge 0$.

Let us now consider $n \ge 0$,
for which we may compute
\begin{equation}
\begin{array}{lcl}
\mathcal{J}^+_{nl}(t)
& = & (c + \dot{\Xi}(t) ) [\Phi'(\xi + n) + \Phi'(\xi -n)]
- [\Delta^\times u^+(t)]_{nl} - g\big( \Phi(\xi + n) + \Phi(\xi - n) \big)
\\[0.2cm]
& = &
\dot{\Xi}(t) [\Phi'(\xi + n) + \Phi'(\xi -n)]
+ G(n, \xi)
+ [\mathcal{I}^+_\Delta(t)]_n -[\Delta^\times u^+(t)]_{nl}
\\[0.2cm]
& \ge &
\dot{\Xi}(t) [\Phi'(\xi + n) + \Phi'(\xi -n)]
+ G(n, \xi)
  - C_1 e^{-(\eta^-_\Phi + \kappa_\Phi) \abs{ \xi} } \mathbf{1}_{n \in [0, \sigma)}.
\end{array}
\end{equation}
Restricting attention to $0 \le n \le  - \xi$,
for which we have $\xi \pm n \le 0$,
we may estimate
\begin{equation}
\begin{array}{lcl}
\mathcal{J}^+_{nl}(t)
& \ge &
M_0 e^{ \eta_0 \xi } \alpha^-_{\mathrm{low}} e^{ - \eta^-_\Phi \abs{\xi+ n } }
- C_2 [\beta^-_{\mathrm{up}}]^2 e^{ - \eta^-_\Phi \abs{\xi + n} } e^{ - \eta^-_\Phi \abs{\xi - n} }
- C_1 e^{-(\eta^-_\Phi + \kappa_\Phi) \abs{ \xi } } \mathbf{1}_{n \in [0, \sigma)}
\\[0.2cm]
& \ge &
M_0 \alpha^-_{\mathrm{low}} e^{ (\eta_0 + \eta^-_\Phi) \xi}
  - C_2 [\beta^-_{\mathrm{up}}]^2 e^{ 2 \eta^-_\Phi \xi}
  - C_1 e^{ (\eta^-_\Phi + \kappa_\Phi) \xi }.
\end{array}
\end{equation}
Choosing $M_0 \gg 1$ to be sufficiently large again ensures that $\mathcal{J}^+_{nl}(t) \ge 0$.

It remains to consider $n > -\xi > \sigma$.
We now have $\xi - n < 0 < \xi + n$ and compute
\begin{equation}
\label{eq:ent:sol:sup:est:n:large}
\begin{array}{lcl}
\mathcal{J}^+_{nl}(t)
& \ge &
M_0 e^{ \eta_0 \xi } \alpha^+_{\mathrm{low}} e^{ - \eta^+_\Phi \abs{\xi+ n } }
- C_2 \beta^-_{\mathrm{up}} e^{ - \eta^-_\Phi \abs{\xi - n} }
\\[0.2cm]
& \ge &
e^{ \eta_0 \xi}
[M_0 \alpha^+_{\mathrm{low}} e^{-\eta^+_\Phi \xi } e^{ - \eta^+_\Phi n}
- C_2 \beta^-_{\mathrm{up}} e^{ (\eta^-_\Phi - \eta_0) \xi} e^{ - \eta^-_\Phi n }]
\\[0.2cm]
& \ge &
e^{ \eta_0 \xi}
[M_0 \alpha^+_{\mathrm{low}} e^{ - \eta^+_\Phi n}
- C_2 \beta^-_{\mathrm{up}}  e^{ - \eta^-_\Phi n }].
\end{array}
\end{equation}
If $\eta^-_\Phi \ge \eta^+_\Phi$, we
immediately get $\mathcal{J}^+_{nl}(t) \ge 0$
upon picking $M_0 \gg 1$ sufficiently large.

On the other hand, if $\eta^-_\Phi < \eta^+_\Phi$,
we recall the constants $L_2$ and $\kappa_2$
from Lemma \ref{lem:ent:bndOnGH}
and note that for all $n \ge - \xi + L_2 > \sigma$
we have
\begin{equation}
\mathcal{J}^+_{nl}(t)
\ge \dot{\Xi}(t) [\Phi'(\xi + n) + \Phi'(\xi - n) ] + G(n, \xi) \ge 0.
\end{equation}
In addition, it is possible to chose $M_0 \gg 1$
in such a way that
\begin{equation}
M_0 \alpha^+_{\mathrm{low}} e^{ - \eta^+_\Phi n } - C_2 \beta^-_{\mathrm{up}} e^{ - \eta^-_\Phi n} > 0
\end{equation}
for all $n$ in the finite range $- \xi \le n \le - \xi + L_2$.
In this case we also have $\mathcal{J}^+_{nl}(t) \ge 0$
from \sref{eq:ent:sol:sup:est:n:large}.
\end{proof}

We are now ready to prove the existence part of Proposition \ref{prp:ent:entSol}.
The proof uses a limiting procedure to construct the entire solution $U$
from a sequence of solutions that are squeezed
between $u^-$ and $u^+$ on compact intervals that converge to $(-\infty, -T_*]$.
We spell out this limiting procedure in detail in the proof below,
because it will be used several more times in the remainder of this paper.
\begin{lem}
\label{lem:ent:exst}
Consider the setting of Proposition
\ref{prp:ent:entSol}.
There exists a $C^1$-smooth function $U: \Real \to \ell^{\infty}(\Lambda^\times; \Real)$
that satisfies \sref{eq:ent:main:lde:nl:coords},
\sref{eq:lem:ent:exst:limit} and \sref{eq:lem:ent:exst:ineqls}.
\end{lem}
\begin{proof}
Consider the constant $T_*$ and the functions $u^\pm$ defined in
Lemmas \ref{lem:ent:subsol} and \ref{lem:ent:supsol}.
By potentially increasing $T_*$, we may assume that
$c > \dot{\Xi}(t)$ holds for  $t \le -T_*$.

For any integer $k \ge T_*$, we write
\begin{equation}
u^{(k)}: [-k, \infty) \to \ell^{\infty}(\Lambda^\times; \Real)
\end{equation}
for the solution to the obstructed LDE \sref{eq:ent:main:lde:nl:coords} with
\begin{equation}
u^{(k)}_{nl}(-k) = u^-_{nl}(-k), \qquad (n,l) \in \Lambda^\times.
\end{equation}
In particular, for all $t \ge -k$ and $(n,l) \in \Lambda^\times$ we have
\begin{equation}
\label{eq:ent:exst:ldesSatisfied}
\begin{array}{lcl}
\frac{d}{dt} u^{(k)}_{nl}(t) & = & [\Delta^\times_{\Lambda^\times} u^{(k)}(t)]_{nl} + g\big( u^{(k)}_{nl}(t) \big),
\\[0.2cm]
\frac{d^2}{dt^2} u^{(k)}_{nl}(t) & = & [\Delta^\times_{\Lambda^\times} \frac{d}{dt} u^{(k)}(t)]_{nl} + g'\big( u^{(k)}_{nl}(t) \big) \frac{d}{dt} u^{(k)}_{nl}(t).
\\[0.2cm]
\end{array}
\end{equation}

We note that $\dot{u}^-_{nl}(t) = 0$ for $n \le 0$
and
\begin{equation}
\dot{u}^-_{nl}(t) = [c - \dot{\Xi}(t) ] [\Phi(ct - \Xi(t) + n) - \Phi(ct - \Xi(t) - n) ] > 0
\end{equation}
for $n > 0$ and $t \le - T_*$. We hence have
\begin{equation}
\begin{array}{lcl}
\dot{u}^{(k)}_{nl}(-k) & = &
[\Delta^\times_{\Lambda^\times} u^-(-k)]_{nl} + g\big( u^-_{nl}(-k)\big)
\\[0.2cm]
&\ge & \dot{u}^-_{nl}(-k)
\\[0.2cm]
& \ge & 0.
\end{array}
\end{equation}
The comparison principle now implies
\begin{equation}
\label{eq:ent:exst:aprioriBnds}
\dot{u}^{(k)}_{nl}(t) > 0, \qquad 0 < u^{(k)}_{nl}(t) < 1
\end{equation}
for all $t> -k$ and $(n,l) \in \Lambda^\times$.
In addition, another application of the comparison principle yields
\begin{equation}
\label{eq:ent:exst:cmpKvsMinusPlus}
u^-_{nl}(t) \le u^{(k)}_{nl}( t ) < u^+_{nl}(t), \qquad -k \le t \le - T_*, \qquad (n,l)\in \Lambda^\times.
\end{equation}

Fix any interval $[t_0, t_1]$.
Combining \sref{eq:ent:exst:ldesSatisfied} with the bounds \sref{eq:ent:exst:aprioriBnds},
we see that for each fixed $(n,l) \in \Lambda^\times$
the sequence of functions $\{ \big( u^{(k)}_{nl}(t), \frac{d}{dt}u^{(k)}_{nl}(t) \big) \}$
is well-defined for large $k$ and equicontinuous on the interval $t_0 \le t \le t_1$.
In particular, potentially passing to a subsequence we can write
\begin{equation}
\label{eq:ent:exst:conv}
\big( u^{(k)}_{nl}(t), \frac{d}{dt}u^{(k)}_{nl}(t) \big) \to \big(U_{nl}(t), \dot{U}_{nl}(t) \big) \qquad k \to \infty,
\end{equation}
where the convergence is uniform on the interval $t_0 \le t \le t_1$.
Via diagonalization, we can pass to a further subsequence
for which  \sref{eq:ent:exst:conv} holds for all $(n,l) \in \Lambda^\times$ and $t \in \Real$,
which can be taken as the definition of the function $U: \Real \to \ell^\infty(\Lambda^\times; \Real)$.
The convergence \sref{eq:ent:exst:conv} is uniform for finite sets of $(n,l)$ and compact intervals of $t$.
In particular, by taking limits in \sref{eq:ent:exst:ldesSatisfied}
we see that
\begin{equation}
\begin{array}{lcl}
\dot{U}_{nl}(t) & = & [\Delta^\times_{\Lambda^\times} U(t)]_{nl} + g\big( U_{nl}(t) \big), \qquad (n,l) \in \Lambda^\times, \qquad t \in \Real,
\\[0.2cm]
\end{array}
\end{equation}
while taking limits in \sref{eq:ent:exst:aprioriBnds}
yields
\begin{equation}
\label{eq:ent:exst:bndsOnU:not:sharp}
\dot{U}_{nl}(t) \ge 0, \qquad 0 \le U_{nl}(t) \le 1.
\end{equation}
Inspection of the definition of $u^-$ readily yields the uniform limit
\begin{equation}
\sup_{ (n,l) \in \Lambda^\times } \abs{ U_{nl}(t) - \Phi(n + ct) } \to 0,
\qquad t \to -\infty.
\end{equation}
In particular, $U$ is not constant which allows us to sharpen \sref{eq:ent:exst:bndsOnU:not:sharp}
to
\begin{equation}
\dot{U}_{nl}(t) > 0, \qquad 0 < U_{nl}(t) < 1.
\end{equation}
\end{proof}

Throughout the remainder of this section we consider the uniqueness of
the function $U$ defined in Lemma \ref{lem:ent:exst}.
The following result establishes a key compactness property.

\begin{lem}
\label{lem:ent:unq:compact}
Consider the setting of Lemma \ref{lem:ent:exst}.
Then for any $\varphi \in (0, \frac{1}{2}]$,
there exist constants $T_4 = T_4(\varphi) > 1$
and $\kappa_4 = \kappa_4(\varphi) > 0$ such that
\begin{equation}
\dot{U}_{nl}(t) \ge \kappa_4, \qquad t \le -T_4, \qquad (n,l) \in \Omega_{\varphi}(t),
\end{equation}
in which
\begin{equation}
\Omega_{\varphi}(t) = \{ (n,l) \in \Lambda^\times : \varphi \le U_{nl}(t) \le 1 - \varphi \}.
\end{equation}
\end{lem}
\begin{proof}
The uniform convergence \sref{eq:lem:ent:exst:limit}
implies that we can pick $T_4 \gg 1$ and $L_4 \gg 1$ in such a way that
\begin{equation}
\Omega_\varphi(t) \subset  \{ \abs{n + ct} \le L_4 \} \subset  \{ n > 1 \}, \qquad t \le -T_4.
\end{equation}
Seeking a contradiction,
let us consider a sequence $\{(t_k, n_k, l_k)\}_{k \ge 0}$
with $t_k \in (-\infty, -T_4]$ and $(n_k, l_k) \in \Omega_\varphi(t_k)$,
for which $\dot{U}_{n_k, l_k}(t_k) \to 0$ as $k \to \infty$.
Introducing the functions
\begin{equation}
U^{(k)}_{nl}(t) = U_{n +n_k, l + l_k}(t + t_k)
\end{equation}
and arguing as in the proof of Lemma \ref{lem:ent:exst},
we can  pass to a subsequence
for which we have the convergence
\begin{equation}
\label{eq:ent:unq:unf:bound:deriv:conv}
U^{(k)}_{nl}(t) \to U^*_{nl}(t), \qquad k \to \infty
\end{equation}
for some function
\begin{equation}
U^*: \Real \to \ell^{\infty}(\Lambda^* ; \Real),
\end{equation}
where $\Lambda^* = \Wholes^2$ if $\abs{n_k} + \abs{l_k} \to \infty$
or $\Lambda^* = \Lambda^\times - (n_*, l_*)$
if $n_k \to n_*$ and $l_k \to l_*$.
In both cases, we may pass to a further subsequence for which
\begin{equation}
\label{eq:ent:unq:unf:bound:deriv:conv:phase}
n_k + c t_k \to \xi_*, \qquad k \to \infty
\end{equation}
for some $\abs{\xi_*} \le L_4$.
By construction, the function $U^*$ satisfies
\begin{equation}
\dot{U}^*_{nl} (t) = [\Delta^\times_{\Lambda^*} U^*(t)]_{nl} + g\big( U^*_{nl}(t) \big), \qquad (n,l) \in \Lambda^*, \qquad t \in \Real,
\end{equation}
with
\begin{equation}
\dot{U}^*_{nl}(t) \ge 0, \qquad \dot{U}^*_{0, 0}(0) = 0.
\end{equation}
In particular, the comparison principle implies that
\begin{equation}
\label{eq:ent:unq:derivDotUZero}
\dot{U}^*_{nl}(t) = 0, \qquad t \le 0, \qquad (n,l) \in \Lambda^*.
\end{equation}

On the other hand, we claim that we have the uniform convergence
\begin{equation}
\label{eq:ent:unq:unifConvUStar}
\lim_{t \to -\infty} \sup_{(n,l) \in \Lambda^*} \abs{ U^*_{nl}(t) - \Phi(n + ct + \xi_*) } = 0,
\end{equation}
which directly contradicts \sref{eq:ent:unq:derivDotUZero}.
To see \sref{eq:ent:unq:unifConvUStar},
we pick any $\epsilon > 0$ and choose $T'(\epsilon)$ in such a way that
\begin{equation}
\abs{ U_{n + n_k , l + l_k}(t + t_k) - \Phi(n + n_k + c( t + t_k) ) } \le \frac{\epsilon}{3}, \qquad k \ge 0, \qquad t \le -T'(\epsilon),
\end{equation}
which is possible because of \sref{eq:lem:ent:exst:limit} and $t_k \le - T_4$.
Now pick any $t \le -T'(\epsilon)$ and any $(n,l) \in \Lambda^*$. On account of \sref{eq:ent:unq:unf:bound:deriv:conv}
and \sref{eq:ent:unq:unf:bound:deriv:conv:phase},
there exists $k_*$ for which
\begin{equation}
\begin{array}{lcl}
\abs{ \Phi\big(n + n_{k*} + c(t + t_{k_*} ) \big) - \Phi(n + ct + \xi_* ) } & \le & \frac{\epsilon}{3},
\\[0.2cm]
\abs{ U^*_{nl}(t) - U_{n + n_{k_*} , l + l_{k_*} }( t + t_{k_*} ) } & \le & \frac{\epsilon}{3}
\end{array}
\end{equation}
both hold, which implies
\begin{equation}
\abs{U^*_{nl}(t) - \Phi(n + ct + \xi_*) } \le \epsilon.
\end{equation}
This establishes \sref{eq:ent:unq:unifConvUStar} and completes the proof.
\end{proof}

\begin{cor}
\label{cor:ent:unq:subsup}
Consider the setting of Lemma \ref{lem:ent:unq:compact}
and pick a sufficiently small $\varphi > 0$.
Then there exist constants $C_5 = C_5(\varphi) > 1$ and $\eta_5 =\eta_5(\varphi) > 0$
such that for any $0 < \epsilon < \varphi$ and any $t_0 \le - T_4(\varphi) - C_5 \epsilon$,
the functions
\begin{equation}
\begin{array}{lcl}
W^-_{nl}( t ; \epsilon, t_0) & = & U_{n l}\big(t_0 - C_5 \epsilon( 1 - e^{-\eta_5 t}) +  t\big)
 - \epsilon e^{- \eta_5 t},
\\[0.2cm]
W^+_{nl}( t ; \epsilon, t_0) & = & U_{n l}\big(t_0 + C_5 \epsilon( 1 - e^{-\eta_5 t}) +  t\big)
 + \epsilon e^{- \eta_5 t}
\\[0.2cm]
\end{array}
\end{equation}
satisfy the differential inequalities
\begin{equation}
\begin{array}{lcl}
\dot{W}^-_{nl}(t) & \le &  [\Delta^\times_{\Lambda^\times} W^-(t)]_{nl} + g\big(W^-_{nl}(t) \big),
\\[0.2cm]
\dot{W}^+_{nl}(t) & \ge &  [\Delta^\times_{\Lambda^\times} W^+(t)]_{nl} + g\big(W^+_{nl}(t) \big),
\end{array}
\end{equation}
for all $0 \le t \le -T_4(\varphi) - t_0 - C_5 \epsilon$.
\end{cor}
\begin{proof}
One can follow the proof of \cite[Prop 4.3]{HJHNEGDIF}
almost verbatim,
noting that the specified restrictions on $t$ imply that
one can use the lower bound for $\dot{U}$
established in Lemma \ref{lem:ent:unq:compact}.
\end{proof}

\begin{proof}[Proof of Proposition \ref{prp:ent:entSol}  ]
It remains to consider the uniqueness of the entire solution $U$
constructed in Lemma \ref{lem:ent:exst}.
Suppose therefore that $V: \Real \to \ell^{\infty}(\Lambda^\times ; \Real)$
satisfies the obstructed LDE \sref{eq:ent:main:lde:nl:coords}
and obeys the limit \sref{eq:lem:ent:exst:limit}.
Picking $\varphi > 0$ to be sufficiently small,
we note that for any $0 < \epsilon < \varphi$
there exists $t_{\epsilon} \le - T_4 - C_5 \epsilon  $ for which
\begin{equation}
\abs{V_{nl}(t) - U_{nl}(t)} \le \epsilon, \qquad t \le  t_{\epsilon}, \qquad (n,l) \in \Lambda^\times.
\end{equation}
In particular, using the functions $W^\pm$ defined in Corollary \ref{cor:ent:unq:subsup},
we obtain
\begin{equation}
W^-_{nl}(0; \epsilon, t_0) \le V_{nl}(t_0) \le W^+_{nl}(0; \epsilon, t_0)
\end{equation}
for any $t_0 \le t_{\epsilon}$.
In particular, we have
\begin{equation}
W^-_{nl}(t - t_0 ; \epsilon, t_0) \le V_{nl}(t) \le W^+_{nl}(t - t_0; \epsilon, t_0)
\end{equation}
for all $t \in [t_0, -T_4  - C_5 \epsilon]$.
Sending $t_0 \to -\infty$, we obtain
\begin{equation}
U_{nl}(t - C_5 \epsilon) \le V_{nl}(t) \le U_{nl}(t + C_5 \epsilon),
\end{equation}
for all $t \le -T_4 - C_5\epsilon$.
In particular, taking the limit $\epsilon \to 0$
we find $V = U$.
\end{proof}

\section{Various Limits}
\label{sec:lims}
In this section we focus on the entire solution $U$
constructed in Proposition \ref{prp:ent:entSol}
and establish a number of useful limits.
In particular, we show that $U$ resembles the planar travelling wave
at the spatial horizon $\abs{n} + \abs{l} \to \infty$. In addition,
we show that $U$ converges pointwise to a
stationary solution of the obstructed LDE
\sref{eq:ent:main:lde:nl:coords},
which under suitable conditions on the obstacle can be shown
to be equal to one identically.

As in \S\ref{sec:ent},
we fix a direction
$(\sigma_h, \sigma_v) \in \Wholes^2 \setminus \{0, 0\}$
with $\gcd(\sigma_h, \sigma_v) = 1$.
We first state our three main results and then proceed
to prove each of them in turn.

\begin{prop}
\label{prp:lims:SpAsympCmpItv}
Consider the obstructed
LDE \sref{eq:ent:main:lde:nl:coords},
assume that (Hg), (HK1)
and $\textrm{(h}\Phi\textrm{)}_{\textrm{\S\ref{sec:prlm}}}$ with $c > 0$
all hold and recall the entire solution $U$
defined in Proposition \ref{prp:ent:entSol}.
Then for every $\epsilon_2 > 0$ and any pair $t_- \le t_+$,
there exists $R = R(\epsilon_2, t_-, t_+)$ such that
\begin{equation}
\abs{ U_{nl}(t) - \Phi(n + ct) } \le \epsilon_2
\end{equation}
for all $t_- \le t \le t_+$ and $\abs{n} + \abs{l} \ge R$.
\end{prop}

\begin{prop}
\label{prp:ltim:convToStat}
Consider the obstructed
LDE \sref{eq:ent:main:lde:nl:coords},
assume that (Hg), (H$\Phi$) and (HK1)
all hold and recall the entire solution $U$
defined in Proposition \ref{prp:ent:entSol}.
Then for each $(n,l) \in \Lambda^\times$ we have
\begin{equation}
\lim_{t \to \infty} U_{nl}(t) = U_{nl; \infty}
\end{equation}
for some sequence
$U_{;\infty} \in \ell^\infty(\Lambda^\times; \Real)$
that admits the bounds $0 < U_{;\infty} \le 1$,
obeys the limits
\begin{equation}
\label{eq:prp:ltim:spLimitsUInfty}
\lim_{ \abs{n} + \abs{l} \to \infty} U_{nl; \infty} = 1
\end{equation}
and satisfies the stationary problem
\begin{equation}
\label{eq:hor:statProblem}
0 = [\Delta^\times_{\Lambda^\times} U_{;\infty} ]_{nl} + g\big( U_{nl;\infty} \big).
\end{equation}
\end{prop}
In order to state our third main result, we
introduce a slight modification of the assumption (HK2).
In the terminology of (HK2), the condition below states
that the line $\ell \subset \Real^2$
goes through the origin and is oriented in the rational direction
$(\sigma_h, \sigma_v)$.
\begin{itemize}
\item[$\textrm{(hK2}\textrm{)}_{\textrm{\S\ref{sec:lims}}}$]{
   For any $(n,l) \in \partial_\times \Lambda^\times$
   and any
   \begin{equation}
    (n',l') \in K^\times_{\mathrm{obs}} \cap \mathcal{N}^\times_{\Wholes^2}(n,l),
   \end{equation}
   we have the bound
   \begin{equation}
     \abs{n'}  \le \abs{n}.
   \end{equation}
}
\end{itemize}
We emphasize that this simpler condition merely serves to ease
our notation. Indeed, inspection of our proofs readily shows
how the full condition (HK2) can be treated.

\begin{prop}
\label{prp:stp:StSolIsOne}
Consider the setting of
Proposition \ref{prp:ltim:convToStat}
and suppose
the obstacle $K_{\mathrm{obs}}$ also satisfies
$\textrm{(hK2}\textrm{)}_{\textrm{\S\ref{sec:lims}}}$.
Then we have
$U_{nl; \infty} = 1$
for all $(n,l) \in \Lambda^\times$.
\end{prop}

Focussing on Proposition \ref{prp:lims:SpAsympCmpItv},
we first treat the transverse horizon $\abs{l} \to \infty$
using a limiting argument and then construct a super-solution
to study the wave horizon $\abs{n} \to - \infty$.
The approach here strongly resembles
the arguments developed in \cite[{\S}4 and {\S}7.2]{BHM}.

\begin{lem}
\label{lem:hor:convToWave}
Consider the obstructed
LDE \sref{eq:ent:main:lde:nl:coords}
and assume that (Hg), (HK1)
and $\textrm{(h}\Phi\textrm{)}_{\textrm{\S\ref{sec:prlm}}}$ with $c \neq 0$
all hold.
Let $U: \Real \to \ell^\infty(\Lambda^\times; [0, 1] )$
be a solution to the obstructed LDE
\sref{eq:ent:main:lde:nl:coords}
that satisfies the limit
\begin{equation}
\sup_{(n,l) \in \Lambda} \abs{ U_{nl}(t) - \Phi(n + ct) } \to 0, \qquad t \to - \infty.
\end{equation}

Consider any sequence $\{l_k\}_{k \ge 1} \subset \Real$
with $\abs{l_k} \to \infty$ as $k \to \infty$.
Pick $(n, l) \in \Lambda^\times$ and consider any interval $[t_0, t_1] \subset \Real$.
Then we have
\begin{equation}
\label{eq:hor:convToWave:limit}
\sup_{t \in [t_0, t_1]} \abs{ U_{n, l + l_k}(t) - \Phi(n + ct) } \to 0,
\qquad k \to \infty.
\end{equation}
\end{lem}
\begin{proof}
Writing
\begin{equation}
u^{(k)}_{nl}(t) = U_{n, l + l_k}(t),
\end{equation}
we can argue as in the proof of
Lemmas \ref{lem:ent:exst}
and \ref{lem:ent:unq:compact}
to show that, after passing to a subsequence,
we have
\begin{equation}
\label{eq:hor:convToWave:limitToUStar}
\lim_{t \to \infty} u^{(k)}_{nl}(t) = U^*_{nl}(t)
\end{equation}
for all $t \in \Real$ and $(n,l) \in \Wholes^2$.
This convergence is uniform for $t$ in compact intervals
and $(n,l)$ in finite subsets of $\Wholes^2$.
In addition, the $C^1$-smooth function $U^*: \Real \to \ell^\infty(\Wholes^2; [0,1])$
satisfies the unobstructed LDE
\begin{equation}
\dot{U}^*_{nl}(t) = [\Delta^\times U^*(t)]_{nl} + g\big( U^*_{nl}(t) \big), \qquad t \in \Real, \qquad (n,l) \in \Wholes^2
\end{equation}
and enjoys the limit
\begin{equation}
\sup_{(n, l) \in \Wholes^2} \abs{ U^*_{nl}(t) - \Phi(n + ct) } \to 0 \hbox{ as } t \to -\infty.
\end{equation}
A standard argument analogous to
the one used in the proof of Proposition \ref{prp:ent:entSol}
now shows that in fact $U^*_{nl}(t) = \Phi(n + ct)$.
In particular, the convergence
\sref{eq:hor:convToWave:limitToUStar} holds for the entire original
sequence $l_k$.
\end{proof}

We write $(c^\pm_{\delta}, \Phi^\pm_{\delta})$
for the waves defined in
Proposition \ref{prp:prlm:subsupWithDelta}
that satisfy the travelling wave MFDE
\begin{equation}
\begin{array}{lcl}
c^\pm_{\delta}[\Phi^\pm_\delta]'(\xi) & = &
   \Phi^\pm_{\delta}(\xi + \sigma_h) + \Phi^\pm_{\delta}(\xi + \sigma_v)
     + \Phi^\pm_{\delta}(\xi - \sigma_h) + \Phi^\pm_{\delta}(\xi - \sigma_v)
     - 4 \Phi^\pm_{\delta}(\xi)
\\[0.2cm]
& & \qquad
 + g^\pm_\delta\big(\Phi_{\delta}(\xi) \big),
\end{array}
\end{equation}
with the limits
\begin{equation}
\Phi^-_{\delta}( - \infty) = - \delta, \qquad
\Phi^-_{\delta}( + \infty) =  1- \delta
\end{equation}
and
\begin{equation}
\Phi^+_{\delta}( - \infty) = + \delta, \qquad
\Phi^+_{\delta}( + \infty) =  1 + \delta.
\end{equation}
We recall that we can arrange for  $c^\pm_{\delta} > 0$
by appropriately restricting $\delta > 0$.

\begin{proof}[Proof of Proposition \ref{prp:lims:SpAsympCmpItv}]
We study the three regimes $n \gg 1$, $n \ll -1$ and $\abs{l} \gg 1$
separately.
First of all, on account of the backward limit
\sref{eq:lem:ent:exst:limit},
we can pick $T'_1 \ge - t_-$ sufficiently large to ensure that
\begin{equation}
\abs{ U_{nl}(-T'_1) - \Phi(n - c T'_1) } \le \frac{1}{2} \epsilon_2, \qquad (n,l) \in \Lambda^\times.
\end{equation}
In particular, we can pick $N'_1 \gg 1 $ in such a way that
$(n,l) \in \Lambda^\times$ whenever $n \ge N'_1$,
while also
\begin{equation}
\Phi(n - c T'_1) \ge 1 - \epsilon_2, \qquad U_{nl}(-T'_1) \ge 1 - \epsilon_2, \qquad n \ge N'_1.
\end{equation}
Since $\Phi' > 0$,  $\dot{U}_{nl} > 0$ and $t_- \ge - T'_1$, we hence have
\begin{equation}
\label{eq:prf:prp:lims:SpAsympCmpItv:reg1}
\abs{ U_{nl}(t) - \Phi(n + c t) } \le \epsilon_2, \qquad n \ge N'_1, \qquad t \ge t_-.
\end{equation}

Moving on to the $n \ll - 1$ regime,
we pick $\delta = \frac{1}{2} \epsilon_2$. Reasoning similarly as above,
we choose $T'_2 \ge -t_-$ and $N'_2 \gg 1$ in such a way that
\begin{equation}
0 \le U_{nl}(-T'_2) \le \delta, \qquad n \le - N'_2.
\end{equation}
Possibly decreasing $\delta > 0$,
we pick $\vartheta \in \Real$ in such a way that $\Phi^+_\delta( \vartheta) = 1$.
This allows us to write
\begin{equation}
U_{nl}(-T'_2) \le \Phi^+_{\delta}( n + \vartheta + N'_2 ), \qquad (n,l) \in \Lambda^\times.
\end{equation}
For $t \ge -T'_2$ and $n \ge -N'_2$, we have
\begin{equation}
\Phi^+_{\delta}\big(n + \vartheta + N'_2 + c^+_{\delta}(t + T'_2) \big) \ge
\Phi^+_{\delta}\big(-N'_2 + \vartheta + N'_2 \big) = 1.
\end{equation}
This allows us to apply the comparison principle on the region
$t \ge -T'_2$ and $n \le -N'_2$
to conclude
\begin{equation}
U_{nl}(t) \le \Phi^+_{\delta}\big(n + \vartheta + N'_2 + c^+_{\delta}(t + T'_2) \big),
\qquad t \ge -T'_2, \qquad n \le -N'_2.
\end{equation}
In particular, since $\Phi^+_{\delta}(-\infty) = \delta = \frac{1}{2} \epsilon_2$,
there exists $N'_3 \gg 1$ such that
\begin{equation}
0 < U_{nl}(t ) \le 2 \delta \le \epsilon_2,
\qquad 0 < \Phi(n + c t ) \le \epsilon_2,
\qquad n \le - N'_3, \qquad t_- \le t \le t_+,
\end{equation}
which again shows
\begin{equation}
\label{eq:prf:prp:lims:SpAsympCmpItv:reg2}
\abs{ U_{nl}(t) - \Phi(n + c t) } \le \epsilon_2, \qquad n \le - N'_3, \qquad t_- \le t \le t_+.
\end{equation}

We conclude by discussing the regime $\abs{l} \to \infty$. By Lemma
\ref{lem:hor:convToWave},
we see that we can pick $L'_4 \gg 1$ in such a way that
\begin{equation}
\label{eq:prf:prp:lims:SpAsympCmpItv:reg3}
\abs{ U_{nl}(t ) - \Phi(n + c t ) } \le \epsilon_2, \qquad t_- \le t \le t_+
\end{equation}
holds whenever $\abs{l} \ge L'_4$
and $-N'_3 \le n \le N'_1$.
We claim that it now suffices to pick
\begin{equation}
R(\epsilon_2, t_-, t_+) = 2 \max\{ L'_4, N'_3, N'_1 \}
\end{equation}
in order to conclude the proof. To see this,
we note that $\abs{l} + \abs{n} \ge R$
implies that either $\abs{l} \ge R/2 \ge L'_4$
or $\abs{n} \ge R/2 \ge \max\{N'_3, N'_1\}$,
which shows that either
\sref{eq:prf:prp:lims:SpAsympCmpItv:reg1},
\sref{eq:prf:prp:lims:SpAsympCmpItv:reg2}
or
\sref{eq:prf:prp:lims:SpAsympCmpItv:reg3}
is satisfied.
\end{proof}

Turning to Proposition \ref{prp:ltim:convToStat},
we introduce the notation
\begin{equation}
B^\times_R(n_0, l_0) = \{(n, l) \in \Wholes^2 : \abs{n -n_0} + \abs{l - l_0} \le R \}.
\end{equation}
We recall also the definition \sref{sec:ent:defPartialTimes}
for the notation $\partial_\times B^\times_R(n_0, l_0)$.
In view of our preparatory work in \S\ref{sec:expblob},
we can closely follow the lines of \cite[{\S}5]{BHM}.

\begin{lem}
\label{lem:hor:defTR12}
Suppose that (Hg) and (H$\Phi$) are both satisfied
and recall the nonlinearities $g^-_{\delta}$
defined in Proposition \ref{prp:prlm:subsupWithDelta}.
Then for any sufficiently small $\delta > 0$, there exist constants
$T = T(\delta)$, $R_1(\delta)$ and $R_2(\delta)$ with
  \begin{equation}
    R_2(\delta) - R_1(\delta) > 2,
  \end{equation}
  such that the solution to the unobstructed LDE
  \begin{equation}
    \label{eq:hor:defTR12:unobstrLDE}
     \dot{u}^-_{nl} = [\Delta^\times u^-]_{nl} + g^-_{\delta} (u^-_{nl} ), \qquad t \ge 0
  \end{equation}
  with initial condition
  \begin{equation}
    \label{eq:hor:defTR12:unobstrLDE:initVal}
    u^-_{nl}(0) =
    \left\{
      \begin{array}{lcl}
         1 - 2 \delta & & (n, l) \in B^\times_{R_1}(0, 0),
         \\[0.2cm]
          - \delta & &    (n,l) \in \Wholes^2 \setminus B^\times_{R_1}(0, 0),
      \end{array}
    \right.
  \end{equation}
  satisfies
  \begin{equation}
    u^-_{nl}(T) \ge 1 - 2 \delta \hbox{ for all } (n,l) \in B^\times_{R_2}(0, 0).
  \end{equation}
\end{lem}
\begin{proof}
This follows directly from Proposition \ref{prp:expblob:blbExp}
with the nonlinearity $g = g_{\delta}^-$.
Here we exploit the fact that this proposition
only requires the weaker condition $\textrm{(hg}\textrm{)}_{\textrm{\S\ref{sec:prlm}}}$
instead of $(Hg)$.
\end{proof}
The following result is the key technical ingredient to the proof of Proposition
\ref{prp:ltim:convToStat}. It can be seen as a generalization of
the spreading result to the obstructed lattice, provided that
one stays far away from the obstacle.
This provides a mechanism by which we can connect
points far in front of the obstacle
to points far behind the obstacle, where we know that $U$ is large
on account of the backward limit \sref{eq:lem:ent:exst:limit}.

\begin{lem}
\label{lem:hor:obst:blob:still:expands}
Consider the setting of Proposition \ref{prp:ltim:convToStat},
pick a sufficiently small $\delta > 0$
and recall the constants $R_1(\delta) < R_2(\delta)$
and $T(\delta)$ defined in Lemma \ref{lem:hor:defTR12}.
Then there exists $R_3 = R_3(\delta) > R_2(\delta)$
such that the following holds true.

Consider any $(n_0, l_0, t_0) \in \Lambda^\times \times \Real$
for which $B^\times_{R_3(\delta)}(n_0, l_0) \subset \Lambda^\times$
and for which $U_{nl}(t_0) \ge 1 - 2 \delta$ for all $(n,l) \in B^\times_{R_1(\delta)}(n_0, l_0)$.
Then we have
\begin{equation}
U_{nl}\big(t_0 + T(\delta) \big) \ge 1 -2 \delta
\end{equation}
for all $(n, l) \in B^\times_{R_2(\delta)}(n_0, l_0)$.
\end{lem}
\begin{proof}
Consider the solution $u^-$
to the unobstructed initial value problem
\sref{eq:hor:defTR12:unobstrLDE}-\sref{eq:hor:defTR12:unobstrLDE:initVal}.
By choosing $\vartheta \gg 1$ we can arrange for
\begin{equation}
u^-_{nl}(0)  = 1-2 \delta \le \min \{
  \Phi^-_{\delta}( +n  + \vartheta) , \Phi^-_{\delta}( - n + \vartheta),
  \Phi^-_{\delta}( +l + \vartheta), \Phi^-_{\delta}( - l + \vartheta) \}
\end{equation}
to hold for all $(n, l) \in B^\times_{R_1}(0, 0)$.
In addition,
for any $(n,l) \in \Wholes^2 \setminus B^\times_{R_1}(0, 0)$
we have
$u^-_{nl}(0) = - \delta < \Phi^-_{\delta}(\xi)$
for all $\xi \in \Real$.
In particular, we see that
\begin{equation}
u^-_{nl}(0) \le
\min \{
  \Phi^-_{\delta}( +n  + \vartheta) , \Phi^-_{\delta}( - n + \vartheta),
  \Phi^-_{\delta}( +l + \vartheta), \Phi^-_{\delta}( - l + \vartheta) \},
\qquad (n,l) \in \Wholes^2.
\end{equation}
Now, notice that the four functions
\begin{equation}
\Phi^-_{\delta}(\pm n + \vartheta + c_{\delta} t),
\qquad \Phi^-_{\delta}(\pm l + \vartheta + c_{\delta} t)
\end{equation}
all satisfy the
unobstructed LDE
\begin{equation}
  \dot{u}_{nl}(t) = [\Delta^\times u(t)]_{nl} + g^-_{\delta}\big( u_{nl}(t) \big),
\end{equation}
since they represent waves
travelling in the directions $(\sigma_h, \sigma_v)$,
$(-\sigma_h, -\sigma_v)$, $(\sigma_v, -\sigma_h)$
and $(-\sigma_v, \sigma_h)$ in the original $(i,j)$ coordinates.
In particular, for all $t \ge 0$
and all $(n,l) \in \Wholes^2$ we have
\begin{equation}
\label{eq:hor:argsForPhiDelta}
u^-_{nl}(t) \le
\min \{
  \Phi^-_{\delta}( +n  + \vartheta + c^-_{\delta} t) ,
  \Phi^-_{\delta}( - n + \vartheta + c^-_{\delta} t),
  \Phi^-_{\delta}( +l + \vartheta + c^-_{\delta} t),
  \Phi^-_{\delta}( - l + \vartheta + c^-_{\delta} t) \}.
\end{equation}
We now choose $R_3 \gg R_2$ in such a way
that
\begin{equation}
\Phi^-_{\delta}( - \frac{1}{2} R_3 + \abs{\sigma_h } + \abs{ \sigma_v } + \vartheta + c^-_{\delta} T) \le 0.
\end{equation}
This allows us to conclude
\begin{equation}
u^-_{nl}(t) \le 0, \qquad (n,l) \in \partial_\times B^\times_{R_3}(0, 0), \qquad 0 \le t \le T,
\end{equation}
because at least one of the four arguments
of $\Phi^-_{\delta}$ appearing in \sref{eq:hor:argsForPhiDelta}
is less than $-\frac{1}{2}R_3 + \abs{\sigma_h} + \abs{ \sigma_v } +  \vartheta + c^-_{\delta} T$.

By construction, we now have
\begin{equation}
U_{nl}(t_0) \ge u^-_{n-n_0,l-l_0}(t - t_0), \qquad
(n,l) \in B^\times_{R_3}(n_0, l_0),
\end{equation}
because $U_{nl}(t_0) \ge 1 - 2 \delta$ for $(n,l) \in B^\times_{R_1}(n_0, l_0)$
and $u^-_{nl}(0) = - \delta \le 0$
for all $(n,l) \in \Wholes^2 \setminus B^\times_{R_1}(0, 0)$.
In addition, we have
\begin{equation}
U_{nl}(t) \ge 0 \ge u^-_{ n -n_0, l - l_0} ( t - t_0)
\qquad (n,l) \in \partial_\times B^\times_{R_3}( n_0, l_0),
\qquad t_0 \le t \le t_0 + T.
\end{equation}
Now, the properties of $g^-_{\delta}$ imply that
\begin{equation}
\dot{u}^-_{nl}(t) = [\Delta^\times  u^-(t)]_{nl} + g^-_{\delta} \big(u^-_{nl}(t) \big) \le
[\Delta^\times  u^-(t)]_{nl} + g\big(u^-_{nl}(t) \big), \qquad (n,l) \in \Wholes^2, \qquad t \ge 0.
\end{equation}
In addition, since $K^\times_{\mathrm{obs}}$ does not intersect $B^\times_{R_3}(n_0, l_0)$,
we have
\begin{equation}
\dot{U}_{nl} = [\Delta^\times U]_{nl} + g\big(U_{nl}(t)\big), \qquad (n,l) \in B^\times_{R_3}(n_0, l_0) \setminus \partial_\times B^\times_{R_3}(n_0, l_0), \qquad t \ge t_0.
\end{equation}
We now conclude that for all $t_0 \le t \le t_0 + T$
and all $(n,l) \in B^\times_{R_3}(n_0, l_0)$ we have
\begin{equation}
U_{nl}(t) \ge u^-_{n - n_0, l- l_0}(t - t_0),
\end{equation}
which directly implies
\begin{equation}
U_{nl}(t_0 + T) \ge 1 - 2 \delta , \qquad (n,l) \in B^\times_{R_2}(n_0, l_0).
\end{equation}
\end{proof}

\begin{proof}[Proof of Proposition \ref{prp:ltim:convToStat}  ]
The fact that $\dot{U}_{nl}(t) \ge 0$ for all $t \in \Real$ and $(n,l) \in \Lambda^\times$
implies that
\begin{equation}
U_{nl;\infty} = \lim_{t \to \infty} U_{nl}(t)
\end{equation}
is well-defined and satisfies $0 < U_{nl;\infty} \le 1$ for all $(n,l) \in \Lambda^\times$.
To see that $U_{;\infty}$ satisfies
the stationary problem \sref{eq:hor:statProblem},
let us suppose to the contrary
 that
for some $(n_*,l_*) \in \Lambda^\times$ we have
\begin{equation}
[\Delta^\times_{\Lambda^\times} U_{; \infty}]_{n_*, l_*} + g\big( U_{n_*, l_* ; \infty} \big)
  = \kappa' \neq 0.
\end{equation}
Picking $T \gg 1$ sufficiently large, we can ensure
\begin{equation}
\abs{ [\Delta^\times_{\Lambda^\times} U(t)]_{n_* l_*} + g\big( U_{n_*  l_* }(t) \big) } \ge \frac{1}{2} \abs{ \kappa' } > 0
\end{equation}
for all $t \ge T$. In particular, this would imply that
$\dot{U}_{n_*, l_*}(t) \ge \frac{1}{2} \abs{\kappa'} > 0$ for all $t \ge T$,
a clear contradiction.

It remains to show that $\lim_{\abs{n} + \abs{l}  \to \infty} U_{nl;\infty} = 1$.
To this end, pick any sufficiently small $\delta > 0$
and recall the constants $R_1 = R_1(\delta)$,
$R_2 = R_2(\delta)$,  $R_3 = R_3(\delta)$
and $T = T(\delta)$ defined in Lemma \ref{lem:hor:obst:blob:still:expands}.
Because of the limit
\sref{eq:lem:ent:exst:limit},
we can certainly find $(n_0, l_0, t_0)$
in such a way that $B^\times_{R_3}(n_0, l_0) \subset \Lambda^\times$
and $U_{nl}(t_0) \ge 1 - 2 \delta$ for $(n,l) \in B^\times_{R_1}(n_0, l_0)$.
Now pick any $(n,l)$
with $\abs{n} + \abs{l}$ sufficiently large
to ensure that $B^\times_{R_3}(n,l) \subset \Lambda^\times$.
We can then find a finite sequence
\begin{equation}
 (n_0, l_0), (n_1, l_1), \ldots , (n_k, l_k) = (n,l)
\end{equation}
with $B^\times_{R_3}(n_i, l_i) \subset \Lambda^\times$
for $0 \le i \le k$ and
\begin{equation}
(n_i, l_i) \subset B^\times_{R_2}(n_{i-1}, l_{i-1}), \qquad 1 \le i \le k.
\end{equation}
In particular, we see $U_{nl}(t_0 + k T) \ge 1 - 2 \delta$,
which in view of $\dot{U}_{nl}(t) \ge 0$ implies
$U_{nl; \infty} \ge 1 - 2 \delta$.
\end{proof}

We now turn our attention to the proof of Proposition \ref{prp:stp:StSolIsOne}.
The material in \cite[{\S}6]{BHM} needs to be adapted in order to
accommodate the fact that the state space for the travelling wave MFDE
is infinite dimensional. In particular,
there is no analogue of the function \cite[(6.4)]{BHM} available
for use in our setting.

For any $\vartheta \in \Real$, we introduce the shorthand
\begin{equation}
\xi_{nl} = \xi_{nl; \vartheta} = \abs{n}  - \vartheta
\end{equation}
and introduce the sequence $w_{; \vartheta} \in \ell^\infty(\Wholes^2; \Real)$
that is given by
\begin{equation}
w_{nl} = w_{nl; \vartheta} = \Phi\big( \xi_{nl; \vartheta} \big).
\end{equation}
A short computation yields
\begin{equation}
\begin{array}{lcl}
[\Delta^\times w]_{nl}
 & = & \Phi( \xi_{nl} +  \abs{n + \sigma_h} - \abs{n} )
 + \Phi( \xi_{nl} +      \abs{n - \sigma_h} - \abs{n} )
\\[0.2cm]
& & \qquad
 + \Phi( \xi_{nl} +      \abs{n + \sigma_v} - \abs{n} )
 + \Phi( \xi_{nl} +      \abs{n - \sigma_v} - \abs{n} )
\\[0.2cm]
& & \qquad
 - 4 \Phi(\xi_{nl}).
\end{array}
\end{equation}
Now, for any $\widetilde{\sigma} \ge 0$ we have
\begin{equation}
\big( \abs{n + \widetilde{\sigma}} - \abs{n} , \abs{n - \widetilde{\sigma}} - \abs{n} \big)
= \left\{ \begin{array}{lcl}
     ( +\widetilde{\sigma} , -\widetilde{\sigma}), & &  n \ge \widetilde{\sigma}, \\
     ( +\widetilde{\sigma} , \widetilde{\sigma} - 2n ), & & 0 \le n < \widetilde{\sigma}, \\
     ( 2n + \widetilde{\sigma}, \widetilde{\sigma} ),   & & -\widetilde{\sigma} < n \le 0, \\
     ( -\widetilde{\sigma} , +\widetilde{\sigma}), & &  n \le - \widetilde{\sigma}. \\
     \end{array}
  \right.
\end{equation}
Assuming for definiteness that $\sigma_h \ge 0$ and $\sigma_v \ge 0$,
we hence have
\begin{equation}
\begin{array}{lcl}
[\Delta^\times w]_{nl}
& = & \Phi(\xi_{nl} + \sigma_h ) + \Phi(\xi_{nl} - \sigma_h)
   + \Phi(\xi_{nl} + \sigma_v) + \Phi(\xi_{nl} - \sigma_v) - 4 \Phi(\xi_{nl})
\\[0.2cm]
& & \qquad + [ \Phi(\xi_{nl} + \sigma_h - 2n ) - \Phi(\xi_{nl} - \sigma_h) ]\mathbf{1}_{0 \le n < \sigma_h }
\\[0.2cm]
& & \qquad + [ \Phi(\xi_{nl} + \sigma_h + 2n ) - \Phi(\xi_{nl} - \sigma_h) ]\mathbf{1}_{-\sigma_h < n < 0 }
\\[0.2cm]
& & \qquad + [ \Phi(\xi_{nl} + \sigma_v - 2n ) - \Phi(\xi_{nl} - \sigma_v) ]\mathbf{1}_{0 \le n < \sigma_v }
\\[0.2cm]
& & \qquad + [ \Phi(\xi_{nl} + \sigma_v + 2n ) - \Phi(\xi_{nl} - \sigma_v) ]\mathbf{1}_{-\sigma_v < n < 0 }
\\[0.2cm]
& \ge &
\Phi(\xi_{nl} + \sigma_h ) + \Phi(\xi_{nl} - \sigma_h)
   + \Phi(\xi_{nl} + \sigma_v) + \Phi(\xi_{nl} - \sigma_v) - 4 \Phi(\xi_{nl}),
\end{array}
\end{equation}
on account of the fact that $\Phi$ is a strictly increasing function.

If $\textrm{(hK2}\textrm{)}_{\textrm{\S\ref{sec:lims}}}$
is satisfied, we may
now estimate
\begin{equation}
\begin{array}{lcl}
[\Delta^\times_{\Lambda^\times} w]_{nl}
& = & [\Delta^\times w]_{nl}
- \sum_{(n', l') \in \mathcal{N}^\times_{\Wholes^2}(n,l) \cap K^\times_{\mathrm{obs}}}
[ w_{ n' l' } - w_{nl} ]
\\[0.2cm]
& = &
[\Delta^\times w]_{nl}
- \sum_{(n', l') \in \mathcal{N}^\times_{\Wholes^2}(n,l) \cap K^\times_{\mathrm{obs}}}
[ \Phi(\xi_{n' l'} ) - \Phi(\xi_{nl} ) ]
\\[0.2cm]
& \ge &
[\Delta^\times w]_{nl},
\\[0.2cm]
\\[0.2cm]
\end{array}
\end{equation}
where the inequality is a consequence of
the bound $\xi_{n' l' } \le \xi_{nl}$.
In particular, we see that
\begin{equation}
\begin{array}{lcl}
[\Delta^\times_{\Lambda^\times} w]_{nl} + g ( w_{nl} )
& \ge &
\Phi(\xi_{nl}  + \sigma_h) + \Phi(\xi_{nl} - \sigma_h)
+ \Phi(\xi_{nl} + \sigma_v) + \Phi(\xi_{nl} - \sigma_v)
 - 4 \Phi(\xi_{nl})
\\[0.2cm]
& & \qquad + g\big( \Phi(\xi_{nl}) \big)
\\[0.2cm]
& = &  c \Phi'(\xi_{nl}).
\end{array}
\end{equation}
Upon writing
\begin{equation}
u^-_{nl; \vartheta}(t) = \Phi\big(\abs{n} + \frac{c}{2} t - \vartheta \big)
= \Phi(\xi_{nl;\vartheta} + \frac{c}{2} t ),
\end{equation}
we immediately see
\begin{equation}
\dot{u}^-_{nl;\vartheta}(t) - [\Delta^\times_{\Lambda^\times} u^-_{; \vartheta}(t)]_{nl} - g\big( u^-_{nl;\vartheta}(t) \big)
\le -  \frac{c}{2} \Phi'\big( \xi_{nl;\vartheta} + \frac{c}{2} t \big) < 0,
\end{equation}
implying that $u^-_{; \vartheta}$ is a sub-solution
to the obstructed LDE \sref{eq:ent:main:lde:nl:coords}
for every $\vartheta \in \Real$,
with
\begin{equation}
u^-_{nl; \vartheta}(0) = w_{nl;\vartheta},
\qquad \lim_{t \to \infty} u^-_{nl; \vartheta}(t) = 1
\end{equation}
for every $(n,l) \in \Lambda^\times$.

\begin{proof}[Proof of Proposition \ref{prp:stp:StSolIsOne}]
In view of the discussion above, it suffices to choose $\vartheta \gg 1$
in such a way that
\begin{equation}
\label{eq:prf:stSollIsOne:idToProve}
U_{nl;\infty} \ge w_{nl; \vartheta}
\end{equation}
holds for all $(n,l) \in \Lambda^\times$.

For any $R > 0$, we introduce the set
\begin{equation}
\begin{array}{lcl}
\Omega^\times_R & = & \{(n,l) \in \Wholes^2 : \abs{n} + \abs{l} \ge R \}.
\\[0.2cm]
\end{array}
\end{equation}
Now, pick $0 < \delta < \frac{1}{2} $ in such a way that $g$ is
strictly decreasing on $[1 - \delta, 1 + \delta]$.
In addition, pick $R \ge 1$ in such a way that
we have $K^\times_{\mathrm{obs}} \cap \Omega^\times_R = \emptyset$
together with
\begin{equation}
U_{nl;\infty} \ge 1 - \delta \qquad (n,l) \in \Omega^\times_R.
\end{equation}
Shifting the wave profile $\Phi$ in such a way that $\Phi(0) = \frac{1}{2}$,
we now pick $\vartheta \gg 1$
sufficiently large to ensure that
\begin{equation}
\vartheta \ge \abs{n}, \qquad (n,l) \in \partial_\times \Omega^\times_R.
\end{equation}

We now claim that
\sref{eq:prf:stSollIsOne:idToProve}
holds for all $(n,l) \in \Omega^\times_R$.
To see this, note first that for $(n,l) \in \partial_\times \Omega^\times_{R}$,
our construction yields
\begin{equation}
\label{eq:stat:lem:bnd:on:st:boundary}
U_{nl;\infty} > \frac{1}{2} = \Phi(0) \ge  w_{nl ; \vartheta} =
\Phi( \abs{n} - \vartheta ).
\end{equation}
On the other hand, our choice for $R$ guarantees
\begin{equation}
1 + \delta \ge U_{nl;\infty} + \delta \ge 1 \ge w_{nl ; \vartheta}
\end{equation}
for all $(n,l) \in \Omega^\times_R$. In particular, we can define
the quantity
\begin{equation}
\epsilon_* = \inf\{ \epsilon \ge 0 :
  U_{nl;\infty} + \epsilon \ge w_{nl; \vartheta} \hbox{ for all } (n,l) \in \Omega^\times_R \}.
\end{equation}
By continuity, it suffices to show that $\epsilon_* = 0$
in order to prove our claim.
Assuming to the contrary that $\epsilon_* > 0$,
we note that the set
\begin{equation}
\widetilde{\Omega}^\times_R := \{(n, l) \in \Omega^\times_R : U_{nl} < 1 - \frac{\epsilon_*}{2} \}
\end{equation}
is bounded because of the spatial limits
\sref{eq:prp:ltim:spLimitsUInfty} satisfied by $U$.
In particular, there exists a pair $(n_*, l_*) \in \widetilde{\Omega}^\times_R$ for which
\begin{equation}
U_{n_* l_*; \infty} + \epsilon_* = w_{n_* l_*; \vartheta}.
\end{equation}
In addition, the inequality \sref{eq:stat:lem:bnd:on:st:boundary}
shows that $(n_*, l_*) \notin \partial_\times \Omega^\times_{R}$,
which by definition implies
\begin{equation}
\mathcal{N}^\times_{\Lambda^\times}(n_*, l_*)
=  \mathcal{N}^\times_{\Wholes^2}(n_*, l_*) \subset \Omega^\times_R.
\end{equation}
Finally, by continuity we also have
\begin{equation}
U_{nl;\infty} + \epsilon_* \ge w_{nl ; \vartheta}, \qquad (n,l) \in \Omega^\times_R.
\end{equation}
Remembering that $g$ is strictly decreasing on $[1-\delta, 1 + \delta]$,
we may now estimate
\begin{equation}
\begin{array}{lcl}
- g(U_{n_* l_*; \infty} + \epsilon_*)
&  > & - g( U_{n_*, l_*; \infty} )
\\[0.2cm]
& = & [\Delta^\times U_{; \infty}]_{n_* l_*}
\\[0.2cm]
& = & \big[\Delta^\times [U_{; \infty} + \epsilon^*]\big]_{n_* l_*}
\\[0.2cm]
& \ge & [\Delta^\times w_{\cdot ; \vartheta}]_{n_* l_*}
\\[0.2cm]
& \ge & - g( w_{n_* l_* ; \vartheta} )
\\[0.2cm]
& = & - g( U_{n_* l_*; \infty} + \epsilon_* ),
\end{array}
\end{equation}
which clearly is impossible.

It remains to establish
\sref{eq:prf:stSollIsOne:idToProve}
for the bounded set $(n,l) \in \Lambda^\times \setminus \Omega^\times_R$.
Note that the connectedness assumption (HK1)
in combination with Corollary \ref{cor:prlm:cmpStrong}
implies that $U_{nl; \infty} > 0$ for all $(n,l) \in \Lambda^\times$.
In particular, we can pick $C' > 1$
in such a way that
\begin{equation}
\Phi( - C' ) < \min \{ U_{nl;\infty} : (n,l) \in \Lambda^\times \setminus \Omega^\times_R \}.
\end{equation}
Possibly increasing $\vartheta$,
we can ensure that
\begin{equation}
\xi_{nl; \vartheta} \le - C' , \qquad (n,l) \in \Lambda^\times \setminus \Omega^\times_R,
\end{equation}
which suffices to
establish \sref{eq:prf:stSollIsOne:idToProve} and complete the proof.
\end{proof}

\section{Proof of Theorem \ref{thm:mr:resKConvex}}
\label{sec:pmr}

We are finally ready to tackle
the second main result of this paper.
In view of the preparatory work in \S\ref{sec:lims},
it will suffice to establish the following result,
which is the analogue of \cite[Thm. 7.1]{BHM}.
\begin{prop}
\label{prp:afterObst:conv}
Consider any angle $\zeta_*$ with $\tan \zeta_* \in \mathbb{Q}$
and suppose that (Hg) and $(HS)_{\zeta_*}$
both hold. Pick $(\sigma_h, \sigma_v) \in \Wholes^2 \setminus \{(0 , 0)\}$
with the property that
\begin{equation}
  \sqrt{\sigma_h^2 + \sigma_v^2}(\cos \zeta_*, \sin \zeta_*) = (\sigma_h, \sigma_v),
  \qquad
  \mathrm{gcd}(\sigma_h, \sigma_v) = 1
\end{equation}
and suppose that $\textrm{(h}\Phi\textrm{)}_{\textrm{\S\ref{sec:prlm}}}$
holds for this pair $(\sigma_h, \sigma_v)$ with $c > 0$.

Suppose that $U: [0, \infty) \to \ell^{\infty}(\Lambda^\times ; [0,1])$
is a $C^1$-smooth solution to the obstructed LDE
\sref{eq:ent:main:lde:nl:coords} for an obstacle that satisfies (HK1).
Suppose furthermore that for every $\epsilon_2 > 0$,
there exists $t_{\epsilon_2} \ge 0$ and a bounded set $K^\times_{\epsilon_2} \supset K^\times_{\mathrm{obs}} $
such that
\begin{equation}
\label{eq:prp:afterObst:conv:critOnCmpSet}
\abs{U_{nl}(t_{\epsilon_2}) - \Phi( n + c t_{\epsilon_2} ) } \le \epsilon_2
\end{equation}
holds for all $(n,l) \in \Wholes^2 \setminus K^\times_{\epsilon_2}$,
while
\begin{equation}
U_{nl}(t) \ge 1 -\epsilon_2
\end{equation}
holds for all $t \ge t_{\epsilon_2}$ and $(n,l) \in \partial_\times \Lambda^\times$.

Then we have the uniform convergence
\begin{equation}
\sup_{(n,l) \in \Lambda^\times} \abs{U_{nl}(t) - \Phi(n + ct) } \to 0 , \qquad t \to \infty.
\end{equation}
\end{prop}
Naturally, we intend to exploit the sub and super-solutions
constructed in Proposition \ref{prp:hom:oblq:mr:sub:sup}
in order to establish the result above.
Our main task in this section
is therefore to construct suitable $C^1$-smooth functions $z:[0, \infty) \to \Real$
that will allow us to absorb error terms caused by the obstacle
into the term $\frac{1}{2} \eta_z z(t)$ that we have to spare in
\sref{eq:prp:oblq:subsub:termsToSpare}.
It is important to note that such estimates
are required only when condition (b) in Proposition \ref{prp:prlm:cmpPrinciple}
fails, i.e., when our sub-solution is larger
than $1 - \epsilon_2$.

This latter event occurs at some time $t_1 \ge t_{\epsilon_2}$
for which no a-priori upper bound is available.
In particular, up to a scaling factor, our function $z(t)$
will follow $z_{\mathrm{hom}}(t)$ until $t$ approaches $t_1$,
after which it increases in a short time interval back to $\frac{2}{3}z_{\mathrm{hom}}(0)$
and then resets back to following $\frac{2}{3}z_{\mathrm{hom}}(t - t_1)$ for $t \ge t_1$.
This way, we can ensure that $z(t)$ is sufficiently large
for the critical time period $t \ge t_1$ where the effects of the obstacle
play a role.

\begin{lem}
\label{lem:pmr:cnstIntPols:minus}
Fix any $0 < \eta_z < 1$. Then there exists a constant $\ell_P = \ell_P(\eta_z) > 0$
together with a polynomial $P_-$
that satisfies the identities
\begin{equation}
\label{eq:pmr:defPminus:ids}
P_-(-\ell_P) = \frac{3}{2}, \qquad P_-'(-\ell_P) = 0,
\qquad P_-(0) = 1,
\qquad P_-'(0) = - \eta_z,
\end{equation}
together with the bounds
\begin{equation}
0 \ge P_-'(x) \ge -\eta_z P_-(x), \qquad -\ell_P \le x \le 0.
\end{equation}
\end{lem}
\begin{proof}
Choosing $\ell_P = \eta_z^{-1}$,
we write
\begin{equation}
P_-(x) =  - \frac{1}{2} \frac{(x+\ell_P)^2}{\ell_P^2} + \frac{3}{2},
\end{equation}
from which the identities \sref{eq:pmr:defPminus:ids} immediately follow.
We now compute
\begin{equation}
\begin{array}{lcl}
\frac{d}{dx}\frac{P_-'(x)}{P_-(x)}
& = & -2 \frac{ (x + \ell_P)^2 + 3 \ell_P^2 }{ \big(( x + \ell_P)^2 - 3 \ell_P^2 \big) ^2 },
\\[0.2cm]
\end{array}
\end{equation}
which shows that for $-\ell_P \le x \le 0$ we have
\begin{equation}
\begin{array}{lcl}
\frac{P_-'(x)}{P_-(x)}
& \ge &  \frac{P_-'(0)}{P_-(0)} = - \eta_z,
\\[0.2cm]
\end{array}
\end{equation}
as desired.
\end{proof}

\begin{lem}
\label{lem:pmr:cnstIntPols:plus}
Fix any $0 < \eta_z < 1$ and recall the constant $\ell_P > 0$
defined in Lemma \ref{lem:pmr:cnstIntPols:minus}.
Then for every $0 < \nu \le \eta_z$,
there exists a polynomial $P_{+; \nu}$
that satisfies the identities
\begin{equation}
\label{eq:pmr:defPplus:ids}
P_{+;\nu}(0)= 1,
\qquad P_{+;\nu}'(0) = - \nu,
\qquad P_{+;\nu}'(\ell_P) = 0,
\end{equation}
together with the bounds
$P_{+;\nu}(\ell_P) \ge \frac{1}{2}$
and
\begin{equation}
0 \ge P_{+;\nu}'(x) \ge -\eta_z P_{+;\nu}(x), \qquad 0 \le x \le \ell_P.
\end{equation}
\end{lem}
\begin{proof}
Upon writing
\begin{equation}
P_{+;\nu}(x) = \frac{\nu}{2 \ell_P} (x - \ell_P)^2 + 1 - \frac{\nu \ell_P}{2}
\end{equation}
and remembering $\ell_P = \eta_z^{-1}$,
we immediately see that the identities \sref{eq:pmr:defPplus:ids}
are satisfied, together with the bound
\begin{equation}
P_{+; \nu}(\ell_P) = 1 - \frac{\nu \ell_P}{2} \ge \frac{1}{2}.
\end{equation}
We now compute
\begin{equation}
\begin{array}{lcl}
\frac{d}{dx}\frac{P_{+;\nu}'(x)}{P_{+;\nu}(x)}
& = & -2 \frac{ (x - \ell_P)^2 - \ell_P^2 ( \frac{2}{\nu \ell_P} - 1 ) }
  { \big(( x - \ell_P)^2 +  \ell_P^2( \frac{2}{\nu \ell_P} - 1)  \big) ^2 }.
\end{array}
\end{equation}
In particular, for $0 \le x \le \ell_P$ we have
\begin{equation}
\begin{array}{lcl}
\frac{P_{+;\nu}'(x)}{P_{+;\nu}(x)}
& \ge & \frac{P_{+;\nu}'(0)}{P_{+;\nu}(0)} = - \nu \ge -\eta_z.
\end{array}
\end{equation}
\end{proof}

We are now ready to construct
template functions $z_{\mathrm{obs} ;t_1}(t)$,
based on
the corresponding function $z_{\mathrm{hom}}(t)$ that was defined in Lemma
\ref{lem:oblq:defZHom}
for the homogeneous lattice.
The crucial point in the result below
is that the constants $\kappa_{\mathrm{obs}}$ and $\mathcal{I}_{\mathrm{obs}}$
do not depend on the size of $t_1$,
which in the sequel we will need to be arbitrarily large.
\begin{lem}
\label{lem:pmr:templateObs}
Fix any $0 < \eta_z < 1$. Then there exists constants
$\mathcal{I}_{\mathrm{obs}} = \mathcal{I}_{\mathrm{obs}}(\eta_z) > 1$ and $\kappa_{\mathrm{obs}} = \kappa_{\mathrm{obs}}(\eta_z) > 0$,
such that for any $t_1 \ge 0$ there exists a $C^1$-smooth function
$z_{\mathrm{obs};t_1}: [0, \infty) \to \Real$
that satisfies the following properties.
\begin{itemize}
\item[(i)]{
  We have $z'_{\mathrm{obs};t_1}(t) \ge -\eta_z z_{\mathrm{obs}; t_1}(t)$ for all $t \ge 0$.
}
\item[(ii)]{
  We have $0 < z_{\mathrm{obs}; t_1}(t ) \le z_{\mathrm{obs}; t_1}(0) = 1$ for all $ t \ge 0$.
}
\item[(iii)]{
  We have $z_{\mathrm{obs}; t_1}(t) \ge \frac{1}{2} z_{\mathrm{hom}}(t)$ for all $t \ge 0$.
}
\item[(iv)]{
  We have $z_{\mathrm{obs}; t_1}(t) \ge \kappa_{\mathrm{obs}} (1 + t - t_1)^{-3/2}$ for all $t \ge t_1$.
}
\item[(v)]{
  We have $\int_{0}^{\infty} z_{\mathrm{obs}; t_1}(t) \, d t < \mathcal{I}_{\mathrm{obs}}$.
}
\end{itemize}
\end{lem}
\begin{proof}
Recall the constant $\ell_P$
and the polynomials defined in Lemmas \ref{lem:pmr:cnstIntPols:minus}-\ref{lem:pmr:cnstIntPols:plus}.
If $0 \le t_1 \le 3 \ell_P$, we define
$z_{\mathrm{obs}; t_1}(t) = z_{\mathrm{hom}}(t)$ and observe
that the properties (i) through (v) follow immediately
from Lemma \ref{lem:oblq:defZHom}.

On the other hand, if $t_1 > 3 \ell_P$, we
define the function $z_{\mathrm{obs}; t_1}$
separately on five different intervals.
In particular, for $0 \le t \le t_1 - 3\ell_P$, we write
\begin{equation}
z_{\mathrm{obs}; t_1}(t) = z_{\mathrm{hom}}(t)
\end{equation}
and define
\begin{equation}
\nu = - z'_{\mathrm{hom}}(t_1 - 3 \ell_P ) / z_{\mathrm{hom}}(t_1 - 3 \ell_P ),
\end{equation}
which implies $0 < \nu \le \eta_z$.
For $t_1 - 3 \ell_P \le t_1 -2 \ell_P$, we write
\begin{equation}
z_{\mathrm{obs}; t_1}(t) = z_{\mathrm{hom}}(t_1 - 3 \ell_P )
   P_{+;\nu}\big(t - (t_1 - 3\ell_P) \big),
\end{equation}
while for $t_1 - \ell_P \le t \le t_1$, we write
\begin{equation}
z_{\mathrm{obs}; t_1}(t) =  \frac{2}{3} P_-(t - t_1).
\end{equation}
Finally, for $t \ge t_1$, we write
\begin{equation}
z_{\mathrm{obs}; t_1}(t) = \frac{2}{3} z_{\mathrm{hom}}(t - t_1).
\end{equation}
It remains to specify $z_{\mathrm{obs}; t_1}(t)$ for $t$ between $t_1 - 2\ell_P$ and $t_1 - \ell_P$.
This can be done in an arbitrary $C^1$-smooth  fashion, under the constraints
\begin{equation}
z_{\mathrm{obs}; t_1}(t_1 - 2 \ell_P) =
z_{\mathrm{hom}}(t_1 - 3 \ell_P )P_{+;\nu}(\ell_P),
\qquad
z_{\mathrm{obs}; t_1}(t_1 - \ell_P) = 1
\end{equation}
together with
\begin{equation}
z'_{\mathrm{obs}; t_1}(t_1 - 2\ell_P ) = z'_{\mathrm{obs}; t_1}(t_1 - \ell_P ) = 0
\end{equation}
and
\begin{equation}
z'_{\mathrm{obs}; t_1}(t ) \ge 0, \qquad t_1 - 2 \ell_P \le t \le t_1 - \ell_P.
\end{equation}
The properties (i) through (iv) follow directly from this construction,
utilizing the observation
\begin{equation}
z_{\mathrm{obs}; t_1}(t_1 - 2 \ell_P)
\ge \frac{1}{2} z_{\mathrm{hom}}(t_1 - 3 \ell_P ).
\end{equation}
In addition, one readily obtains the bound
\begin{equation}
\int_{t =0}^{\infty} z_{\mathrm{obs}; t_1}(t) \, dt \le 2 \int_{t = 0}^{\infty} z_{\mathrm{hom}}(t) \, dt + 3 \ell_P,
\end{equation}
which establishes (v).
\end{proof}

\begin{proof}[Proof of Proposition \ref{prp:afterObst:conv}]
Pick any $\delta_* > 0$. We restrict ourselves here to showing that
\begin{equation}
\label{eq:prp:afterObst:finLimitToShow}
\liminf_{t \to \infty} \inf_{(n,l) \in \Lambda^\times} [ U_{nl}(t) - \Phi(n + ct ) ] \ge - \delta_*,
\end{equation}
noting that the companion bound
\begin{equation}
\label{eq:prp:afterObst:finLimitToIgnore}
\limsup_{t \to \infty} \sup_{(n,l) \in \Lambda^\times} [ U_{nl}(t) - \Phi(n + ct ) ] \le + \delta_*,
\end{equation}
can be obtained in a similar but less involved fashion.

Setting out to establish \sref{eq:prp:afterObst:finLimitToShow},
we pick $\epsilon_1 > 0$ in such a way that both
\begin{equation}
\label{eq:prp:afterObst:bndOnEps1}
\epsilon_1 K_Z \mathcal{I}_{\mathrm{obs}} \norm{\Phi'} \le \delta_*,
\qquad \epsilon_1 K_Z \mathcal{I}_{\mathrm{obs}} \le 1.
\end{equation}
In addition, for any $\frac{1}{2} > \epsilon_2 > 0$,
we define $\xi_2 = \xi_2(\epsilon_2)$ in such a way that
$\Phi(\xi_2) = 1 - 2 \epsilon_2$.
Note that $\xi_2 \to + \infty$ as $\epsilon_2 \downarrow 0$.
We now choose $\epsilon_2 > 0$ to be sufficiently small to ensure that
\begin{equation}
e^{ - \eta_\mathcal{N} \xi_2(\epsilon_2) } 3 K_{\mathcal{N} }
  e^{\eta_\mathcal{N} [ \mathrm{diam}(\partial_\times \Lambda^\times) + 2 ] }
   e^{ -\eta_{\mathcal{N}} \frac{c}{2} t}
\le \frac{1}{2} \epsilon_1 \eta_z \kappa_{\mathrm{obs}} (1 + t)^{-3/2}
\end{equation}
holds for all $t \ge 0$, which is possible because there exists $\kappa' > 0$
such that $\kappa' (1 + t)^{-3/2} \ge \exp[ - \eta_\mathcal{N} \frac{c}{2} t ]$ for all $t \ge 0$.
Possibly decreasing $\epsilon_2$, we ensure that $0 < 2 \epsilon_2 < \epsilon_1$.
Recalling the constant $\kappa_{\mathrm{hom}}$ defined in Lemma \ref{lem:oblq:defZHom},
we write $\epsilon_3 =  \frac{1}{2}\epsilon_1 \kappa_{\mathrm{hom}}$
in view of item (iii) of Lemma \ref{lem:pmr:templateObs}.

Now, consider the time $t_{\epsilon_2}$ and set
$K^\times_{\epsilon_2} \supset K^\times_{\mathrm{obs}} $ specified
by the assumptions in the statement of this result.
If necessary, increase the size of $K^\times_{\epsilon_2}$
to ensure that $\partial_\times \Lambda^\times \subset K^\times_{\epsilon_2}$.
Without loss of generality, we will assume $t_{\epsilon_2} > 0$,
which by the comparison principle implies $0 < U(t_{\epsilon_2} ) < 1$.
Pick $\Omega_\perp = \Omega_\perp(\epsilon_2) > 0$
in such a way that for all $(n,l) \in K^\times_{\epsilon_2}$ we have $\abs{l} \le \Omega_{\perp}$.
Pick $\Omega_{\mathrm{phase}} = \Omega_{\mathrm{phase}}(\epsilon_2)$
in such a way that
\begin{equation}
 \Phi( n + c t_{\epsilon_2} - \Omega_{\mathrm{phase}} ) \le U_{nl}(t_{\epsilon_2}), \qquad (n,l) \in K^\times_{\epsilon_2} \cap \Lambda^\times,
\end{equation}
which is possible by the boundedness of $K^\times_{\epsilon_2}$.

Consider now the function $\theta: [0, \infty) \to \ell^{\infty}(\Wholes; \Real)$
defined in Proposition \ref{prp:hom:oblq:mr:sub:sup}.
Upon introducing the phase shift $\vartheta = c t_{\epsilon_2}$,
we write
\begin{equation}
t_1 = \inf  \{ t \ge 0 \hbox{ for which } n + ct + \vartheta - \theta_l(t) \ge \xi_2(\epsilon_2 )  \hbox{ for some } (n,l) \in \partial_\times \Lambda^\times \}.
\end{equation}
Since $\theta_l(t) \to 0$ as $t \to \infty$, we have $0 \le t_1 < \infty$.
By continuity and boundedness of $\partial_\times \Lambda^\times$,
there exist $(n_1,l_1) \in \partial_\times \Lambda^\times$
with
\begin{equation}
n_1 + c t_1 + \vartheta - \theta_{l_1}(t_1) = \xi_2(\epsilon_2).
\end{equation}

We now write
\begin{equation}
z(t) = \epsilon_1 z_{\mathrm{obs}; t_1}(t ), \qquad Z(t) = \epsilon_1 K_Z \int_{0}^t z_{\mathrm{obs};t_1}(t')\, d t'
\end{equation}
and consider the functions $W^-$ and $\xi^-$ defined in Proposition \ref{prp:hom:oblq:mr:sub:sup}.
By construction, we have
\begin{equation}
W^-_{nl}(0) \le U_{nl}(t_{\epsilon_2}), \qquad (n,l) \in \Lambda^\times.
\end{equation}
In addition, we have
\begin{equation}
\lim_{t \to \infty} \sup_{(n,l) \in \Lambda^\times}
\big[ W^-_{nl}(t) - \Phi\big(n + ct + \vartheta - Z(t) \big) \big]  = 0,
\end{equation}
together with
\begin{equation}
\abs{\Phi(n + ct + \vartheta ) - \Phi\big(n + ct + \vartheta - Z(t) \big)}
\le \norm{\Phi'} \abs{Z(t)} \le \epsilon_1 \norm{\Phi'} K_Z \mathcal{I}_{\mathrm{obs}} \le \delta_*
\end{equation}
on account of \sref{eq:prp:afterObst:bndOnEps1}.
In particular, the comparison
principle now implies the bound \sref{eq:prp:afterObst:finLimitToShow}
provided we can show that $W^-$ is indeed a sub-solution for
the obstructed LDE \sref{eq:ent:main:lde:nl:coords}.

In order to establish this, we note that
\begin{equation}
[\Delta^\times W^-]_{nl} = [\Delta^\times_{\Lambda^\times} W^-]_{nl}
\end{equation}
for all $(n,l) \in \Lambda^\times \setminus \partial_\times \Lambda^\times$.
We hence only have to consider $(n,l) \in \partial_\times \Lambda^\times$,
in which case we have
\begin{equation}
\begin{array}{lcl}
\abs{ [\Delta^\times W^-(t)]_{nl} - [\Delta^\times_{\Lambda^\times} W^-(t)]_{nl} }
& \le & \sum_{(n',l') \in \mathcal{N}^\times_{\Wholes^2}(n,l) \cap K^\times_{\mathrm{obs}} }
  \abs{ W^-_{n'l'}(t) - W^-_{nl}(t)  }
\\[0.2cm]
& \le & 3 K_{\mathcal{N}} e^{ - \eta_{\mathcal{N}} \abs{\xi^-_{nl}(t) } },
\end{array}
\end{equation}
on account of the fact that $\Lambda^\times$ is connected.

In addition, since $U_{nl}(t) \ge 1 - \epsilon_2$ for all $t \ge t_{\epsilon_2}$
and $(n,l) \in \partial_\times \Lambda^\times$,
it suffices to show that
\begin{equation}
3 K_{\mathcal{N}} e^{ - \eta_{\mathcal{N}} \abs{\xi^-_{nl}(t) } }
\le \frac{1}{2} \eta_z z(t)
\end{equation}
for all $(n,l) \in \partial_\times \Lambda^\times$
and $t \ge 0$
for which
\begin{equation}
W^-_{nl}(t) \ge 1 - \epsilon_2.
\end{equation}
For any such triplet $(n,l,t)$,
we compute
\begin{equation}
\Phi\big(\xi^-_{nl}(t) \big) + \epsilon_2  \ge
\Phi\big(\xi^-_{nl}(t) \big) + \epsilon_2 (1 + t)^{-1/2} \ge W^-_{nl}(t) + z(t)
\ge W^-_{nl}(t) \ge 1 -\epsilon_2,
\end{equation}
which implies $\xi^-_{nl}(t) \ge \xi_2(\epsilon_2)$.
In particular,  we have
\begin{equation}
n + ct + \vartheta - \theta_l(t) \ge Z(t) + \xi_2(\epsilon_2) \ge \xi_2(\epsilon_2),
\end{equation}
which allows us to conclude that in fact $t \ge t_1$.

In addition, for any $t \ge t_1$ and any $(n,l) \in \partial_\times \Lambda^\times$,
we have
\begin{equation}
\begin{array}{lcl}
\xi_{nl}(t) & \ge &
 \xi_{n_1,l_1}(t) - [\mathrm{diam}(\partial_\times \Lambda^\times) + 1]
\\[0.2cm]
& \ge &
\xi_{n_1, l_1}(t_1) + (t - t_1) \frac{c}{2}
   - [\mathrm{diam}(\partial_\times \Lambda^\times) + 1]
\\[0.2cm]
& = &
n_1 + c t_1 + \vartheta - \theta_{l_1}(t) - Z(t)
  - [\mathrm{diam}(\partial_\times \Lambda^\times) + 1]
      + (t - t_1) \frac{c}{2}
\\[0.2cm]
& = &
\xi_2(\epsilon_2) - Z(t) -  [\mathrm{diam}(\partial_\times \Lambda^\times) + 1]
 + (t - t_1) \frac{c}{2}
\\[0.2cm]
& \ge &
\xi_2(\epsilon_2)  -  [\mathrm{diam}(\partial_\times \Lambda^\times) + 2]
 + (t - t_1) \frac{c}{2},
\end{array}
\end{equation}
remembering that \sref{eq:prp:afterObst:bndOnEps1} implies $0 \le Z(t) \le 1$.
In particular, for any $t \ge t_1$ and $(n,l) \in \partial_\times \Lambda^\times$,
we compute
\begin{equation}
\begin{array}{lcl}
3 K_{\mathcal{N}} e^{ - \eta_{\mathcal{N}} \abs{\xi^-_{nl}(t) } }
& \le &
3 K_{\mathcal{N}} e^{ \eta_{\mathcal{N}} [\mathrm{diam}(\partial_\times \Lambda^\times) + 2] }
e^{ - \eta_{\mathcal{N}} \xi_2(\epsilon_2) } e^{ - \eta_{\mathcal{N}} \frac{c}{2} (t - t_1) }
\\[0.2cm]
& \le & \frac{1}{2} \epsilon_1 \eta_z \kappa_{\mathrm{obs}} ( 1 + t - t_1)^{-3/2}
\\[0.2cm]
& \le & \frac{1}{2} \epsilon_1 \eta_z z_{\mathrm{obs}; t_1}(t)
\\[0.2cm]
& = & \frac{1}{2} \eta_z z(t),
\end{array}
\end{equation}
as desired.
\end{proof}

\begin{proof}[Proof of Theorem \ref{thm:mr:resKConvex}]
Consider the entire solution $U_{nl}: \Real \to \ell^{\infty}(\Lambda^\times; \Real)$
constructed in Proposition \ref{prp:ent:entSol}.
We claim that $U$ satisfies the assumptions
of Proposition \ref{prp:afterObst:conv} above.

To see this, pick any $\epsilon_2 > 0$.
Propositions \ref{prp:ltim:convToStat}
and \ref{prp:stp:StSolIsOne}
together imply that
for each $(n,l) \in \Lambda^\times$
we have the limit $U_{nl}(t) \to 1$ as $t \to \infty$.
In particular, we can pick $t_{\epsilon_2} \ge 0$
in such a way that $U_{nl}(t) \ge 1 - \epsilon_2$ for all $t \ge t_{\epsilon_2}$
and $(n,l) \in \partial_\times \Lambda^\times$.
In addition, the existence of the set $K^\times_{\epsilon_2}$
with the property \sref{eq:prp:afterObst:conv:critOnCmpSet}
follows directly from an application of Proposition \ref{prp:lims:SpAsympCmpItv} with
$t_-  = t_+ = t_{\epsilon_2}$.

The temporal limit \sref{eq:mr:resKConvex:tempLimit}
now is a direct consequence of Proposition \ref{prp:afterObst:conv}.
To establish the spatial limit \sref{eq:mr:resKConvex:spLimit},
we pick any $\epsilon_2 > 0$
and note that the temporal limit \sref{eq:mr:resKConvex:tempLimit}
implies the existence of $t_- \le t_+$
for which the bound
\begin{equation}
\label{eq:prfmr:bndUVsWave}
 \abs{ U_{nl}(t) - \Phi(n + ct) } \le \epsilon_2
\end{equation}
holds for all $(n,l) \in \Lambda^\times$
and all $t \in \Real$ for which either $t \le t_-$ or $t \ge t_+$ holds.
One can then again use
Proposition \ref{prp:lims:SpAsympCmpItv} to obtain the same conclusion \sref{eq:prfmr:bndUVsWave}
for $t_- \le t \le t_+$ and $\abs{l} + \abs{n} \ge R(\epsilon_2 , t_-, t_+)$.
\end{proof}


\section{Discussion}
\label{sec:disc}

In this paper we studied
planar travelling wave solutions
to a scalar bistable reaction-diffusion system
posed on $\Wholes^2$. In particular,
we established the stability of these waves under
a class of perturbations that includes large but
localized distortions.
In addition, we proved that these planar waves persist in an appropriate sense
after removing a finite cluster of grid points.

By using the comparison principle, we were able
to construct a much larger basin of attraction
for the travelling waves than
was possible in our previous paper \cite{HJHSTB2D},
where we only used spectral methods and Green's function
techniques. On the other hand, the construction of our sub
and super-solutions required an extra order of expansion
in the Taylor series as compared to \cite{HJHSTB2D},
leading to significantly more involved computations.

As in \cite{HJHSTB2D}, we restricted ourselves
to studying planar waves that travel
in rational directions for technical convenience.
Inspection of the arguments in the present paper
however suggest that this restriction can
be removed with considerably less trouble
than would be required for \cite{HJHSTB2D}. Indeed,
the function $v$ constructed in \S\ref{sec:oblq:plt}
is well-defined even for $l \in \Real$
and the computations in \S\ref{sec:oblq:subsup}
remain valid in this setting.

Compared to the PDE results obtained in \cite{BHM},
there are three obvious differences that immediately stand out.
The first is that we have only considered (the analogue of)
directionally convex obstacles but not star-shaped obstacles.
In order to remedy this, one would need to extend
Proposition \ref{prp:stp:StSolIsOne}
to include the latter class of obstacles. The key technical difficulty is that
one would need to work with radially expanding sub-solutions,
which are considerably harder to construct in the LDE case than in the PDE case.
Nevertheless, using the techniques in \S\ref{sec:expblob} we are confident
that this can be done.

The second difference is that we have not formulated an
analogue for \cite[Thm. 1.6]{BHM},
which handles arbitrary compact obstacles. The price that needed
to be paid there is that the system no longer necessarily
converges pointwise to one, the homogeneous equilibrium state.
The arguments in \cite[{\S}8]{BHM} leading to this result are fairly technical,
but we believe that there is no fundamental reason that the ideas
cannot be carried over to the discrete setting.

The third difference is the most intriguing from our perspective
and concerns our assumption (H$\Phi$) that requires the waves to
have strictly positive speed in every direction. In the present paper
we exploit this assumption to build a mechanism by which the invading state
present behind the incoming wave
can travel in a wide berth around the obstacle to move into the region
on the other side of the obstacle. If waves cannot travel
in the horizontal and vertical directions,
one could imagine that this mechanism becomes blocked,
potentially protecting a zone in front of the obstacle from seeing the invading state.

On the other hand, our understanding of the spreading of perturbations
through the lattice is rather crude at present and our proof
technique for establishing the spatial limit \sref{eq:prp:ltim:spLimitsUInfty}
can easily be considered 
much too coarse.
At present we are conducting numerical experiments to distinguish between
these two scenarios.

\bibliographystyle{klunumHJ}
\bibliography{ref}

\end{document}